\documentclass[12pt]{article}

\usepackage{a4wide,latexsym,amsfonts,amssymb,enumerate,amscd}
\usepackage{amsmath}
\usepackage{amsthm}
\usepackage{pstricks}

\psset{linewidth=0.3pt,dimen=middle}
\psset{xunit=.70cm,yunit=0.70cm}
\psset{arrowsize=1pt 5,arrowlength=0.6,arrowinset=0.7}

\usepackage[all]{xy}
\SelectTips{cm}{}

\usepackage{graphicx}

\newcommand{\ep}{\underline{\epsilon}}
\newcommand{\onen}{{\mathbf 1}_{n}}
\newcommand{\onenn}[1]{{\mathbf 1}_{#1}}
\newcommand{\onenp}{{\mathbf 1}_{n'}}
\newcommand{\onenpp}{{\mathbf 1}_{n''}}
\newcommand{\onem}{{\mathbf 1}_{m}}
\newcommand{\E}{\cal{E}}
\newcommand{\F}{\cal{F}}

\newcommand{\tsigma}{\sigma}
\newcommand{\tomega}{\omega}
\newcommand{\tpsi}{\psi}

\newcommand{\Ucas}{\cal{C}}
\newcommand{\UupD}{\cal{E}}
\newcommand{\UdownD}{\cal{F}}

\newcommand{\ogam}{\xi^+}
\newcommand{\odelt}{\hat{\xi}^+}
\newcommand{\gam}{\xi^{-}}
\newcommand{\delt}{\hat{\xi}^{-}}
\newcommand{\gams}{\xi^{-}_{\tsigma}}
\newcommand{\delts}{\hat{\xi}^{-}_{\tsigma}}
\newcommand{\ogams}{\xi^+_{\tsigma}}
\newcommand{\odelts}{\hat{\xi}^+_{\tsigma}}
\newcommand{\gamp}{\xi'^{-}}

\newcommand{\Uq}{{\bf U}_q(\mathfrak{sl}_2)}
\newcommand{\U}{\dot{{\bf U}}}
\newcommand{\Ucat}{\cal{U}}

\newcommand{\UcatD}{\dot{\cal{U}}}

\newcommand{\B}{\dot{\mathbb{B}}}

\newcommand{\UA}{{_{\cal{A}}\dot{{\bf U}}}}

\newcommand{\und}[1]{\underline{#1}}

\newcommand{\xsum}[2]{
  \xy
  (0,.4)*{\sum};
  (0,3.7)*{\scs #2};
  (0,-2.9)*{\scs #1};
  \endxy
}

\newcommand{\refequal}[1]{\xy {\ar@{=}^{#1}
(-1,0)*{};(1,0)*{}};
\endxy}
\usepackage{fancyheadings}
\newcommand{\To}{\Rightarrow}

\newcommand{\Hom}{{\rm Hom}}
\newcommand{\HOM}{{\rm HOM}}
\renewcommand{\to}{\rightarrow}
\newcommand{\maps}{\colon}
\newcommand{\op}{{\rm op}}
\newcommand{\co}{{\rm co}}

\newcommand{\im}{{\rm im\ }}

\newcommand{\scs}{\scriptstyle}
\def\Id{\mathrm{Id}}

\def\mf{\mathfrak}
\def\shuffle{\,\raise 1pt\hbox{$\scriptscriptstyle\cup{\mskip
               -4mu}\cup$}\,}

\theoremstyle{definition}
\newtheorem{thm}{Theorem}[section]
\newtheorem{cor}[thm]{Corollary}

\newtheorem{rem}[thm]{Remark}
\newtheorem{prop}[thm]{Proposition}
\newtheorem{defn}[thm]{Definition}

\numberwithin{equation}{section}

\def\emph#1{{\sl #1\/}}

\let\hat=\widehat
\let\tilde=\widetilde

\let\phi=\varphi
\let\theta=\vartheta
\let\epsilon=\varepsilon

\usepackage{bbm}

\def\Z{{\mathbbm Z}}
\def\Q{{\mathbbm Q}}

\def\cal#1{\mathcal{#1}}%
\def\1{\mathbbm{1}}%
\def\nn{\notag}

\newcommand{\Uup}{
    \xy {\ar (0,-3)*{};(0,3)*{} };(1.5,0)*{};(-1.5,0)*{};\endxy}

\newcommand{\Udown}{
    \xy {\ar (0,3)*{};(0,-3)*{} };(1.5,0)*{};(-1.5,0)*{};\endxy}

\newcommand{\Uupdot}{
   \xy {\ar (0,-3)*{};(0,3)*{} };(0,0)*{\bullet};(1.5,0)*{};(-1.5,0)*{};\endxy}

\newcommand{\Udowndot}{
   \xy {\ar (0,3)*{};(0,-3)*{} };(0,0)*{\bullet};(1.5,0)*{};(-1.5,0)*{};\endxy}

\newcommand{\Ucupr}{\;\;
    \vcenter{\xy (-2,3)*{}; (2,3)*{} **\crv{(-2,-1) & (2,-1)}?(1)*\dir{>};
            (2,-3)*{};(-2,3)*{}; \endxy} \;\; }

\newcommand{\Ucupl}{\;\;
    \vcenter{\xy (2,3)*{}; (-2,3)*{} **\crv{(2,-1) & (-2,-1)}?(1)*\dir{>};
            (2,-3)*{};(-2,3)*{}; \endxy} \;\; }

\newcommand{\Ucapr}{\;\;
    \vcenter{\xy (-2,-3)*{}; (2,-3)*{} **\crv{(-2,1) & (2,1)}?(1)*\dir{>};
            (2,-3)*{};(-2,3)*{}; \endxy} \;\; }

\newcommand{\Ucapl}{\;\;
    \vcenter{\xy (2,-3)*{}; (-2,-3)*{} **\crv{(2,1) & (-2,1)}?(1)*\dir{>};
            (2,-3)*{};(-2,3)*{}; \endxy} \;\; }

\newcommand{\Ucross}{\;\;
    \vcenter{\xy {\ar (2.5,-2.5)*{};(-2.5,2.5)*{}}; {\ar (-2.5,-2.5)*{};(2.5,2.5)*{}};
    (4,0)*{};(-4,0)*{};\endxy} \;\; }

\newcommand{\Ucrossu}{\xy {\ar (2.5,-2.5)*{};(-2.5,2.5)*{}}; {\ar (-2.5,-2.5)*{};(2.5,2.5)*{} };
\endxy}

\newcommand{\Ucrossd}{
    \vcenter{\xy {\ar (2.5,2.5)*{};(-2.5,-2.5)*{}}; {\ar (-2.5,2.5)*{};(2.5,-2.5)*{} }; (4,0)*{};(-4,0)*{};\endxy}  }

\newcommand{\sUup}{
    \xy {\ar (0,-2)*{};(0,2)*{} };(1.5,0)*{};(-1.5,0)*{};\endxy}

\newcommand{\sUdown}{
    \xy {\ar (0,2)*{};(0,-2)*{} };(1.5,0)*{};(-1.5,0)*{};\endxy}

\newcommand{\sUupdot}{
   \xy {\ar (0,-2)*{};(0,2)*{} };(0,0)*{\scs \bullet};(1.5,0)*{};(-1.5,0)*{};\endxy}

\newcommand{\sUdowndot}{
   \xy {\ar (0,2)*{};(0,-2)*{} };(0,0)*{\scs \bullet};(1.5,0)*{};(-1.5,0)*{};\endxy}

\newcommand{\sUcupr}{\;\;
    \vcenter{\xy (-1.5,2)*{}; (1.5,2)*{} **\crv{(-2,-1) & (2,-1)}?(1)*\dir{>};
            (1.5,-2)*{};(-1.5,2)*{}; \endxy} \;\; }

\newcommand{\sUcupl}{\;\;
    \vcenter{\xy (1.5,2)*{}; (-1.5,2)*{} **\crv{(2,-1) & (-2,-1)}?(1)*\dir{>};
            (1.5,-2)*{};(-1.5,2)*{}; \endxy} \;\; }

\newcommand{\sUcapr}{\;\;
    \vcenter{\xy (-1.5,-2)*{}; (1.5,-2)*{} **\crv{(-2,1) & (2,1)}?(1)*\dir{>};
            (1.5,-2)*{};(-1.5,2)*{}; \endxy} \;\; }

\newcommand{\sUcapl}{\;\;
    \vcenter{\xy (1.5,-2)*{}; (-1.5,-2)*{} **\crv{(2,1) & (-2,1)}?(1)*\dir{>};
            (1.5,-2)*{};(-1.5,2)*{}; \endxy} \;\; }

\newcommand{\ccbub}[1]{
\xybox{%
 (-6,0)*{};
  (6,0)*{};
  (-4,0)*{}="t1";
  (4,0)*{}="t2";
  "t2";"t1" **\crv{(4,6) & (-4,6)};
   ?(1)*\dir{>};
  "t2";"t1" **\crv{(4,-6) & (-4,-6)};
   ?(.3)*\dir{}+(0,0)*{\bullet}+(0,-3)*{\scs {#1}};
}}

\newcommand{\cbub}[1]{
\xybox{%
 (-6,0)*{};
  (6,0)*{};
  (-4,0)*{}="t1";
  (4,0)*{}="t2";
  "t2";"t1" **\crv{(4,6) & (-4,6)};
    ?(.95)*\dir{<};
  "t2";"t1" **\crv{(4,-6) & (-4,-6)};
   ?(.3)*\dir{}+(0,0)*{\bullet}+(0,-3)*{\scs {#1}};
}}
\newcommand{\bbe}[1]{\xybox{%
  (-2,0)*{};
  (2,0)*{};
  (0,0);(0,-18) **\dir{-}; ?(.5)*\dir{<}+(2.3,0)*{\scriptstyle{#1}};
}}

\newcommand{\bbsid}{\xybox{%
  (-2,0)*{};
  (2,0)*{};
  (0,10);(0,4) **\dir{-};
}}
\newcommand{\bbpef}{\xybox{%
  (-6,0)*{};
  (6,0)*{};
  (-4,0)*{}="t1";
  (4,0)*{}="t2";
  "t1";"t2" **\crv{(-4,-6) & (4,-6)}; ?(.15)*\dir{>} ?(.9)*\dir{>};
}}
\newcommand{\bbpfe}{\xybox{%
  (-6,0)*{};
  (6,0)*{};
  (-4,0)*{}="t1";
  (4,0)*{}="t2";
  "t2";"t1" **\crv{(4,-6) & (-4,-6)}; ?(.15)*\dir{>} ?(.9)*\dir{>};
}}

\newcommand{\bbcfe}[1]{\xybox{%
  (-6,0)*{};
  (6,0)*{};
  (-4,0)*{}="t1";
  (4,0)*{}="t2";
  "t1";"t2" **\crv{(-4,6) & (4,6)}; ?(.15)*\dir{>} ?(.9)*\dir{>}
  ?(.5)*\dir{}+(0,2)*{\scriptstyle{#1}};
}}
\newcommand{\bbcef}[1]{\xybox{%
  (-6,0)*{};
  (6,0)*{};
  (-4,0)*{}="t1";
  (4,0)*{}="t2";
  "t2";"t1" **\crv{(4,6) & (-4,6)}; ?(.15)*\dir{>}
  ?(.9)*\dir{>} ?(.5)*\dir{}+(0,2)*{\scriptstyle{#1}};
}}

\newcommand{\lowrru}[1]{\xybox{%
  (-8,0)*{};
  (8,0)*{};
  (-6,-18)*{};(6,-9)*{} **\crv{(-6,-13) & (6,-15)} ?(1)*\dir{>};
  (6,-9)*{};(6,0)*{}  **\dir{-} ?(.3)*\dir{ }+(2,0)*{\scs {\bf j}};
}}

\newcommand{\lowllu}[1]{\xybox{%
  (-8,0)*{};
  (8,0)*{};
  (6,-18)*{};(-6,-9)*{} **\crv{(6,-13) & (-6,-15)} ?(1)*\dir{>};
  (-6,-9)*{};(-6,0)*{}  **\dir{-} ?(.3)*\dir{ }+(-2,0)*{\scs {\bf j}};
}}

\newcommand{\bbdl}[1]{\xybox{%
  (2,0);(0,-8) **\crv{(2,-2)&(0,-6)}; ?(.5)*\dir{>}
}}
\newcommand{\bbdlu}[1]{\xybox{%
  (2,0);(0,-8) **\crv{(2,-2)&(0,-6)}; ?(.5)*\dir{<}
}}
\newcommand{\bbdr}[1]{\xybox{%
  (-2,0);(0,-8) **\crv{(-2,-2)&(0,-6)}; ?(.5)*\dir{>}
}}
\newcommand{\bbdru}[1]{\xybox{%
  (-2,0);(0,-8) **\crv{(-2,-2)&(0,-6)}; ?(.5)*\dir{<}
}}

\newgray{whitegray}{.9}

\title{A categorification of the Casimir of quantum sl(2)}
      \author{
      Anna Beliakova, Mikhail Khovanov, and Aaron D.\ Lauda}

\date{January 10, 2011}

\begin{document}
%
% ==============================================================================

\maketitle

\begin{abstract}
\begin{center}
We categorify the Casimir element of the idempotented form of quantum sl(2).
\end{center}
\end{abstract}

\setcounter{tocdepth}{2} \tableofcontents

\newpage

%#################################################################
%
\section{Introduction}
%
%##################################################################

The Witten-Reshetikhin-Turaev invariant~\cite{RT,Witten} of a 3-manifold presented by surgery along a framed link is given by summing over labellings of the components of the link by irreducible representations of the quantum group $U_q(\mathfrak{sl}_2)$, then evaluating the corresponding coloured Jones polynomial at a root of unity. Crane and Frenkel~\cite{CF} conjectured that quantum 3-manifold invariants could be categorified using the categorified representation theory of $U_q(\mathfrak{sl}_2)$.  While the Jones polynomial and coloured Jones polynomial have been categorified~\cite{BW,Kh1}, the problem of categorification at a root of unity has seen little progress.

A universal invariant of knots  taking values in (an appropriate completion of) the center of the quantum $U_q(\mathfrak{sl}_2)$ was constructed by Lawrence~\cite{Law2,Law}. This universal invariant dominates all coloured Jones polynomials.  The center of $U_q(\mathfrak{sl}_2)$ as a ${\Z}[q,q^{-1}]$-algebra  is freely generated by  the Casimir element $C$. For $p\geq 0$, let us introduce
$$\sigma_p=\prod^p_{i=1} \left(C^2-(q^{2i}+2+q^{-2i})\right)$$
which are monic polynomials of degree $p$ in $C^2$. The universal invariant $J_K$ of a knot $K$ can be written \cite[Theorem 4.5]{Hab} as
$$J_K=\sum_{p\geq 0} a_p(K) \sigma_p\, .$$
The  coefficients $a_p(K)\in {\mathbb Z}[q,q^{-1}]$, $p\in \mathbb N$,
determine
the Witten-Reshetikhin-Turaev invariant of any integral homology 3--sphere
obtained by surgery on the knot $K$. Therefore it is natural to seek a categorification of this universal invariant within the context of categorified representation theory of $U_q(\mf{sl}_2)$.  Here we take a first step in this ambitious program by categorifying the quantum Casimir element for $U_q(\mf{sl}_2)$.

Igor Frenkel conjectured~\cite{Fren} the existence of a categorification of the integral idempotented version $\UA$ of the quantum enveloping algebra of $\mf{sl}_2$ at generic $q$. The algebra $\UA$ is the $\Z[q,q^{-1}]$-subalgebra of the algebra $\U$ defined by Beilinson, Lusztig and MacPherson~\cite{BLM} and generalized to arbitrary types by Lusztig~\cite{Lus4}.  In $\U$ the identity element $1 \in \mathbf{U}_q(\mathfrak{sl}_2)$ is replaced by a collection of orthogonal idempotents $1_n$ indexed by the weight lattice for $\mathfrak{sl}_2$. We recall the definition of $\mathbf{U}_q(\mathfrak{sl}_2)$ in Section~\ref{subsec_old_casimir} and $\U$ in Section~\ref{subsec_udot}.

In \cite{Lau1} the third author introduced a categorification $\UcatD$ of $\UA$ given by the idempotent completion of an additive 2-category $\Ucat$ whose objects $n \in \Z$ are parameterized by the integral weight lattice of $\mf{sl}_2$.  The 1-morphisms from $n$ to $m$ are given by direct sums of 1-morphisms $\cal{E}_{\ep}\onen\{t\} = \cal{E}_{\epsilon_1}\dots \cal{E}_{\epsilon_k}\onen\{t\}$ where $\ep=\epsilon_1\dots\epsilon_k$ with $\epsilon_1,\dots,\epsilon_k \in \{+,-\}$, $m-n=2\sum_{i=1}^k \epsilon_i 1$,  $\cal{E}_{+}=\cal{E}$, $\cal{E_{-}}=\cal{F}$ and $t\in \Z$. The 1-morphisms $\cal{E}$ and $\cal{F}$ can be thought of as categorifications of the generators $E$ and $F$ of quantum $\mf{sl}_2$. The 2-morphisms are given by $\Bbbk$-linear combinations of certain planar diagrams modulo local relations.  In \cite{Lau1} it was shown that the split Grothendieck ring of $\UcatD$ is isomorphic to $\UA$,
\begin{equation}
  K_0(\UcatD) \cong \UA,
\end{equation}
when the ground ring is a field $\Bbbk$.  In \cite{KLMS} this result was proven with $\Bbbk$ replaced by the integers.

The quantum Casimir element for $U_q(\mathfrak{sl}_2)$ has the form
\begin{equation} \label{eq_usual_casimir_INTRO}
c := EF +\frac{q^{-1}K+qK^{-1}}{(q-q^{-1})^2} = FE +\frac{qK+q^{-1}K^{-1}}{(q-q^{-1})^2}.
\end{equation}
This element is central and is preserved by various (anti)linear (anti)involutions defined on $\mathbf{U}_q(\mathfrak{sl}_2)$. In this paper we category the integral idempotent version of the Casimir element obtained from \eqref{eq_usual_casimir_INTRO} by clearing denominators by multiplying by $(q-q^{-1})^2$ and projecting via $1_n$.  We also multiply by $-1$ for convenience and obtain the integral idempotented Casimir element $\dot{C}$ for $\U$:
\begin{eqnarray}
 \dot{C} &=& \prod_{n\in\Z}C1_n, \\
  C1_n =1_nC &:=& (-q^2+2-q^{-2})EF1_n-(q^{n-1}+q^{1-n})1_n, \label{eq_casimir_INTRO}\\
            &=& (-q^2+2-q^{-2})FE1_n-(q^{n+1}+q^{-1-n})1_n. \label{eq_casimir2_INTRO}
\end{eqnarray}
This element belongs to the center, defined in Section~\ref{subsec_idempotented_rings}, of the idempotented ring $\U$.

To categorify the component $C1_n$ of the idempotented Casimir element of $\U$ in the form given in \eqref{eq_casimir_INTRO} we must lift elements $q^aEF1_n$ and $q^b1_n$ to 1-morphisms $\cal{E}\cal{F}\onen\{a\}$ and $\onen\{b\}$.  We follow the now standard procedure of lifting powers of $q$ to grading shifts and using complexes whenever minus signs are present. This requires us to work with the 2-category $Kom(\UcatD)$ of bounded complexes over the 2-category $\UcatD$ whose objects are integers $n \in \Z$, 1-morphisms are bounded complexes of 1-morphisms in $\UcatD$, and 2-morphisms are chain maps constructed from the 2-morphisms in $\UcatD$.   We also consider the 2-category $Com(\UcatD)$ which has the same objects and 1-morphisms as $Kom(\Ucat)$, and whose 2-morphisms are chain maps up to homotopy.

We are looking for a complex with four copies of $\cal{E}\cal{F}\onen$ and two copies of $\onen$ with grading shifts:
\begin{equation}
  \text{$\cal{E}\cal{F}\onen\{-2\}$, \;\;
  $\cal{E}\cal{F}\onen$, \;\;
  $\cal{E}\cal{F}\onen$, \;\;
  $\cal{E}\cal{F}\onen\{2\}$, \;\;
  $\onen\{1-n\}$, \;\; $\onen\{n-1\}$}.  \nn
\end{equation}
The minus signs in \eqref{eq_casimir_INTRO} indicate that the terms $\cal{E}\cal{F}\onen\{-2\}$,  $\cal{E}\cal{F}\onen\{2\}$, $\onen\{1-n\}$, $\onen\{n-1\}$ should live in odd homological degrees, and the remaining two copies of $\cal{E}\cal{F}\onen$ in even degrees.

The positioning of these terms in the complex is naturally dictated
by the $q$-degrees of the possible maps between them. Negative degree
endomorphisms of $\cal{E}\cal{F}\onen$ exist only for $n>1$, while
there are obvious degree two endomorphisms given by placing a dot on
one of two vertical lines in the diagram of the identity map:
$\text{$\Uupdot\Udown$}$ and $\text{$\Uup\Udowndot$}$ (see Section~\ref{subsec_def_ucat} for a review of the 2-category $\Ucat$). We
can arrange the above four copies of $\cal{E}\cal{F}\onen$ with the
appropriate shifts and cohomological degrees into a complex just
using these maps
\[
\xy
 (-30,0)*+{\cal{E}\cal{F}\onen\{2\}}="1";
 (0,15)*+{\cal{E}\cal{F}\onen}="2";
 (0,-15)*+{\cal{E}\cal{F}\onen}="3";
 (30,0)*+{\cal{E}\cal{F}\onen\{-2\}}="4";
  {\ar "1";"2"};
  {\ar "1";"3"};
  {\ar "2";"4"};
  {\ar "3";"4"};
  (-17,12)*{\text{$\Uupdot\Udown$}};
  (17,12)*{-\text{$\Uup\Udowndot$}};
  (-17,-12)*{\text{$\Uup\Udowndot$}};
  (17,-12)*{\text{$\Uupdot\Udown$}};
\endxy
\]
where the exact position of the minus sign is unimportant. To find the room for the two shifted copies of $\onen$ we observe that clockwise cup and cap 2-morphisms have degree $1-n$, perfectly matching the difference in degrees of these $\onen$ and those of the middle $\cal{E}\cal{F}\onen$ in the complex, leading to a commutative square
 \[
\xy
 (-30,0)*+{\onen\{1-n\}}="1";
 (0,15)*+{\cal{E}\cal{F}\onen}="2";
 (0,-15)*+{\cal{E}\cal{F}\onen}="3";
 (30,0)*+{\onen\{n-1\}}="4";
  {\ar^{\vcenter{\xy (2,3)*{}; (-2,3)*{} **\crv{(2,-1) & (-2,-1)}?(1)*\dir{>};
             \endxy} } "1";"2"};
  {\ar_{\text{$\Ucupl$}} "1";"3"};
  {\ar^{-\text{$\Ucapr$}} "2";"4"};
  {\ar_{\vcenter{\xy (-2,-3)*{}; (2,-3)*{} **\crv{(-2,1) & (2,1)}?(1)*\dir{>};
            \endxy}} "3";"4"};
\endxy
\]
These two commutative squares can be glued into a single complex $\cal{C}\onen$ centered in homological degree zero:
\begin{equation} \label{eq_casimir_EF_INTRO}
 \cal{C}\onen :=\xy
  (-55,15)*+{\E\F \onen \{2\}}="1";
  (-55,-15)*+{\onen \{1-n\}}="2";
  (0,15)*+{\und{\E\F \onen }}="3";
  (0,-15)*+{\und{\E\F\onen}}="4";
  (55,-15)*+{\onen \{n-1\}}="5";
  (55,15)*+{\E\F \onen \{-2\}}="6";
   {\ar^{\Uupdot\Udown} "1";"3"};
   {\ar^{} "1";"4"};
   %% Had problems with the label placing it by hand
   "1"+(18,-6)*{\Uup\Udowndot};
   {\ar_(.35){\Ucupl} "2";"3"}; %% The <<< adjusts the where the label on
   {\ar_-{\Ucupl} "2";"4"};     %% the arrow goes
   {\ar^-<<<<<<<{\Ucapr} "3";"5"};
   "3"+(20,4)*{-\;\Uup\Udowndot};
   {\ar "3";"6"};
   {\ar_{-\;\;\Ucapr} "4";"5"};
   {\ar_{} "4";"6"};
   "4"+(24,8)*{\Uupdot\Udown};
   (-57,0)*{\bigoplus};(0,0)*{\bigoplus};(57,0)*{\bigoplus};
 \endxy
\end{equation}
We call the above complex the Casimir complex.  The image of this
complex in the Grothendieck ring of $Com(\UcatD)$ is $C1_n$ expressed in the form given by the right hand side of equation \eqref{eq_casimir_INTRO}.

Starting with the form of the idempotented Casimir element given in \eqref{eq_casimir2_INTRO} we obtain a different complex:
\begin{equation} \label{eq_casimir_FE_INTRO}
\cal{C}'\onen :=
\xy
  (-55,15)*+{\F\E \onen \{2\}}="1";
  (-55,-15)*+{\onen \{1+n\}}="2";
  (0,15)*+{\und{\F\E \onen }}="3";
  (0,-15)*+{\und{\F\E\onen}}="4";
  (55,-15)*+{\onen \{-n-1\}}="5";
  (55,15)*+{\F\E \onen \{-2\}}="6";
   {\ar^{\Udown\Uupdot} "1";"3"};
   {\ar^{} "1";"4"};
   %% Had problems with the label placing it by hand
   "1"+(18,-6)*{\Udowndot\Uup};
   {\ar_(.35){\Ucupr} "2";"3"}; %% The <<< adjusts the where the label on
   {\ar_-{\Ucupr} "2";"4"};     %% the arrow goes
   {\ar^-<<<<<<<{\Ucapl} "3";"5"};
   "3"+(20,4)*{-\;\Udowndot\Uup};
   {\ar "3";"6"};
   {\ar_{-\;\;\Ucapl} "4";"5"};
   {\ar_{} "4";"6"};
   "4"+(24,8)*{\Udown\Uupdot};
   (-57,0)*{\bigoplus};(0,0)*{\bigoplus};(57,0)*{\bigoplus};
 \endxy
\end{equation}
However, we will show that the complex $\cal{C}'\onen$ is homotopy equivalent to $\cal{C}\onen$.  These two complexes behave well under certain symmetry 2-functors $\tpsi$, $\tomega$, $\tsigma$ defined for the 2-category $\Ucat$ in \cite{Lau1}, and extended here in Section~\ref{sec_sym} to the 2-categories $Kom(\UcatD)$ and $Com(\UcatD)$. In particular, $\cal{C}'\onen=\tsigma(\cal{C}\onenn{-n})$. These symmetry 2-functors categorify certain (anti)linear (anti)involutions on the algebras $\U$ with the various (anti)linearity and (anti)involution properties being reflected in the (contravariant)covariant behaviour of the 2-functors.  Just as one can go between the two forms of the Casimir in \eqref{eq_casimir_INTRO} and \eqref{eq_casimir2_INTRO} using these (anti)involutions on the algebras $\U$, we relate the complexes above together with their alternative versions obtained by moving the minus signs and reordering the dot 2-morphisms via these categorified symmetries of $\U$.

Our results can be summarized as follows:
\begin{thm} \label{thm_main}\hfill
\begin{enumerate}[a)]
 \item
There are canonical mutually-inverse isomorphisms
\begin{equation}
      \varrho^{\tsigma} \maps \cal{C}\onen \to \cal{C}'\onen,
      \qquad
      \hat{\varrho}^{\tsigma} \maps \cal{C}'\onen \to \cal{C}\onen
    \end{equation}
     in $Com(\UcatD)$. If $n \leq 0$ the complex $\cal{C}\onen$ is indecomposable in $Kom(\UcatD)$, and complex $\cal{C}'\onen$ is isomorphic to a direct sum of $\cal{C}\onen$ and a contractible complex.
 If $n \geq0$ the complex $\cal{C}'\onen$  is indecomposable in $Kom(\UcatD)$, and complex $\cal{C}\onen$ is isomorphic to a direct sum of $\cal{C}'\onen$ and a contractible complex.
 If $n=0$ complexes $\cal{C}\onen$ and $\cal{C}'\onen$ are isomorphic in $Kom(\UcatD)$.

 \item Under the isomorphism $K_0(\UcatD) \cong K_0(Com(\UcatD)) \cong \UA$
we have $[\cal{C}\onen]=C1_n$,  so that the complex $\cal{C}\onen$ in $Com(\UcatD)$ descends to the component $C1_n$ of the Casimir element $\dot{C}$ of $\U$ after passing to the Grothendieck ring.

  \item  The complex $\cal{C}\onen$ is invariant under the symmetries $\tpsi$ and $\tomega\tsigma$ of $Kom(\Ucat)$.  Symmetry $\tsigma$ takes $\cal{C}\onen=\eqref{eq_casimir_EF_INTRO}$ to the complex $\cal{C}'\onenn{-n}$ given by \eqref{eq_casimir_FE_INTRO} for $-n$.

  \item \emph{Commutativity}: There exists a collection of invertible 2-morphisms of complexes
\begin{equation}
  \kappa_{X} \maps X \cal{C} \to \cal{C} X,
\end{equation}
with inverses
\begin{equation}
  \hat{\kappa}_{X} \maps \cal{C} X \to X \cal{C},
\end{equation}
for all $X$ in $Com(\UcatD)$.

\emph{Naturality:} The collection of invertible chain maps $\kappa_X$ is natural in the sense that for any 2-morphism $f \maps X \to Y$ the squares
\begin{equation}
    \xy
   (-10,-10)*+{X\cal{C}}="tl";
   (10,-10)*+{\cal{C}X}="tr";
   (-10,10)*+{Y\cal{C}}="bl";
   (10,10)*+{\cal{C}Y}="br";
   {\ar_-{\kappa_{X}} "tl";"tr"};
   {\ar^{f\Ucas} "tl";"bl"};
   {\ar^{\kappa_{Y}} "bl";"br"};
   {\ar_{\Ucas f} "tr";"br"};
  \endxy
  \qquad
      \xy
   (-10,-10)*+{\cal{C}X}="tl";
   (10,-10)*+{X\cal{C}}="tr";
   (-10,10)*+{\cal{C}Y}="bl";
   (10,10)*+{Y\cal{C}}="br";
   {\ar_-{\hat{\kappa}_{X}} "tl";"tr"};
   {\ar^{\Ucas f} "tl";"bl"};
   {\ar^{\hat{\kappa}_{Y}} "bl";"br"};
   {\ar_{f\Ucas} "tr";"br"};
  \endxy
\end{equation}
commute in $Com(\UcatD)$.  By construction these invertible chain maps are compatible with composition in $Com(\UcatD)$ given by the tensor product of complexes and juxtaposition of diagrams. That is, for complexes $Y=\onenp Y\onen$ and $X=\onenpp X \onenp$, with $\cal{C}XY\onen = \onenpp\cal{C}\onenpp X\onenp Y\onen$,
we have a commutative diagram
\begin{equation}
  \xy
   (-20,15)*+{\cal{C}XY\onen}="l";
   (20,15)*+{XY\cal{C}\onen}="r";
   (0,0)*+{X\cal{C}Y\onen}="b";
   {\ar^-{\kappa_{XY}} "l";"r"};
   {\ar_{\kappa_X Y} "l";"b"};
   {\ar_{X\kappa_Y} "b";"r"};
  \endxy
\end{equation}
in $Com(\UcatD)$.
\end{enumerate}
\end{thm}

Parts a) and d) are difficult, while parts b) and c) are obvious. The indecomposability of the Casimir and the resulting simplifications are discussed in section~\ref{sec_indec}. The rest of part a) and part c) of the Theorem can be found in Proposition~\ref{prop_homotopy_sym}. The construction of the commutativity chain isomorphisms is given in Section~\ref{subsec_commutativity}, while the naturality of these maps is proven in Section~\ref{sec_nat}.

As explained in \cite[Section 3.7]{KL3} an additive 2-category can be viewed as an idempotented monoidal category by regarding  1-morphisms as objects of the monoidal category.  The 2-morphisms in the 2-category become 1-morphisms in the monoidal category.  The composition operation for 1-morphisms and the horizontal composition for 2-morphisms in the original 2-category gives rise to the monoidal structure, allowing objects and morphisms to be tensored together. It is sometimes convenient to view $Kom(\UcatD)$ and $Com(\UcatD)$ as idempotented additive monoidal categories in this way.

The commutativity and naturality statements in the third property above imply that the complex $\cal{C}\onen$ is in the (Drinfeld) center of the additive monoidal category $Com(\UcatD)$ \cite{JS,Ma,KV}.  The collection of chain maps $\kappa_X$ define an invertible natural transformation of functors
$\kappa\maps - \otimes \cal{C} \To \cal{C} \otimes -$, where $- \otimes \cal{C}$ and $\cal{C}\otimes -$ are the endofunctors of $Com(\UcatD)$ given by tensoring on the right, respectively left, with the complex $\cal{C}\onen$ for appropriate $n$.

\medskip

The categorification of the Casimir element for quantum $\mathfrak{sl}_2$ presented here demonstrates the increase in combinatorial complexity that arises when lifting structures to the categorical level: the Casimir complex only commutes with other complexes up to chain homotopies, which are rather involved.   By appealing extensively to the graphical calculus for categorified $\U$ and its symmetries, we are able to study the Casimir complex and construct explicit chain maps giving commutativity of the Casimir up to chain homotopy.  This paper presents new identities that are used for simplifying 2-morphisms in categorified $\mathfrak{sl}_2$.  We hope that the calculations in this paper will serve to further illustrate how complex computations can be performed in the graphical calculus for $\Ucat$.

\smallskip
%%%%%%%%%%%%%%%%%%%%%%%%%%%%%%%%%%%%%%%%%%%%%%%%%%%%%%%%%%%%%%%%%%
\noindent {\it Acknowledgments:}
A.B. would like to acknowledge the Swiss National Science Foundation for support via grant PP002-119088.  M.K. is grateful to the NSF for partially supporting him via grants DMS-0706924 and DMS-0739392.  A.L. was partially supported by the NSF grants  DMS-0739392 and DMS-0855713 and would like to thank the MSRI for support in Spring 2010 when this work was almost completed.

%%%%%%%%%%%%%%%%%%%%%%%%%%%%%%%%%%%%%%%%%%%%%%%%%%%%%%%%%%%%%%%

%#################################################################
%
\section{Casimir element and idempotented form of quantum $\mathfrak{sl}_2$ }
%
%##################################################################

% ====================================================================
%
\subsection{Quantum $\mathfrak{sl}_2$ and the Casimir element} \label{subsec_old_casimir}
%
% ====================================================================

The quantum group $\Uq$ is the associative algebra (with unit) over
$\Q(q)$ with generators $E$, $F$, $K$, $K^{-1}$ and relations
\begin{eqnarray}
  KK^{-1}=&1&=K^{-1}K, \label{eq_UqI}\\
  KE &=& q^2EK, \\
  KF&=&q^{-2}FK, \\
  EF-FE&=&\frac{K-K^{-1}}{q-q^{-1}}. \label{eq_UqIV}
\end{eqnarray}
For simplicity the algebra $\Uq$ is written ${\bf U}$. For more details on quantum groups see \cite{Kassel}.

For $a\geq 0$ we put $[a]=\frac{q^a-q^{-a}}{q-q^{-1}}$,
$[a]!=[a][a-1]\dots [1]$ and
$ E^{(a)}=\frac{E^a}{[a]!}$, $ F^{(a)}=\frac{F^a}{[a]!}$.
We further define the integral form
${}_{\cal{A}}{\bf U}$ to be the $\Z[q,q^{-1}]$-subalgebra of $\bf U$
generated by
\begin{equation}
  \{ E^{(a)}, F^{(a)}, K^{\pm 1}|\; a \in \Z_{+}
\}.
\end{equation}

There are several $\Z[q,q^{-1}]$-(anti)linear (anti)automorphisms that will be used in this paper. Let $\bar{\;}$ be the $\Q$-linear involution of $\Q(q)$ which maps $q$ to $q^{-1}$.
\begin{itemize}
  \item The $\Q(q)$-antilinear algebra involution $\und{\psi} \maps {\bf U} \to {\bf U}$ is given by
\[
 \und{\psi}(E)=E, \quad \und{\psi}(F)=F, \quad \und{\psi}(K) = K^{-1}, \quad
 \und{\psi}(fx)=\bar{f}\und{\psi}(x) \quad \text{for $f \in \Q(q)$ and $x \in {\bf U}$}.\]
 \item The $\Q(q)$-linear algebra involution $\und{\omega}\maps {\bf U} \to {\bf U}$ is given by
\begin{eqnarray*}
\und{\omega}(E)=F, && \und{\omega}(F)=E, \qquad \und{\omega}(K) = K^{-1}, \\
   \und{\omega}(fx)=f\und{\omega}(x), && {\rm for} \; f\in \Q(q) \; {\rm and} \; x \in {\bf U},\\
  \und{\omega}(xy)=\und{\omega}(x)\und{\omega}(y), && {\rm for} \; x,y \in {\bf U}.
\end{eqnarray*}
 \item  The $\Q(q)$-linear algebra antiinvolution $\und{\sigma} \maps {\bf U} \to {\bf U}$ is given by
\begin{eqnarray*}
\und{\sigma}(E)=E, && \und{\sigma}(F)=F, \qquad \und{\sigma}(K) = K^{-1}, \\
   \und{\sigma}(fx)=f\und{\sigma}(x), && {\rm for} \; f\in \Q(q) \; {\rm and} \; x \in {\bf U},\\
  \und{\sigma}(xy)=\und{\sigma}(y)\und{\sigma}(x), && {\rm for} \; x,y \in {\bf U}.
\end{eqnarray*}
\end{itemize}

The (anti)linear (anti)involutions pairwise commute and generate the group $G=(\Z_2)^3$ of (anti)linear (anti)automorphisms acting on ${\bf U}$. Throughout the paper we will also use the index two subgroup $G_1 =\{1, \und{\psi},\und{\omega}\und{\sigma}, \und{\psi}\und{\omega}\und{\sigma}\}$ of $G$ and the coset $G\setminus G_1=\{\und{\omega}, \und{\sigma}, \und{\psi}\und{\omega}, \und{\psi}\und{\sigma} \}$.

The Casimir element for ${\bf U}$ is given by
\begin{equation} \label{eq_usual_casimir}
 c:= EF +\frac{q^{-1}K+qK^{-1}}{(q-q^{-1})^2} = FE +\frac{qK+q^{-1}K^{-1}}{(q-q^{-1})^2}.
\end{equation}
It is easy to verify that $Ec=cE$, $Fc=cF$ and $Kc=cK$. Moreover, $c$ generates the center of ${\bf U}$, and
\begin{equation}
  Z({\bf U}) = \Q(q)[c].
\end{equation}
We will be most interested in the element
\begin{eqnarray}
  C:= -(q-q^{-1})^2c &=& (-q^2+2-q^{-2})EF-q^{-1}K-qK^{-1}, \label{eq_nonidemp_casimir1}\\
   &=& (-q^2+2-q^{-2})FE-q K-q^{-1}K^{-1}. \label{eq_nonidemp_casimir2}
\end{eqnarray}
Of course,
\begin{equation}
  Z({\bf U}) = \Q(q)[C].
\end{equation}
The element $C$ belongs to the integral form ${}_{\cal{A}}{\bf U}$ of {\bf U}, and  we call $C$ the Casimir element.

The symmetries in $G$ preserve the Casimir element:
\begin{equation}
  \und{\psi}(C) = \und{\omega}(C)=\und{\sigma}(C)= C.
\end{equation}
Notice that the involutions $\und{\psi}$, $\und{\omega}\und{\sigma}$,  $\und{\psi}\und{\omega}\und{\sigma}$ in $G_1$ preserve the form of the Casimir in \eqref{eq_nonidemp_casimir1} and \eqref{eq_nonidemp_casimir2}, while the involutions $\und{\omega}$, $\und{\sigma}$, $\und{\psi}\und{\omega}$, $\und{\psi}\und{\sigma}$ in $G\setminus G_1$ map one form of the Casimir element in \eqref{eq_nonidemp_casimir1} and \eqref{eq_nonidemp_casimir2} to the other.

All of these symmetries preserve the integral form ${}_{\cal{A}}{\bf U}$ of $\mathbf{U}$.

% ====================================================================
%
\subsection{Idempotented rings and their centers} \label{subsec_idempotented_rings}
%
% ====================================================================
An idempotented ring $A$ is a not necessarily unital associative ring equipped with a family of mutually-orthogonal idempotents $1_i$, indexed by elements $i$ of a set $I$, such that
\begin{equation}
  A = \bigoplus_{i,j \in I} 1_i A 1_j.
\end{equation}
The center $Z(A)$ of $A$ is a subspace of $\prod_{i \in I}1_iA1_i$ consisting of elements $\prod_{i \in I}z_i$ such that
\begin{equation}
  z_i x = x z_j
\end{equation}
for any $i,j \in I$ and $x \in 1_iA1_j$.  $Z(A)$ is a commutative ring isomorphic to the center of the category of idempotented $A$-modules, i.e. $A$-modules $M$ such that
\begin{equation}
  M = \bigoplus_{i \in I} 1_iM.
\end{equation}
An idempotented ring $A$ has a unit element if and only if the set $I$ is finite, in which case
\begin{equation}
  1 = \sum_{i \in I} 1_i.
\end{equation}
For unital $A$, the center of $A$ defined as above coincides with the usual center of $A$.

% ====================================================================
%
\subsection{BLM $\U$} \label{subsec_udot}
%
% ====================================================================

The Belinson-Lusztig-MacPherson~\cite{BLM} algebra $\U$ is the $\Q(q)$-algebra
obtained by modifying ${\bf U}$ by replacing the unit element with a collection of orthogonal
idempotents $1_n$ for $n \in \Z$,
\begin{equation} \label{eq_orthog_idemp}
  1_n1_m=\delta_{n,m}1_n,
\end{equation}
indexed by the weight lattice of $\mathfrak{sl}_2$, such that
\begin{equation}
K1_n =1_nK= q^n 1_n, \quad E1_n = 1_{n+2}E=1_{n+2}E1_n, \quad F1_n = 1_{n-2}F=1_{n-2}F1_n.
\end{equation}
Similarly, the $\Z[q,q^{-1}]$-subalgebra $\UA$ of $\U$ is obtained
from ${}_{\cal{A}}{\bf U}$ by replacing the unit element by a collection of orthogonal
idempotents \eqref{eq_orthog_idemp}, such that
\begin{eqnarray} \label{eq_onesubn}
K1_n &=& 1_nK = 1_nK1_n = q^n 1_n, \\
  E^{(a)}1_n &=& 1_{n+2a}E^{(a)} = 1_{n+2a}E^{(a)}1_n, \nn\\
F^{(a)}1_n &=& 1_{n-2a}F^{(a)}=1_{n-2a}F^{(a)}. \nn
\end{eqnarray}
The diagram below illustrates the various algebras considered so far
\begin{equation}
  \xymatrix{
  {}_{\cal{A}}{\bf U} \ar@{~>}[d] \ar@{^{(}->}[r]& {\bf U} \ar@{~>}[d]\\
  \UA \ar@{^{(}->}[r]  & \U,
  }
\end{equation}
where the rightward pointing arrows are the inclusions of subalgebras, and the squiggly arrows denote
passing to the idempotent form of the algebra. See~\cite{Lus4} and the references therein for more details on the algebra $\U$.

There are direct sum decompositions
\[
 \U = \bigoplus_{n,m \in \Z}1_m\U1_n \qquad \qquad \UA = \bigoplus_{n,m \in
 \Z}1_m(\UA)1_n
\]
with $1_m(\UA)1_n$ the $\Z[q,q^{-1}]$-submodule spanned by
$1_mE^{(a)}F^{(b)}1_n$ and $1_mF^{(b)}E^{(a)}1_n$ for $a,b \in \Z_+$ (these elements are zero unless $m=n+2a-2b$).

The algebra $\U$ does not have the unit since the infinite sum
$\sum_{n\in \Z}1_n$ does not belong to $\U$; instead, the system
of idempotents $\{1_n | n \in \Z \}$ serves as a
substitute for $1$.  Lusztig's basis $\B$ of $\U$ consists of the following elements of $\U$:
\begin{enumerate}[(i)]
     \item $E^{(a)}F^{(b)}1_{n} \quad $ for $a$,$b\in \Z_+$,
     $n\in\Z$, $n\leq b-a$,
     \item $F^{(b)}E^{(a)}1_{n} \quad$ for $a$,$b\in\Z_+$, $n\in\Z$,
     $n\geq
     b-a$,
\end{enumerate}
where $E^{(a)}F^{(b)}1_{b-a}=F^{(b)}E^{(a)}1_{b-a}$.

The (anti)involutions $\und{\psi}$, $\und{\omega}$, and  $\und{\sigma}$ all naturally
extend to  $\U$ if we set
\begin{equation}
    \und{\psi}(1_{n}) = 1_{n}, \quad \und{\omega}(1_{n}) = 1_{-n},
    \quad \und{\sigma}(1_{n})=1_{-n}.\nn
\end{equation}
Taking direct sums of the induced maps on each summand $1_m\U1_n$ allows these
maps to be extended to $\U$ and $\UA$.  These
$\Z[q,q^{-1}]$-(anti)linear (anti)algebra homomorphisms are recorded below on
some elements of $\UA$:
\begin{eqnarray}
 \und{\omega} &\maps&
   q^s1_m E^{(a)} F^{(b)}1_n  \mapsto  q^s 1_{-m} F^{(a)} E^{(b)}1_{-n} ,
  \label{eq_def_omega}\\
 \nn\\
\und{\sigma}&\maps&
   q^s 1_{n} E^{(a)} F^{(b)}1_n  \mapsto  q^s 1_{-n} F^{(b)}E^{(a)}1_{-m},
  \label{eq_def_sigma}\\
 \nn\\
 \und{\psi} &\maps&
    q^s  1_m E^{(a)} F^{(b)}1_n
    \mapsto
 q^{-s}  1_m E^{(a)} F^{(b)}1_n,
 \label{eq_def_psi}
\end{eqnarray}
where $m=n+2a-2b$. The group $G$ acts on both $\U$ and $\UA$.

There is a natural homomorphism from the center of $\mathbf{U}$ to the center of its idempotented form $\U$ that sends $x \in Z(\mathbf{U})$ to $\prod_{n \in \Z}1_nx1_n$.  It is not hard to check that this homomorphism is, in fact, an isomorphism.  Denote by
\begin{equation}
  \dot{C} = \prod_{n \in \Z} 1_nC
\end{equation}
the image of $C$ in $\U$ under this homomorphism.  Let $\dot{C}_{ev}=\prod_{n \in 2\Z} 1_nC$ and $\dot{C}_{od}=\prod_{n \in 2\Z+1} 1_nC$.
Then $\dot{C}=\dot{C}_{ev}+\dot{C}_{ev}$ and
\begin{equation}
  Z(\mathbf{\U}) \cong \Q(q)[\dot{C}_{ev}] \times \Q(q)[\dot{C}_{od}].
\end{equation}
%the ring of polynomials in $\dot{C}$ with coefficients in $\Q(q)$.

We call $C1_n$ components or terms of the Casimir element $\dot{C}$. They are given by
\begin{eqnarray}
  1_nC =C1_n &=& (-q^2+2-q^{-2})EF1_n-(q^{n-1}+q^{1-n})1_n, \label{eq_casimir}\\
            &=& (-q^2+2-q^{-2})FE1_n-(q^{n+1}+q^{-1-n})1_n. \label{eq_casimir2}
\end{eqnarray}
Components are preserved by the symmetries in $G_1=\{1,\und{\psi},\und{\omega}\und{\sigma},\und{\psi}\und{\sigma}\und{\omega}\}$:
\begin{equation}
 g(C1_n) = C1_n
\end{equation}
for any $g \in G_1$.  They are interchanged by the elements in the coset $G\setminus G_1=\{\und{\sigma},\und{\omega}, \und{\psi}\und{\sigma}, \und{\psi}\und{\omega}\}$:
\begin{equation}
 g'(C1_n) = C1_{-n}
\end{equation}
for any $g' \in G\setminus G_1$. In addition to sending $1_n$ to $1_{-n}$, the involutions in $G \setminus G_1$ map one form of the Casimir in \eqref{eq_casimir} and \eqref{eq_casimir2} to the other form.

%##############################################################################
%
\section{Brief review of sl(2)-calculus}
%
%##############################################################################

% ======================================================================
%
\subsection{The 2-category $\Ucat$} \label{subsec_def_ucat}
%
% ======================================================================

% --------------------------------------------------------------------
%
\subsubsection{Definition of $\Ucat$}
%
% --------------------------------------------------------------------
Fix a field $\Bbbk$. Here we recall the definition of the $\Bbbk$-linear 2-category $\Ucat=\Ucat(\mf{sl}_2)$ introduced in \cite{Lau1}, see also \cite{Lau2,KL3}.

\begin{defn} \label{def_Ucat}
The 2-category $\Ucat$ is the additive 2-category consisting of
\begin{itemize}
  \item objects: $n$ for $n \in \Z$.
\end{itemize}
The hom $\Ucat(n,n')$ between two objects $n$, $n'$ is an additive $\Bbbk$-linear category:
\begin{itemize}
  \item objects of $\Ucat(n,n')$: for a signed sequence $\ep = (\epsilon_1,\epsilon_2, \dots, \epsilon_m)$, where $\epsilon_1, \dots, \epsilon_m \in \{ +,-\}$, define
 $$\cal{E}_{\ep} := \cal{E}_{\epsilon_1} \cal{E}_{\epsilon_2}\dots \cal{E}_{\epsilon_m}$$
where $\cal{E}_{+}:= \cal{E}$ and $\cal{E}_{-}:= \cal{F}$.  An object of $\Ucat(n,n')$, called a 1-morphism
in $\Ucat$, is a formal finite direct sum of 1-morphisms
  \[
 \cal{E}_{\ep} \onen\{t\} =\onenp \cal{E}_{\ep} \onen\{t\}
  \]
for any $t\in \Z$ and signed sequence $\ep$ such that $n'=n+\sum_{j=1}^m \epsilon_j2 $.

  \item morphisms of $\Ucat(n,n')$: given objects $\cal{E}_{\ep} \onen\{t\}
  ,\cal{E}_{\ep'} \onen\{t'\} \in \Ucat(n,n')$, the hom
sets $\Ucat(\cal{E}_{\ep} \onen\{t\},\cal{E}_{\ep'} \onen\{t'\})$ of $\Ucat(n,n')$
are $\Bbbk$-vector spaces given by linear combinations of diagrams with degree $t-t'$, modulo certain relations, built from composites of:
\begin{enumerate}[i)]
  \item  Degree zero identity 2-morphisms $1_x$ for each 1-morphism $x$ in
$\Ucat$; the identity 2-morphisms $1_{\cal{E} \onen}\{t\}$ and
$1_{\cal{F} \onen}\{t\}$  are represented graphically by
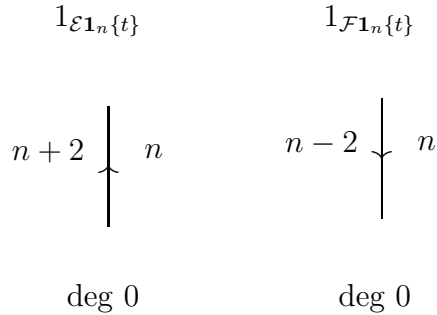
\[
\begin{array}{ccc}
  1_{\cal{E} \onen\{t\}} &\quad  & 1_{\cal{F} \onen\{t\}} \\ \\
    \xy
 (0,8);(0,-8); **\dir{-} ?(.5)*\dir{>}+(2.3,0)*{\scriptstyle{}};
 (0,11)*{};
 (6,2)*{ n};
 (-8,2)*{ n +2};
 (-10,0)*{};(10,0)*{};
 \endxy
 & &
 \;\;   \xy
 (0,8);(0,-8); **\dir{-} ?(.5)*\dir{<}+(2.3,0)*{\scriptstyle{}};
 (0,-11)*{};(0,11)*{};
 (6,2)*{ n};
 (-8,2)*{ n -2};
 (-12,0)*{};(12,0)*{};
 \endxy
\\ \\
   \;\;\text{ {\rm deg} 0}\;\;
 & &\;\;\text{ {\rm deg} 0}\;\;
\end{array}
\]
and more generally, for a signed sequence $\ep = (\epsilon_1,\epsilon_2, \dots, \epsilon_m)$, the identity $1_{\cal{E}_{\ep} \onen\{t\}}$ 2-morphism is
represented as
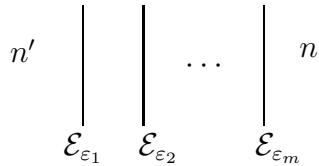
\begin{equation*}
\begin{array}{ccc}
  \xy
 (-12,8);(-12,-8); **\dir{-};
 (-4,8);(-4,-8); **\dir{-};
 (4,0)*{\cdots};
 (12,8);(12,-8); **\dir{-};
  (-12,-11)*{\cal{E}_{\epsilon_1}}; (-2,-11)*{\cal{E}_{\epsilon_2}};(14,-11)*{\cal{E}_{\epsilon_m}};
 (18,2)*{ n}; (-20,2)*{ n'};
 \endxy
\end{array}
\end{equation*}
where the strand labelled $\cal{E}_{\epsilon_{\alpha}}$ is oriented up if $\epsilon_{\alpha}=+$
and oriented down if $\epsilon_{\alpha}=-$. We will often omit labels from the strands since the labels can be deduced from the orientation of a strand.

  \item For each $n \in X$ the 2-morphisms
\[
\begin{tabular}{|l|c|c|c|c|}
 \hline
 {\bf 2-morphism:} &   \xy
 (0,7);(0,-7); **\dir{-} ?(.75)*\dir{>};
 (0,-2)*{\txt\large{$\bullet$}};
 (6,4)*{n}; (-8,4)*{n +2}; (-10,0)*{};(10,0)*{};
 \endxy
 &
     \xy
 (0,7);(0,-7); **\dir{-} ?(.75)*\dir{<};
 (0,-2)*{\txt\large{$\bullet$}};
 (-6,4)*{n}; (8,4)*{n+2}; (-10,0)*{};(10,9)*{};
 \endxy
 &
   \xy
  (0,0)*{\xybox{
    (-4,-4)*{};(4,4)*{} **\crv{(-4,-1) & (4,1)}?(1)*\dir{>} ;
    (4,-4)*{};(-4,4)*{} **\crv{(4,-1) & (-4,1)}?(1)*\dir{>};
     (8,1)*{n};     (-12,0)*{};(12,0)*{};     }};
  \endxy
 &
   \xy
  (0,0)*{\xybox{
    (-4,4)*{};(4,-4)*{} **\crv{(-4,1) & (4,-1)}?(1)*\dir{>} ;
    (4,4)*{};(-4,-4)*{} **\crv{(4,1) & (-4,-1)}?(1)*\dir{>};
     (8,1)*{ n};     (-12,0)*{};(12,0)*{};     }};
  \endxy
\\ & & & &\\
\hline
 {\bf Degree:} & \;\;\text{  2 }\;\;
 &\;\;\text{  2}\;\;& \;\;\text{ -2}\;\;
 & \;\;\text{  -2}\;\; \\
 \hline
\end{tabular}
\]
\[
\begin{tabular}{|l|c|c|c|c|}
 \hline
  {\bf 2-morphism:} &  \xy
    (0,-3)*{\bbpef{}};
    (8,-5)*{n};    (-12,0)*{};(12,0)*{};
    \endxy
  & \xy
    (0,-3)*{\bbpfe{}};
    (8,-5)*{n};    (-12,0)*{};(12,0)*{};
    \endxy
  & \xy
    (0,-2)*{\bbcef{}};
    (8,0)*{n};     (-12,0)*{};(12,0)*{};
    \endxy
  & \xy
    (0,-2)*{\bbcfe{}};
    (8,0)*{n};    (-12,0)*{};(12,0)*{};
    \endxy\\& & &  &\\ \hline
 {\bf Degree:} & \;\;\text{  $1+n$}\;\;
 & \;\;\text{ $1-n$}\;\;
 & \;\;\text{ $1+n$}\;\;
 & \;\;\text{  $1-n$}\;\;
 \\
 \hline
\end{tabular}
\]
\end{enumerate}
such that the following identities hold:

\item  cups and caps are biadjointness morphisms  up to grading shifts:
\begin{equation} \label{eq_biadjoint1}
  \xy   0;/r.18pc/:
    (-8,0)*{}="1";
    (0,0)*{}="2";
    (8,0)*{}="3";
    (-8,-10);"1" **\dir{-};
    "1";"2" **\crv{(-8,8) & (0,8)} ?(0)*\dir{>} ?(1)*\dir{>};
    "2";"3" **\crv{(0,-8) & (8,-8)}?(1)*\dir{>};
    "3"; (8,10) **\dir{-};
    (12,-9)*{n};
    (-6,9)*{n+2};
    \endxy
    \; =
    \;
\xy   0;/r.18pc/:
    (-8,0)*{}="1";
    (0,0)*{}="2";
    (8,0)*{}="3";
    (0,-10);(0,10)**\dir{-} ?(.5)*\dir{>};
    (5,8)*{n};
    (-9,8)*{n+2};
    \endxy
\qquad \quad  \xy   0;/r.18pc/:
    (-8,0)*{}="1";
    (0,0)*{}="2";
    (8,0)*{}="3";
    (-8,-10);"1" **\dir{-};
    "1";"2" **\crv{(-8,8) & (0,8)} ?(0)*\dir{<} ?(1)*\dir{<};
    "2";"3" **\crv{(0,-8) & (8,-8)}?(1)*\dir{<};
    "3"; (8,10) **\dir{-};
    (12,-9)*{n+2};
    (-6,9)*{ n};
    \endxy
    \; =
    \;
\xy   0;/r.18pc/:
    (-8,0)*{}="1";
    (0,0)*{}="2";
    (8,0)*{}="3";
    (0,-10);(0,10)**\dir{-} ?(.5)*\dir{<};
   (9,8)*{n+2};
    (-6,8)*{ n};
    \endxy
\end{equation}

\begin{equation}
 \xy   0;/r.18pc/:
    (8,0)*{}="1";
    (0,0)*{}="2";
    (-8,0)*{}="3";
    (8,-10);"1" **\dir{-};
    "1";"2" **\crv{(8,8) & (0,8)} ?(0)*\dir{>} ?(1)*\dir{>};
    "2";"3" **\crv{(0,-8) & (-8,-8)}?(1)*\dir{>};
    "3"; (-8,10) **\dir{-};
    (12,9)*{n};
    (-5,-9)*{n+2};
    \endxy
    \; =
    \;
      \xy 0;/r.18pc/:
    (8,0)*{}="1";
    (0,0)*{}="2";
    (-8,0)*{}="3";
    (0,-10);(0,10)**\dir{-} ?(.5)*\dir{>};
    (5,-8)*{n};
    (-9,-8)*{n+2};
    \endxy
\qquad \quad \xy  0;/r.18pc/:
    (8,0)*{}="1";
    (0,0)*{}="2";
    (-8,0)*{}="3";
    (8,-10);"1" **\dir{-};
    "1";"2" **\crv{(8,8) & (0,8)} ?(0)*\dir{<} ?(1)*\dir{<};
    "2";"3" **\crv{(0,-8) & (-8,-8)}?(1)*\dir{<};
    "3"; (-8,10) **\dir{-};
    (12,9)*{n+2};
    (-6,-9)*{ n};
    \endxy
    \; =
    \;
\xy  0;/r.18pc/:
    (8,0)*{}="1";
    (0,0)*{}="2";
    (-8,0)*{}="3";
    (0,-10);(0,10)**\dir{-} ?(.5)*\dir{<};
    (9,-8)*{n+2};
    (-6,-8)*{ n};
    \endxy
\end{equation}

\item NilHecke relations hold:
 \begin{equation}
  \vcenter{\xy 0;/r.18pc/:
    (-4,-4)*{};(4,4)*{} **\crv{(-4,-1) & (4,1)}?(1)*\dir{>};
    (4,-4)*{};(-4,4)*{} **\crv{(4,-1) & (-4,1)}?(1)*\dir{>};
    (-4,4)*{};(4,12)*{} **\crv{(-4,7) & (4,9)}?(1)*\dir{>};
    (4,4)*{};(-4,12)*{} **\crv{(4,7) & (-4,9)}?(1)*\dir{>};
 \endxy}
 =0, \qquad \quad
 \vcenter{
 \xy 0;/r.18pc/:
    (-4,-4)*{};(4,4)*{} **\crv{(-4,-1) & (4,1)}?(1)*\dir{>};
    (4,-4)*{};(-4,4)*{} **\crv{(4,-1) & (-4,1)}?(1)*\dir{>};
    (4,4)*{};(12,12)*{} **\crv{(4,7) & (12,9)}?(1)*\dir{>};
    (12,4)*{};(4,12)*{} **\crv{(12,7) & (4,9)}?(1)*\dir{>};
    (-4,12)*{};(4,20)*{} **\crv{(-4,15) & (4,17)}?(1)*\dir{>};
    (4,12)*{};(-4,20)*{} **\crv{(4,15) & (-4,17)}?(1)*\dir{>};
    (-4,4)*{}; (-4,12) **\dir{-};
    (12,-4)*{}; (12,4) **\dir{-};
    (12,12)*{}; (12,20) **\dir{-};
  (18,8)*{n};
\endxy}
 \;\; =\;\;
 \vcenter{
 \xy 0;/r.18pc/:
    (4,-4)*{};(-4,4)*{} **\crv{(4,-1) & (-4,1)}?(1)*\dir{>};
    (-4,-4)*{};(4,4)*{} **\crv{(-4,-1) & (4,1)}?(1)*\dir{>};
    (-4,4)*{};(-12,12)*{} **\crv{(-4,7) & (-12,9)}?(1)*\dir{>};
    (-12,4)*{};(-4,12)*{} **\crv{(-12,7) & (-4,9)}?(1)*\dir{>};
    (4,12)*{};(-4,20)*{} **\crv{(4,15) & (-4,17)}?(1)*\dir{>};
    (-4,12)*{};(4,20)*{} **\crv{(-4,15) & (4,17)}?(1)*\dir{>};
    (4,4)*{}; (4,12) **\dir{-};
    (-12,-4)*{}; (-12,4) **\dir{-};
    (-12,12)*{}; (-12,20) **\dir{-};
  (10,8)*{n};
\endxy} \label{eq_nil_rels}
  \end{equation}
\begin{eqnarray}
  \xy
  (4,4);(4,-4) **\dir{-}?(0)*\dir{<}+(2.3,0)*{};
  (-4,4);(-4,-4) **\dir{-}?(0)*\dir{<}+(2.3,0)*{};
  (9,2)*{n};
 \endxy
 \quad =
\xy
  (0,0)*{\xybox{
    (-4,-4)*{};(4,4)*{} **\crv{(-4,-1) & (4,1)}?(1)*\dir{>}?(.25)*{\bullet};
    (4,-4)*{};(-4,4)*{} **\crv{(4,-1) & (-4,1)}?(1)*\dir{>};
     (8,1)*{ n};
     (-10,0)*{};(10,0)*{};
     }};
  \endxy
 \;\; -
 \xy
  (0,0)*{\xybox{
    (-4,-4)*{};(4,4)*{} **\crv{(-4,-1) & (4,1)}?(1)*\dir{>}?(.75)*{\bullet};
    (4,-4)*{};(-4,4)*{} **\crv{(4,-1) & (-4,1)}?(1)*\dir{>};
     (8,1)*{ n};
     (-10,0)*{};(10,0)*{};
     }};
  \endxy
 \;\; =
\xy
  (0,0)*{\xybox{
    (-4,-4)*{};(4,4)*{} **\crv{(-4,-1) & (4,1)}?(1)*\dir{>};
    (4,-4)*{};(-4,4)*{} **\crv{(4,-1) & (-4,1)}?(1)*\dir{>}?(.75)*{\bullet};
     (8,1)*{ n};
     (-10,0)*{};(10,0)*{};
     }};
  \endxy
 \;\; -
  \xy
  (0,0)*{\xybox{
    (-4,-4)*{};(4,4)*{} **\crv{(-4,-1) & (4,1)}?(1)*\dir{>} ;
    (4,-4)*{};(-4,4)*{} **\crv{(4,-1) & (-4,1)}?(1)*\dir{>}?(.25)*{\bullet};
     (8,1)*{ n};
     (-10,0)*{};(10,0)*{};
     }};
  \endxy \nn \\ \label{eq_nil_dotslide}
\end{eqnarray}

  \item All 2-morphisms are cyclic\footnote{See \cite{Lau1} and the references therein for
the definition of a cyclic 2-morphism with respect to a biadjoint structure.} with respect to the above biadjoint structure.  This is ensured by the relations:
\begin{equation} \label{eq_cyclic_dot}
    \xy
    (-8,5)*{}="1";
    (0,5)*{}="2";
    (0,-5)*{}="2'";
    (8,-5)*{}="3";
    (-8,-10);"1" **\dir{-};
    "2";"2'" **\dir{-} ?(.5)*\dir{<};
    "1";"2" **\crv{(-8,12) & (0,12)} ?(0)*\dir{<};
    "2'";"3" **\crv{(0,-12) & (8,-12)}?(1)*\dir{<};
    "3"; (8,10) **\dir{-};
    (15,-9)*{ n+2};
    (-12,9)*{n};
    (0,4)*{\txt\large{$\bullet$}};
    (10,8)*{\scs };
    (-10,-8)*{\scs };
    \endxy
    \quad = \quad
      \xy
 (0,10);(0,-10); **\dir{-} ?(.75)*\dir{<}+(2.3,0)*{\scriptstyle{}}
 ?(.1)*\dir{ }+(2,0)*{\scs };
 (0,0)*{\txt\large{$\bullet$}};
 (-6,5)*{ n};
 (8,5)*{ n +2};
 (-10,0)*{};(10,0)*{};(-2,-8)*{\scs };
 \endxy
    \quad = \quad
    \xy
    (8,5)*{}="1";
    (0,5)*{}="2";
    (0,-5)*{}="2'";
    (-8,-5)*{}="3";
    (8,-10);"1" **\dir{-};
    "2";"2'" **\dir{-} ?(.5)*\dir{<};
    "1";"2" **\crv{(8,12) & (0,12)} ?(0)*\dir{<};
    "2'";"3" **\crv{(0,-12) & (-8,-12)}?(1)*\dir{<};
    "3"; (-8,10) **\dir{-};
    (15,9)*{n+2};
    (-12,-9)*{n};
    (0,4)*{\txt\large{$\bullet$}};
    (-10,8)*{\scs };
    (10,-8)*{\scs };
    \endxy
\end{equation}
\begin{equation} \label{eq_cyclic_cross-gen}
\xy 0;/r.19pc/:
  (0,0)*{\xybox{
    (-4,-4)*{};(4,4)*{} **\crv{(-4,-1) & (4,1)}?(1)*\dir{>};
    (4,-4)*{};(-4,4)*{} **\crv{(4,-1) & (-4,1)};
     (4,4)*{};(-18,4)*{} **\crv{(4,16) & (-18,16)} ?(1)*\dir{>};
     (-4,-4)*{};(18,-4)*{} **\crv{(-4,-16) & (18,-16)} ?(1)*\dir{<}?(0)*\dir{<};
     (18,-4);(18,12) **\dir{-};(12,-4);(12,12) **\dir{-};
     (-18,4);(-18,-12) **\dir{-};(-12,4);(-12,-12) **\dir{-};
     (8,1)*{ n};
     (-10,0)*{};(10,0)*{};
      (4,-4)*{};(12,-4)*{} **\crv{(4,-10) & (12,-10)}?(1)*\dir{<}?(0)*\dir{<};
      (-4,4)*{};(-12,4)*{} **\crv{(-4,10) & (-12,10)}?(1)*\dir{>}?(0)*\dir{>};
     }};
  \endxy
\quad =  \quad \xy
  (0,0)*{\xybox{
    (-4,-4)*{};(4,4)*{} **\crv{(-4,-1) & (4,1)}?(0)*\dir{<} ;
    (4,-4)*{};(-4,4)*{} **\crv{(4,-1) & (-4,1)}?(0)*\dir{<};
     (-8,0)*{ n};
     (-12,0)*{};(12,0)*{};
     }};
  \endxy \quad =  \quad
 \xy 0;/r.19pc/:
  (0,0)*{\xybox{
    (4,-4)*{};(-4,4)*{} **\crv{(4,-1) & (-4,1)}?(1)*\dir{>};
    (-4,-4)*{};(4,4)*{} **\crv{(-4,-1) & (4,1)};
     (-4,4)*{};(18,4)*{} **\crv{(-4,16) & (18,16)} ?(1)*\dir{>};
     (4,-4)*{};(-18,-4)*{} **\crv{(4,-16) & (-18,-16)} ?(1)*\dir{<}?(0)*\dir{<};
     (-18,-4);(-18,12) **\dir{-};(-12,-4);(-12,12) **\dir{-};
     (18,4);(18,-12) **\dir{-};(12,4);(12,-12) **\dir{-};
     (8,1)*{ n};
     (-10,0)*{};(10,0)*{};
     (-4,-4)*{};(-12,-4)*{} **\crv{(-4,-10) & (-12,-10)}?(1)*\dir{<}?(0)*\dir{<};
      (4,4)*{};(12,4)*{} **\crv{(4,10) & (12,10)}?(1)*\dir{>}?(0)*\dir{>};
     }};
  \endxy
\end{equation}
The cyclic condition on 2-morphisms expressed by \eqref{eq_cyclic_dot} and
\eqref{eq_cyclic_cross-gen} ensures that isotopic diagrams represent
the same 2-morphism in $\Ucat$.

It will be convenient to introduce degree zero 2-morphisms:
\begin{equation} \label{eq_crossl-gen}
  \xy
  (0,0)*{\xybox{
    (-4,-4)*{};(4,4)*{} **\crv{(-4,-1) & (4,1)}?(1)*\dir{>} ;
    (4,-4)*{};(-4,4)*{} **\crv{(4,-1) & (-4,1)}?(0)*\dir{<};
     (8,2)*{ n};
     (-12,0)*{};(12,0)*{};
     }};
  \endxy
:=
 \xy 0;/r.19pc/:
  (0,0)*{\xybox{
    (4,-4)*{};(-4,4)*{} **\crv{(4,-1) & (-4,1)}?(1)*\dir{>};
    (-4,-4)*{};(4,4)*{} **\crv{(-4,-1) & (4,1)};
   %  (-4,4)*{};(18,4)*{} **\crv{(-4,16) & (18,16)} ?(1)*\dir{>};
    % (4,-4)*{};(-18,-4)*{} **\crv{(4,-16) & (-18,-16)} ?(1)*\dir{<}?(0)*\dir{<};
     (-4,4);(-4,12) **\dir{-};
     (-12,-4);(-12,12) **\dir{-};
     (4,-4);(4,-12) **\dir{-};(12,4);(12,-12) **\dir{-};
     (16,1)*{n};
     (-10,0)*{};(10,0)*{};
     (-4,-4)*{};(-12,-4)*{} **\crv{(-4,-10) & (-12,-10)}?(1)*\dir{<}?(0)*\dir{<};
      (4,4)*{};(12,4)*{} **\crv{(4,10) & (12,10)}?(1)*\dir{>}?(0)*\dir{>};
     }};
  \endxy
  \quad = \quad
  \xy 0;/r.19pc/:
  (0,0)*{\xybox{
    (-4,-4)*{};(4,4)*{} **\crv{(-4,-1) & (4,1)}?(1)*\dir{<};
    (4,-4)*{};(-4,4)*{} **\crv{(4,-1) & (-4,1)};
     (4,4);(4,12) **\dir{-};
     (12,-4);(12,12) **\dir{-};
     (-4,-4);(-4,-12) **\dir{-};(-12,4);(-12,-12) **\dir{-};
     (16,1)*{n};
     (10,0)*{};(-10,0)*{};
     (4,-4)*{};(12,-4)*{} **\crv{(4,-10) & (12,-10)}?(1)*\dir{>}?(0)*\dir{>};
      (-4,4)*{};(-12,4)*{} **\crv{(-4,10) & (-12,10)}?(1)*\dir{<}?(0)*\dir{<};
     }};
  \endxy
\end{equation}
\begin{equation} \label{eq_crossr-gen}
  \xy
  (0,0)*{\xybox{
    (-4,-4)*{};(4,4)*{} **\crv{(-4,-1) & (4,1)}?(0)*\dir{<} ;
    (4,-4)*{};(-4,4)*{} **\crv{(4,-1) & (-4,1)}?(1)*\dir{>};
     (-8,2)*{ n};
     (-12,0)*{};(12,0)*{};
     }};
  \endxy
:=
 \xy 0;/r.19pc/:
  (0,0)*{\xybox{
    (-4,-4)*{};(4,4)*{} **\crv{(-4,-1) & (4,1)}?(1)*\dir{>};
    (4,-4)*{};(-4,4)*{} **\crv{(4,-1) & (-4,1)};
     (4,4);(4,12) **\dir{-};
     (12,-4);(12,12) **\dir{-};
     (-4,-4);(-4,-12) **\dir{-};(-12,4);(-12,-12) **\dir{-};
     (-16,1)*{n};
     (10,0)*{};(-10,0)*{};
     (4,-4)*{};(12,-4)*{} **\crv{(4,-10) & (12,-10)}?(1)*\dir{<}?(0)*\dir{<};
      (-4,4)*{};(-12,4)*{} **\crv{(-4,10) & (-12,10)}?(1)*\dir{>}?(0)*\dir{>};
     }};
  \endxy
  \quad = \quad
  \xy 0;/r.19pc/:
  (0,0)*{\xybox{
    (4,-4)*{};(-4,4)*{} **\crv{(4,-1) & (-4,1)}?(1)*\dir{<};
    (-4,-4)*{};(4,4)*{} **\crv{(-4,-1) & (4,1)};
   %  (-4,4)*{};(18,4)*{} **\crv{(-4,16) & (18,16)} ?(1)*\dir{>};
    % (4,-4)*{};(-18,-4)*{} **\crv{(4,-16) & (-18,-16)} ?(1)*\dir{<}?(0)*\dir{<};
     (-4,4);(-4,12) **\dir{-};
     (-12,-4);(-12,12) **\dir{-};
     (4,-4);(4,-12) **\dir{-};(12,4);(12,-12) **\dir{-};
     (-16,1)*{n};
     (-10,0)*{};(10,0)*{};
     (-4,-4)*{};(-12,-4)*{} **\crv{(-4,-10) & (-12,-10)}?(1)*\dir{>}?(0)*\dir{>};
      (4,4)*{};(12,4)*{} **\crv{(4,10) & (12,10)}?(1)*\dir{<}?(0)*\dir{<};
     % (-14,11)*{\scs i};(-2,11)*{\scs j};(14,-11)*{\scs i};(2,-11)*{\scs j};
     }};
  \endxy
\end{equation}
where the second equality in \eqref{eq_crossl-gen} and \eqref{eq_crossr-gen}
follow from \eqref{eq_cyclic_cross-gen}.  We also write
\[
  \xy
 (0,7);(0,-7); **\dir{-} ?(.75)*\dir{>}+(2.3,0)*{\scriptstyle{}};
 (0.1,-2)*{\txt\large{$\bullet$}};
 (6,4)*{ n};(-3,-1)*{\alpha};
 (-10,0)*{};(10,0)*{};(0,-10)*{};(0,10)*{};
 \endxy \quad := \quad \left(
  \xy
 (0,7);(0,-7); **\dir{-} ?(.75)*\dir{>}+(2.3,0)*{\scriptstyle{}};
 (0.1,-2)*{\txt\large{$\bullet$}};
 (6,4)*{ n};(-3,-1)*{};
 (-10,0)*{};(10,0)*{};(0,-10)*{};(0,10)*{};
 \endxy
 \right)^{\alpha}
\]
to denote the $\alpha$-fold vertical composite of a dot with itself.

\item  All dotted bubbles of negative degree are zero. That is,
\begin{equation} \label{eq_positivity_bubbles}
 \xy
 (-12,0)*{\cbub{\alpha}};
 (-8,8)*{n};
 \endxy
  = 0
 \qquad
  \text{if $\alpha<n-1$,} \qquad
 \xy
 (-12,0)*{\ccbub{\alpha}};
 (-8,8)*{n};
 \endxy = 0\quad
  \text{if $\alpha< -n-1$}
\end{equation}
for all $\alpha \in \Z_+$.  A dotted bubble of degree zero equals 1:
\[
\xy 0;/r.18pc/:
 (0,0)*{\cbub{n-1}};
  (4,8)*{n};
 \endxy
  = 1 \quad \text{for $n \geq 1$,}
  \qquad \quad
  \xy 0;/r.18pc/:
 (0,0)*{\ccbub{-n-1}};
  (4,8)*{n};
 \endxy
  = 1 \quad \text{for $n \leq -1$.}
\]
It is often convenient to express dotted bubbles using a notation introduced in \cite{KLMS} that emphasizes the degree:
\[
\xy 0;/r.18pc/:
 (0,0)*{\cbub{\spadesuit+\alpha}};
  (4,8)*{n};
 \endxy\quad := \quad
\xy 0;/r.18pc/:
 (0,0)*{\cbub{(n-1)+\alpha}};
  (4,8)*{n};
 \endxy
   \qquad \quad
   \qquad
   \xy 0;/r.18pc/:
 (0,0)*{\ccbub{\spadesuit+\alpha}};
  (4,8)*{n};
 \endxy \quad := \quad
   \xy 0;/r.18pc/:
 (0,0)*{\ccbub{(-n-1)+\alpha}};
  (4,8)*{n};
 \endxy
\]
so that
\[
 \deg\left(\xy 0;/r.18pc/:
 (0,0)*{\cbub{\spadesuit+\alpha}};
  (4,8)*{n};
 \endxy \right) \;\; = \;\; 2\alpha
 \qquad \qquad
  \deg\left(    \xy 0;/r.18pc/:
 (0,0)*{\ccbub{\spadesuit+\alpha}};
  (4,8)*{n};
 \endxy\right) \;\; = \;\; 2\alpha.
\]
The value of $\spadesuit$ depends on the orientation, $\spadesuit = n-1$ for clockwise oriented bubbles and $\spadesuit = -n-1$ for counter-clockwise oriented bubbles.  Notice that for some values of $n$ it is possible that $\spadesuit +\alpha$ is a negative number even though $\alpha \geq 0$.  While vertically composing a generator with itself a negative number of times does not make sense, having these symbols around greatly simplifies the calculus.  For each $\spadesuit + \alpha <0$, where
\[
 \deg\left(\xy 0;/r.18pc/:
 (0,0)*{\cbub{\spadesuit+\alpha}};
  (4,8)*{n};
 \endxy \right) \geq 0
 \qquad \qquad
  \deg\left(    \xy 0;/r.18pc/:
 (0,0)*{\ccbub{\spadesuit+\alpha}};
  (4,8)*{n};
 \endxy\right) \geq 0,
\]
we introduce formal symbols, called {\em fake bubbles},  inductively defined by the equation
\begin{center}
\begin{eqnarray}
 \makebox[0pt]{ $
\left( \xy 0;/r.15pc/:
 (0,0)*{\ccbub{\spadesuit+0}};
  (4,8)*{n};
 \endxy
 +
 \xy 0;/r.15pc/:
 (0,0)*{\ccbub{\spadesuit+1}};
  (4,8)*{n};
 \endxy t
 + \cdots +
 \xy 0;/r.15pc/:
 (0,0)*{\ccbub{\spadesuit+\alpha}};
  (4,8)*{n};
 \endxy t^{\alpha}
 + \cdots
\right)
%%%%
%
%%%%
\left( \xy 0;/r.15pc/:
 (0,0)*{\cbub{\spadesuit+0}};
  (4,8)*{n};
 \endxy
 + \cdots +
 \xy 0;/r.15pc/:
 (0,0)*{\cbub{\spadesuit+\alpha}};
 (4,8)*{n};
 \endxy t^{\alpha}
 + \cdots
\right) =1 .$ } \nn \\ \label{eq_infinite_Grass}
\end{eqnarray}
\end{center}
and the additional condition
\[
\xy 0;/r.18pc/:
 (0,0)*{\cbub{\spadesuit +0}};
  (4,8)*{n};
 \endxy
 \quad = \quad
  \xy 0;/r.18pc/:
 (0,0)*{\ccbub{\spadesuit +0}};
  (4,8)*{n};
 \endxy
  \quad = \quad 1.
\]
Equation~\eqref{eq_infinite_Grass} is called the infinite Grassmannian relation.  It remains valid even in high degree when most of the bubbles involved are not fake bubbles.  See \cite{Lau1} for more details.

\item For the following relations we employ the convention that all summations
are increasing, so that $\sum_{f=0}^{\alpha}$ is zero if $\alpha < 0$.
\begin{equation} \label{eq_reduction}
  \text{$\xy 0;/r.18pc/:
  (14,8)*{n};
  (-3,-8)*{};(3,8)*{} **\crv{(-3,-1) & (3,1)}?(1)*\dir{>};?(0)*\dir{>};
    (3,-8)*{};(-3,8)*{} **\crv{(3,-1) & (-3,1)}?(1)*\dir{>};
  (-3,-12)*{\bbsid};  (-3,8)*{\bbsid};
  (3,8)*{}="t1";  (9,8)*{}="t2";
  (3,-8)*{}="t1'";  (9,-8)*{}="t2'";
   "t1";"t2" **\crv{(3,14) & (9, 14)};
   "t1'";"t2'" **\crv{(3,-14) & (9, -14)};
   "t2'";"t2" **\dir{-} ?(.5)*\dir{<};
   (9,0)*{}; (-6,-8)*{\scs };
 \endxy$} \;\; = \;\; -\sum_{f_1+f_2=-n}
   \xy
  (19,4)*{n};
  (0,0)*{\bbe{}};(-2,-8)*{\scs };
  (12,-2)*{\cbub{\spadesuit+f_2}};
  (0,6)*{\bullet}+(3,-1)*{\scs f_1};
 \endxy
\qquad \qquad
  \text{$ \xy 0;/r.18pc/:
  (-12,8)*{n};
   (-3,-8)*{};(3,8)*{} **\crv{(-3,-1) & (3,1)}?(1)*\dir{>};?(0)*\dir{>};
    (3,-8)*{};(-3,8)*{} **\crv{(3,-1) & (-3,1)}?(1)*\dir{>};
  (3,-12)*{\bbsid};
  (3,8)*{\bbsid}; (6,-8)*{\scs };
  (-9,8)*{}="t1";
  (-3,8)*{}="t2";
  (-9,-8)*{}="t1'";
  (-3,-8)*{}="t2'";
   "t1";"t2" **\crv{(-9,14) & (-3, 14)};
   "t1'";"t2'" **\crv{(-9,-14) & (-3, -14)};
  "t1'";"t1" **\dir{-} ?(.5)*\dir{<};
 \endxy$} \;\; = \;\;
 \sum_{g_1+g_2=n}^{}
   \xy
  (-12,8)*{n};
  (0,0)*{\bbe{}};(2,-8)*{\scs};
  (-12,-2)*{\ccbub{\spadesuit+g_2}};
  (0,6)*{\bullet}+(3,-1)*{\scs g_1};
 \endxy
\end{equation}
\begin{eqnarray}
 \vcenter{\xy 0;/r.18pc/:
  (-8,0)*{};
  (8,0)*{};
  (-4,10)*{}="t1";
  (4,10)*{}="t2";
  (-4,-10)*{}="b1";
  (4,-10)*{}="b2";(-6,-8)*{\scs };(6,-8)*{\scs };
  "t1";"b1" **\dir{-} ?(.5)*\dir{<};
  "t2";"b2" **\dir{-} ?(.5)*\dir{>};
  (10,2)*{n};
  (-10,2)*{n};
  \endxy}
&\quad = \quad&
 -\;\;
 \vcenter{   \xy 0;/r.18pc/:
    (-4,-4)*{};(4,4)*{} **\crv{(-4,-1) & (4,1)}?(1)*\dir{>};
    (4,-4)*{};(-4,4)*{} **\crv{(4,-1) & (-4,1)}?(1)*\dir{<};?(0)*\dir{<};
    (-4,4)*{};(4,12)*{} **\crv{(-4,7) & (4,9)};
    (4,4)*{};(-4,12)*{} **\crv{(4,7) & (-4,9)}?(1)*\dir{>};
  (8,8)*{n};(-6,-3)*{\scs };
     (6.5,-3)*{\scs };
 \endxy}
  \quad + \quad
   \sum_{ \xy  (0,3)*{\scs f_1+f_2+f_3}; (0,0)*{\scs =n-1};\endxy}
    \vcenter{\xy 0;/r.18pc/:
    (-10,10)*{n};
    (-8,0)*{};
  (8,0)*{};
  (-4,-15)*{}="b1";
  (4,-15)*{}="b2";
  "b2";"b1" **\crv{(5,-8) & (-5,-8)}; ?(.05)*\dir{<} ?(.93)*\dir{<}
  ?(.8)*\dir{}+(0,-.1)*{\bullet}+(-3,2)*{\scs f_3};
  (-4,15)*{}="t1";
  (4,15)*{}="t2";
  "t2";"t1" **\crv{(5,8) & (-5,8)}; ?(.15)*\dir{>} ?(.95)*\dir{>}
  ?(.4)*\dir{}+(0,-.2)*{\bullet}+(3,-2)*{\scs \; f_1};
  (0,0)*{\ccbub{\scs \quad \spadesuit+f_2}};
  \endxy} \nn
 \\  \; \nn \\
 \vcenter{\xy 0;/r.18pc/:
  (-8,0)*{};(-6,-8)*{\scs };(6,-8)*{\scs };
  (8,0)*{};
  (-4,10)*{}="t1";
  (4,10)*{}="t2";
  (-4,-10)*{}="b1";
  (4,-10)*{}="b2";
  "t1";"b1" **\dir{-} ?(.5)*\dir{>};
  "t2";"b2" **\dir{-} ?(.5)*\dir{<};
  (10,2)*{n};
  (-10,2)*{n};
  \endxy}
&\quad = \quad&
 -\;\;
   \vcenter{\xy 0;/r.18pc/:
    (-4,-4)*{};(4,4)*{} **\crv{(-4,-1) & (4,1)}?(1)*\dir{<};?(0)*\dir{<};
    (4,-4)*{};(-4,4)*{} **\crv{(4,-1) & (-4,1)}?(1)*\dir{>};
    (-4,4)*{};(4,12)*{} **\crv{(-4,7) & (4,9)}?(1)*\dir{>};
    (4,4)*{};(-4,12)*{} **\crv{(4,7) & (-4,9)};
  (8,8)*{n};(-6,-3)*{\scs };  (6,-3)*{\scs };
 \endxy}
  \quad + \quad
    \sum_{ \xy  (0,3)*{\scs g_1+g_2+g_3}; (0,0)*{\scs =-n-1};\endxy}
    \vcenter{\xy 0;/r.18pc/:
    (-8,0)*{};
  (8,0)*{};
  (-4,-15)*{}="b1";
  (4,-15)*{}="b2";
  "b2";"b1" **\crv{(5,-8) & (-5,-8)}; ?(.1)*\dir{>} ?(.95)*\dir{>}
  ?(.8)*\dir{}+(0,-.1)*{\bullet}+(-3,2)*{\scs g_3};
  (-4,15)*{}="t1";
  (4,15)*{}="t2";
  "t2";"t1" **\crv{(5,8) & (-5,8)}; ?(.15)*\dir{<} ?(.9)*\dir{<}
  ?(.4)*\dir{}+(0,-.2)*{\bullet}+(3,-2)*{\scs g_1};
  (0,0)*{\cbub{\scs \quad\; \spadesuit + g_2}};
  (-10,10)*{n};
  \endxy} \label{eq_ident_decomp}
\end{eqnarray}
for all $n\in \Z$.  In equations \eqref{eq_reduction} and \eqref{eq_ident_decomp} whenever the summations are nonzero they utilize fake bubbles.

\item the additive composition functor $\Ucat(n,n')
 \times  \Ucat(n',n'')\to\Ucat(n,n'') $ is given on
 1-morphisms of $\Ucat$ by
\begin{equation}
  \cal{E}_{\ep'}\mathbf{1}_{n'}\{t'\} \times \cal{E}_{\ep}\onen\{t\} \mapsto
  \cal{E}_{\ep'\ep}\onen\{t+t'\}
\end{equation}
for $n'=n+\sum_{j=1}^m \epsilon_i 2$, and on 2-morphisms of $\Ucat$ by juxtaposition of diagrams
\[
\left(\;\;\vcenter{\xy 0;/r.16pc/:
 (-4,-15)*{}; (-20,25) **\crv{(-3,-6) & (-20,4)}?(0)*\dir{<}?(.6)*\dir{}+(0,0)*{\bullet};
 (-12,-15)*{}; (-4,25) **\crv{(-12,-6) & (-4,0)}?(0)*\dir{<}?(.6)*\dir{}+(.2,0)*{\bullet};
 ?(0)*\dir{<}?(.75)*\dir{}+(.2,0)*{\bullet};?(0)*\dir{<}?(.9)*\dir{}+(0,0)*{\bullet};
 (-28,25)*{}; (-12,25) **\crv{(-28,10) & (-12,10)}?(0)*\dir{<};
  ?(.2)*\dir{}+(0,0)*{\bullet}?(.35)*\dir{}+(0,0)*{\bullet};
 (-36,-15)*{}; (-36,25) **\crv{(-34,-6) & (-35,4)}?(1)*\dir{>};
 (-28,-15)*{}; (-42,25) **\crv{(-28,-6) & (-42,4)}?(1)*\dir{>};
 (-42,-15)*{}; (-20,-15) **\crv{(-42,-5) & (-20,-5)}?(1)*\dir{>};
 (6,10)*{\cbub{}{}};
 (-23,0)*{\cbub{}{}};
 (8,-4)*{n'};(-44,-4)*{n''};
 \endxy}\;\;\right) \;\; \times \;\;
\left(\;\;\vcenter{ \xy 0;/r.18pc/: (-14,8)*{\xybox{
 (0,-10)*{}; (-16,10)*{} **\crv{(0,-6) & (-16,6)}?(.5)*\dir{};
 (-16,-10)*{}; (-8,10)*{} **\crv{(-16,-6) & (-8,6)}?(1)*\dir{}+(.1,0)*{\bullet};
  (-8,-10)*{}; (0,10)*{} **\crv{(-8,-6) & (-0,6)}?(.6)*\dir{}+(.2,0)*{\bullet}?
  (1)*\dir{}+(.1,0)*{\bullet};
  (0,10)*{}; (-16,30)*{} **\crv{(0,14) & (-16,26)}?(1)*\dir{>};
 (-16,10)*{}; (-8,30)*{} **\crv{(-16,14) & (-8,26)}?(1)*\dir{>};
  (-8,10)*{}; (0,30)*{} **\crv{(-8,14) & (-0,26)}?(1)*\dir{>}?(.6)*\dir{}+(.25,0)*{\bullet};
   }};
 (-2,-4)*{n}; (-26,-4)*{n'};
 \endxy} \;\;\right)
 \;\;\mapsto \;\;
\vcenter{\xy 0;/r.16pc/:
 (-4,-15)*{}; (-20,25) **\crv{(-3,-6) & (-20,4)}?(0)*\dir{<}?(.6)*\dir{}+(0,0)*{\bullet};
 (-12,-15)*{}; (-4,25) **\crv{(-12,-6) & (-4,0)}?(0)*\dir{<}?(.6)*\dir{}+(.2,0)*{\bullet};
 ?(0)*\dir{<}?(.75)*\dir{}+(.2,0)*{\bullet};?(0)*\dir{<}?(.9)*\dir{}+(0,0)*{\bullet};
 (-28,25)*{}; (-12,25) **\crv{(-28,10) & (-12,10)}?(0)*\dir{<};
  ?(.2)*\dir{}+(0,0)*{\bullet}?(.35)*\dir{}+(0,0)*{\bullet};
 (-36,-15)*{}; (-36,25) **\crv{(-34,-6) & (-35,4)}?(1)*\dir{>};
 (-28,-15)*{}; (-42,25) **\crv{(-28,-6) & (-42,4)}?(1)*\dir{>};
 (-42,-15)*{}; (-20,-15) **\crv{(-42,-5) & (-20,-5)}?(1)*\dir{>};
 (6,10)*{\cbub{}};
 (-23,0)*{\cbub{}};
 \endxy}
 \vcenter{ \xy 0;/r.16pc/: (-14,8)*{\xybox{
 (0,-10)*{}; (-16,10)*{} **\crv{(0,-6) & (-16,6)}?(.5)*\dir{};
 (-16,-10)*{}; (-8,10)*{} **\crv{(-16,-6) & (-8,6)}?(1)*\dir{}+(.1,0)*{\bullet};
  (-8,-10)*{}; (0,10)*{} **\crv{(-8,-6) & (-0,6)}?(.6)*\dir{}+(.2,0)*{\bullet}?
  (1)*\dir{}+(.1,0)*{\bullet};
  (0,10)*{}; (-16,30)*{} **\crv{(0,14) & (-16,26)}?(1)*\dir{>};
 (-16,10)*{}; (-8,30)*{} **\crv{(-16,14) & (-8,26)}?(1)*\dir{>};
  (-8,10)*{}; (0,30)*{} **\crv{(-8,14) & (-0,26)}?(1)*\dir{>}?(.6)*\dir{}+(.25,0)*{\bullet};
   }};
 (0,-5)*{n};
 \endxy}
\]
\end{itemize}
\end{defn}

% --------------------------------------------------------------------
%
\subsubsection{Relations in $\Ucat$}
%
% --------------------------------------------------------------------

In this section we collect some relations that follow from the
definition of  $\Ucat$.  These relations were proven in \cite{Lau1}.

\begin{equation} \label{eq_ind_dotslide}
\xy
  (0,0)*{\xybox{
    (-4,-4)*{};(4,4)*{} **\crv{(-4,-1) & (4,1)}?(1)*\dir{>}?(.25)*{\bullet}+(-2.5,1)*{\alpha};
    (4,-4)*{};(-4,4)*{} **\crv{(4,-1) & (-4,1)}?(1)*\dir{>};
     (8,-4)*{n};
     (-10,0)*{};(10,0)*{};
     }};
  \endxy
 \;\; -
 \xy
  (0,0)*{\xybox{
    (-4,-4)*{};(4,4)*{} **\crv{(-4,-1) & (4,1)}?(1)*\dir{>}?(.75)*{\bullet}+(2.5,-1)*{\alpha};
    (4,-4)*{};(-4,4)*{} **\crv{(4,-1) & (-4,1)}?(1)*\dir{>};
     (8,-4)*{n};
     (-10,0)*{};(10,0)*{};
     }};
  \endxy
 \;\; =
\xy
  (0,0)*{\xybox{
    (-4,-4)*{};(4,4)*{} **\crv{(-4,-1) & (4,1)}?(1)*\dir{>};
    (4,-4)*{};(-4,4)*{} **\crv{(4,-1) & (-4,1)}?(1)*\dir{>}?(.75)*{\bullet}+(-2.5,-1)*{\alpha};
     (8,3)*{ n};
     (-10,0)*{};(10,0)*{};
     }};
  \endxy
 \;\; -
  \xy
  (0,0)*{\xybox{
    (-4,-4)*{};(4,4)*{} **\crv{(-4,-1) & (4,1)}?(1)*\dir{>} ;
    (4,-4)*{};(-4,4)*{} **\crv{(4,-1) & (-4,1)}?(1)*\dir{>}?(.25)*{\bullet}+(2.5,1)*{\alpha};
     (8,3)*{n};
     (-10,0)*{};(10,0)*{};
     }};
  \endxy
  \;\; = \;\;
  \sum_{f_1 + f_2 = \alpha-1}
  \xy
  (3,4);(3,-4) **\dir{-}?(0)*\dir{<} ?(.5)*\dir{}+(0,0)*{\bullet}+(2.5,1)*{\scs f_2};
  (-3,4);(-3,-4) **\dir{-}?(0)*\dir{<}?(.5)*\dir{}+(0,0)*{\bullet}+(-2.5,1)*{\scs f_1};;
  (9,-4)*{n};
 \endxy
\end{equation}

For all $n \in \Z$ the following equation holds
\begin{equation}
 \vcenter{
 \xy 0;/r.17pc/:
    (-4,-4)*{};(4,4)*{} **\crv{(-4,-1) & (4,1)};
    (4,-4)*{};(-4,4)*{} **\crv{(4,-1) & (-4,1)}  ?(0)*\dir{<};
    (4,4)*{};(12,12)*{} **\crv{(4,7) & (12,9)};
    (12,4)*{};(4,12)*{} **\crv{(12,7) & (4,9)};
    (-4,12)*{};(4,20)*{} **\crv{(-4,15) & (4,17)};
    (4,12)*{};(-4,20)*{} **\crv{(4,15) & (-4,17)}?(1)*\dir{>};
    (-4,4)*{}; (-4,12) **\dir{-};
    (12,-4)*{}; (12,4) **\dir{-};
    (12,12)*{}; (12,20) **\dir{-}?(1)*\dir{>};;
  (18,8)*{n};
\endxy}
-\;
   \vcenter{
 \xy 0;/r.17pc/:
    (4,-4)*{};(-4,4)*{} **\crv{(4,-1) & (-4,1)};
    (-4,-4)*{};(4,4)*{} **\crv{(-4,-1) & (4,1)}?(0)*\dir{<};;
    (-4,4)*{};(-12,12)*{} **\crv{(-4,7) & (-12,9)};
    (-12,4)*{};(-4,12)*{} **\crv{(-12,7) & (-4,9)};
    (4,12)*{};(-4,20)*{} **\crv{(4,15) & (-4,17)};
    (-4,12)*{};(4,20)*{} **\crv{(-4,15) & (4,17)}?(1)*\dir{>};
    (4,4)*{}; (4,12) **\dir{-};
    (-12,-4)*{}; (-12,4) **\dir{-};
    (-12,12)*{}; (-12,20) **\dir{-}?(1)*\dir{>};;
  (10,8)*{n};
\endxy}
  \; = \;
\sum_{\xy (0,2)*{\scs f_1+f_2+f_3+f_4}; (0,-1)*{\scs =n}\endxy} \; \xy 0;/r.17pc/:
    (-4,12)*{}="t1";
    (4,12)*{}="t2";
  "t2";"t1" **\crv{(5,5) & (-5,5)}; ?(.15)*\dir{} ?(1)*\dir{>}
  ?(.2)*\dir{}+(0,-.2)*{\bullet}+(3,-2)*{\scs f_1};
    (-4,-12)*{}="t1";
    (4,-12)*{}="t2";
  "t2";"t1" **\crv{(5,-5) & (-5,-5)};  ?(0)*\dir{<}
  ?(.15)*\dir{}+(0,-.2)*{\bullet}+(3,2)*{\scs f_3};
    (-8,1)*{\ccbub{\scs \spadesuit +f_4}{}};
    (13,12)*{};(13,-12)*{} **\dir{-} ?(.5)*\dir{<};
    (13,8)*{\bullet}+(3,2)*{\scs f_2};
  (19,-6)*{n};
  \endxy
+\;
 \sum_{\xy (0,2)*{\scs g_1+g_2+g_3+g_4}; (0,-1)*{\scs =-n-2}\endxy} \; \xy 0;/r.17pc/: (-10,12)*{};(-10,-12)*{} **\dir{-}
?(.5)*\dir{<};
  (-10,8)*{\bullet}+(-3,2)*{\scs g_2};
  (-4,12)*{}="t1";
  (4,12)*{}="t2";
  "t1";"t2" **\crv{(-4,5) & (4,5)};  ?(1)*\dir{>}
  ?(.4)*\dir{}+(0,-.2)*{\bullet}+(-3,-2)*{\scs \;\; g_1};
  (-4,-12)*{}="t1";
  (4,-12)*{}="t2";
  "t2";"t1" **\crv{(4,-5) & (-4,-5)};  ?(1)*\dir{>}
  ?(.8)*\dir{}+(0,-.2)*{\bullet}+(1,4)*{\scs g_3};
  (12,1)*{\cbub{\scs \spadesuit+g_4}{}};
  (24,6)*{n};
  \endxy
 \label{eq_r3_extra}
\end{equation}
where the first sum is over all $f_1, f_2, f_3, f_4 \geq 0$ with
$f_1+f_2+f_3+f_4=n$ and the second sum is over all $g_1, g_2, g_3,
g_4 \geq 0$ with $g_1+g_2+g_3+g_4=-n -2$.  Recall that all
summations in this paper are increasing, so that the first summation
is zero if $n<0$ and the second is zero when $-2<n$.  By rotating this equation and shifting $n$ we also have
\begin{equation}
   \vcenter{
 \xy 0;/r.17pc/:
    (4,-4)*{};(-4,4)*{} **\crv{(4,-1) & (-4,1)}?(0)*\dir{<};
    (-4,-4)*{};(4,4)*{} **\crv{(-4,-1) & (4,1)};
    (-4,4)*{};(-12,12)*{} **\crv{(-4,7) & (-12,9)};
    (-12,4)*{};(-4,12)*{} **\crv{(-12,7) & (-4,9)};
    (4,12)*{};(-4,20)*{} **\crv{(4,15) & (-4,17)}?(1)*\dir{>};;
    (-4,12)*{};(4,20)*{} **\crv{(-4,15) & (4,17)}?(1)*\dir{};
    (4,4)*{}; (4,12) **\dir{-};
    (-12,-4)*{}; (-12,4) **\dir{-}?(0)*\dir{<};
    (-12,12)*{}; (-12,20) **\dir{-};
  (10,8)*{n};
\endxy}
-\;
 \vcenter{
 \xy 0;/r.17pc/:
    (-4,-4)*{};(4,4)*{} **\crv{(-4,-1) & (4,1)}?(0)*\dir{<};
    (4,-4)*{};(-4,4)*{} **\crv{(4,-1) & (-4,1)};
    (4,4)*{};(12,12)*{} **\crv{(4,7) & (12,9)};
    (12,4)*{};(4,12)*{} **\crv{(12,7) & (4,9)};
    (-4,12)*{};(4,20)*{} **\crv{(-4,15) & (4,17)}?(1)*\dir{>};
    (4,12)*{};(-4,20)*{} **\crv{(4,15) & (-4,17)}?(1)*\dir{};
    (-4,4)*{}; (-4,12) **\dir{-};
    (12,-4)*{}; (12,4) **\dir{-}?(0)*\dir{<};
    (12,12)*{}; (12,20) **\dir{-};
  (18,8)*{n};
\endxy}
  \; = \;
 \sum_{\xy (0,2)*{\scs f_1+f_2+f_3+f_4}; (0,-1)*{\scs =n-2}\endxy} \; \xy 0;/r.17pc/:
    (-4,12)*{}="t1";
    (4,12)*{}="t2";
  "t2";"t1" **\crv{(5,5) & (-5,5)}; ?(.15)*\dir{} ?(1)*\dir{>}
  ?(.8)*\dir{}+(0,-.2)*{\bullet}+(-3,-2)*{\scs f_1};
    (-4,-12)*{}="t1";
    (4,-12)*{}="t2";
  "t2";"t1" **\crv{(5,-5) & (-5,-5)};  ?(0)*\dir{<}
  ?(.85)*\dir{}+(0,-.2)*{\bullet}+(-3,2)*{\scs f_3};
    (8,1)*{\ccbub{\scs \spadesuit +f_4}{}};
    (-13,12)*{};(-13,-12)*{} **\dir{-} ?(.5)*\dir{>};
    (-13,8)*{\bullet}+(-3,2)*{\scs f_2};
  (19,6)*{n};
  \endxy
+\;
  \sum_{\xy (0,2)*{\scs g_1+g_2+g_3+g_4}; (0,-1)*{\scs =-n}\endxy} \; \xy 0;/r.17pc/:
(10,12)*{};(10,-12)*{} **\dir{-}?(.5)*\dir{>};
  (10,8)*{\bullet}+(3,2)*{\scs g_2};
  (-4,12)*{}="t1";
  (4,12)*{}="t2";
  "t1";"t2" **\crv{(-4,5) & (4,5)};  ?(1)*\dir{>}
  ?(.7)*\dir{}+(0,-.2)*{\bullet}+(3,-2)*{\scs \;\; g_1};
  (-4,-12)*{}="t1";
  (4,-12)*{}="t2";
  "t2";"t1" **\crv{(4,-5) & (-4,-5)}; ?(1)*\dir{>}
  ?(.2)*\dir{}+(0,-.2)*{\bullet}+(1,4)*{\scs g_3};
  (-10,0)*{\cbub{\scs \spadesuit+g_4}{}};
  (24,6)*{n};
  \endxy
 \label{eq_r3_extra2}
\end{equation}

Dotted curl relations:
\begin{equation} \label{eq_reduction-dots}
  \text{$\xy 0;/r.18pc/:
  (14,8)*{n};
  (-3,-8)*{};(3,8)*{} **\crv{(-3,-1) & (3,1)}?(1)*\dir{>};?(0)*\dir{>};
    (3,-8)*{};(-3,8)*{} **\crv{(3,-1) & (-3,1)}?(1)*\dir{>};
  (-3,-12)*{\bbsid};  (-3,8)*{\bbsid};
  (3,8)*{}="t1";  (9,8)*{}="t2";
  (3,-8)*{}="t1'";  (9,-8)*{}="t2'";
   "t1";"t2" **\crv{(3,14) & (9, 14)};
   "t1'";"t2'" **\crv{(3,-14) & (9, -14)};
   (9,-4)*{\bullet}+(3,-1)*{\scs x};
   "t2'";"t2" **\dir{-} ?(.5)*\dir{<};
   (9,0)*{}; (-6,-8)*{\scs };
 \endxy$} \;\; = \;\; -\sum_{f_1+f_2=x-n}
   \xy
  (19,4)*{n};
  (0,0)*{\bbe{}};(-2,-8)*{\scs };
  (12,-2)*{\cbub{\spadesuit+f_2}{}};
  (0,6)*{\bullet}+(3,-1)*{\scs f_1};
 \endxy
\qquad \qquad
  \text{$ \xy 0;/r.18pc/:
  (-12,8)*{n};
   (-3,-8)*{};(3,8)*{} **\crv{(-3,-1) & (3,1)}?(1)*\dir{>};?(0)*\dir{>};
    (3,-8)*{};(-3,8)*{} **\crv{(3,-1) & (-3,1)}?(1)*\dir{>};
  (3,-12)*{\bbsid};
  (3,8)*{\bbsid}; (6,-8)*{\scs };
  (-9,8)*{}="t1";
  (-3,8)*{}="t2";
  (-9,-8)*{}="t1'";
  (-3,-8)*{}="t2'";
   "t1";"t2" **\crv{(-9,14) & (-3, 14)};
   "t1'";"t2'" **\crv{(-9,-14) & (-3, -14)};
  "t1'";"t1" **\dir{-} ?(.5)*\dir{<};  (-9,-4)*{\bullet}+(-3,-1)*{\scs x};
 \endxy$} \;\; = \;\;
 \sum_{g_1+g_2=x+n}^{}
   \xy
  (-12,8)*{n};
  (0,0)*{\bbe{}};(2,-8)*{\scs};
  (-12,-2)*{\ccbub{\spadesuit+g_2}{}};
  (0,6)*{\bullet}+(3,-1)*{\scs g_1};
 \endxy
\end{equation}
One can also show the relations:
\begin{eqnarray}
 \vcenter{\xy 0;/r.18pc/:
  (-8,0)*{};
  (8,0)*{};
  (-4,10)*{}="t1";
  (4,10)*{}="t2";
  (-4,-10)*{}="b1";
  (4,-10)*{}="b2";(-6,-8)*{\scs };(6,-8)*{\scs };
  "t1";"b1" **\dir{-} ?(.5)*\dir{<};
  "t2";"b2" **\dir{-} ?(.5)*\dir{>};
   (-4,-4)*{\bullet}+(-3,-1)*{\scs x};
   (4,-4)*{\bullet}+(3,-1)*{\scs y};
  (10,2)*{n};
  (-10,2)*{n};
  \endxy}
&\quad = \quad&
 -\;\;
 \vcenter{   \xy 0;/r.18pc/:
    (-4,-4)*{};(4,4)*{} **\crv{(-4,-1) & (4,1)};
    (4,-4)*{};(-4,4)*{} **\crv{(4,-1) & (-4,1)}?(0)*\dir{<};
    (-4,4)*{};(4,12)*{} **\crv{(-4,7) & (4,9)};
    (4,4)*{};(-4,12)*{} **\crv{(4,7) & (-4,9)}?(1)*\dir{>};
  (8,8)*{n};(-6,-3)*{\scs };
     (6.5,-3)*{\scs };
   (-4,4)*{\bullet}+(-3,-1)*{\scs y};
   (4,4)*{\bullet}+(3,-1)*{\scs x};
 \endxy}
  \quad + \quad
   \sum_{ \xy  (0,3)*{\scs f_1+f_2+f_3}; (0,0)*{\scs =x+y+n-1};\endxy}
    \vcenter{\xy 0;/r.18pc/:
    (-10,10)*{n};
    (-8,0)*{};
  (8,0)*{};
  (-4,-15)*{}="b1";
  (4,-15)*{}="b2";
  "b2";"b1" **\crv{(5,-8) & (-5,-8)}; ?(.05)*\dir{<} ?(.93)*\dir{<}
  ?(.8)*\dir{}+(0,-.1)*{\bullet}+(-3,2)*{\scs f_3};
  (-4,15)*{}="t1";
  (4,15)*{}="t2";
  "t2";"t1" **\crv{(5,8) & (-5,8)}; ?(.15)*\dir{>} ?(.95)*\dir{>}
  ?(.4)*\dir{}+(0,-.2)*{\bullet}+(3,-2)*{\scs \; f_1};
  (0,0)*{\ccbub{\scs \quad \spadesuit+f_2}{}};
  \endxy} \nn
 \\  \; \nn \\
 \vcenter{\xy 0;/r.18pc/:
  (-8,0)*{};(-6,-8)*{\scs };(6,-8)*{\scs };
  (8,0)*{};
  (-4,10)*{}="t1";
  (4,10)*{}="t2";
  (-4,-10)*{}="b1";
  (4,-10)*{}="b2";
  "t1";"b1" **\dir{-} ?(.5)*\dir{>};
  "t2";"b2" **\dir{-} ?(.5)*\dir{<};
  (-4,-4)*{\bullet}+(-3,-1)*{\scs x};
   (4,-4)*{\bullet}+(3,-1)*{\scs y};
  (10,2)*{n};
  (-10,2)*{n};
  \endxy}
&\quad = \quad&
 -\;\;
   \vcenter{\xy 0;/r.18pc/:
    (-4,-4)*{};(4,4)*{} **\crv{(-4,-1) & (4,1)}?(0)*\dir{<};
    (4,-4)*{};(-4,4)*{} **\crv{(4,-1) & (-4,1)};
    (-4,4)*{};(4,12)*{} **\crv{(-4,7) & (4,9)}?(1)*\dir{>};
    (4,4)*{};(-4,12)*{} **\crv{(4,7) & (-4,9)};
  (8,8)*{n};(-6,-3)*{\scs };  (6,-3)*{\scs };
   (-4,4)*{\bullet}+(-3,-1)*{\scs y};
   (4,4)*{\bullet}+(3,-1)*{\scs x};
 \endxy}
  \quad + \quad
    \sum_{ \xy  (0,3)*{\scs g_1+g_2+g_3}; (0,0)*{\scs =x+y+-n-1};\endxy}
    \vcenter{\xy 0;/r.18pc/:
    (-8,0)*{};
  (8,0)*{};
  (-4,-15)*{}="b1";
  (4,-15)*{}="b2";
  "b2";"b1" **\crv{(5,-8) & (-5,-8)}; ?(.1)*\dir{>} ?(.95)*\dir{>}
  ?(.8)*\dir{}+(0,-.1)*{\bullet}+(-3,2)*{\scs g_3};
  (-4,15)*{}="t1";
  (4,15)*{}="t2";
  "t2";"t1" **\crv{(5,8) & (-5,8)}; ?(.15)*\dir{<} ?(.9)*\dir{<}
  ?(.4)*\dir{}+(0,-.2)*{\bullet}+(3,-2)*{\scs g_1};
  (0,0)*{\cbub{\scs \quad\; \spadesuit + g_2}{}};
  (-10,10)*{n};
  \endxy} \label{eq_ident_decomp-dots}
\end{eqnarray}

Bubble slide equations:
\begin{eqnarray} \label{eq_bubslide1}
 \xy
  (-5,8)*{n};
  (0,0)*{\bbe{}};
  (12,-2)*{\cbub{\spadesuit+j}{}};
 \endxy
  & =&
     \xy
  (0,8)*{n};
  (12,0)*{\bbe{}};
  (0,-2)*{\cbub{\spadesuit+(j-2)}{}};
  (12,6)*{\bullet}+(3,-1)*{\scs 2};
 \endxy
   -2 \;
         \xy
  (0,8)*{n};
  (12,0)*{\bbe{}};
  (0,-2)*{\cbub{\spadesuit+(j-1)}{}};
  (12,6)*{\bullet}+(8,-1)*{\scs };
 \endxy
 + \;\;
     \xy
  (0,8)*{n};
  (12,0)*{\bbe{}};
  (0,-2)*{\cbub{\spadesuit+j}{}};
  (12,6)*{}+(8,-1)*{\scs };
 \endxy
 \\ \nn \\ \nn \\
  \xy
  (17,8)*{n};
  (12,0)*{\bbe{}};
  (0,-2)*{\ccbub{\spadesuit+j}{}};
  (12,6)*{}+(8,-1)*{\scs };
 \endxy
  &=&
    \xy
  (15,8)*{n};
  (0,0)*{\bbe{}};
  (12,-2)*{\ccbub{\quad\spadesuit+(j-2)}{}};
  (0,6)*{\bullet }+(3,1)*{\scs 2};
 \endxy
  -2 \;
      \xy
  (15,8)*{n};
  (0,0)*{\bbe{}};
  (12,-2)*{\ccbub{\quad\spadesuit+(j-1)}{}};
  (0,6)*{\bullet }+(5,-1)*{\scs };
 \endxy
 + \;\;
      \xy
  (15,8)*{n};
  (0,0)*{\bbe{}};
  (12,-2)*{\ccbub{\spadesuit+j}{}};
  %(0,6)*{\bullet }+(5,-1)*{\scs \alpha-f};
 \endxy
 \end{eqnarray}

\begin{eqnarray}
\xy
  (14,8)*{n};
  (0,0)*{\bbe{}};
  (12,-2)*{\ccbub{\spadesuit+j}{}};
  (0,6)*{ }+(7,-1)*{\scs  };
 \endxy
 & = &
  \xsum{f=0}{j}(j+1-f)
   \xy
  (0,8)*{n+2};
  (12,0)*{\bbe{}};
  (0,-2)*{\ccbub{\spadesuit+f}{}};
  (12,6)*{\bullet}+(5,-1)*{\scs j-f};
 \endxy \label{eq_bubbleslide_cc_r}\\ \nn \\  \nn \\
     \xy
  (15,8)*{n};
  (11,0)*{\bbe{}};
  (0,-2)*{\cbub{\spadesuit+j\quad }{}};
 \endxy
   &= &
     \xsum{f=0}{j}(j+1-f)
     \xy
  (18,8)*{n};
  (0,0)*{\bbe{}};
  (14,-4)*{\cbub{\spadesuit+f}{}};
  (0,6)*{\bullet }+(5,-1)*{\scs j-f};
 \endxy \label{eq_bubslide2}
\end{eqnarray}

Below we collect a few additional identities that have not appeared in the literature previously.

The following relation together with its image under the various symmetries of the 2-category $\Ucat$ will be used extensively in the paper.
\begin{prop}
\begin{equation}
  \vcenter{\xy 0;/r.17pc/:
    (4,-4)*{};(-4,4)*{} **\crv{(4,-1) & (-4,1)}?(0)*\dir{<};
    (-4,-4)*{};(4,4)*{} **\crv{(-4,-1) & (4,1)};
    (-4,4)*{};(-12,12)*{} **\crv{(-4,7) & (-12,9)};
    (-12,4)*{};(-4,12)*{} **\crv{(-12,7) & (-4,9)};
    (4,12)*{};(-4,20)*{} **\crv{(4,15) & (-4,17)}?(1)*\dir{>};;
    (-4,12)*{};(4,20)*{} **\crv{(-4,15) & (4,17)}?(1)*\dir{};
    (4,4)*{}; (4,12) **\dir{-};
    (-12,-4)*{}; (-12,4) **\dir{-}?(0)*\dir{<};
    (-12,12)*{}; (-12,20) **\dir{-};
    (-2,-1)*{\bullet};
  (8,12)*{n};
\endxy}
 \;\; -\;\;
  \vcenter{\xy 0;/r.17pc/:
    (4,-4)*{};(-4,4)*{} **\crv{(4,-1) & (-4,1)}?(0)*\dir{<};
    (-4,-4)*{};(4,4)*{} **\crv{(-4,-1) & (4,1)};
    (-4,4)*{};(-12,12)*{} **\crv{(-4,7) & (-12,9)};
    (-12,4)*{};(-4,12)*{} **\crv{(-12,7) & (-4,9)};
    (4,12)*{};(-4,20)*{} **\crv{(4,15) & (-4,17)}?(1)*\dir{>};;
    (-4,12)*{};(4,20)*{} **\crv{(-4,15) & (4,17)}?(1)*\dir{};
    (4,4)*{}; (4,12) **\dir{-};
    (-12,-4)*{}; (-12,4) **\dir{-}?(0)*\dir{<};
    (-12,12)*{}; (-12,20) **\dir{-};
    (2,-1)*{\bullet};
  (8,12)*{n};
\endxy}
\;\; - \;\;
 \vcenter{ \xy 0;/r.17pc/:
    (-4,-4)*{};(4,4)*{} **\crv{(-4,-1) & (4,1)}?(0)*\dir{<};
    (4,-4)*{};(-4,4)*{} **\crv{(4,-1) & (-4,1)};
    (4,4)*{};(12,12)*{} **\crv{(4,7) & (12,9)};
    (12,4)*{};(4,12)*{} **\crv{(12,7) & (4,9)};
    (-4,12)*{};(4,20)*{} **\crv{(-4,15) & (4,17)}?(1)*\dir{>};
    (4,12)*{};(-4,20)*{} **\crv{(4,15) & (-4,17)}?(1)*\dir{};
    (-4,4)*{}; (-4,12) **\dir{-};
    (12,-4)*{}; (12,4) **\dir{-}?(0)*\dir{<};
    (12,12)*{}; (12,20) **\dir{-};
  (16,12)*{n}; (-4,8)*{\bullet};
\endxy}
\;\; + \;\;
 \vcenter{ \xy 0;/r.17pc/:
    (-4,-4)*{};(4,4)*{} **\crv{(-4,-1) & (4,1)}?(0)*\dir{<};
    (4,-4)*{};(-4,4)*{} **\crv{(4,-1) & (-4,1)};
    (4,4)*{};(12,12)*{} **\crv{(4,7) & (12,9)};
    (12,4)*{};(4,12)*{} **\crv{(12,7) & (4,9)};
    (-4,12)*{};(4,20)*{} **\crv{(-4,15) & (4,17)}?(1)*\dir{>};
    (4,12)*{};(-4,20)*{} **\crv{(4,15) & (-4,17)}?(1)*\dir{};
    (-4,4)*{}; (-4,12) **\dir{-};
    (12,-4)*{}; (12,4) **\dir{-}?(0)*\dir{<};
    (12,12)*{}; (12,20) **\dir{-};
  (16,12)*{n}; (4,12)*{\bullet};
\endxy}
\;\; -\;\;
 \xy 0;/r.17pc/:
    (-12,12)*{};(-12,-12)*{} **\dir{-} ?(.5)*\dir{>};
    (-4,-12)*{};(-4,12)*{} **\dir{-} ?(.5)*\dir{>};
    (4,12)*{};(4,-12)*{} **\dir{-} ?(.5)*\dir{>};
  (8,8)*{n};
  \endxy
+\;
  \xy 0;/r.17pc/:
  (-4,12)*{};(12,-12)*{} **\crv{(-4,4) & (12,-4)}?(.5)*\dir{>};
  (4,12)*{}="t1";
  (12,12)*{}="t2";
  "t1";"t2" **\crv{(4,5) & (12,5)};  ?(1)*\dir{>};
  (-4,-12)*{}="t1";  (4,-12)*{}="t2";
  "t2";"t1" **\crv{(4,-5) & (-4,-5)}; ?(1)*\dir{>};
  (19,4)*{n};
  \endxy
 \;\; =\;\;0.
\end{equation}
\end{prop}

The proof utilizes the nilHecke relations to slide dots as well as \eqref{eq_r3_extra2}.

\begin{prop}[Step functions]
\begin{equation}
 \vcenter{   \xy %0;/r.18pc/:
    (-4,-4)*{};(4,4)*{} **\crv{(-4,-1) & (4,1)}?(1)*\dir{>};
    (4,-4)*{};(-4,4)*{} **\crv{(4,-1) & (-4,1)}?(1)*\dir{};?(0)*\dir{<};
    (-4,4)*{};(4,12)*{} **\crv{(-4,7) & (4,9)};
    (4,4)*{};(-4,12)*{} **\crv{(4,7) & (-4,9)};
    (-4,12)*{};(4,12)*{} **\crv{(-4,18) & (4,18)}?(1)*\dir{>};
  (8,8)*{n};
 \endxy}
 \quad = \quad
 \left\{
\begin{array}{ccl}
  -\;\;\vcenter{\xy (-4,0)*{};(4,0)*{} **\crv{(-4,6) & (4,6)}?(1)*\dir{>};
  (8,8)*{n};
  \endxy} & \quad & \text{if $n \leq 0$} \\ & & \\
  0 & \quad & \text{otherwise.}
\end{array}
 \right.
\end{equation}
\begin{equation}
 \vcenter{   \xy %0;/r.18pc/:
    (4,-4)*{};(-4,4)*{} **\crv{(4,-1) & (-4,1)}?(1)*\dir{>};
    (-4,-4)*{};(4,4)*{} **\crv{(-4,-1) & (4,1)}?(1)*\dir{};?(0)*\dir{<};
    (4,4)*{};(-4,12)*{} **\crv{(4,7) & (-4,9)};
    (-4,4)*{};(4,12)*{} **\crv{(-4,7) & (4,9)};
    (4,12)*{};(-4,12)*{} **\crv{(4,18) & (-4,18)}?(1)*\dir{>};
  (8,8)*{n};
 \endxy}
 \quad = \quad
 \left\{
\begin{array}{ccl}
  -\;\;\vcenter{\xy (4,0)*{};(-4,0)*{} **\crv{(4,6) & (-4,6)}?(1)*\dir{>};
  (8,8)*{n};
  \endxy} & \quad & \text{if $n \geq 0$} \\ & & \\
  0 & \quad & \text{otherwise.}
\end{array}
 \right.
\end{equation}
\end{prop}

\begin{proof}
This proposition follows immediately from the Dotted Curl Relations \eqref{eq_reduction-dots}.
\end{proof}

\begin{prop}
\begin{equation}
 \vcenter{   \xy %0;/r.18pc/:
    (-4,-4)*{};(4,4)*{} **\crv{(-3,-1) & (3,1)}?(1)*\dir{>};
    (4,-4)*{};(-4,4)*{} **\crv{(3,-1) & (-3,1)}?(1)*\dir{>};;
    (-4,4)*{};(-12,4)*{} **\crv{(-4,8) & (-12,8)};
    (-4,-4)*{};(-12,-4)*{} **\crv{(-4,-8) & (-12,-8)};
    (4,4)*{};(12,4)*{} **\crv{(4,8) & (12,8)};
    (4,-4)*{};(12,-4)*{} **\crv{(4,-8) & (12,-8)};
    (12,-4)*{};(12,4)*{} **\dir{-};
    (-12,-4)*{};(-12,4)*{} **\dir{-};
  (2,9)*{n};
 \endxy}
 \quad = \quad
 \left\{
\begin{array}{ccl}
  \vcenter{\xy (0,0)*{\cbub{\spadesuit+1}{}};
  (8,8)*{0};
  \endxy} & \quad & \text{if $n =0$,} \\ & & \\
  0 & \quad & \text{otherwise.}
\end{array}
 \right.
\end{equation}
\end{prop}

\begin{proof}
This proposition follows from the Curl relations in $\Ucat$ together with the positivity of bubbles axiom.
\end{proof}

% ====================================================================
%
\subsection{The 2-categories $\UcatD$, $Kom(\UcatD)$, and $Com(\UcatD)$} \label{sec_KomU}
%
% ====================================================================

% --------------------------------------------------------------------
%
\subsubsection{Additive categories, homotopy categories, and Karoubi envelopes}
%
% --------------------------------------------------------------------

For an additive category $\cal{M}$ we write $Kom(\cal{M})$ for the
category of bounded complexes in $\cal{M}$.  An object $(X,d)$ of
$Kom(\cal{M})$ is a collection of objects $X^i$ of $\cal{M}$
together with maps $d_i$
\begin{equation}
  \xymatrix{\dots \ar[r]^{d} &
   X^{i-1} \ar[r]^{d_{i-1}} &
   X^i \ar[r]^{d_i} &
   X^{i+1} \ar[r]^{d_{i+1}} & \dots}
\end{equation}
such that $d_{i+1} d_i =0$ and only finitely many objects are
nonzero. A morphisms $f \maps (X,d) \to (Y,d)$ in $Kom(\cal{M})$ is
a collection of morphisms $f_i \maps X^i \to Y^i$ such that
\begin{equation}
    \xymatrix{
\dots \ar[r]^{d} &
X^{i-1} \ar[r]^{d_{i-1}} \ar[d]_{f_{i-1}}&
X^i \ar[r]^{d_i} \ar[d]_{f_i}&
 X^{i+1} \ar[r]^{d_{i+1}} \ar[d]_{f_{i+1}}&
 \dots \\
 \dots \ar[r]_{d} &
 Y^{i-1} \ar[r]_{d_{i-1}} &
 Y^i \ar[r]_{d_i} &
 Y^{i+1} \ar[r]_{d_{i+1}} & \dots}
\end{equation}
commutes.

Given a pair of morphisms $f,g \maps (X,d) \to (Y,d)$ in
$Kom(\cal{M})$,  we say that $f$ is {\em homotopic} to $g$ if there
exists morphisms $h^i \maps X^{i} \to Y^{i-1}$ such that $f_i-g_i =
h^{i+1}d_i+d_{i-1}h^i$ for all $i$. A morphism of complexes is said
to be null-homotopic if it is homotopic to the zero map.

\begin{defn}
The homotopy category $Com(\cal{M})$ has the same objects as $Kom(\cal{M})$, and morphisms are morphisms in $Kom(\cal{M})$ modulo null-homotopic morphisms.
\end{defn}

The Karoubi envelope $Kar(\cal{M})$ of a category $\cal{M}$ is an
enlargement of $\cal{M}$ in which all idempotents split. An idempotent
$e \maps b\to b$ in a category $\cal{M}$  is said to split if there
exist morphisms
\[
 \xymatrix{ b \ar[r]^g & b' \ar[r]^h &b}
\]
such that $e=hg$ and $gh = \Id_{b'}$. More precisely, the
Karoubi envelope $Kar(\cal{M})$ is a category whose objects are pairs $(b,e)$
where $e \maps b \to b$ is an idempotent of $\cal{M}$ and whose
morphisms are triples of the form
\[
 (e,f,e') \maps (b,e) \to (b',e')
\]
where $f \maps b \to b'$ in $\cal{M}$ making the diagram
\begin{equation} \label{eq_Kar_morph}
 \xymatrix{
 b \ar[r]^f \ar[d]_e \ar[dr]^{f} & b' \ar[d]^{e'} \\ b \ar[r]_f & b'
 }
\end{equation}
commute. Thus, $f$ must satisfy $f=e'fe$, which is equivalent to $f=e'f=fe$. Composition is induced from the composition in $\cal{M}$,
and the identity morphism is $(e,e,e) \maps (b,e) \to (b,e)$.  When
$\cal{M}$ is an additive category, the splitting of idempotents
allows us to write $(b,e)\in Kar(\cal{M})$ as $\im e$, and we have $b
\cong \im e \oplus \im (\Id_b-e)$.

The identity map $\Id_b\maps b \to b$ is an idempotent, and this gives rise
to a fully faithful functor $\cal{M} \to Kar(\cal{M})$. In
$Kar(\cal{M})$ all idempotents of $\cal{M}$ split and this
functor is universal with respect to functors which split
idempotents in $\cal{M}$.  When $\cal{M}$ is additive the inclusion $\cal{M} \to Kar(\cal{M})$ is an additive functor (see \cite[Section 9]{Lau1} and references therein).

\begin{prop} \label{prop_Kar_Kom}
For any additive category $\cal{M}$ there exists a canonical equivalence
\begin{equation}
  Kom\left(Kar(\cal{M})\right) \cong Kar\left(Kom(\cal{M})\right).
\end{equation}
\end{prop}

\begin{proof}
Define the functor
\begin{equation} \label{eq_rhoM}
  \rho_\cal{M} \maps Kom(Kar(\cal{M})) \to Kar(Kom(\cal{M}))
\end{equation}
as follows.  An object of $Kom(Kar(\cal{M}))$ has the form
\begin{equation}
(X,e) =
  \xymatrix{\dots \ar[r]^-{d_{i-1}} &
   (X^i,e_i) \ar[r]^-{d_i} &
   (X^{i+1},e_{i+1}) \ar[r]^-{d_{i+1}} & \dots}
\end{equation}
where $e_i^2=e_i$, $d_{i+1}d_i=0$, and $d_i=e_{i+1}d_ie_i$.  Here $e_i\maps X^i \to X^i$ is an idempotent and $d_i\maps X^i \to X^{i+1}$.  The functor
$\rho_{\cal{M}}$ takes this object to the pair in $Kar(Kom(\cal{M}))$ consisting of the complex
\begin{equation}
  X =
    \xymatrix{\dots \ar[r]^-{d_{i-1}} &
   X^i \ar[r]^-{d_i} &
   X^{i+1} \ar[r]^-{d_{i+1}} & \dots}
\end{equation}
and the idempotent chain map $(\dots,e_i,e_{i+1}, \dots)$.  A morphism $f \maps (X,e) \to (X',e')$ in $Kom(Kar(\cal{M}))$ is a collection of maps $f_i \maps X^i\to (X')^i$ such that the squares
\begin{equation}
\xymatrix{
  X^i \ar[r]^{d_i} \ar[d]_{f_i} & X^{i+1} \ar[d]^{f_{i+1}} \\ (X')^i \ar[r]_{d'_i} & (X')^{i+1} }
\end{equation}
commute, and $f_i=e'_if_ie_i$.  The functor $\rho_{\cal{M}}$ takes the morphism $f$ to the ``same" morphism $\{f_i\}$ of complexes equipped with idempotents $(\dots, e_i, e_{i+1}, \dots)$ and $(\dots, e'_i, e'_{i+1}, \dots)$.

It is easy to see that $\rho_{\cal{M}}$ is fully-faithful.  To show $\rho_{\cal{M}}$ is essentially surjective note that any object $Y$ of $Kar(Kom(\cal{M}))$ is isomorphic to $\rho_{\cal{M}}(\tilde{Y})$ for some object $\tilde{Y}$ of $Kom(Kar(\cal{M}))$. The object $Y$ consists of a complex
\begin{equation}
  Y =
    \xymatrix{\dots \ar[r]^-{d_{i-1}} &
   Y^i \ar[r]^-{d_i} &
   Y^{i+1} \ar[r]^-{d_{i+1}} & \dots}
\end{equation}
in $Kom(\cal{M})$ together with idempotents $e_i\maps Y^i \to Y^i$ such that
\begin{equation}
  e_{i+1}d_i = d_i e_i.
\end{equation}
Let $\tilde{Y}$ be the object
\begin{equation}
    \xymatrix{\dots \ar[rr]^-{d_{i-1}e_{i-1}} &&
   (Y^i,e_i) \ar[rr]^-{d_i e_i} & &
   (Y^{i+1},e_{i+1}) \ar[rr]^-{d_{i+1}e_{i+1}} & & \dots}
\end{equation}
in $Kom(Kar(\cal{M}))$. We can define morphisms $\phi_1 \maps Y \to \rho_{\cal{M}}(\tilde{Y})$ and $\phi_2 \maps \rho_{\cal{M}}(\tilde{Y}) \to Y$ in $Kar(Kom(\cal{M}))$ that are both given on $Y^i$ as multiplication by $e_i$.
Then $\phi_2 \phi_1 = \Id_{Y}$ and $\phi_1 \phi_2 =\Id_{\rho_{\cal{M}}(\tilde{Y})}$, showing that $Y$ and $\rho_{\cal{M}}(\tilde{Y})$ are isomorphic.  Together with fully-faithfulness of $\rho_{\cal{M}}$ this completes the proof that $\rho_{\cal{M}}$ is an equivalence of categories.
\end{proof}

\begin{prop} \label{prop_Kar_Com}
For any additive $\Bbbk$-linear category $\cal{M}$ with finite dimensional hom spaces there exists a canonical equivalence
\begin{equation}
  Com\left(Kar(\cal{M})\right) \cong Kar\left(Com(\cal{M})\right).
\end{equation}
\end{prop}

\begin{proof}
The functor $\rho_{\cal{M}}$ descends to a functor
\begin{equation}
  \rho^c_{\cal{M}} \maps Com(Kar(\cal{M})) \to
  Kar(Com(\cal{M})).
\end{equation}
Given an object $X$ of $Com(Kar(\cal{M}))$ we can view it as an object of $Kom(Kar(\cal{M}))$.  By the idempotent-lifting property for finite-dimensional algebras \cite[Chapter 1]{Benson} an idempotent of $Com(Kar(\cal{M}))$ lifts to an idempotent in $Kom(Kar(\cal{M}))$, and the latter category is idempotent-complete by the previous proposition.  Therefore, $Com(Kar(\cal{M}))$ is idempotent-complete as well, allowing us to define a functor
\begin{equation}
  \tilde{\rho}_{\cal{M}} \maps Kar(Com(\cal{M})) \to
  Kar(Com(Kar(\cal{M}))) \simeq Com(Kar(\cal{M}))
\end{equation}
such that $\rho^c_{\cal{M}}\tilde{\rho}_{\cal{M}} \cong \Id_{Kar(Com(\cal{M}))}$, $\tilde{\rho}_{\cal{M}}\rho^c_{\cal{M}} \cong \Id_{Com(Kar(\cal{M}))}$, showing that $\rho^c_{\cal{M}}$ is an equivalence.
\end{proof}

Alternatively, the result follows from \cite[Corollary 2.12]{BS}.

% --------------------------------------------------------------------
%
\subsubsection{Karoubian envelope of $\Ucat$}
%
% --------------------------------------------------------------------

\begin{defn}
Define the additive 2-category $\UcatD$ to have the
same objects as $\Ucat$ and hom additive categories
given by $\UcatD(n,m) =
Kar\left(\Ucat(n,m)\right)$. The fully-faithful
additive functors $\Ucat(n,m) \to
\UcatD(n,m)$ combine to form an additive
2-functor $\Ucat \to \UcatD$ universal with respect to splitting
idempotents in the hom categories $\UcatD(n,m)$.  The
composition functor $\UcatD(n,n') \times
\UcatD(n',n'') \to \UcatD(n,n'')$ is induced
by the universal property of the Karoubi envelope from the
composition functor for $\Ucat$. The 2-category $\UcatD$ has graded
2-homs given by
\begin{equation}
\HOM_{\UcatD}(x,y) := \bigoplus_{t\in \Z}\Hom_{\UcatD}(x\{t\},y).
\end{equation}
\end{defn}

\begin{thm}[Theorem 9.1.3 \cite{Lau1}]  \label{thm_Groth}
There is an isomorphism
\begin{equation}
  \gamma \maps \U_{\cal{A}} \to K_0(\UcatD),
\end{equation}
where $K_0(\UcatD)$ is the split Grothendieck ring of $\UcatD$.
\end{thm}

In \cite[Corollary 5.14]{KLMS} this result is proven when the ground field $\Bbbk$ is replaced by the commutative ring $\Z$.

% ====================================================================
%
\subsubsection{Karoubian envelopes of $Kom(\Ucat)$ and $Com(\Ucat)$}
%
% ====================================================================

\begin{defn}
Define $Kom(\Ucat)$ to be the additive 2-category with objects $n \in \Z$ and additive hom categories $Kom(\Ucat)(n,m):=Kom\left( \Ucat\left( n,m\right)\right)$.  The additive composition functor
$Kom\left(\Ucat(n',n'')\right)\times Kom\left(\Ucat(n,n')\right) \to Kom(\Ucat(n,n''))$ is given by the tensor product of complexes using the additive composition functor on $\Ucat$ to tensor 1-morphisms via composition.
\end{defn}

\begin{defn}
Define $Com(\Ucat)$ to be the additive 2-category with the same objects and 1-morphisms as $Kom(\Ucat)$ and 2-morphisms given by identifying homotopy equivalent 2-morphisms in $Kom(\Ucat)$.
\end{defn}

Recall that $\UcatD=Kar(\Ucat)$. By Propositions~\ref{prop_Kar_Kom} and \ref{prop_Kar_Com} there are equivalences
\begin{equation}
  Kar(Kom(\Ucat)) \cong Kom(\UcatD), \qquad Kar(Com(\Ucat)) \cong Com(\UcatD).
\end{equation}

The 2-categories we consider fit into the following table where the horizontal arrows denote passage to the Karoubian envelope and vertical arrows stand for passage to complexes and modding out by null-homotopic maps.
\begin{equation} \label{eq_all_2cats}
 \xy
  (-25,20)*+{\Ucat}="tl";
  (25,20)*+{\UcatD = Kar(\Ucat)}="tr";
  (-25,0)*+{Kom(\Ucat)}="ml";
  (25,0)*+{Kom(\UcatD) \cong Kar(Kom(\Ucat))}="mr";
  (-25,-20)*+{Com(\Ucat)}="bl";
  (25,-20)*+{Com(\UcatD) \cong Kar(Com(\Ucat))}="br";
    {\ar "tl";"ml"};
    {\ar "ml";"bl"};
    {\ar "tr";"mr"};
    {\ar "mr";"br"};
    {\ar "tl";"tr"};
    {\ar "ml";"mr"};
    {\ar "bl";"br"};
 \endxy
\end{equation}

% ====================================================================
%
\subsection{Symmetry 2-functors} \label{sec_sym}
%
% ====================================================================

A covariant/contravariant functor $\alpha \maps \cal{M} \to \cal{M}'$ extends canonically to a functor $$Kar(\alpha)\maps Kar(\cal{M}) \to Kar(\cal{M}').$$  An additive covariant/contravariant functor $\alpha \maps \cal{M} \to \cal{M}'$  between additive categories extends canonically to an additive functor $$Kom(\alpha) \maps Kom(\cal{M}) \to Kom(\cal{M}')$$ and an exact functor $Com(\alpha) \maps Com(\cal{M}) \to Com(\cal{M}')$ between triangulated categories.

Given an exact endofunctor $\alpha \maps \cal{M} \to \cal{M}$ these extensions respect the equivalence $\rho_{\cal{M}}$ in \eqref{eq_rhoM}, in the sense that the diagrams
\begin{eqnarray}
  \xy
   (-25,10)*+{Kom(Kar(\cal{M}))}="tl";
   (25,10)*+{Kar(Kom(\cal{M}))}="tr";
   (-25,-10)*+{Kom(Kar(\cal{M}))}="bl";
   (25,-10)*+{Kar(Kom(\cal{M}))}="br";
    {\ar^{\rho_{\cal{M}}} "tl";"tr"};
    {\ar_{Kom(Kar(\alpha))} "tl";"bl"};
    {\ar_{\rho_{\cal{M}}} "bl";"br"};
    {\ar^{Kar(Kom(\alpha))} "tr";"br"};
  \endxy
\\ \nn \\ \nn \\
  \xy
   (-25,10)*+{Com(Kar(\cal{M}))}="tl";
   (25,10)*+{Kar(Com(\cal{M}))}="tr";
   (-25,-10)*+{Com(Kar(\cal{M}))}="bl";
   (25,-10)*+{Kar(Com(\cal{M}))}="br";
    {\ar^{\rho_{\cal{M}}} "tl";"tr"};
    {\ar_{Com(Kar(\alpha))} "tl";"bl"};
    {\ar_{\rho_{\cal{M}}} "bl";"br"};
    {\ar^{Kar(Com(\alpha))} "tr";"br"};
  \endxy
\end{eqnarray}
commute.

In this section we recall several 2-functor involutions $\tomega$, $\tpsi$, $\tsigma$ on the 2-category $\Ucat$ defined in \cite{Lau1} and extend them to 2-functors on all the 2-categories in \eqref{eq_all_2cats}.
We use the same notation for these extended 2-functors.

Denote by $\Ucat^{\op}$ the 2-category with the same objects as
$\Ucat$ but the 1-morphisms reversed.  The direction of the
2-morphisms remain fixed. The 2-category $\Ucat^{\co}$ has the same
objects and 1-morphism as $\Ucat$, but the directions of the
2-morphisms is reversed. That is, $\Ucat^{\co}(x,y)=\Ucat(y,x)$ for
1-morphisms $x$ and $y$. Finally, $\Ucat^{\co\op}$ denotes the
2-category with the same objects as $\Ucat$, but the directions of
the 1-morphisms and 2-morphisms have been reversed.

Using the symmetries of the diagrammatic relations imposed on
$\Ucat$ 2-functors were defined in \cite{Lau1} that categorify various $\Z[q,q^{-1}]$-(anti)linear
(anti)automorphisms of the algebra $\U$. The various forms of
contravariant behaviour for 2-functors on $\Ucat$ translate into
properties of the corresponding homomorphism in $\U$ as the
following table summarizes:
\begin{center}
\begin{tabular}{|l|l|}
  \hline
  % after \\: \hline or \cline{col1-col2} \cline{col3-col4} ...
  {\bf 2-functors} & {\bf Algebra maps} \\ \hline \hline
  $\Ucat \to \Ucat$ &  $\Z[q,q^{-1}]$-linear
 homomorphisms\\
  $\Ucat \to \Ucat^{\op}$ & $\Z[q,q^{-1}]$-linear
antihomomorphisms \\
  $\Ucat \to \Ucat^{\co}$ & $\Z[q,q^{-1}]$-antilinear
 homomorphisms \\
  $\Ucat \to \Ucat^{\co\op}$ & $\Z[q,q^{-1}]$-antilinear
antihomomorphisms \\
  \hline
\end{tabular}
\end{center}

\paragraph{Rescale, invert the orientation, and send $n \mapsto -n$:}

Consider the operation on the diagrammatic calculus that rescales
the crossing $\Ucross \mapsto -\Ucross$, inverts
the orientation of each strand and sends $n \mapsto -n$:
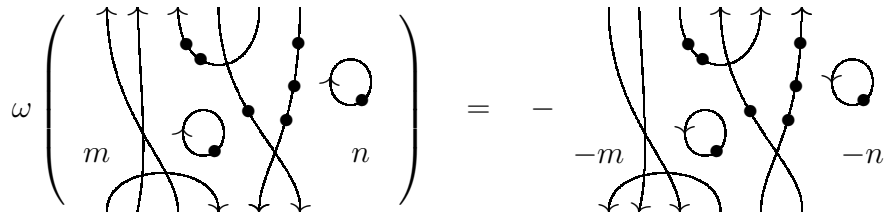
\[
\tomega\left(\;\;\vcenter{\xy 0;/r.16pc/:
 (-4,-15)*{}; (-20,25) **\crv{(-3,-6) & (-20,4)}?(0)*\dir{<}?(.6)*\dir{}+(0,0)*{\bullet};
 (-12,-15)*{}; (-4,25) **\crv{(-12,-6) & (-4,0)}?(0)*\dir{<}?(.6)*\dir{}+(.2,0)*{\bullet};
 ?(0)*\dir{<}?(.75)*\dir{}+(.2,0)*{\bullet};?(0)*\dir{<}?(.9)*\dir{}+(0,0)*{\bullet};
 (-28,25)*{}; (-12,25) **\crv{(-28,10) & (-12,10)}?(0)*\dir{<};
  ?(.2)*\dir{}+(0,0)*{\bullet}?(.35)*\dir{}+(0,0)*{\bullet};
 (-36,-15)*{}; (-36,25) **\crv{(-34,-6) & (-35,4)}?(1)*\dir{>};
 (-28,-15)*{}; (-42,25) **\crv{(-28,-6) & (-42,4)}?(1)*\dir{>};
 (-42,-15)*{}; (-20,-15) **\crv{(-42,-5) & (-20,-5)}?(1)*\dir{>};
 (6,10)*{\cbub{}{}};
 (-23,0)*{\cbub{}{}};
 (8,-4)*{n};(-44,-4)*{m};
 \endxy}\;\;\right) \quad = \quad - \;\;
 \vcenter{\xy 0;/r.16pc/:
 (-4,-15)*{}; (-20,25) **\crv{(-3,-6) & (-20,4)}?(1)*\dir{>}?(.6)*\dir{}+(0,0)*{\bullet};
 (-12,-15)*{}; (-4,25) **\crv{(-12,-6) & (-4,0)}?(1)*\dir{>}?(.6)*\dir{}+(.2,0)*{\bullet};
 ?(1)*\dir{>}?(.75)*\dir{}+(.2,0)*{\bullet};?(.9)*\dir{}+(0,0)*{\bullet};
 (-28,25)*{}; (-12,25) **\crv{(-28,10) & (-12,10)}?(1)*\dir{>};
  ?(.2)*\dir{}+(0,0)*{\bullet}?(.35)*\dir{}+(0,0)*{\bullet};
 (-36,-15)*{}; (-36,25) **\crv{(-34,-6) & (-35,4)}?(0)*\dir{<};
 (-28,-15)*{}; (-42,25) **\crv{(-28,-6) & (-42,4)}?(0)*\dir{<};
 (-42,-15)*{}; (-20,-15) **\crv{(-42,-5) & (-20,-5)}?(0)*\dir{<};
 (6,10)*{\ccbub{}{}};
 (-23,0)*{\ccbub{}{}};
 (8,-4)*{-n};(-44,-4)*{-m};
 \endxy}
\]
This gives a strict invertible 2-functor
$\tomega\maps \Ucat \to \Ucat$
\begin{eqnarray}
  \tomega \maps \Ucat &\to& \Ucat \nn \\
  n &\mapsto&  -n \nn \\
  \onem\cal{E}^{\alpha_1} \cal{F}^{\beta_1}\cal{E}^{\alpha_2} \cdots
 \cal{E}^{\alpha_k}\cal{F}^{\beta_k}\onen\{s\}
 &\mapsto &
 \mathbf{1}_{-m} \cal{F}^{\alpha_1} \cal{E}^{\beta_1}\cal{F}^{\alpha_2} \cdots
\cal{F}^{\alpha_k}\cal{E}^{\beta_k}\mathbf{1}_{-n}\{s\}.
\end{eqnarray}
This 2-functor extends to a 2-endofunctor
\begin{eqnarray}
  \tomega \maps Kom(\Ucat) &\to& Kom(\Ucat), \nn \\
  n &\mapsto&  -n \nn \\
 (X,d)
 &\mapsto &
   \xy
    (-50,0)*+{\cdots}="1";
    (-30,0)*+{\tomega(X^{i-1})}="2";
    (0,0)*+{\tomega(X^i)}="3";
    (30,0)*+{ \tomega(X^{i+1})}="4";
    (50,0)*+{\cdots}="5";
    {\ar^-{} "1";"2"};
    {\ar^-{\tomega(d_{i-1})} "2";"3"};
    {\ar^-{\tomega(d_i)} "3";"4"};
    {\ar^-{} "4";"5"};
   \endxy \nn \\
   f_i \maps X \to Y & \mapsto &
   \tomega(f_i) \maps \tomega(X) \to \tomega(Y),
\end{eqnarray}
and a 2-endofunctor on all the other 2-categories in \eqref{eq_all_2cats}.

\paragraph{Rescale, reflect across the vertical axis, and send $n \mapsto -n$: }

The operation on diagrams that rescales the crossing $\Ucross
\mapsto -\Ucross$, reflects a diagram across the
vertical axis, and sends $n$ to $-n$ leaves invariant the relations on the
2-morphisms of $\Ucat$. This operation
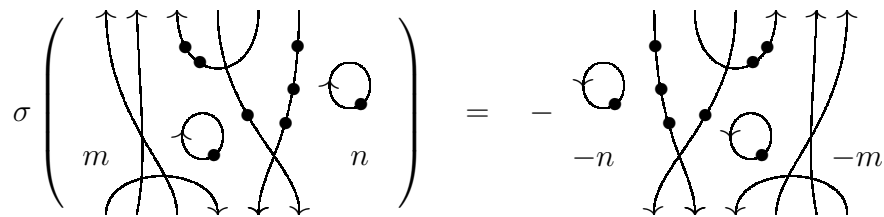
\[
\tsigma\left(\;\;\vcenter{\xy 0;/r.16pc/:
 (-4,-15)*{}; (-20,25) **\crv{(-3,-6) & (-20,4)}?(0)*\dir{<}?(.6)*\dir{}+(0,0)*{\bullet};
 (-12,-15)*{}; (-4,25) **\crv{(-12,-6) & (-4,0)}?(0)*\dir{<}?(.6)*\dir{}+(.2,0)*{\bullet};
 ?(0)*\dir{<}?(.75)*\dir{}+(.2,0)*{\bullet};?(0)*\dir{<}?(.9)*\dir{}+(0,0)*{\bullet};
 (-28,25)*{}; (-12,25) **\crv{(-28,10) & (-12,10)}?(0)*\dir{<};
  ?(.2)*\dir{}+(0,0)*{\bullet}?(.35)*\dir{}+(0,0)*{\bullet};
 (-36,-15)*{}; (-36,25) **\crv{(-34,-6) & (-35,4)}?(1)*\dir{>};
 (-28,-15)*{}; (-42,25) **\crv{(-28,-6) & (-42,4)}?(1)*\dir{>};
 (-42,-15)*{}; (-20,-15) **\crv{(-42,-5) & (-20,-5)}?(1)*\dir{>};
 (6,10)*{\cbub{}{}};
 (-23,0)*{\cbub{}{}};
 (8,-4)*{n};(-44,-4)*{m};
 \endxy}\;\;\right) \quad = \quad -\;\;
 \vcenter{\xy 0;/r.16pc/:
 (4,-15)*{}; (20,25) **\crv{(3,-6) & (20,4)}?(0)*\dir{<}?(.6)*\dir{}+(0,0)*{\bullet};
 (12,-15)*{}; (4,25) **\crv{(12,-6) & (4,0)}?(0)*\dir{<}?(.6)*\dir{}+(.2,0)*{\bullet};
 ?(0)*\dir{<}?(.75)*\dir{}+(.2,0)*{\bullet};?(0)*\dir{<}?(.9)*\dir{}+(0,0)*{\bullet};
 (28,25)*{}; (12,25) **\crv{(28,10) & (12,10)}?(0)*\dir{<};
  ?(.2)*\dir{}+(0,0)*{\bullet}?(.35)*\dir{}+(0,0)*{\bullet};
 (36,-15)*{}; (36,25) **\crv{(34,-6) & (35,4)}?(1)*\dir{>};
 (28,-15)*{}; (42,25) **\crv{(28,-6) & (42,4)}?(1)*\dir{>};
 (42,-15)*{}; (20,-15) **\crv{(42,-5) & (20,-5)}?(1)*\dir{>};
 (-6,10)*{\ccbub{}{}};
 (23,0)*{\ccbub{}{}};
 (-8,-4)*{-n};(44,-4)*{-m};
 \endxy}
\]
is contravariant for composition of 1-morphisms, covariant for
composition of 2-morphisms, and preserves the degree of a diagram.
This symmetry gives an invertible 2-functor
\begin{eqnarray}
  \tsigma \maps \Ucat &\to& \Ucat^{\op}, \nn \\
  n &\mapsto&  -n \nn \\
  \onem\cal{E}^{\alpha_1} \cal{F}^{\beta_1}\cal{E}^{\alpha_2} \cdots
 \cal{E}^{\alpha_k}\cal{F}^{\beta_k}\onen\{s\}
 &\mapsto &
 \mathbf{1}_{-n} \cal{F}^{\beta_k} \cal{E}^{\alpha_k}\cal{F}^{\beta_{k-1}} \cdots
\cal{F}^{\beta_1}\cal{E}^{\alpha_1}\mathbf{1}_{-m}\{s\} \nn
\end{eqnarray}
that acts on 2-morphisms via the symmetry described above.  This 2-functor extends to a 2-functor
\begin{eqnarray}
  \tsigma \maps Kom(\Ucat) &\to& Kom(\Ucat) \nn \\
  n &\mapsto&  -n \nn \\
 (X,d)
 &\mapsto &
   \xy
    (-50,0)*+{\cdots}="1";
    (-30,0)*+{\tsigma(X^{i-1})}="2";
    (0,0)*+{\tsigma(X^i)}="3";
    (30,0)*+{ \tsigma(X^{i+1})}="4";
    (50,0)*+{\cdots}="5";
    {\ar^-{} "1";"2"};
    {\ar^-{\tsigma(d_{i-1})} "2";"3"};
    {\ar^-{\tsigma(d_i)} "3";"4"};
    {\ar^-{} "4";"5"};
   \endxy \nn \\
   f_i \maps X \to Y & \mapsto &
   \tsigma(f_i) \maps \tsigma(X) \to \tsigma(Y),
\end{eqnarray}
and, likewise, a 2-endofunctor on all the other 2-categories from \eqref{eq_all_2cats}.

\paragraph{Reflect across the x-axis and invert orientation:}
Here we are careful to keep track of what happens to the shifts of
sources and targets
\[
\tpsi\left(\;\;\vcenter{\xy 0;/r.16pc/:
 (-4,-15)*{}; (-20,25) **\crv{(-3,-6) & (-20,4)}?(0)*\dir{<}?(.6)*\dir{}+(0,0)*{\bullet};
 (-12,-15)*{}; (-4,25) **\crv{(-12,-6) & (-4,0)}?(0)*\dir{<}?(.6)*\dir{}+(.2,0)*{\bullet};
 ?(0)*\dir{<}?(.75)*\dir{}+(.2,0)*{\bullet};?(0)*\dir{<}?(.9)*\dir{}+(0,0)*{\bullet};
 (-28,25)*{}; (-12,25) **\crv{(-28,10) & (-12,10)}?(0)*\dir{<};
  ?(.2)*\dir{}+(0,0)*{\bullet}?(.35)*\dir{}+(0,0)*{\bullet};
 (-36,-15)*{}; (-36,25) **\crv{(-34,-6) & (-35,4)}?(1)*\dir{>};
 (-28,-15)*{}; (-42,25) **\crv{(-28,-6) & (-42,4)}?(1)*\dir{>};
 (-42,-15)*{}; (-20,-15) **\crv{(-42,-5) & (-20,-5)}?(1)*\dir{>};
 (6,10)*{\cbub{}{}};
 (-23,0)*{\cbub{}{}};
 (8,-4)*{n};(-44,-4)*{m};
  (3,-19)*{\scs \{t\}};(3,29)*{\scs \{t'\}};
 \endxy}\;\;\right) \quad = \quad
  \vcenter{\xy 0;/r.16pc/:
 (-4,15)*{}; (-20,-25) **\crv{(-3,6) & (-20,-4)}?(1)*\dir{>}?(.6)*\dir{}+(0,0)*{\bullet};
 (-12,15)*{}; (-4,-25) **\crv{(-12,6) & (-4,0)}?(1)*\dir{>}?(.6)*\dir{}+(.2,0)*{\bullet};
 ?(1)*\dir{>}?(.75)*\dir{}+(.2,0)*{\bullet};?(.9)*\dir{}+(0,0)*{\bullet};
 (-28,-25)*{}; (-12,-25) **\crv{(-28,-10) & (-12,-10)}?(1)*\dir{>};
  ?(.2)*\dir{}+(0,0)*{\bullet}?(.35)*\dir{}+(0,0)*{\bullet};
 (-36,15)*{}; (-36,-25) **\crv{(-34,6) & (-35,-4)}?(0)*\dir{<};
 (-28,15)*{}; (-42,-25) **\crv{(-28,6) & (-42,-4)}?(0)*\dir{<};
 (-42,15)*{}; (-20,15) **\crv{(-42,5) & (-20,5)}?(0)*\dir{<};
 (6,-10)*{\cbub{}{}};
 (-23,0)*{\cbub{}{}};
 (8,4)*{n};(-44,4)*{m};
 (3,-29)*{\scs \;\;\{-t'\}};(3,19)*{\scs\;\; \{-t\}};
 \endxy}
\]
Shift reversals on the right hand side are required for this
transformation to preserve the degree of a diagram.  This gives an invertible 2-functor
\begin{eqnarray}
  \tpsi \maps \Ucat &\to& \Ucat^{\co}, \nn \\
  n &\mapsto&  n \nn \\
  \onem\cal{E}^{\alpha_1} \cal{F}^{\beta_1}\cal{E}^{\alpha_2} \cdots
 \cal{E}^{\alpha_k}\cal{F}^{\beta_k}\onen\{s\}
 &\mapsto &
 \onem\cal{E}^{\alpha_1} \cal{F}^{\beta_1}\cal{E}^{\alpha_2} \cdots
\cal{E}^{\alpha_k}\cal{F}^{\beta_k}\onen\{-s\},
\end{eqnarray}
and on 2-morphisms $\tpsi$ reflects the diagrams across the $x$-axis and inverts the orientation.

Since $\tpsi$ is contravariant on 2-morphisms in $\Ucat$, this 2-functor extends to a 2-functor
\begin{eqnarray}
  \tpsi \maps Kom(\Ucat) &\to& Kom(\Ucat), \nn \\
  n &\mapsto&  n \nn \\
 (X,d)
 &\mapsto &
   \xy
    (-50,0)*+{\cdots}="1";
    (-30,0)*+{\tpsi(X^{i+1})}="2";
    (0,0)*+{\tpsi(X^i)}="3";
    (30,0)*+{ \tpsi(X^{i-1})}="4";
    (50,0)*+{\cdots}="5";
    {\ar^-{} "1";"2"};
    {\ar^-{\tpsi(d_{i})} "2";"3"};
    {\ar^-{\tpsi(d_{i-1})} "3";"4"};
    {\ar^-{} "4";"5"};
   \endxy \nn \\
   f_i \maps X \to Y & \mapsto &
   \tpsi(f_i) \maps \tpsi(Y) \to \tpsi(X),
\end{eqnarray}
and, likewise, a 2-endofunctor on all 2-categories in \eqref{eq_all_2cats}. Notice that $\tpsi$ inverts the homological degree so that
$\tpsi$ acts on a complex $(X,d)$ in $Kom(\Ucat)$
by $(\tpsi X)^{i}=\tpsi(X^{-i})$.

\bigskip

These 2-functors are involutive and commute with each other `on-the-nose':
\begin{equation}
  \tomega \tsigma = \tomega \tsigma, \qquad \tsigma \tpsi = \tpsi \tsigma, \qquad
  \tomega \tpsi = \tpsi \tomega,
\end{equation}
generating a group $\cal{G} = (\Z_2)^3$ of 2-functors acting on all the 2-categories in \eqref{eq_all_2cats}. Equivalences in table \eqref{eq_all_2cats} respect this action.  On the Grothendieck group of $\UcatD$ the 2-functors $\tpsi$, $\tomega$, $\tsigma$ descend to (anti)involutions $\und{\psi}$, $\und{\omega}$, and $\und{\sigma}$ on $\UA$.

% ###############################################################################
%
\section{The Casimir complex}
%
% ###############################################################################

%==============================================================================
%
\subsection{The Casimir complex and its symmetries} \label{subsec_casimir_complex}
%
% ==============================================================================

We sometimes represent the Casimir complex \eqref{eq_casimir_EF_INTRO} using the notation
\begin{equation}
  \xy
  (-60,0)*+{ \scs
    \E{}\F{} \onen \{2\}
    \oplus \onen \{1-n\}  }="1";
  (-5,0)*+{ \scs
     \und{ \E{}\F{} \onen
      \oplus \E{}\F{} \onen}}="3";
  (60,0)*+{ \scs
    \E{}\F{} \onen \{-2\} \oplus\onen \{n-1\}.
     }="5";
   {\ar^-{
  \left(
    \begin{array}{cc}
      \text{$\Uupdot\Udown$} & \Ucupl \\ & \\
      \text{$\Uup\Udowndot $} & \Ucupl \\
    \end{array}
  \right)
   } "1";"3"};
   {\ar^-{
  \left(
    \begin{array}{cc}
      -\;\Uup\Udowndot  &  \text{$\Uupdot\Udown $}  \\ & \\
      \Ucapr & -\;\Ucapr\\
    \end{array}
  \right)
   } "3";"5"}; (-72,15)*{\cal{C}\onen :=};
 \endxy\nn
\end{equation}
or
\begin{equation} \label{eq_4_casimir}
\xy
  (-50,0)*+{\left(\begin{array}{c}
      \scs  \E\F \onen \{2\}\\ \scs \onen \{1-n\}
    \end{array}\right) }="1";
  (0,0)*+{\und{\left(\begin{array}{c}
      \scs  \E\F \onen\\ \scs \E\F \onen
    \end{array}\right)} }="3";
  (60,0)*+{\left(\begin{array}{c}
      \scs \E\F \onen \{-2\}\\ \scs \onen \{n-1\}
    \end{array}\right) }="5";
   {\ar^-{
  \left(
    \begin{array}{cc}
      \text{$\Uupdot\Udown$} & \Ucupl \\ & \\
      \text{$\Uup\Udowndot $} & \Ucupl \\
    \end{array}
  \right)
   } "1";"3"};
   {\ar^-{
  \left(
    \begin{array}{cc}
      -\;\Uup\Udowndot  &  \text{$\Uupdot\Udown $}  \\ & \\
      \Ucapr & -\;\Ucapr\\
    \end{array}
  \right)
   } "3";"5"}; (-72,15)*{\cal{C}\onen :=};
 \endxy
\end{equation}
We will interchange freely between these methods of depicting the complex $\cal{C}\onen$.

The placement of minus signs in the above complex is arbitrary as long as each square anticommutes.  In fact, we get different placements of the minus signs and dots using the symmetry 2-functors defined in Section~\ref{sec_sym}:
\begin{align}
 \tpsi(\cal{C}\onen) &=
   \xy
  (-60,0)*+{
  \left(\begin{array}{c}
      \scs  \E{}\F{} \onen \{2\}\\ \scs \onen \{1-n\}
    \end{array}\right) }="1";
  (0,0)*+{\und{\left(\begin{array}{c}
      \scs  \E{}\F{} \onen\\ \scs \E{}\F{} \onen
    \end{array}\right)}  }="3";
  (50,0)*+{\left(\begin{array}{c}
      \scs  \E{}\F{} \onen \{-2\}\\ \scs \onen \{n-1\}
    \end{array}\right) }="5";
   {\ar^-{
  \left(
    \begin{array}{cc}
      \text{$-\;\Uup\Udowndot$} & \Ucupl \\ & \\
      \text{$\Uupdot\Udown $} & -\;\Ucupl \\
    \end{array}
  \right)
   } "1";"3"};
   {\ar^-{
  \left(
    \begin{array}{cc}
      \text{$\Uupdot\Udown$}  &  \text{$\Uup\Udowndot $}  \\ & \\
      \Ucapr & \Ucapr\\
    \end{array}
  \right)
   } "3";"5"}; %(-72,15)*{  \widetilde{\cal{C}'\onen} :=};
 \endxy \nn \\
 \label{eq_casimir_EF_tilde}
\end{align}
\begin{align}
 \tsigma\tomega(\cal{C}\onen) &= \xy
 (-50,0)*+{
    \left(\begin{array}{c}
      \scs  \E\F \onen \{2\} \\ \scs  \onen \{1-n\}
    \end{array}\right)
 }="1";
 (0,0)*+{
   \und{\left(\begin{array}{c}
   \scs   \scs   \E\F \onen \\ \scs \E\F \onen
 \end{array}\right)} }="3";
  (60,0)*+{
   \left(\begin{array}{c}
      \scs     \E\F \onen \{-2\} \\ \scs \onen \{n-1\}
    \end{array}\right)      }="5";
   {\ar^-{
  \left( \begin{array}{cc}
      \text{$\Uup\Udowndot$} & \Ucupl \\ & \\
      \text{$\Uupdot \Uup$} & \Ucupl \\
    \end{array}  \right)
   } "1";"3"};
   {\ar^-{
  \left(
    \begin{array}{cc}
      -\;\Uupdot \Udown &  \text{$\Uup \Udowndot$}  \\ & \\
      \Ucapr & -\;\Ucapr\\
    \end{array}
  \right)
   } "3";"5"};
  %(-72,15)*{\cal{C}\onen :=};
 \endxy \nn \\
 \label{eq_casimir_EF_bar}
 \end{align}
 \begin{align}
  \tsigma \tomega\tpsi(\cal{C}\onen) &=
   \xy
  (-60,0)*+{
  \left(\begin{array}{c}
      \scs  \E{}\F{} \onen \{2\}\\ \scs \onen \{1-n\}
    \end{array}\right) }="1";
  (0,0)*+{\und{\left(\begin{array}{c}
      \scs  \E{} \F{}\onen\\ \scs \E{}\F{} \onen
    \end{array}\right)}  }="3";
  (50,0)*+{\left(\begin{array}{c}
      \scs  \E{}\F{} \onen \{-2\}\\ \scs \onen \{n-1\}
    \end{array}\right) }="5";
   {\ar^-{
  \left(
    \begin{array}{cc}
      \text{$-\;\Uupdot\Udown$} & \Ucupl \\ & \\
      \text{$\Uup\Udowndot $} & -\;\Ucupl \\
    \end{array}
  \right)
   } "1";"3"};
   {\ar^-{
  \left(
    \begin{array}{cc}
      \text{$\Uup\Udowndot$}  &  \text{$\Uupdot\Udown $}  \\ & \\
      \Ucapr & \Ucapr\\
    \end{array}
  \right)
   } "3";"5"}; %(-72,15)*{  \widetilde{\cal{C}'\onen} :=};
 \endxy \nn \\
 \end{align}
It is trivial to check (see also Proposition~\ref{prop_homotopy_sym} below) that these complexes are all isomorphic to the Casimir complex $\cal{C}\onen$.  We will write $\cal{G}_1=\{\Id, \tpsi, \tomega\tsigma, \tpsi\tomega\tsigma\}$ for the subgroup of symmetry 2-functors in $\cal{G}=(\Z_2)^3$ that preserve the Casimir complex.

Just as the symmetries in $G\setminus G_1=\{\und{\omega}, \und{\sigma}, \und{\psi}\und{\omega}, \und{\psi}\und{\sigma}\}$ interchange between the two forms of the Casimir element
\begin{eqnarray}
 \xy
  (-40,0)*++{(-q^2+2-q^{-2})EF-q^{-1}K-qK^{-1}}="1";
  (40,0)*++{(-q^2+2-q^{-2})FE-qK-q^{-1}K^{-1},}="2";
    {\ar@/^3pc/^{G\setminus G_1} "1";"2"};
    {\ar@/^3pc/^{G\setminus G_1} "2";"1"};
 \endxy \nn
\end{eqnarray}
we can write down a categorification of the idempotented  Casimir element in the form
\begin{equation}
  (-q^2+2-q^{-2})FE1_n-(q^{-n-1}+q^{1+n})1_n
\end{equation}
by applying symmetry 2-functors in $\cal{G}\setminus\cal{G}_1=\{\tomega, \tsigma, \tpsi\tomega, \tpsi\tomega\}$ to the Casimir complex $\cal{C}\onenn{-n}$.  Depending on which 2-functor in $\cal{G}\setminus\cal{G}_1$ is chosen, we will get a different placement of minus signs and dots:
\begin{align}
\tsigma(\cal{C}\onenn{-n}) &= \xy
 (-50,0)*+{
    \left(\begin{array}{c}
      \scs  \F\E \onen \{2\} \\ \scs  \onen \{1+n\}
    \end{array}\right)
 }="1";
 (0,0)*+{
   \und{\left(\begin{array}{c}
   \scs   \scs   \F\E \onen \\ \scs \F\E \onen
 \end{array}\right)} }="3";
  (60,0)*+{
   \left(\begin{array}{c}
      \scs     \F\E \onen \{-2\} \\ \scs \onen \{-n-1\}
    \end{array}\right)      }="5";
   {\ar^-{
  \left( \begin{array}{cc}
      \text{$\Udown\Uupdot$} & \Ucupr \\ & \\
      \text{$\Udowndot \Uup$} & \Ucupr \\
    \end{array}  \right)
   } "1";"3"};
   {\ar^-{
  \left(
    \begin{array}{cc}
      -\;\Udowndot \Uup &  \text{$\Udown \Uupdot$}  \\ & \\
      \Ucapl & -\;\Ucapl\\
    \end{array}
  \right)
   } "3";"5"};
  %(-72,15)*{\cal{C}\onen :=};
 \endxy \nn \\
 \label{eq_casimir_FE}
\end{align}
\begin{align}
 \tomega(\cal{C}\onenn{-n}) &=
  \xy
  (-50,0)*+{\left(\begin{array}{c}
      \scs  \F\E \onen \{2\}\\ \scs \onen \{1+n\}
    \end{array}\right) }="1";
  (0,0)*+{\und{\left(\begin{array}{c}
      \scs  \F\E \onen\\ \scs \F\E \onen
    \end{array}\right)} }="3";
  (60,0)*+{\left(\begin{array}{c}
      \scs \F\E \onen \{-2\}\\ \scs \onen \{-n-1\}
    \end{array}\right) }="5";
   {\ar^-{
  \left(
    \begin{array}{cc}
      \text{$\Udowndot\Uup$} & \Ucupr \\ & \\
      \text{$\Udown\Uupdot $} & \Ucupr \\
    \end{array}
  \right)
   } "1";"3"};
   {\ar^-{
  \left(
    \begin{array}{cc}
      -\;\Udown\Uupdot  &  \text{$\Udowndot\Uup $}  \\ & \\
      \Ucapl & -\;\Ucapl\\
    \end{array}
  \right)
   } "3";"5"};
 \endxy \nn \\
 \label{eq_casimir_FE_bar}
\end{align}
 \begin{align}
  \tsigma\tpsi(\cal{C}\onenn{-n}) &=
   \xy
  (-60,0)*+{
  \left(\begin{array}{c}
      \scs  \F{}\E{} \onen \{2\}\\ \scs \onen \{1+n\}
    \end{array}\right) }="1";
  (0,0)*+{\und{\left(\begin{array}{c}
      \scs  \F{}\E{} \onen\\ \scs \F{}\E{} \onen
    \end{array}\right)}  }="3";
  (50,0)*+{\left(\begin{array}{c}
      \scs  \F{} \E{}\onen \{-2\}\\ \scs \onen \{-n-1\}
    \end{array}\right) }="5";
   {\ar^-{
  \left(
    \begin{array}{cc}
      \text{$-\;\Udowndot\Uup$} & \Ucupr \\ & \\
      \text{$\Udown\Uupdot $} & -\;\Ucupr \\
    \end{array}
  \right)
   } "1";"3"};
   {\ar^-{
  \left(
    \begin{array}{cc}
      \text{$\Udown\Uupdot$}  &  \text{$\Udowndot\Uup $}  \\ & \\
      \Ucapl & \Ucapl\\
    \end{array}
  \right)
   } "3";"5"}; %(-72,15)*{  \widetilde{\cal{C}'\onen} :=};
 \endxy \nn \\
 \label{eq_casimir_EF_tilde} \\
 \tomega\tpsi(\cal{C}\onenn{-n}) &=
   \xy
  (-60,0)*+{
  \left(\begin{array}{c}
      \scs  \F{}\E{} \onen \{2\}\\ \scs \onen \{1+n\}
    \end{array}\right) }="1";
  (0,0)*+{\und{\left(\begin{array}{c}
      \scs  \F{}\E{} \onen\\ \scs \F{}\E{} \onen
    \end{array}\right)}  }="3";
  (50,0)*+{\left(\begin{array}{c}
      \scs  \F{}\E{} \onen \{-2\}\\ \scs \onen \{-n-1\}
    \end{array}\right) }="5";
   {\ar^-{
  \left(
    \begin{array}{cc}
      \text{$-\;\Udown\Uupdot$} & \Ucupr \\ & \\
      \text{$\Udowndot\Uup $} & -\;\Ucupr \\
    \end{array}
  \right)
   } "1";"3"};
   {\ar^-{
  \left(
    \begin{array}{cc}
      \text{$\Udowndot\Uup$}  &  \text{$\Udown\Uupdot $}  \\ & \\
      \Ucapl & \Ucapl\\
    \end{array}
  \right)
   } "3";"5"}; %(-72,15)*{  \widetilde{\cal{C}'\onen} :=};
 \endxy \nn \\
\end{align}
The first complex is identical to the complex $\cal{C}'\onen$ in \eqref{eq_casimir_FE_INTRO}.  The four complexes above are isomorphic to one another in $Kom(\UcatD)$. In Proposition~\ref{prop_homotopy_sym} below we will show that they are homotopy equivalent to the complex $\cal{C}\onen$.

\begin{prop} \label{prop_homotopy_sym} \hfill
\begin{enumerate}[a)]
  \item \label{part_a} For any $g \in \cal{G}_1$ and $g'\in \cal{G}\setminus\cal{G}_1$ there are
chain isomorphisms $\cal{C}\onen \simeq g(\cal{C}\onen)$ and
homotopy equivalences $\cal{C}\onen \simeq g'(\cal{C}\onenn{-n})$ given by chain maps
\begin{equation} \label{eq_bigsym}
 \xy
 (-40,-5)*+{\cal{C}\onen}="t";
 (-10,15)*+{\tsigma(\cal{C}\onenn{-n})}="mt2";
 (-10,-15)*+{\tomega(\cal{C}\onenn{-n})}="mt1";
 (10,-5)*+{\tpsi(\cal{C}\onen)}="mt3";
 (-40,-45)*+{\tsigma\tomega(\cal{C}\onen)}="mb1";
 (40,15)*+{\tsigma\tpsi(\cal{C}\onenn{-n})}="mb3";
 (40,-15)*+{\tomega\tpsi(\cal{C}\onenn{-n})}="mb2";
 (10,-45)*+{\tsigma\tomega\tpsi(\cal{C}\onen)}="b";
  {\ar@{.>}^{\varrho^{\tsigma}} "t";"mt2"};
  {\ar@{.>}_{\varrho^{\tomega}} "t";"mt1"};
  {\ar^{\varrho^{\tpsi}} "t";"mt3"};
  {\ar_{\varrho^{\tsigma\tomega}} "t";"mb1"};
  {\ar|(.67){\hole}^(.3){\tsigma(\varrho^{\tsigma\tomega})} "mt2";"mt1"};
  {\ar_{\tpsi(\hat{\varrho}^{\tsigma\tomega})} "mt3";"b"};
  %{\ar@{.>} "mt1";"mb1"};
  {\ar@{.>}_{\tomega(\hat{\varrho}^{\tsigma})} "mb1";"mt1"};
  {\ar|(.4){\hole}_(.6){\tomega(\varrho^{\tpsi})} "mt1";"mb2"};
  {\ar@{.>}^{\tpsi(\hat{\varrho}^{\tomega})} "mt3";"mb2"};
  {\ar@{.>}^{\tpsi(\hat{\varrho}^{\tsigma})} "mt3";"mb3"};
  {\ar_{\tsigma\tomega(\varrho^{\tpsi})} "mb1";"b"};
  {\ar^{\tpsi\tsigma(\hat{\varrho}^{\tsigma\tomega})} "mb3";"mb2"};
  {\ar@{.>}_{\tpsi\tomega(\varrho^{\tsigma})} "b";"mb2"};
  %{\ar@{.>}^{\varrho^{\tsigma\tpsi,\tsigma\tomega\tpsi}} "mb3";"b"}
  {\ar^{\tsigma(\varrho^{\tpsi})} "mt2";"mb3"};
 \endxy
\end{equation}
where the solid lines represent chain isomorphisms and the dotted lines represent chain homotopy equivalences.  We denote the inverse of a map $\varrho^{g}$ by $\hat{\varrho}^{g}$ for $g \in \cal{G}$.

 \item \label{rem_varrho_nat}
The triangles on the left and right in \eqref{eq_bigsym} commute and the four remaining squares anticommute.

 \item \label{part_c} Complexes $\cal{C}\onen$ and $\cal{C}'\onen$ are homotopy equivalent.
\end{enumerate}
\end{prop}

\begin{proof}
We define explicit chain isomorphisms
 \[
 \xy
  (-14,0)*+{\cal{C}\onen}="1";
  (14,0)*+{\tpsi(\cal{C}\onen)}="2";
   {\ar@/^1pc/^{\varrho^{\tpsi}} "1";"2"};
   {\ar@/^1pc/^{\hat{\varrho}^{\tpsi}} "2";"1"};
 \endxy
  \qquad
   \xy
  (-14,0)*+{\cal{C}\onen}="1";
  (14,0)*+{\tsigma\tomega(\cal{C}\onen),}="2";
   {\ar@/^1pc/^{\varrho^{\tsigma\tomega}} "1";"2"};
   {\ar@/^1pc/^{\hat{\varrho}^{\tsigma\tomega}} "2";"1"};
 \endxy
\]
as well as a homotopy equivalence
\[
 \xy
  (-14,0)*+{\cal{C}\onen}="1";
  (14,0)*+{\tsigma(\cal{C}\onen).}="2";
   {\ar@/^1pc/^{\varrho^{\tsigma}} "1";"2"};
   {\ar@/^1pc/^{\hat{\varrho}^{\tsigma}} "2";"1"};
 \endxy
\]
While the maps $\varrho^{\tpsi}$ and $\varrho^{\tsigma\tomega}$ are rather uninteresting, it is convenient to fix them.

Define $\varrho^{\tpsi} \maps \cal{C}\onen \to \tpsi(\cal{C}\onen)$ and its inverse $\hat{\varrho}^{\tpsi} \maps \tpsi(\cal{C}\onen) \to \cal{C}\onen$ by
\begin{equation} %% tilde Rho
  \left(\varrho^{\tpsi}\right)_{-1} :=
  \left(
    \begin{array}{cc}
      \text{$-\;\Uup\Udown$} & 0 \\
      0 & 1 \\
    \end{array}
  \right)
\quad
  \left(\varrho^{\tpsi}\right)_{0} :=
  \left(
    \begin{array}{cc}
      0 & \text{$\Uup\Udown$} \\
      \text{$-\;\Uup\Udown$} & 0 \\
    \end{array}
  \right)
\quad
  \left(\varrho^{\tpsi}\right)_{+1} :=
  \left(
    \begin{array}{cc}
      \text{$\Uup\Udown$} & 0 \\
      0 & -1 \\
    \end{array}
  \right)
\end{equation}
\begin{equation} %% tilde Rho inv
  \left(\hat{\varrho}^{\tpsi}\right)_{-1} :=
  \left(
    \begin{array}{cc}
      \text{$-\;\Uup\Udown$} & 0 \\
      0 & 1 \\
    \end{array}
  \right)
\quad
  \left(\hat{\varrho}^{\tpsi}\right)_{0} :=
  \left(
    \begin{array}{cc}
      0 & \text{$-\;\Uup\Udown$} \\
      \text{$\Uup\Udown$} & 0 \\
    \end{array}
  \right)
\quad
  \left(\hat{\varrho}^{\tpsi}\right)_{+1} :=
  \left(
    \begin{array}{cc}
      \text{$\Uup\Udown$} & 0 \\
      0 & -1 \\
    \end{array}
  \right)
\end{equation}
where $\cal{C}\onen$ and $\tpsi(\cal{C}\onen)$ are given in \eqref{eq_4_casimir} and \eqref{eq_casimir_EF_tilde}.
The map $\varrho^{\tsigma\tomega} \maps \cal{C}\onen \to \tsigma\tomega(\cal{C}\onen)$ and its inverse $\hat{\varrho}^{\tsigma\tomega} \maps \tsigma\tomega(\cal{C}\onen) \to \cal{C}\onen$ are given by the chain maps
\begin{equation} %% BAR Rho'
  \left(\varrho^{\tsigma\tomega}\right)_{-1} :=
  \left(
    \begin{array}{cc}
      \text{$\Uup\Udown$} & 0 \\
      0 & 1 \\
    \end{array}
  \right)
\quad
  \left(\varrho^{\tsigma\tomega}\right)_{0} :=
  \left(
    \begin{array}{cc}
      0 & \text{$\Uup\Udown$} \\
      \text{$\Uup\Udown$} & 0 \\
    \end{array}
  \right)
\quad
  \left(\varrho^{\tsigma\tomega}\right)_{+1} :=
  \left(
    \begin{array}{cc}
      \text{$-\;\Uup\Udown$} & 0 \\
      0 & -1 \\
    \end{array}
  \right)
\end{equation}
\begin{equation} %% BAR Rho' inv
  \left(\hat{\varrho}^{\tsigma\tomega}\right)_{-1} :=
  \left(
    \begin{array}{cc}
      \text{$\Uup\Udown$} & 0 \\
      0 & 1 \\
    \end{array}
  \right)
\quad
  \left(\hat{\varrho}^{\tsigma\tomega}\right)_{0} :=
  \left(
    \begin{array}{cc}
      0 & \text{$\Uup\Udown$} \\
      \text{$\Uup\Udown$} & 0 \\
    \end{array}
  \right)
\quad
  \left(\hat{\varrho}^{\tsigma\tomega}\right)_{+1} :=
  \left(
    \begin{array}{cc}
      \text{$-\;\Uup\Udown$} & 0 \\
      0 & -1 \\
    \end{array}
  \right)
\end{equation}
with $\tsigma \tomega(\cal{C}\onen)$ given by \eqref{eq_casimir_EF_bar}.
We sometimes express chain maps using cube-like diagrams.  For example, the (rather obvious) chain maps $\varrho^{\tsigma\tomega} \maps \cal{C}\onen \to \tsigma\tomega(\cal{C}\onen)$ and $\hat{\varrho}^{\tsigma\tomega} \maps \tsigma\tomega(\cal{C}\onen)\to  \cal{C}\onen$ can be depicted as
\begin{equation}
 \xy
  (-68,-15)*+{\E{}\F{} \onen \{2\}}="1";
  (-40,-35)*+{\onen \{1-n\}}="2";
  (-10,-15)*+{\E{}\F{} \onen }="3";
  (10,-35)*+{\cal{E}\cal{F} \onen}="4";
  (68,-35)*+{\onen \{n-1\}}="5";
  (40,-15)*+{\E{}\F{} \onen \{-2\}}="6";
   {\ar^(.65){\sUupdot\sUdown} "1";"3"};
   {\ar^{} "1";"4"};
   %% Had problems with the label placing it by hand
   "1"+(12,-8)*{\sUup\sUdowndot};
   {\ar^(.55){\sUcupl} "2";"3"}; %% The <<< adjusts the where the label on
   {\ar_-{\sUcupl} "2";"4"};     %% the arrow goes
   {\ar^-(.75){\sUcapr} "3";"5"};
   "4"+(14,4)*{\sUupdot\sUdown};
   {\ar "3";"6"};
   {\ar_{-\;\;\sUcapr} "4";"5"};
   {\ar "4";"6"};
   "6"+(-18,3)*{-\;\sUup\sUdowndot};
  (-68,35)*+{\E \F\onen \{2\}}="1e";
  (-40,15)*+{\onen \{1-n\}}="2e";
  (-10,35)*+{\E{}\F{} \onen }="3e";
  (10,15)*+{\E{}\F{} \onen}="4e";
  (68,15)*+{\onen \{n-1\}}="5e";
  (40,35)*+{\E{}\F{} \onen \{-2\}}="6e";
   {\ar^{\sUup\sUdowndot} "1e";"3e"};
   {\ar^{} "1e";"4e"};
   %% Had problems with the label placing it by hand
   "1e"+(12,-8)*{\sUupdot\sUdown};
   {\ar^(.6){\sUcupl} "2e";"3e"}; %% The <<< adjusts the where the label on
   {\ar_-(.35){\sUcupl} "2e";"4e"};     %% the arrow goes
      {\ar^-(.25){\sUcapr} "3e";"5e"};
   "4e"+(14,4)*{\sUup\sUdowndot};
   {\ar "3e";"6e"};
   {\ar_(.3){-\;\;\sUcapr} "4e";"5e"};
   {\ar "4e";"6e"};
   "6e"+(-22,3)*{-\;\sUupdot\sUdown};
   %% Green chain maps
   {\textcolor[rgb]{0.00,0.50,0.25}{\ar@/_0.3pc/_(.4){\varrho^{\tsigma\tomega}_1} "1"+(-3,3); "1e"+(-3,-3)}};
   {\textcolor[rgb]{0.00,0.50,0.25}{\ar@/_0.3pc/_(.65){\varrho^{\tsigma\tomega}_2} "2"+(-3,3); "2e"+(-3,-3)}};
   {\textcolor[rgb]{0.00,0.50,0.25}{\ar@/_0.3pc/_(.35){\varrho^{\tsigma\tomega}_6} "5"+(1,3); "5e"+(1,-3)}};
   {\textcolor[rgb]{0.00,0.50,0.25}{\ar@/_0.4pc/_(.2){\varrho^{\tsigma\tomega}_3} "3"+(-1,3); "4e"}};
   {\textcolor[rgb]{0.00,0.50,0.25}{\ar@/_0.3pc/_(.1){\varrho^{\tsigma\tomega}_4} "4"+(3,3); "3e"+(2,-3)}};
   {\textcolor[rgb]{0.00,0.50,0.25}{\ar@/_0.3pc/_(.3){\varrho^{\tsigma\tomega}_5} "6"+(5,3); "6e"+(5,-3)}};
      %% Blue chain maps
   {\textcolor[rgb]{0.00,0.00,1.00}{\ar@/_0.3pc/_(.6){\hat{\varrho}^{\tsigma\tomega}_1} "1e"+(-5,-3); "1"+(-5,3)}};
   {\textcolor[rgb]{0.00,0.00,1.00}{\ar@/_0.3pc/_(.35){\hat{\varrho}^{\tsigma\tomega}_2} "2e"+(-5,-3); "2"+(-5,3)}};
   {\textcolor[rgb]{0.00,0.00,1.00}{\ar@/_0.3pc/_(.65){\hat{\varrho}^{\tsigma\tomega}_6} "5e"+(-1,-3); "5"+(-1,3)}};
   {\textcolor[rgb]{0.00,0.00,1.00}{\ar@/_0.3pc/_(.9){\hat{\varrho}^{\tsigma\tomega}_4} "3e"+(0,-3); "4"+(1,3)}};
   {\textcolor[rgb]{0.00,0.00,1.00}{\ar@/_0.4pc/_(.7){\hat{\varrho}^{\tsigma\tomega}_3} "4e"; "3"+(-3,3)}};
   {\textcolor[rgb]{0.00,0.00,1.00}{\ar@/_0.3pc/_(.7){\hat{\varrho}^{\tsigma\tomega}_5} "6e"+(3,-3); "6"+(3,3)}};
 \endxy
\end{equation}
where
\begin{equation}
  \varrho^{\tsigma\tomega}_1 = \varrho^{\tsigma\tomega}_3 = \varrho^{\tsigma\tomega}_4 = -\varrho^{\tsigma\tomega}_5 = \Id_{\cal{E}\cal{F}\onen} \qquad \varrho^{\tsigma\tomega}_2 = - \varrho^{\tsigma\tomega}_6 = \Id_{1_n},
\end{equation}
and
\begin{equation}
  \hat{\varrho}^{\tsigma\tomega}_1 = \hat{\varrho}^{\tsigma\tomega}_3 = \hat{\varrho}^{\tsigma\tomega}_4 = -\hat{\varrho}^{\tsigma\tomega}_5 = \Id_{\cal{F}\cal{E}\onen} \qquad \hat{\varrho}^{\tsigma\tomega}_2 = - \hat{\varrho}^{\tsigma\tomega}_6 = \Id_{1_n}.
\end{equation}

 The interesting maps here are $\varrho^{\tsigma} \maps \cal{C}\onen \to \tsigma(\cal{C}\onenn{-n})$ and its homotopy inverse $\hat{\varrho}^{\tsigma} \maps \tsigma(\cal{C}\onenn{-n}) \to \cal{C}\onen$, given by the diagram
\begin{equation}
 \xy
  (-68,-15)*+{\F\E \onen \{2\}}="1";
  (-40,-35)*+{\onen \{1+n\}}="2";
  (-10,-15)*+{\F\E \onen }="3";
  (10,-35)*+{\F\E \onen}="4";
  (40,-15)*+{\F\E \onen \{-2\}}="6";
  (68,-35)*+{\onen \{-n-1\}}="5";
   {\ar^(.65){\sUdown\sUupdot} "1";"3"};
   {\ar^{} "1";"4"};
   %% Had problems with the label placing it by hand
   "1"+(12,-8)*{\sUdowndot\sUup};
   {\ar^(.55){\sUcupr} "2";"3"}; %% The <<< adjusts the where the label on
   {\ar_-{\sUcupr} "2";"4"};     %% the arrow goes
      {\ar^-(.35){\sUcapl} "3";"5"};
   "4"+(14,6)*{\sUdown\sUupdot};
   {\ar "3";"6"};
   {\ar^(.6){-\sUcapl} "4";"5"};
   {\ar "4";"6"};
   "6"+(-20,3)*{-\;\sUdowndot\sUup};
  (-68,35)*+{\E \F\onen \{2\}}="1e";
  (-40,15)*+{\onen \{1-n\}}="2e";
  (-10,35)*+{\E\F \onen }="3e";
  (10,15)*+{\E\F \onen}="4e";
  (40,35)*+{\E\F \onen \{-2\}}="6e";
  (68,15)*+{\onen \{n-1\}}="5e";
   {\ar^{\sUupdot\sUdown} "1e";"3e"};
   {\ar^{} "1e";"4e"};
   %% Had problems with the label placing it by hand
   "1e"+(12,-8)*{\sUup\sUdowndot};
   {\ar^(.6){\sUcupl} "2e";"3e"}; %% The <<< adjusts the where the label on
   {\ar_-(.35){\sUcupl} "2e";"4e"};     %% the arrow goes
      {\ar^-(.25){\sUcapr} "3e";"5e"};
   "4e"+(6,8)*{\sUupdot\sUdown};
   {\ar^{-\sUup\sUdowndot} "3e";"6e"};
   {\ar^(.58){-\sUcapr} "4e";"5e"};
   {\ar "4e";"6e"};
   %% Green chain maps
   {\textcolor[rgb]{0.00,0.50,0.25}{\ar@/_0.3pc/_(.4){\hat{\varrho}^{\tsigma}_1} "1"+(-3,3); "1e"+(-3,-3)}};
   {\textcolor[rgb]{0.00,0.50,0.25}{\ar@/_0.4pc/_(.7){\hat{\varrho}^{\tsigma}_2} "1"; "2e"}};
   {\textcolor[rgb]{0.00,0.50,0.25}{\ar@/_0.3pc/_(.75){\hat{\varrho}^{\tsigma}_3} "2"+(0,3); "1e"+(2,-3)}};
   {\textcolor[rgb]{0.00,0.50,0.25}{\ar@/_0.3pc/_(.65){\hat{\varrho}^{\tsigma}_4} "2"+(5,3); "2e"+(5,-3)}};
   {\textcolor[rgb]{0.00,0.50,0.25}{\ar@/_0.3pc/_(.4){\hat{\varrho}^{\tsigma}_5} "3"+(1,3); "3e"+(1,-3)}};
   {\textcolor[rgb]{0.00,0.50,0.25}{\ar@/_0.3pc/_(.65){\hat{\varrho}^{\tsigma}_6} "4"+(1,3); "4e"+(1,-3)}};
   {\textcolor[rgb]{0.00,0.50,0.25}{\ar@/_0.3pc/_(.35){\hat{\varrho}^{\tsigma}_{10}} "5"+(5,3); "5e"+(5,-3)}};
   {\textcolor[rgb]{0.00,0.50,0.25}{\ar@/_0.4pc/_(.3){\hat{\varrho}^{\tsigma}_8} "5"+(-1,3); "6e"+(7,-3)}};
   {\textcolor[rgb]{0.00,0.50,0.25}{\ar@/_0.3pc/_(.25){\hat{\varrho}^{\tsigma}_9} "6"+(3,3); "5e"+(-3,-3)}};
   {\textcolor[rgb]{0.00,0.50,0.25}{\ar@/_0.3pc/_(.3){\hat{\varrho}^{\tsigma}_{7}} "6"+(-3,3); "6e"+(-3,-3)}};
      %% Blue chain maps
   {\textcolor[rgb]{0.00,0.00,1.00}{\ar@/_0.3pc/_(.6){\varrho^{\tsigma}_1} "1e"+(-5,-3); "1"+(-5,3)}};
   {\textcolor[rgb]{0.00,0.00,1.00}{\ar@/_0.3pc/_(.25){\varrho^{\tsigma}_2} "1e"+(0,-3); "2"+(-2,3)}};
   {\textcolor[rgb]{0.00,0.00,1.00}{\ar@/_0.4pc/_(.3){\varrho^{\tsigma}_3} "2e"; "1"}};
   {\textcolor[rgb]{0.00,0.00,1.00}{\ar@/_0.3pc/_(.35){\varrho^{\tsigma}_4} "2e"+(3,-3); "2"+(3,3)}};
   {\textcolor[rgb]{0.00,0.00,1.00}{\ar@/_0.3pc/_(.6){\varrho^{\tsigma}_5} "3e"+(-1,-3); "3"+(-1,3)}};
   {\textcolor[rgb]{0.00,0.00,1.00}{\ar@/_0.3pc/_(.35){\varrho^{\tsigma}_6} "4e"+(-1,-3); "4"+(-1,3)}};
   {\textcolor[rgb]{0.00,0.00,1.00}{\ar@/_0.3pc/_(.65){\varrho^{\tsigma}_{10}} "5e"+(3,-3); "5"+(3,3)}};
   {\textcolor[rgb]{0.00,0.00,1.00}{\ar@/_0.3pc/_(.75){\varrho^{\tsigma}_9} "5e"+(-5,-3); "6"+(1,3)}};
   {\textcolor[rgb]{0.00,0.00,1.00}{\ar@/_0.3pc/_(.7){\varrho^{\tsigma}_8} "6e"+(5,-3); "5"+(-3,3)}};
   {\textcolor[rgb]{0.00,0.00,1.00}{\ar@/_0.3pc/_(.7){\varrho^{\tsigma}_{7}} "6e"+(-5,-3); "6"+(-5,3)}};
 \endxy
\end{equation}
where
\[
\varrho^{\tsigma}_1 =\varrho^{\tsigma}_5=\varrho^{\tsigma}_6=\varrho^{\tsigma}_{7}\;\; =\;\; \vcenter{
 \xy 0;/r.17pc/:
    (-4,-4)*{};(4,4)*{} **\crv{(-4,-1) & (4,1)}?(1)*\dir{>};
    (4,-4)*{};(-4,4)*{} **\crv{(4,-1) & (-4,1)}?(0)*\dir{<};
\endxy}
    \qquad
\varrho^{\tsigma}_2 \;\; =\;\;
  -\;  \vcenter{\xy 0;/r.15pc/:
    %\varrho^{\tsigma}(-10,10)*{n};  (-8,0)*{};  (8,0)*{};
  (-4,-12)*{}="t1";
  (4,-12)*{}="t2";
  "t2";"t1" **\crv{(5,-4) & (-5,-4)}; ?(0)*\dir{<}
  \endxy}
    \qquad
\varrho^{\tsigma}_3 \;\; =\;\; -\;     \sum_{ \xy  (0,3)*{\scs f_1+f_2}; (0,0)*{\scs
=-n-1};\endxy}
    \vcenter{\xy 0;/r.15pc/:
    %\varrho^{\tsigma}(-10,10)*{n};  (-8,0)*{};  (8,0)*{};
  (-4,15)*{}="t1";
  (4,15)*{}="t2";
  "t2";"t1" **\crv{(5,8) & (-5,8)}; ?(0)*\dir{<}
  ?(.4)*\dir{}+(0,-.2)*{\bullet}+(3,-2)*{\scs \; f_1};
  (0,0)*{\cbub{\scs \quad \spadesuit+f_2}{}};
  \endxy}
\]
\[
\varrho^{\tsigma}_4 \;\; =\;\;
  - \; \vcenter{\xy 0;/r.15pc/:
    (0,0)*{\cbub{\scs \quad \spadesuit-n}{}};
  \endxy}
    \qquad
\varrho^{\tsigma}_8 \;\; =\;\; -\;     \sum_{ \xy  (0,3)*{\scs f_1+f_2}; (0,0)*{\scs
=n-1};\endxy}
    \vcenter{\xy 0;/r.15pc/:
    %\varrho^{\tsigma}(-10,10)*{n};  (-8,0)*{};  (8,0)*{};
  (-4,-15)*{}="t1";
  (4,-15)*{}="t2";
  "t2";"t1" **\crv{(5,-8) & (-5,-8)}; ?(0)*\dir{<}
  ?(.67)*\dir{}+(0,-.2)*{\bullet}+(-4,2)*{\scs \; f_1};
  (2,0)*{\ccbub{\scs \quad \spadesuit+f_2}{}};
  \endxy}
      \qquad
\varrho^{\tsigma}_9\;\;=\;\; -\;  \vcenter{\xy 0;/r.15pc/:
    %\varrho^{\tsigma}(-10,10)*{n};  (-8,0)*{};  (8,0)*{};
  (-4,12)*{}="t1";
  (4,12)*{}="t2";
  "t2";"t1" **\crv{(5,4) & (-5,4)}; ?(0)*\dir{<}
  \endxy}
        \qquad
\varrho^{\tsigma}_{10} \;\; =\;\; \vcenter{\xy 0;/r.15pc/:
    (0,0)*{\ccbub{\scs \quad \spadesuit+n}{}};
  \endxy}
\]
\[
\hat{\varrho}^{\tsigma}_1=\hat{\varrho}^{\tsigma}_5=\hat{\varrho}^{\tsigma}_6 =\hat{\varrho}^{\tsigma}_{7} \;\; =\;\; -\;\vcenter{
 \xy 0;/r.17pc/:
    (-4,-4)*{};(4,4)*{} **\crv{(-4,-1) & (4,1)}?(0)*\dir{<};
    (4,-4)*{};(-4,4)*{} **\crv{(4,-1) & (-4,1)}?(1)*\dir{>};
\endxy}
    \qquad
\hat{\varrho}^{\tsigma}_2 \;\; =\;\;
  -\;  \vcenter{\xy 0;/r.15pc/:
    %\varrho^{\tsigma}(-10,10)*{n};  (-8,0)*{};  (8,0)*{};
  (-4,-12)*{}="t1";
  (4,-12)*{}="t2";
  "t2";"t1" **\crv{(5,-4) & (-5,-4)}; ?(1)*\dir{>}
  \endxy}
    \qquad
\hat{\varrho}^{\tsigma}_3 \;\; =\;\; -\;     \sum_{ \xy  (0,3)*{\scs f_1+f_2}; (0,0)*{\scs
=n-1};\endxy}
    \vcenter{\xy 0;/r.15pc/:
    %\varrho^{\tsigma}(-10,10)*{n};  (-8,0)*{};  (8,0)*{};
  (-4,15)*{}="t1";
  (4,15)*{}="t2";
  "t2";"t1" **\crv{(5,8) & (-5,8)}; ?(.15)*\dir{>} ?(.95)*\dir{>}
  ?(.4)*\dir{}+(0,-.2)*{\bullet}+(3,-2)*{\scs \; f_1};
  (0,0)*{\ccbub{\scs \quad \spadesuit+f_2}{}};
  \endxy}
\]
\[
\hat{\varrho}^{\tsigma}_4 \;\; =\;\;
  - \; \vcenter{\xy 0;/r.15pc/:
    (0,0)*{\ccbub{\scs \quad \spadesuit+n}{}};
  \endxy}
      \qquad
\hat{\varrho}^{\tsigma}_8 \;\; =\;\; -\;  \vcenter{\xy 0;/r.15pc/:
    %\varrho^{\tsigma}(-10,10)*{n};  (-8,0)*{};  (8,0)*{};
  (-4,12)*{}="t1";
  (4,12)*{}="t2";
  "t2";"t1" **\crv{(5,4) & (-5,4)}; ?(1)*\dir{>}
  \endxy}
    \qquad
\hat{\varrho}^{\tsigma}_9\;\;=\;\; -\;   \sum_{ \xy  (0,3)*{\scs f_1+f_2}; (0,0)*{\scs
=-n-1};\endxy}
    \vcenter{\xy 0;/r.15pc/:
    %\varrho^{\tsigma}(-10,10)*{n};  (-8,0)*{};  (8,0)*{};
  (-4,-15)*{}="t1";
  (4,-15)*{}="t2";
  "t2";"t1" **\crv{(5,-8) & (-5,-8)}; ?(1)*\dir{>}
  ?(.67)*\dir{}+(0,-.2)*{\bullet}+(-4,2)*{\scs \; f_1};
  (2,0)*{\cbub{\scs \quad \spadesuit+f_2}{}};
  \endxy}
    \qquad
\hat{\varrho}^{\tsigma}_{10} \;\; =\;\; \vcenter{\xy 0;/r.15pc/:
    (0,0)*{\cbub{\scs \quad \spadesuit-n}{}};
  \endxy}
\]
The chain homotopies
$\varrho^{\tsigma}\hat{\varrho}^{\tsigma}-\Id \simeq 0$, $\hat{\varrho}^{\tsigma}\varrho^{\tsigma}-\Id \simeq 0$ are given by

\begin{equation}
 \xy
  (-68,-15)*+{\F\E \onen \{2\}}="1";
  (-40,-35)*+{\onen \{1+n\}}="2";
  (-10,-15)*+{\F\E \onen }="3";
  (10,-35)*+{\F\E \onen}="4";
  (40,-15)*+{\F\E \onen \{-2\}}="6";
  (68,-35)*+{\onen \{-n-1\}}="5";
   {\ar^(.65){\sUdown\sUupdot} "1";"3"};
   {\ar^{} "1";"4"};
   %% Had problems with the label placing it by hand
   "1"+(12,-8)*{\sUdowndot\sUup};
   {\ar^(.55){\sUcupr} "2";"3"}; %% The <<< adjusts the where the label on
   {\ar_-{\sUcupr} "2";"4"};     %% the arrow goes
      {\ar^-(.35){\sUcapl} "3";"5"};
   "4"+(14,6)*{\sUdown\sUupdot};
   {\ar "3";"6"};
   {\ar^(.6){-\;\;\sUcapl} "4";"5"};
   {\ar "4";"6"};
   "6"+(-20,3)*{-\;\sUdowndot\sUup};
  (-68,35)*+{\E \F\onen \{2\}}="1e";
  (-40,15)*+{\onen \{1-n\}}="2e";
  (-10,35)*+{\E\F \onen }="3e";
  (10,15)*+{\E\F \onen}="4e";
  (40,35)*+{\E\F \onen \{-2\}}="6e";
  (68,15)*+{\onen \{n-1\}}="5e";
   {\ar^{\sUupdot\sUdown} "1e";"3e"};
   {\ar^{} "1e";"4e"};
   %% Had problems with the label placing it by hand
   "1e"+(12,-8)*{\sUup\sUdowndot};
   {\ar^(.6){\sUcupl} "2e";"3e"}; %% The <<< adjusts the where the label on
   {\ar_-(.35){\sUcupl} "2e";"4e"};     %% the arrow goes
      {\ar^-(.25){\sUcapr} "3e";"5e"};
   "4e"+(6,8)*{\sUupdot\sUdown};
   {\ar^{-\sUup\sUdowndot} "3e";"6e"};
   {\ar^(.58){-\;\;\sUcapr} "4e";"5e"};
   {\ar "4e";"6e"};
  {\textcolor[rgb]{1.00,0.00,0.00}{\ar@/_1.6pc/_(.4){h_1} "3e"; "1e"}};
  {\textcolor[rgb]{1.00,0.00,0.00}{\ar@/^1.0pc/^(.2){h_2} "3e"; "2e"}};
  {\textcolor[rgb]{1.00,0.00,0.00}{\ar@/^1.0pc/^(.1){h_3} "6e"+(2,-3); "4e"}};
  {\textcolor[rgb]{1.00,0.00,0.00}{\ar@/^1.2pc/^(.45){h_4} "5e"; "4e"}};
   %% Bottom homotopy
  {\textcolor[rgb]{1.00,0.00,0.50}{\ar@/_1.6pc/_(.3){h'_1} "3"; "1"}};
  {\textcolor[rgb]{1.00,0.00,0.50}{\ar@/^1.0pc/^(.2){h'_2} "3"; "2"}};
  {\textcolor[rgb]{1.00,0.00,0.50}{\ar@/^1.0pc/^(.1){h'_3} "6"+(2,-3); "4"}};
  {\textcolor[rgb]{1.00,0.00,0.50}{\ar@/^1.0pc/^(.45){h'_4} "5"; "4"}};
 \endxy
\end{equation}
where
\begin{align}
h_1=h_3 \;\; &=\;\; -\; \sum_{ \xy  (0,3)*{\scs f_1+f_2+f_3};
(0,0)*{\scs =n-2};\endxy}
    \vcenter{\xy 0;/r.15pc/:
    %(-10,10)*{n};
    (-8,0)*{};  (8,0)*{};
  (-4,-15)*{}="b1";
  (4,-15)*{}="b2";
  "b2";"b1" **\crv{(5,-8) & (-5,-8)}; ?(0)*\dir{<}
  ?(.8)*\dir{}+(0,-.1)*{\bullet}+(-4,2)*{\scs f_3};
  (-4,15)*{}="t1";
  (4,15)*{}="t2";
  "t2";"t1" **\crv{(5,8) & (-5,8)};  ?(1)*\dir{>}
  ?(.4)*\dir{}+(0,-.2)*{\bullet}+(3,-2)*{\scs \; f_1};
  (0,0)*{\ccbub{\scs \quad \spadesuit+f_2}};
  \endxy}  \qquad
  \qquad h_2  &= \;\; - \;  \sum_{ \xy  (0,3)*{\scs f_1+f_2}; (0,0)*{\scs
=n-1};\endxy}
    \vcenter{\xy 0;/r.15pc/:
    %\varrho^{\tsigma}(-10,10)*{n};  (-8,0)*{};  (8,0)*{};
  (-4,-15)*{}="t1";
  (4,-15)*{}="t2";
  "t2";"t1" **\crv{(5,-8) & (-5,-8)}; ?(0)*\dir{<}
  ?(.67)*\dir{}+(0,-.2)*{\bullet}+(-4,2)*{\scs \; f_1};
  (2,0)*{\ccbub{\scs \quad \spadesuit+f_2}{}};
  \endxy} \qquad
h_4 \;\; &= \;\;     \sum_{ \xy  (0,3)*{\scs f_1+f_2}; (0,0)*{\scs
=n-1};\endxy}
    \vcenter{\xy 0;/r.15pc/:
    %\varrho^{\tsigma}(-10,10)*{n};  (-8,0)*{};  (8,0)*{};
  (-4,15)*{}="t1";
  (4,15)*{}="t2";
  "t2";"t1" **\crv{(5,8) & (-5,8)}; ?(.15)*\dir{>} ?(.95)*\dir{>}
  ?(.4)*\dir{}+(0,-.2)*{\bullet}+(3,-2)*{\scs \; f_1};
  (0,0)*{\ccbub{\scs \quad \spadesuit+f_2}{}};
  \endxy}
\\
h'_1=h'_3 \;\; &=\;\; -\; \sum_{ \xy  (0,3)*{\scs g_1+g_2+g_3};
(0,0)*{\scs =x+y-n-2};\endxy}
    \vcenter{\xy 0;/r.15pc/:
    (-8,0)*{};
  (8,0)*{};
  (-4,-15)*{}="b1";
  (4,-15)*{}="b2";
  "b2";"b1" **\crv{(5,-8) & (-5,-8)}; ?(.1)*\dir{>} ?(.95)*\dir{>}
  ?(.8)*\dir{}+(0,-.1)*{\bullet}+(-3,2)*{\scs g_3};
  (-4,15)*{}="t1";
  (4,15)*{}="t2";
  "t2";"t1" **\crv{(5,8) & (-5,8)}; ?(.15)*\dir{<} ?(.9)*\dir{<}
  ?(.4)*\dir{}+(0,-.2)*{\bullet}+(3,-2)*{\scs g_1};
  (0,0)*{\cbub{\scs \quad\; \spadesuit + g_2}{}};
  %(-10,10)*{n};
  \endxy} \qquad
\quad h_2' &= \;\; -\;
 \sum_{ \xy  (0,3)*{\scs f_1+f_2};
(0,0)*{\scs =-n-1};\endxy}
    \vcenter{\xy 0;/r.15pc/:
    %\varrho^{\tsigma}(-10,10)*{n};  (-8,0)*{};  (8,0)*{};
  (-4,-15)*{}="t1";
  (4,-15)*{}="t2";
  "t2";"t1" **\crv{(5,-8) & (-5,-8)}; ?(1)*\dir{>}
  ?(.67)*\dir{}+(0,-.2)*{\bullet}+(-4,2)*{\scs \; f_1};
  (2,0)*{\cbub{\scs \quad \spadesuit+f_2}{}};
  \endxy}
  \qquad
h'_4 \;\; &= \;\;
 \sum_{ \xy  (0,3)*{\scs f_1+f_2}; (0,0)*{\scs
=-n-1};\endxy}
    \vcenter{\xy 0;/r.15pc/:
    %\varrho^{\tsigma}(-10,10)*{n};  (-8,0)*{};  (8,0)*{};
  (-4,15)*{}="t1";
  (4,15)*{}="t2";
  "t2";"t1" **\crv{(5,8) & (-5,8)}; ?(0)*\dir{<}
  ?(.4)*\dir{}+(0,-.2)*{\bullet}+(3,-2)*{\scs \; f_1};
  (0,0)*{\cbub{\scs \quad \spadesuit+f_2}{}};
  \endxy}
\end{align}
One can verify the following equations
\begin{eqnarray}
  \varrho^{\tsigma}\hat{\varrho}^{\tsigma}-\Id &=&hd + dh,\\
  \hat{\varrho}^{\tsigma} \varrho^{\tsigma}-\Id &=&h'd + dh'.
\end{eqnarray}

The chain maps just defined satisfy
\begin{align}
  \tsigma(\varrho^{\tsigma}) &= \hat{\varrho}^{\tsigma},
  &
  \tsigma(\hat{\varrho}^{\tsigma}) &= \varrho^{\tsigma},
   \\
  \tpsi(\varrho^{\tpsi}) &=  - \varrho^{\tpsi},
  &  \tpsi(\hat{\varrho}^{\tpsi}) &=  - \hat{\varrho}^{\tpsi},\\
  \tomega(\varrho^{\tsigma\tomega}) &= \tsigma(\hat{\varrho}^{\tsigma\tomega}),
  &
  \tomega(\hat{\varrho}^{\tsigma\tomega}) &= \tsigma(\varrho^{\tsigma\tomega}).
\end{align}

To prove part \ref{rem_varrho_nat} of the proposition one can check by direct computation that the front square in \eqref{eq_bigsym} anticommutes.  The back solid square is just $\tsigma$ applied to the front solid square so it also anticommutes. The leftmost square commutes on the nose and we define maps
\begin{align}
  \varrho^{\tomega} &:= \tsigma(\varrho^{\tsigma\tomega}) \circ \varrho^{\tsigma}
  & \hat{\varrho}^{\tomega} &:= \hat{\varrho}^{\tsigma} \circ \tsigma(\hat{\varrho}^{\tsigma\tomega}),
\end{align}
using these commutative squares.  The rightmost square is just $\tpsi$ applied to the leftmost square so it also commutes on the nose.  The top square can be shown to anticommute. After observing that $\tpsi\tomega(\varrho^{\tsigma})=\tpsi\tomega\tsigma(\hat{\varrho}^{\tsigma})$ and $\tomega(\hat{\varrho}^{\tsigma})=\tomega\tsigma(\varrho^{\tsigma})$ the anticommutativity of the bottom square follows since it is just $\tomega\tsigma$ applied to the top square.

Part \ref{part_c} follows immediately from part \ref{part_a} since $\cal{C}'\onen=\tsigma(\cal{C}\onenn{-n})$.
\end{proof}

Theorem~\ref{thm_Groth} and the results in section~\ref{sec_KomU} imply that
\begin{equation}
K_0\left( Kar\left( Kom(\Ucat) \right)\right) \cong
K_0\left( Kom\left( Kar(\Ucat) \right)\right) =
K_0\left( Kar(\Ucat) \right) \cong K_0(\UcatD) \cong \UA.
\end{equation}
Under this isomorphism
\begin{equation}
  [\cal{C}\onen] = C1_n.
\end{equation}

% ==============================================================================
%
\subsection{Commutativity of the Casimir complex} \label{subsec_commutativity}
%
% ==============================================================================

% ------------------------------------------------------------------------------
%
\subsubsection{Commutativity chain maps $\gam$ and $\delt$}
%
% ------------------------------------------------------------------------------

\begin{defn} \label{def_xi_minus}
Define  chain maps $\gam = \gam\onen \maps \F \cal{C}\onen\to\cal{C}\F\onen$ and
$\delt = \delt\onen \maps \cal{C}\F\onen \to \F \cal{C}\onen$ as follows:
\begin{equation}
 \xy
  (-65,35)*+{\E\F\F \onen \{2\}}="1'";
  (-45,15)*+{\F\onen \{-1-n\}}="2'";
  (-10,35)*+{\E\F\F \onen }="3'";
  (10,15)*+{\E\F\F \onen}="4'";
  (45,35)*+{\E\F\F \onen \{-2\}}="5'";
  (65,15)*+{\F\onen \{n-3\}}="6'";
   {\ar^{\sUupdot\sUdown\sUdown} "1'";"3'"};
   {\ar^{} "1'" ;"4'"};
   "1'"+(15,-8)*{\sUup\sUdowndot\sUdown};
   "3'"+(-15,-5)*{\sUcupl\sUdown};
   "3'"+(28,-4)*{\text{$\sUcapr\sUdown$}};
   "4'"+(4,8)*{\sUupdot\sUdown\sUdown};
   {\ar "2'";"3'"};
   {\ar^(.4){\text{$\sUcupl\sUdown$}} "2'";"4'"};
   {\ar^-{-\;\sUup\sUdowndot\sUdown} "3'";"5'"};
   {\ar "3'";"6'"};
   {\ar "4'";"5'"};
   {\ar^(.4){-\;\;\text{$\sUcapr\sUdown$}} "4'";"6'"};
  (-65,-15)*+{\F\E\F \onen \{2\}}="1";
  (-45,-35)*+{\F\onen \{1-n\}}="2";
  (-10,-15)*+{\F\E\F \onen }="3";
  (10,-35)*+{\F\E\F \onen}="4";
  (45,-15)*+{\F\E\F \onen \{-2\}}="5";
  (65,-35)*+{\F\onen \{n-1\}}="6";
   {\ar^{\sUdown\sUupdot\sUdown} "1";"3"};
   {\ar^{} "1" ;"4"};
   "4"+(-12,8)*{\sUdown\sUup\sUdowndot};
   "3"+(-16,-5)*{\sUdown\sUcupl};
   "6"+(-18,10)*{\text{$\sUdown\sUcapr$}};
   "4"+(14,4)*{\sUdown\sUupdot\sUdown};
   {\ar "2";"3"};
   {\ar_(.4){\text{$\sUdown\sUcupl$}} "2";"4"};
   {\ar^(.63){-\;\sUdown\sUup\sUdowndot} "3";"5"};
   {\ar "3";"6"};
   {\ar "4";"5"};
   {\ar_(.4){-\text{$\sUdown\sUcapr$}} "4";"6"};
   {\textcolor[rgb]{0.00,0.50,0.25}{\ar@/^0.25pc/^{\gam_1} "1"+(-6,4);"1'"+(-6,-4)}};
   {\textcolor[rgb]{0.00,0.50,0.25}{\ar@/^0.3pc/^{\gam_2} "1";"2'"+(-4,-3)}};
   {\textcolor[rgb]{0.00,0.50,0.25}{\ar@/^0.3pc/^{\gam_3} "3";"3'"}};
   {\textcolor[rgb]{0.00,0.50,0.25}{\ar^(.3){\gam_4} "3"+(4,3);"4'"}};
   {\textcolor[rgb]{0.00,0.50,0.25}{\ar@/^0.35pc/^(.7){\gam_5} "4"+(1,4);"4'"+(1,-4)}};
    {\textcolor[rgb]{0.00,0.50,0.25}{\ar^{\gam_6} "5";"5'"+(-1,-3)}};
   {\textcolor[rgb]{0.00,0.50,0.25}{\ar@/^0.3pc/^(.3){\gam_9} "5"+(4,3);"6'"+(-1,-3)}};
   {\textcolor[rgb]{0.00,0.50,0.25}{\ar@/^0.3pc/^(.2){\gam_7} "6";"5'"+(4,-4)}};
   {\textcolor[rgb]{0.00,0.50,0.25}{\ar^{\gam_8} "6"+(4,4);"6'"+(4,-4)}};
   {\textcolor[rgb]{0.00,0.00,1.00}{\ar@/^0.25pc/^{\delt_1} "1'"+(-4,-4);"1"+(-4,4)}};
   {\textcolor[rgb]{0.00,0.00,1.00}{\ar@/^0.3pc/^{\delt_3} "2'";"1"+(3,3)}};
   {\textcolor[rgb]{0.00,0.00,1.00}{\ar@/^0.25pc/^(.7){\delt_2} "1'"+(4,-4);"2"+(-4,4)}};
   {\textcolor[rgb]{0.00,0.00,1.00}{\ar@/^0.3pc/^(.3){\delt_4} "2'";"2"}};
   {\textcolor[rgb]{0.00,0.00,1.00}{\ar@/^0.35pc/^{\delt_5} "3'"+(2,-3);"3"+(2,3)}};
   {\textcolor[rgb]{0.00,0.00,1.00}{\ar^(.55){\delt_6} "3'"+(5,-3);"4"+(-5,3)}};
   {\textcolor[rgb]{0.00,0.00,1.00}{\ar@/^0.3pc/^(.3){\delt_7} "4'"+(3,-4);"4"+(3,4)}};
   {\textcolor[rgb]{0.00,0.00,1.00}{\ar@/^0.3pc/^(.55){\delt_9} "6'"+(2,-3);"5"+(6,3)}};
   {\textcolor[rgb]{0.00,0.00,1.00}{\ar@/^0.25pc/^(.55){\delt_8} "5'"+(1,-3);"5"+(2,3)}};
 \endxy
\end{equation}
where
\[
 \gam_1 =\;\; \vcenter{
 \xy 0;/r.13pc/:
    (-4,-12)*{};(4,4)*{} **\crv{(-4,-1) & (4,1)} ?(0)*\dir{<};
    (4,-12)*{};(-4,4)*{} **\crv{(4,-1) & (-4,1)};
    (4,4)*{};(12,12)*{} **\crv{(4,7) & (12,9)};
    (12,4)*{};(4,12)*{} **\crv{(12,7) & (4,9)};
    (-4,4)*{}; (-4,12) **\dir{-}?(1)*\dir{>};
    (12,-12)*{}; (12,4) **\dir{-}?(0)*\dir{<};
  (10.5,9.5)*{\bullet}; %(18,8)*{n};
\endxy}
 - \;\; \vcenter{
 \xy 0;/r.13pc/:
    (-4,-12)*{};(4,4)*{} **\crv{(-4,-1) & (4,1)} ?(0)*\dir{<};
    (4,-12)*{};(-4,4)*{} **\crv{(4,-1) & (-4,1)};
    (4,4)*{};(12,12)*{} **\crv{(4,7) & (12,9)};
    (12,4)*{};(4,12)*{} **\crv{(12,7) & (4,9)};
    (-4,4)*{}; (-4,12) **\dir{-}?(1)*\dir{>};
    (12,-12)*{}; (12,4) **\dir{-}?(0)*\dir{<};
    (-4,6)*{\bullet};% (18,8)*{n};
\endxy}
\;\;- \sum_{\xy (0,2)*{\scs f_1+f_2+f_3+f_4}; (0,-1)*{\scs =n-2}; (0,-4)*{\scs 1 \leq f_3};\endxy}
\vcenter{  \xy 0;/r.13pc/:
  (4,12)*{};(-12,-12)*{} **\crv{(4,4) & (-12,-4)} ?(1)*\dir{>}
  ?(.75)*\dir{}+(0,0)*{\bullet}+(-4.5,1)*{\scs f_2};
  (-4,12)*{}="t1";
  (-12,12)*{}="t2";
  "t1";"t2" **\crv{(-4,5) & (-12,5)};?(1)*\dir{>}?
   ?(.5)*\dir{}+(0,0)*{\bullet}+(-1,-3)*{\scs f_1};
  (4,-12)*{}="t1";  (-4,-12)*{}="t2";
  "t2";"t1" **\crv{(-4,-5) & (4,-5)}; ?(1)*\dir{>}
    ?(.5)*\dir{}+(0,0)*{\bullet}+(2,3.5)*{\scs f_3};
  %(16,-10)*{n};
  (14,4)*{\ccbub{\spadesuit+f_4}{}};
  \endxy }
\qquad
\gam_2  =\;\; \vcenter{
 \xy 0;/r.13pc/:
    (-4,-4)*{};(4,-4)*{} **\crv{(-4,4) & (4,4)} ?(0)*\dir{<};
    (12,-4)*{}; (12,12) **\dir{-}?(0)*\dir{<};
   %(4,8)*{\bullet}; (18,8)*{n};
\endxy}
\;\; - \sum_{\xy (0,2)*{\scs f_1+f_2+f_3}; (0,-1)*{\scs =n-1}; (0,-4)*{\scs 1 \leq f_3};\endxy}
\vcenter{  \xy 0;/r.13pc/:
  (-12,12)*{}; (-12,-12)*{} **\dir{-} ?(1)*\dir{>};
   ?(.5)*\dir{}+(0,0)*{\bullet}+(4,1)*{\scs f_1};
  (4,-12)*{}="t1";  (-4,-12)*{}="t2";
  "t2";"t1" **\crv{(-4,-5) & (4,-5)}; ?(1)*\dir{>}
    ?(.75)*\dir{}+(0,0)*{\bullet}+(4,1)*{\scs f_3};
  %(13,10)*{n};
  (2,5)*{\ccbub{\spadesuit+f_2}{}};
  \endxy }
\]
\[
 \gam_3  =\;\; \vcenter{
 \xy 0;/r.13pc/:
    (-4,-12)*{};(4,4)*{} **\crv{(-4,-1) & (4,1)} ?(0)*\dir{<};
    (4,-12)*{};(-4,4)*{} **\crv{(4,-1) & (-4,1)};
    (4,4)*{};(12,12)*{} **\crv{(4,7) & (12,9)};
    (12,4)*{};(4,12)*{} **\crv{(12,7) & (4,9)};
    (-4,4)*{}; (-4,12) **\dir{-}?(1)*\dir{>};
    (12,-12)*{}; (12,4) **\dir{-}?(0)*\dir{<};
  (10.5,9.5)*{\bullet}; %(18,8)*{n};
\endxy}
 - \;\; \vcenter{
 \xy 0;/r.13pc/:
    (-4,-12)*{};(4,4)*{} **\crv{(-4,-1) & (4,1)} ?(0)*\dir{<};
    (4,-12)*{};(-4,4)*{} **\crv{(4,-1) & (-4,1)};
    (4,4)*{};(12,12)*{} **\crv{(4,7) & (12,9)};
    (12,4)*{};(4,12)*{} **\crv{(12,7) & (4,9)};
    (-4,4)*{}; (-4,12) **\dir{-}?(1)*\dir{>};
    (12,-12)*{}; (12,4) **\dir{-}?(0)*\dir{<};
    (-4,6)*{\bullet}; %(18,8)*{n};
\endxy}
\;\;- \sum_{\xy (0,2)*{\scs f_1+f_2+f_3+f_4}; (0,-1)*{\scs =n-2}; \endxy}
\vcenter{  \xy 0;/r.13pc/:
  (4,12)*{};(-12,-12)*{} **\crv{(4,4) & (-12,-4)} ?(1)*\dir{>}
  ?(.75)*\dir{}+(0,0)*{\bullet}+(-4.5,1)*{\scs f_2};
  (-4,12)*{}="t1";
  (-12,12)*{}="t2";
  "t1";"t2" **\crv{(-4,5) & (-12,5)};?(1)*\dir{>}?
   ?(.5)*\dir{}+(0,0)*{\bullet}+(-1,-3)*{\scs f_1};
  (4,-12)*{}="t1";  (-4,-12)*{}="t2";
  "t2";"t1" **\crv{(-4,-5) & (4,-5)}; ?(1)*\dir{>}
    ?(.5)*\dir{}+(0,0)*{\bullet}+(2,3.5)*{\scs f_3};
  %(16,-10)*{n};
  (14,4)*{\ccbub{\spadesuit+f_4}{}};
  \endxy }
\qquad
\gam_4  = \;\; \vcenter{
 \xy 0;/r.13pc/:
    (-4,-12)*{};(-12,12)*{} **\crv{(-4,3) & (-12,5)}?(1)*\dir{>};
    (-12,-12)*{};(-4,12)*{} **\crv{(-12,3) & (-4,5)}?(0)*\dir{<};
    (4,-12)*{}; (4,12) **\dir{-}?(0)*\dir{<};
\endxy}
  \; -
 \sum_{\xy (0,2)*{\scs f_1+f_2+f_3}; (0,-1)*{\scs =n-1}; \endxy}
\vcenter{  \xy 0;/r.13pc/:
  (-4,12)*{};(-12,-12)*{} **\crv{(-4,4) & (-12,-4)} ?(1)*\dir{>};
  (4,12)*{}="t1";
  (-12,12)*{}="t2";
  "t1";"t2" **\crv{(4,5) & (-12,5)};?(1)*\dir{>}?
   ?(.25)*\dir{}+(0,0)*{\bullet}+(-1,-3.5)*{\scs f_1};
  (4,-12)*{}="t1";  (-4,-12)*{}="t2";
  "t2";"t1" **\crv{(-4,-5) & (4,-5)}; ?(1)*\dir{>}
    ?(.5)*\dir{}+(0,0)*{\bullet}+(2,3.5)*{\scs f_3};
  %(16,-10)*{n};
  (14,2)*{\ccbub{\spadesuit+f_2}{}};
  \endxy }
\]
\[
\gam_5  =\;\; \vcenter{
 \xy 0;/r.13pc/:
    (-4,-12)*{};(4,4)*{} **\crv{(-4,-1) & (4,1)} ?(0)*\dir{<};
    (4,-12)*{};(-4,4)*{} **\crv{(4,-1) & (-4,1)};
    (4,4)*{};(12,12)*{} **\crv{(4,7) & (12,9)};
    (12,4)*{};(4,12)*{} **\crv{(12,7) & (4,9)};
    (-4,4)*{}; (-4,12) **\dir{-}?(1)*\dir{>};
    (12,-12)*{}; (12,4) **\dir{-}?(0)*\dir{<};
    (4,4)*{\bullet}; %(18,8)*{n};
\endxy}
\;\; - \;\; \vcenter{
 \xy 0;/r.13pc/:
    (-4,-12)*{};(4,4)*{} **\crv{(-4,-1) & (4,1)} ?(0)*\dir{<};
    (4,-12)*{};(-4,4)*{} **\crv{(4,-1) & (-4,1)};
    (4,4)*{};(12,12)*{} **\crv{(4,7) & (12,9)};
    (12,4)*{};(4,12)*{} **\crv{(12,7) & (4,9)};
    (-4,4)*{}; (-4,12) **\dir{-}?(1)*\dir{>};
    (12,-12)*{}; (12,4) **\dir{-}?(0)*\dir{<};
    (-4.3,6)*{\bullet}; %(18,8)*{n};
\endxy}
 \;\; +\;\;
 \sum_{\xy (0,2)*{\scs f_1+f_2+f_3}; (0,-1)*{\scs =n-1}; \endxy}
\vcenter{  \xy 0;/r.13pc/:
  (-4,12)*{};(-12,-12)*{} **\crv{(-4,4) & (-12,-4)} ?(1)*\dir{>};
  (4,12)*{}="t1";
  (-12,12)*{}="t2";
  "t1";"t2" **\crv{(4,5) & (-12,5)};?(1)*\dir{>}?
   ?(.75)*\dir{}+(0,0)*{\bullet}+(-2,-3.5)*{\scs f_1};
  (4,-12)*{}="t1";  (-4,-12)*{}="t2";
  "t2";"t1" **\crv{(-4,-5) & (4,-5)}; ?(1)*\dir{>}
    ?(.5)*\dir{}+(0,0)*{\bullet}+(2,3.5)*{\scs f_3};
  %(16,-10)*{n};
  (14,2)*{\ccbub{\spadesuit+f_2}{}};
  \endxy }
\]
\[
\gam_6 =\;\; \vcenter{
 \xy 0;/r.13pc/:
    (-4,-12)*{};(4,4)*{} **\crv{(-4,-1) & (4,1)} ?(0)*\dir{<};
    (4,-12)*{};(-4,4)*{} **\crv{(4,-1) & (-4,1)};
    (4,4)*{};(12,12)*{} **\crv{(4,7) & (12,9)};
    (12,4)*{};(4,12)*{} **\crv{(12,7) & (4,9)};
    (-4,4)*{}; (-4,12) **\dir{-}?(1)*\dir{>};
    (12,-12)*{}; (12,4) **\dir{-}?(0)*\dir{<};
    (4,4)*{\bullet}; %(18,8)*{n};
\endxy}
 - \;\; \vcenter{
 \xy 0;/r.13pc/:
    (-4,-12)*{};(4,4)*{} **\crv{(-4,-1) & (4,1)} ?(0)*\dir{<};
    (4,-12)*{};(-4,4)*{} **\crv{(4,-1) & (-4,1)};
    (4,4)*{};(12,12)*{} **\crv{(4,7) & (12,9)};
    (12,4)*{};(4,12)*{} **\crv{(12,7) & (4,9)};
    (-4,4)*{}; (-4,12) **\dir{-}?(1)*\dir{>};
    (12,-12)*{}; (12,4) **\dir{-}?(0)*\dir{<};
    (-4.3,6)*{\bullet}; %(18,8)*{n};
\endxy}
 \;\; +\;\;
 \sum_{\xy (0,2)*{\scs f_1+f_2+f_3}; (0,-1)*{\scs =n-1};(0,-4)*{\scs 1 \leq f_1}; \endxy}
\vcenter{  \xy 0;/r.13pc/:
  (-4,12)*{};(-12,-12)*{} **\crv{(-4,4) & (-12,-4)} ?(1)*\dir{>};
  (4,12)*{}="t1";
  (-12,12)*{}="t2";
  "t1";"t2" **\crv{(4,5) & (-12,5)};?(1)*\dir{>}?
   ?(.75)*\dir{}+(0,0)*{\bullet}+(-2,-3.5)*{\scs f_1};
  (4,-12)*{}="t1";  (-4,-12)*{}="t2";
  "t2";"t1" **\crv{(-4,-5) & (4,-5)}; ?(1)*\dir{>}
    ?(.5)*\dir{}+(0,0)*{\bullet}+(-3,3.5)*{\scs f_3};
  (14,2)*{\ccbub{\spadesuit+f_2}{}};
  %(16,-10)*{n};
  \endxy }
\qquad
\gam_7\;\; =\;\; -\;\;
 \sum_{\xy (0,2)*{\scs f_1+f_2=n}; (0,-1)*{\scs 1 \leq f_1}; \endxy}
\vcenter{  \xy 0;/r.13pc/:
  (-4,12)*{};(-4,-12)*{} **\crv{(-4,4) & (-4,-4)} ?(1)*\dir{>};
  (4,12)*{}="t1";
  (-12,12)*{}="t2";
  "t1";"t2" **\crv{(4,5) & (-12,5)};?(1)*\dir{>}?
   ?(.75)*\dir{}+(0,0)*{\bullet}+(-2,-3.5)*{\scs f_1};
  (10,-3)*{\ccbub{\spadesuit+f_2}{}};
  %(22,-9)*{n};
  \endxy }
\]
\[
\gam_8  =\;\;
 \vcenter{   \xy 0;/r.13pc/:
    (-4,-12)*{};(4,4)*{} **\crv{(-1,-4) & (1,4)}?(0)*\dir{<};
    (-4,12)*{};(4,-4)*{} **\crv{(-1,4) & (1,-4)}?(1)*\dir{};;
    (4,-4)*{};(12,4)*{} **\crv{(7,-4) & (9,4)};
    (4,4)*{};(12,-4)*{} **\crv{(7,4) & (9,-4)};
    (12,-4)*{};(12,4)*{} **\crv{(18,-4) & (18,4)}?(.5)*\dir{<};
  %(8,8)*{n};
  (-2,-7)*{\bullet};
 \endxy}
\;\;+\;\; \delta_{n,1} \;
 \vcenter{
 \xy 0;/r.13pc/:
    (0,-12)*{}; (0,12) **\dir{-}?(0)*\dir{<};
    (12.3,-4)*{\ccbub{\spadesuit+1}{}}; (9,8)*{1};
\endxy}
\qquad \qquad
\gam_9  =\;\; \vcenter{   \xy 0;/r.13pc/:
    (-12,18)*{}; (-12,-4)*{} **\dir{-} ?(1)*\dir{>};
    (-4,-4)*{};(4,4)*{} **\crv{(-4,-1) & (4,1)}?(1)*\dir{>};
    (4,-4)*{};(-4,4)*{} **\crv{(4,-1) & (-4,1)}?(1)*\dir{};?(0)*\dir{<};
    (-4,4)*{};(4,12)*{} **\crv{(-4,7) & (4,9)};
    (4,4)*{};(-4,12)*{} **\crv{(4,7) & (-4,9)};
    (-4,12)*{};(4,12)*{} **\crv{(-4,18) & (4,18)}?(1)*\dir{>};
   %(8,8)*{n};
 \endxy}
\;\; -\;\; \delta_{n,1}\;\vcenter{
 \xy 0;/r.13pc/:
    (4,-12)*{};(-4,-12)*{} **\crv{(4,-4) & (-4,-4)} ?(0)*\dir{<};
    (-12,-12)*{}; (-12,12) **\dir{-}?(0)*\dir{<};
    (2,8)*{1};
\endxy}
\]

The map $\delt\onen$ is defined via $\gam\onen$ using symmetry 2-functors:
\begin{equation} \label{eq_delt_def}
   \delt\onen:= \cal{F}\hat{\varrho}^{\tpsi}\circ\tpsi(\gam\onen)\circ \varrho^{\tpsi}\cal{F}.
\end{equation}
More explicitly,
\[
\delt_1 =\;\; \vcenter{
 \xy 0;/r.13pc/:
    (-4,12)*{};(4,-4)*{} **\crv{(-4,1) & (4,-1)};
    (4,12)*{};(-4,-4)*{} **\crv{(4,1) & (-4,-1)}?(0)*\dir{<};
    (4,-4)*{};(12,-12)*{} **\crv{(4,-7) & (12,-9)}?(1)*\dir{>};
    (12,-4)*{};(4,-12)*{} **\crv{(12,-7) & (4,-9)}?(1)*\dir{>};
    (-4,-4)*{}; (-4,-12) **\dir{-};
    (12,12)*{}; (12,-4) **\dir{-};
    (4,-4)*{\bullet}; %(18,8)*{n};
\endxy}
 - \;\; \vcenter{
 \xy 0;/r.13pc/:
    (-4,12)*{};(4,-4)*{} **\crv{(-4,1) & (4,-1)};
    (4,12)*{};(-4,-4)*{} **\crv{(4,1) & (-4,-1)}?(0)*\dir{<};
    (4,-4)*{};(12,-12)*{} **\crv{(4,-7) & (12,-9)}?(1)*\dir{>};
    (12,-4)*{};(4,-12)*{} **\crv{(12,-7) & (4,-9)}?(1)*\dir{>};
    (-4,-4)*{}; (-4,-12) **\dir{-};
    (12,12)*{}; (12,-4) **\dir{-};
    (-4,-6)*{\bullet}; %(18,8)*{n};
\endxy}
 \;\; +\;\;
 \sum_{\xy (0,2)*{\scs f_1+f_2+f_3}; (0,-1)*{\scs =n-1};(0,-4)*{\scs 1 \leq f_1}; \endxy}
\vcenter{  \xy 0;/r.13pc/:
  (-4,-12)*{};(-12,12)*{} **\crv{(-4,-4) & (-12,4)} ?(0)*\dir{<};
  (4,-12)*{}="t1";  (-12,-12)*{}="t2";
  "t1";"t2" **\crv{(4,-5) & (-12,-5)};?(0)*\dir{<}
   ?(.75)*\dir{}+(0,0)*{\bullet}+(-2,3.5)*{\scs f_1};
  (4,12)*{}="t1";  (-4,12)*{}="t2";
  "t2";"t1" **\crv{(-4,5) & (4,5)}; ?(0)*\dir{<}
    ?(.5)*\dir{}+(0,0)*{\bullet}+(2,-3.5)*{\scs f_3};
  (14,-2)*{\ccbub{\spadesuit+f_2}{}}; %(16,10)*{n};
  \endxy }
\qquad
\delt_2 =\;\;
 \sum_{\xy (0,2)*{\scs f_1+f_2}; (0,-1)*{\scs =n-1};(0,-4)*{\scs 1 \leq f_1}; \endxy}
\vcenter{  \xy 0;/r.13pc/:
  (-4,-12)*{};(-12,12)*{} **\crv{(-4,-4) & (-12,4)} ?(0)*\dir{<};
  (4,-12)*{}="t1";  (-12,-12)*{}="t2";
  "t1";"t2" **\crv{(4,-5) & (-12,-5)};?(0)*\dir{<}
   ?(.75)*\dir{}+(0,0)*{\bullet}+(-3,3)*{\scs f_1};
  (4,2)*{\ccbub{\spadesuit+f_2}{}}; %(16,10)*{n};
  \endxy }
\]
\[
\delt_3  =\;\; -\;
\vcenter{   \xy 0;/r.13pc/:
    (-12,-18)*{}; (-12,4)*{} **\dir{-} ?(0)*\dir{<};
    (-4,4)*{};(4,-4)*{} **\crv{(-4,1) & (4,-1)}?(0)*\dir{<};
    (4,4)*{};(-4,-4)*{} **\crv{(4,1) & (-4,-1)}?(1)*\dir{};?(1)*\dir{>};
    (-4,-4)*{};(4,-12)*{} **\crv{(-4,-7) & (4,-9)};
    (4,-4)*{};(-4,-12)*{} **\crv{(4,-7) & (-4,-9)};
    (-4,-12)*{};(4,-12)*{} **\crv{(-4,-18) & (4,-18)}?(0)*\dir{<};
  %(8,-8)*{n};
 \endxy}
\;\; +\;\; \delta_{n,1}\;\vcenter{
 \xy 0;/r.13pc/:
    (4,12)*{};(-4,12)*{} **\crv{(4,4) & (-4,4)} ?(1)*\dir{>};
    (-12,12)*{}; (-12,-12) **\dir{-}?(1)*\dir{>};
    (2,-8)*{1};
\endxy}
\qquad
\delt_4 \;\; =\;\;
 \vcenter{   \xy 0;/r.13pc/:
    (-4,-12)*{};(4,4)*{} **\crv{(-1,-4) & (1,4)}?(0)*\dir{<};
    (-4,12)*{};(4,-4)*{} **\crv{(-1,4) & (1,-4)}?(1)*\dir{};;
    (4,-4)*{};(12,4)*{} **\crv{(7,-4) & (9,4)};
    (4,4)*{};(12,-4)*{} **\crv{(7,4) & (9,-4)};
    (12,-4)*{};(12,4)*{} **\crv{(18,-4) & (18,4)}?(.5)*\dir{<};
  (-2,-7)*{\bullet}; %(8,8)*{n};
 \endxy}
\;\;+\;\; \delta_{n,1}\;
 \vcenter{
 \xy 0;/r.13pc/:
    (0,-12)*{}; (0,12) **\dir{-}?(0)*\dir{<};
    (12.3,-4)*{\ccbub{\spadesuit+1}{}}; (9,8)*{1};
\endxy}
 \]
\[
\delt_5 =\;\; \vcenter{
 \xy 0;/r.13pc/:
    (-4,12)*{};(4,-4)*{} **\crv{(-4,1) & (4,-1)};
    (4,12)*{};(-4,-4)*{} **\crv{(4,1) & (-4,-1)}?(0)*\dir{<};
    (4,-4)*{};(12,-12)*{} **\crv{(4,-7) & (12,-9)}?(1)*\dir{>};
    (12,-4)*{};(4,-12)*{} **\crv{(12,-7) & (4,-9)}?(1)*\dir{>};
    (-4,-4)*{}; (-4,-12) **\dir{-};
    (12,12)*{}; (12,-4) **\dir{-};
    (4,-4)*{\bullet}; %(18,8)*{n};
\endxy}
 - \;\; \vcenter{
 \xy 0;/r.13pc/:
    (-4,12)*{};(4,-4)*{} **\crv{(-4,1) & (4,-1)};
    (4,12)*{};(-4,-4)*{} **\crv{(4,1) & (-4,-1)}?(0)*\dir{<};
    (4,-4)*{};(12,-12)*{} **\crv{(4,-7) & (12,-9)}?(1)*\dir{>};
    (12,-4)*{};(4,-12)*{} **\crv{(12,-7) & (4,-9)}?(1)*\dir{>};
    (-4,-4)*{}; (-4,-12) **\dir{-};
    (12,12)*{}; (12,-4) **\dir{-};
    (-4,-6)*{\bullet}; %(18,8)*{n};
\endxy}
 \;\; +\;\;
 \sum_{\xy (0,2)*{\scs f_1+f_2+f_3}; (0,-1)*{\scs =n-1}; \endxy}
\vcenter{  \xy 0;/r.13pc/:
  (-4,-12)*{};(-12,12)*{} **\crv{(-4,-4) & (-12,4)} ?(0)*\dir{<};
  (4,-12)*{}="t1";  (-12,-12)*{}="t2";
  "t1";"t2" **\crv{(4,-5) & (-12,-5)};?(0)*\dir{<}
   ?(.75)*\dir{}+(0,0)*{\bullet}+(-2,3.5)*{\scs f_1};
  (4,12)*{}="t1";  (-4,12)*{}="t2";
  "t2";"t1" **\crv{(-4,5) & (4,5)}; ?(0)*\dir{<}
    ?(.5)*\dir{}+(0,0)*{\bullet}+(2,-3.5)*{\scs f_3};
  (14,-2)*{\ccbub{\spadesuit+f_2}{}}; %(16,10)*{n};
  \endxy }
\qquad
\delt_6 \;\; = \;\;-\; \vcenter{
 \xy 0;/r.13pc/:
    (-4,-4)*{};(-12,12)*{} **\crv{(-4,3) & (-12,5)}?(0)*\dir{<};
    (-12,-4)*{};(-4,12)*{} **\crv{(-12,3) & (-4,5)}?(1)*\dir{>};
    (4,-4)*{}; (4,12) **\dir{-}?(0)*\dir{<};
\endxy}
 \;\; +\;\;
 \sum_{\xy (0,2)*{\scs f_1+f_2+f_3}; (0,-1)*{\scs =n-1}; \endxy}
\vcenter{  \xy 0;/r.13pc/:
  (-4,-12)*{};(-12,12)*{} **\crv{(-4,-4) & (-12,4)} ?(0)*\dir{<};
  (4,-12)*{}="t1";  (-12,-12)*{}="t2";
  "t1";"t2" **\crv{(4,-5) & (-12,-5)};?(0)*\dir{<}
   ?(.25)*\dir{}+(0,0)*{\bullet}+(1,4)*{\scs f_1};
  (4,12)*{}="t1";  (-4,12)*{}="t2";
  "t2";"t1" **\crv{(-4,5) & (4,5)}; ?(0)*\dir{<}
    ?(.5)*\dir{}+(0,0)*{\bullet}+(2,-3.5)*{\scs f_3};
  (14,-2)*{\ccbub{\spadesuit+f_2}{}}; %(16,10)*{n};
  \endxy }
\]
\[
 \delt_7 = \;\; \vcenter{
 \xy 0;/r.13pc/:
    (4,-12)*{};(-4,4)*{} **\crv{(4,-1) & (-4,1)}?(0)*\dir{<};
    (-4,-12)*{};(4,4)*{} **\crv{(-4,-1) & (4,1)}?(0)*\dir{<};
    (-4,4)*{};(-12,12)*{} **\crv{(-4,7) & (-12,9)};
    (-12,4)*{};(-4,12)*{} **\crv{(-12,7) & (-4,9)}?(1)*\dir{>};;
    (4,4)*{}; (4,12) **\dir{-};
    (-12,-12)*{}; (-12,4) **\dir{-};
  (3,-5.5)*{\bullet}; %(10,8)*{n};
\endxy}
 - \;\;
 \vcenter{
 \xy 0;/r.13pc/:
    (4,-12)*{};(-4,4)*{} **\crv{(4,-1) & (-4,1)}?(0)*\dir{<};
    (-4,-12)*{};(4,4)*{} **\crv{(-4,-1) & (4,1)}?(0)*\dir{<};
    (-4,4)*{};(-12,12)*{} **\crv{(-4,7) & (-12,9)};
    (-12,4)*{};(-4,12)*{} **\crv{(-12,7) & (-4,9)}?(1)*\dir{>};;
    (4,4)*{}; (4,12) **\dir{-};
    (-12,-12)*{}; (-12,4) **\dir{-};
  (-12.3,2)*{\bullet}; %(10,8)*{n};
\endxy}
  -\;\;
 \sum_{\xy (0,2)*{\scs f_1+f_2+f_3+f_4}; (0,-1)*{\scs =n-2}; \endxy}
\vcenter{  \xy 0;/r.13pc/:
  (4,-12)*{};(-12,12)*{} **\crv{(4,-4) & (-12,4)} ?(0)*\dir{<}
   ?(.75)*\dir{}+(0,0)*{\bullet}+(-4,-1)*{\scs f_2};;
  (-4,-12)*{}="t1";  (-12,-12)*{}="t2";
  "t1";"t2" **\crv{(-4,-5) & (-12,-5)};?(0)*\dir{<}
   ?(.25)*\dir{}+(0,0)*{\bullet}+(0,3.5)*{\scs f_1};
  (4,12)*{}="t1";  (-4,12)*{}="t2";
  "t2";"t1" **\crv{(-4,5) & (4,5)}; ?(0)*\dir{<}
    ?(.5)*\dir{}+(0,0)*{\bullet}+(2,-3.5)*{\scs f_3};
  (14,-2)*{\ccbub{\spadesuit+f_4}{}}; %(16,10)*{n};
  \endxy }
\]
\[
 \delt_8= \;\; \vcenter{
 \xy 0;/r.13pc/:
    (4,-12)*{};(-4,4)*{} **\crv{(4,-1) & (-4,1)}?(0)*\dir{<};
    (-4,-12)*{};(4,4)*{} **\crv{(-4,-1) & (4,1)}?(0)*\dir{<};
    (-4,4)*{};(-12,12)*{} **\crv{(-4,7) & (-12,9)};
    (-12,4)*{};(-4,12)*{} **\crv{(-12,7) & (-4,9)}?(1)*\dir{>};;
    (4,4)*{}; (4,12) **\dir{-};
    (-12,-12)*{}; (-12,4) **\dir{-};
  (3,-5.5)*{\bullet}; %(10,8)*{n};
\endxy}
 - \;\;
 \vcenter{
 \xy 0;/r.13pc/:
    (4,-12)*{};(-4,4)*{} **\crv{(4,-1) & (-4,1)}?(0)*\dir{<};
    (-4,-12)*{};(4,4)*{} **\crv{(-4,-1) & (4,1)}?(0)*\dir{<};
    (-4,4)*{};(-12,12)*{} **\crv{(-4,7) & (-12,9)};
    (-12,4)*{};(-4,12)*{} **\crv{(-12,7) & (-4,9)}?(1)*\dir{>};;
    (4,4)*{}; (4,12) **\dir{-};
    (-12,-12)*{}; (-12,4) **\dir{-};
  (-12.3,2)*{\bullet}; %(10,8)*{n};
\endxy}
 \;\; -\;\;
 \sum_{\xy (0,2)*{\scs f_1+f_2+f_3+f_4}; (0,-1)*{\scs =n-2};
  (0,-4)*{\scs 1 \leq f_3};\endxy}
\vcenter{  \xy 0;/r.13pc/:
  (4,-12)*{};(-12,12)*{} **\crv{(4,-4) & (-12,4)} ?(0)*\dir{<}
   ?(.75)*\dir{}+(0,0)*{\bullet}+(-4,-1)*{\scs f_2};;
  (-4,-12)*{}="t1";  (-12,-12)*{}="t2";
  "t1";"t2" **\crv{(-4,-5) & (-12,-5)};?(0)*\dir{<}
   ?(.25)*\dir{}+(0,0)*{\bullet}+(0,3.5)*{\scs f_1};
  (4,12)*{}="t1";  (-4,12)*{}="t2";
  "t2";"t1" **\crv{(-4,5) & (4,5)}; ?(0)*\dir{<}
    ?(.5)*\dir{}+(0,0)*{\bullet}+(2,-3.5)*{\scs f_3};
  (14,-2)*{\ccbub{\spadesuit+f_4}{}}; %(16,10)*{n};
  \endxy }
\quad
\delt_9  =\;\; - \;
 \vcenter{
 \xy 0;/r.13pc/:
    (-4,12)*{};(4,12)*{} **\crv{(-4,4) & (4,4)}?(1)*\dir{>};
    (12,-12)*{}; (12,12) **\dir{-}?(0)*\dir{<};
    %(18,8)*{n};
\endxy}
 \;\; +\;\;\sum_{\xy (0,2)*{\scs f_1+f_2+f_3}; (0,-1)*{\scs =n-1}; (0,-4)*{\scs 1 \leq f_3};\endxy}
\vcenter{  \xy 0;/r.13pc/:
  (-12,12)*{}; (-12,-12)*{} **\dir{-} ?(1)*\dir{>};
   ?(.5)*\dir{}+(0,0)*{\bullet}+(4,1)*{\scs f_1};
  (4,12)*{}="t1";  (-4,12)*{}="t2";
  "t2";"t1" **\crv{(-4,5) & (4,5)}; ?(1)*\dir{>}
    ?(.75)*\dir{}+(0,0)*{\bullet}+(3,-1)*{\scs f_3};
 (2,-5)*{\ccbub{\spadesuit+f_2}{}}; %(16,-10)*{n};
  \endxy }
\]
\end{defn}

\begin{rem}
Note that when $n\leq 0$ all terms involving bubbles are zero.
\end{rem}

\begin{prop} \label{prop_gam_delt}
The maps $\gam$ and $\delt$ defined above are mutually-inverse chain homotopy equivalences between $\cal{F}\cal{C}\onen$ and $\cal{C}\cal{F}\onen$.
\end{prop}

\begin{proof}
To check that the maps $\gam$ and $\delt$ are chain maps it suffices to verify this claim for the map $\gam$ since $\delt$ is defined from $\gam$ via \eqref{eq_delt_def}. This is proven by direct computation.

Below we give explicit chain homotopies $\Id - \gam \delt =hd+dh$ and $\Id-\delt\gam = h'd+dh'$ showing that $\gam$ and $\delt$ are mutually-inverse homotopy equivalences.

\begin{equation}
 \xy
  (-65,35)*+{\E{}\F{}\F{} \onen \{2\}}="1'";
  (-45,15)*+{\F{}\onen \{1+n\}}="2'";
  (-10,35)*+{\E{}\F{}\F{} \onen }="3'";
  (10,15)*+{\E{}\F{}\F{} \onen}="4'";
  (45,35)*+{\E{}\F{}\F{} \onen \{-2\}}="5'";
  (65,15)*+{\F{}\onen \{n-3\}}="6'";
   {\ar^{\sUupdot\sUdown\sUdown} "1'";"3'"};
   {\ar^{} "1'" ;"4'"};
   "1'"+(22,-3)*{\sUup\sUdowndot\sUdown};
   "3'"+(-7,-9)*{\sUcupl\sUdown};
   "3'"+(15,-8)*{\text{$\sUcapr\sUdown$}};
   "5'"+(-6.5,-8)*{\sUupdot\sUdown\sUdown};
   {\ar "2'";"3'"};
   {\ar_(.4){\text{$\sUcupl\sUdown$}} "2'";"4'"};
   {\ar^-{-\;\sUup\sUdowndot\sUdown} "3'";"5'"};
   {\ar "3'";"6'"};
   {\ar "4'";"5'"};
   {\ar_(.4){-\;\;\text{$\sUcapr\sUdown$}} "4'";"6'"};
  (-65,-15)*+{\F\E\F \onen \{2\}}="1";
  (-45,-35)*+{\F\onen \{1-n\}}="2";
  (-10,-15)*+{\F\E\F \onen }="3";
  (10,-35)*+{\F\E\F \onen}="4";
  (45,-15)*+{\F\E\F \onen \{-2\}}="5";
  (65,-35)*+{\F\onen \{n-1\}}="6";
   {\ar^{\sUdown\sUupdot\sUdown} "1";"3"};
   {\ar^{} "1" ;"4"};
   "4"+(-12,8)*{\sUdown\sUup\sUdowndot};
   "3"+(-16,-5)*{\sUdown\sUcupl};
   "6"+(-18,10)*{\text{$\sUdown\sUcapr$}};
   "4"+(14,4)*{\sUdown\sUupdot\sUdown};
   {\ar "2";"3"};
   {\ar_(.4){\text{$\sUdown\sUcupl$}} "2";"4"};
   {\ar^(.63){-\;\sUdown\sUup\sUdowndot} "3";"5"};
   {\ar "3";"6"};
   {\ar "4";"5"};
   {\ar_(.4){-\text{$\sUdown\sUcapr$}} "4";"6"};
  %% %%
%%   {\textcolor[rgb]{0.00,0.50,0.25}{\ar@/^0.25pc/^{\gam_1} "1"+(-6,4);"1'"+(-6,-4)}};
%%   {\textcolor[rgb]{0.00,0.50,0.25}{\ar@/^0.3pc/^{\gam_2} "1";"2'"+(-4,-3)}};
%%   {\textcolor[rgb]{0.00,0.50,0.25}{\ar@/^0.3pc/^{\gam_3} "3";"3'"}};
%%   {\textcolor[rgb]{0.00,0.50,0.25}{\ar^(.3){\gam_4} "3"+(4,3);"4'"}};
%%   {\textcolor[rgb]{0.00,0.50,0.25}{\ar@/^0.35pc/^(.7){\gam_5} "4"+(-1,4);"4'"+(1,-4)}};
%%    {\textcolor[rgb]{0.00,0.50,0.25}{\ar^{\gam_6} "5";"5'"+(-1,-3)}};
%%   {\textcolor[rgb]{0.00,0.50,0.25}{\ar@/^0.3pc/^(.3){\gam_9} "5"+(4,3);"6'"+(-1,-3)}};
%%   {\textcolor[rgb]{0.00,0.50,0.25}{\ar@/^0.3pc/^(.2){\gam_7} "6";"5'"+(4,-4)}};
%%   {\textcolor[rgb]{0.00,0.50,0.25}{\ar^{\gam_8} "6"+(4,4);"6'"+(4,-4)}};
%%   %%
%%   {\textcolor[rgb]{0.00,0.00,1.00}{\ar@/^0.25pc/^{\delt_1} "1'"+(-4,-4);"1"+(-4,4)}};
%%   {\textcolor[rgb]{0.00,0.00,1.00}{\ar@/^0.3pc/^{\delt_3} "2'";"1"+(3,3)}};
%%   {\textcolor[rgb]{0.00,0.00,1.00}{\ar@/^0.25pc/^(.7){\delt_2} "1'"+(4,-4);"2"+(-4,4)}};
%%   {\textcolor[rgb]{0.00,0.00,1.00}{\ar@/^0.3pc/^(.3){\delt_4} "2'";"2"}};
%%   {\textcolor[rgb]{0.00,0.00,1.00}{\ar@/^0.35pc/^{\delt_5} "3'"+(2,-3);"3"+(2,3)}};
%%   {\textcolor[rgb]{0.00,0.00,1.00}{\ar^(.55){\delt_6} "3'"+(5,-3);"4"+(-5,3)}};
%%   {\textcolor[rgb]{0.00,0.00,1.00}{\ar@/^0.3pc/^(.3){\delt_7} "4'"+(3,-4);"4"+(1,4)}};
%%   {\textcolor[rgb]{0.00,0.00,1.00}{\ar@/^0.3pc/^(.55){\delt_9} "6'"+(2,-3);"5"+(6,3)}};
%%   {\textcolor[rgb]{0.00,0.00,1.00}{\ar@/^0.25pc/^(.55){\delt_8} "5'"+(1,-3);"5"+(2,3)}};
%%   %%
%%      %%
   {\textcolor[rgb]{1.00,0.00,0.00}{\ar@/_1.8pc/_{h_1} "3'";"1'"}};
   {\textcolor[rgb]{1.00,0.00,0.00}{\ar@/_1.4pc/_{h_2} "3'";"2'"}};
   {\textcolor[rgb]{1.00,0.00,0.00}{\ar@/^1.4pc/^(.75){h_3}
   "4'";"1'"}};
   {\textcolor[rgb]{1.00,0.00,0.00}{\ar@/_1.8pc/_{h_4} "5'";"3'"}};
   {\textcolor[rgb]{1.00,0.00,0.00}{\ar@/_1.4pc/_{h_5} "5'";"4'"}};
   {\textcolor[rgb]{1.00,0.00,0.00}{\ar@/_1.2pc/_(.65){h_6}
   "6'";"4'"}};
   {\textcolor[rgb]{1.00,0.00,0.00}{\ar@/_1.8pc/_{h'_1} "3";"1"}};
   {\textcolor[rgb]{1.00,0.00,0.00}{\ar@/^.8pc/^(.75){h'_2}
   "4";"1"}};
   {\textcolor[rgb]{1.00,0.00,0.00}{\ar@/^.8pc/^(.75){h'_7}
   "6";"3"}};
   {\textcolor[rgb]{1.00,0.00,0.00}{\ar@/_1.8pc/_(.35){h'_{5}}
   "5";"3"}};
   {\textcolor[rgb]{1.00,0.00,0.00}{\ar@/^1.6pc/^{h'_8}
   "6";"4"}};
   {\textcolor[rgb]{1.00,0.00,0.00}{\ar@/^1.6pc/^{h'_{4}}
   "4";"2"}};
   {\textcolor[rgb]{1.00,0.00,0.00}{\ar@/^.8pc/^(.25){h'_3}
   "3";"2"}};
    {\textcolor[rgb]{1.00,0.00,0.00}{\ar@/_1.3pc/_(.58){h'_{6}}
   "5";"4"+(3,3)}};
 \endxy
\end{equation}

To simplify notation we write
\begin{equation}
  \Theta(n) :=
\left\{
\begin{array}{cc}
  1 & \text{if $n \geq 0$,} \\
  0 & \text{if $n<0$.}
\end{array}
\right.
\end{equation}

\begin{equation}
h_1 \;\; = \;\;
\vcenter{
 \xy 0;/r.17pc/:
    (-4,-4)*{};(-12,12)*{} **\crv{(-4,3) & (-12,5)}?(0)*\dir{<};
    (-12,-4)*{};(-4,12)*{} **\crv{(-12,3) & (-4,5)}?(0)*\dir{<};
    (-20,-4)*{}; (-20,12) **\dir{-}?(1)*\dir{>};
\endxy}
\;\; + \;\;
\sum_{\xy (0,2)*{\scs f_1+f_2+f_3+f_4}; (0,-1)*{\scs =n-3};\endxy}
\vcenter{  \xy 0;/r.15pc/:
  (-4,-12)*{};(4,12)*{} **\crv{(-4,-4) & (4,4)} ?(0)*\dir{<}
  ?(.75)*\dir{}+(0,0)*{\bullet}+(3.5,1)*{\scs f_2};
  (4,-12)*{}="t1";  (-12,-12)*{}="t2";
  "t1";"t2" **\crv{(4,-5) & (-12,-5)};?(0)*\dir{<}
   ?(.75)*\dir{}+(0,0)*{\bullet}+(-1,3.5)*{\scs f_1};
  (-4,12)*{}="t1";  (-12,12)*{}="t2";
  "t1";"t2" **\crv{(-4,5) & (-12,5)}; ?(1)*\dir{>}
    ?(.5)*\dir{}+(0,0)*{\bullet}+(1,-3)*{\scs f_3};
  %(16,10)*{n};
  (14,-2)*{\ccbub{\spadesuit+f_4}{}};
  \endxy }
\end{equation}
\begin{equation}
h_2 \;\; = \;\;
\sum_{\xy (0,2)*{\scs f_1+f_2+f_3}; (0,-1)*{\scs =n-2};\endxy}
\vcenter{  \xy 0;/r.15pc/:
  (-4,-12)*{};(-4,12)*{} **\crv{(-4,-4)  &(-4,4)} ?(0)*\dir{<}
  ?(.75)*\dir{}+(0,0)*{\bullet}+(3.5,1)*{\scs f_2};
  (4,-12)*{}="t1";  (-12,-12)*{}="t2";
  "t1";"t2" **\crv{(4,-5) & (-12,-5)};?(0)*\dir{<}
   ?(.75)*\dir{}+(0,0)*{\bullet}+(-1,3.5)*{\scs f_1};
  %(16,10)*{n};
  (14,0)*{\ccbub{\spadesuit+f_3}{}};
  \endxy }
  \qquad
-h_3 = h_4 \;\; = \;\;
\vcenter{
 \xy 0;/r.17pc/:
    (-4,-4)*{};(-12,12)*{} **\crv{(-4,3) & (-12,5)}?(0)*\dir{<};
    (-12,-4)*{};(-4,12)*{} **\crv{(-12,3) & (-4,5)}?(0)*\dir{<};
    (-20,-4)*{}; (-20,12) **\dir{-}?(1)*\dir{>};
\endxy}
\end{equation}
\begin{equation}
h_5 \;\; = \;\;
\vcenter{
 \xy 0;/r.17pc/:
    (-4,-4)*{};(-12,12)*{} **\crv{(-4,3) & (-12,5)}?(0)*\dir{<};
    (-12,-4)*{};(-4,12)*{} **\crv{(-12,3) & (-4,5)}?(0)*\dir{<};
    (-20,-4)*{}; (-20,12) **\dir{-}?(1)*\dir{>};
\endxy}
\;\; + \;\;
\sum_{\xy (0,2)*{\scs f_1+f_2+f_3+f_4}; (0,-1)*{\scs =n-3};\endxy}
\vcenter{  \xy 0;/r.15pc/:
  (-4,12)*{};(4,-12)*{} **\crv{(-4,4) & (4,-4)} ?(1)*\dir{>}
  ?(.75)*\dir{}+(0,0)*{\bullet}+(3.5,-1)*{\scs f_2};
  (4,12)*{}="t1";  (-12,12)*{}="t2";
  "t1";"t2" **\crv{(4,5) & (-12,5)};?(1)*\dir{>}
   ?(.75)*\dir{}+(0,0)*{\bullet}+(-1,-3.5)*{\scs f_1};
  (-4,-12)*{}="t1";  (-12,-12)*{}="t2";
  "t1";"t2" **\crv{(-4,-5) & (-12,-5)}; ?(0)*\dir{<}
    ?(.5)*\dir{}+(0,0)*{\bullet}+(1,3)*{\scs f_3};
  %(16,10)*{n};
  (14,2)*{\ccbub{\spadesuit+f_4}{}};
  \endxy }
\qquad
h_6 \;\; = \;\; -\;
\sum_{\xy (0,2)*{\scs f_1+f_2+f_3}; (0,-1)*{\scs =n-2};\endxy}
\vcenter{  \xy 0;/r.15pc/:
  (-4,12)*{};(-4,-12)*{} **\crv{(-4,4)  &(-4,-4)} ?(1)*\dir{>}
  ?(.75)*\dir{}+(0,0)*{\bullet}+(3.5,-1)*{\scs f_2};
  (4,12)*{}="t1";  (-12,12)*{}="t2";
  "t1";"t2" **\crv{(4,5) & (-12,5)};?(1)*\dir{>}
   ?(.75)*\dir{}+(0,0)*{\bullet}+(-1,-3.5)*{\scs f_1};
  %(16,10)*{n};
  (14,0)*{\ccbub{\spadesuit+f_3}{}};
  \endxy }
\end{equation}
\begin{equation}
h'_1  =   h'_{6} \;\; = \;\; -\;
 \vcenter{
 \xy 0;/r.15pc/:
    (4,-4)*{};(-4,4)*{} **\crv{(4,-1) & (-4,1)}?(0)*\dir{<};
    (-4,-4)*{};(4,4)*{} **\crv{(-4,-1) & (4,1)}?(1)*\dir{};
    (-4,4)*{};(-12,12)*{} **\crv{(-4,7) & (-12,9)}?(1)*\dir{};
    (-12,4)*{};(-4,12)*{} **\crv{(-12,7) & (-4,9)}?(1)*\dir{};
    (4,12)*{};(-4,20)*{} **\crv{(4,15) & (-4,17)} ?(1)*\dir{>};
    (-4,12)*{};(4,20)*{} **\crv{(-4,15) & (4,17)}?(0)*\dir{};
    (4,4)*{}; (4,12) **\dir{-};
    (-12,-4)*{}; (-12,4) **\dir{-} ?(0)*\dir{<};
    (-12,12)*{}; (-12,20) **\dir{-};
  %(12,8)*{n};
\endxy}
 \qquad
h'_2 = -h'_5 \;\; = \;\;
 \vcenter{
 \xy 0;/r.15pc/:
    (4,-4)*{};(-4,4)*{} **\crv{(4,-1) & (-4,1)}?(0)*\dir{<};
    (-4,-4)*{};(4,4)*{} **\crv{(-4,-1) & (4,1)}?(1)*\dir{};
    (-4,4)*{};(-12,12)*{} **\crv{(-4,7) & (-12,9)}?(1)*\dir{};
    (-12,4)*{};(-4,12)*{} **\crv{(-12,7) & (-4,9)}?(1)*\dir{};
    (4,12)*{};(-4,20)*{} **\crv{(4,15) & (-4,17)} ?(1)*\dir{>};
    (-4,12)*{};(4,20)*{} **\crv{(-4,15) & (4,17)}?(0)*\dir{};
    (4,4)*{}; (4,12) **\dir{-};
    (-12,-4)*{}; (-12,4) **\dir{-} ?(0)*\dir{<};
    (-12,12)*{}; (-12,20) **\dir{-};
  %(12,8)*{n};
\endxy}
 \;\; - \sum_{\xy (0,2)*{\scs f_1+f_2+f_3}; (0,-1)*{\scs =n-2}; \endxy} \;
\vcenter{  \xy 0;/r.15pc/:
  (-12,12)*{}; (-12,-12)*{} **\dir{-} ?(1)*\dir{>};
  % ?(.5)*\dir{}+(0,0)*{\bullet}+(4,1)*{\scs f_1};
  (4,-12)*{}="t1";  (-4,-12)*{}="t2";
  "t2";"t1" **\crv{(-4,-5) & (4,-5)}; ?(1)*\dir{>}
    ?(.25)*\dir{}+(0,0)*{\bullet}+(-3,3)*{\scs f_3};
  (4,12)*{}="t1";  (-4,12)*{}="t2";
  "t2";"t1" **\crv{(-4,5) & (4,5)}; ?(0)*\dir{<}
    ?(.25)*\dir{}+(0,0)*{\bullet}+(-3,-1)*{\scs f_1};
  %(18,-10)*{n};
  (10,1)*{\ccbub{\spadesuit+f_2}{}};
  \endxy }
\end{equation}

\begin{equation}
h'_3 \;\; = \;\;   \vcenter{   \xy 0;/r.15pc/:
    (-18,-12)*{};(-10,-12)*{} **\crv{(-18,-4) & (-10,-4)}?(0)*\dir{<};
    (-4,-12)*{};(4,4)*{} **\crv{(-1,-4) & (1,4)}?(0)*\dir{<};
    (-4,12)*{};(4,-4)*{} **\crv{(-1,4) & (1,-4)}?(1)*\dir{};;
    (4,-4)*{};(12,4)*{} **\crv{(7,-4) & (9,4)};
    (4,4)*{};(12,-4)*{} **\crv{(7,4) & (9,-4)};
    (12,-4)*{};(12,4)*{} **\crv{(18,-4) & (18,4)}?(.5)*\dir{<};
  (8,8)*{n};
 \endxy}
 \;\; -\;\;
  \vcenter{
 \xy 0;/r.15pc/:
    (4,-4)*{};(-4,4)*{} **\crv{(4,-1) & (-4,1)}?(0)*\dir{<};
    (-4,-4)*{};(4,4)*{} **\crv{(-4,-1) & (4,1)}?(1)*\dir{};
    (-4,4)*{};(-12,12)*{} **\crv{(-4,7) & (-12,9)}?(1)*\dir{};
    (-12,4)*{};(-4,12)*{} **\crv{(-12,7) & (-4,9)}?(1)*\dir{};
    (-4,12)*{};(4,12)*{} **\crv{(-3,15) & (4,15)}?(0)*\dir{};
    (4,4)*{}; (4,12) **\dir{-};
    (-12,-4)*{}; (-12,4) **\dir{-} ?(0)*\dir{<};
    (-12,12)*{}; (-12,20) **\dir{-};
  (12,8)*{n};
\endxy}
\;\; =\;\;
\Theta(n-1)\;
 \vcenter{
 \xy 0;/r.15pc/:
    (-4,-12)*{};(4,-12)*{} **\crv{(-4,-4) & (4,-4)}?(0)*\dir{<};
    (12,12)*{}; (12,-12) **\dir{-}?(1)*\dir{>};
    (18,-8)*{n};
\endxy}
 \;\; -\;\;\sum_{\xy (0,2)*{\scs f_1+f_2+f_3}; (0,-1)*{\scs =n-1};\endxy}
\vcenter{  \xy 0;/r.15pc/:
  (-12,-12)*{}; (-12,12)*{} **\dir{-} ?(0)*\dir{<};
   ?(.5)*\dir{}+(0,0)*{\bullet}+(4,-1)*{\scs f_1};
  (4,-12)*{}="t1";  (-4,-12)*{}="t2";
  "t2";"t1" **\crv{(-4,-5) & (4,-5)}; ?(1)*\dir{>}
    ?(.75)*\dir{}+(0,0)*{\bullet}+(3,1)*{\scs f_3};
  (16,10)*{n};
  (2,5)*{\ccbub{\spadesuit+f_2}{}};
  \endxy }
\end{equation}
\begin{equation}
h'_4 \;\; = \;\; -\;
 \vcenter{
 \xy 0;/r.15pc/:
    (-4,-12)*{};(4,-12)*{} **\crv{(-4,-4) & (4,-4)}?(0)*\dir{<};
    (12,12)*{}; (12,-12) **\dir{-}?(1)*\dir{>};
    %(18,-8)*{n};
\endxy}
 \;\; +\;\;\sum_{\xy (0,2)*{\scs f_1+f_2+f_3}; (0,-1)*{\scs =n-1}; (0,-4)*{\scs 1 \leq f_1};\endxy}
\vcenter{  \xy 0;/r.15pc/:
  (-12,-12)*{}; (-12,12)*{} **\dir{-} ?(0)*\dir{<};
   ?(.5)*\dir{}+(0,0)*{\bullet}+(4,-1)*{\scs f_1};
  (4,-12)*{}="t1";  (-4,-12)*{}="t2";
  "t2";"t1" **\crv{(-4,-5) & (4,-5)}; ?(1)*\dir{>}
    ?(.75)*\dir{}+(0,0)*{\bullet}+(3,1)*{\scs f_3};
  %(16,10)*{n};
  (2,5)*{\ccbub{\spadesuit+f_2}{}};
  \endxy }
\;\; = \;\;   \vcenter{
 \xy 0;/r.15pc/:
    (4,-4)*{};(-4,4)*{} **\crv{(4,-1) & (-4,1)}?(0)*\dir{<};
    (-4,-4)*{};(4,4)*{} **\crv{(-4,-1) & (4,1)}?(1)*\dir{};
    (-4,4)*{};(-12,12)*{} **\crv{(-4,7) & (-12,9)}?(1)*\dir{};
    (-12,4)*{};(-4,12)*{} **\crv{(-12,7) & (-4,9)}?(1)*\dir{};
    (-4,12)*{};(4,12)*{} **\crv{(-3,15) & (4,15)}?(0)*\dir{};
    (4,4)*{}; (4,12) **\dir{-};
    (-12,-4)*{}; (-12,4) **\dir{-} ?(0)*\dir{<};
    (-12,12)*{}; (-12,20) **\dir{-};
  %(12,8)*{n};
\endxy}
 \;\; -\;\;\sum_{\xy (0,2)*{\scs f_1+f_2}; (0,-1)*{\scs =n-1}; \endxy} \;
\vcenter{  \xy 0;/r.15pc/:
  (-12,-12)*{}; (-12,12)*{} **\dir{-} ?(0)*\dir{<};
   ?(.5)*\dir{};
  (4,-12)*{}="t1";  (-4,-12)*{}="t2";
  "t2";"t1" **\crv{(-4,-5) & (4,-5)}; ?(1)*\dir{>}
    ?(.75)*\dir{}+(0,0)*{\bullet}+(3,1)*{\scs f_1};
  %(16,10)*{n};
  (2,5)*{\ccbub{\spadesuit+f_2}{}};
  \endxy }
\end{equation}
\begin{equation}
h'_7 \;\; = \;\; -\;
 \vcenter{
 \xy 0;/r.15pc/:
    (-4,12)*{};(4,12)*{} **\crv{(-4,4) & (4,4)}?(1)*\dir{>};
    (12,12)*{}; (12,-12) **\dir{-}?(1)*\dir{>};
    %(18,-8)*{n};
\endxy}
 \;\; +\;\;\sum_{\xy (0,2)*{\scs f_1+f_2+f_3}; (0,-1)*{\scs =n-1}; (0,-4)*{\scs 1 \leq f_1};\endxy}
\vcenter{  \xy 0;/r.15pc/:
  (-12,-12)*{}; (-12,12)*{} **\dir{-} ?(0)*\dir{<};
   ?(.5)*\dir{}+(0,0)*{\bullet}+(4,-1)*{\scs f_1};
  (4,12)*{}="t1";  (-4,12)*{}="t2";
  "t2";"t1" **\crv{(-4,5) & (4,5)}; ?(0)*\dir{<}
    ?(.75)*\dir{}+(0,0)*{\bullet}+(3,-1)*{\scs f_3};
  %(16,-10)*{n};
  (2,-5)*{\ccbub{\spadesuit+f_2}{}};
  \endxy }
 \;\; = \;\;
  \vcenter{
 \xy 0;/r.15pc/:
    (4,4)*{};(-4,-4)*{} **\crv{(4,1) & (-4,-1)}?(0)*\dir{};
    (-4,4)*{};(4,-4)*{} **\crv{(-4,1) & (4,-1)}?(0)*\dir{<};
    (-4,-4)*{};(-12,-12)*{} **\crv{(-4,-7) & (-12,-9)}?(0)*\dir{};
    (-12,-4)*{};(-4,-12)*{} **\crv{(-12,-7) & (-4,-9)}?(1)*\dir{};
    (-4,-12)*{};(4,-12)*{} **\crv{(-3,-15) & (4,-15)}?(0)*\dir{};
    (4,-4)*{}; (4,-12) **\dir{-};
    (-12,4)*{}; (-12,-4) **\dir{-} ?(0)*\dir{};
    (-12,-12)*{}; (-12,-20) **\dir{-}?(1)*\dir{>};
  %(12,-8)*{n};
\endxy}
 \;\; -\;\;\sum_{\xy (0,2)*{\scs f_1+f_2}; (0,-1)*{\scs =n-1}; \endxy}
\vcenter{  \xy 0;/r.15pc/:
  (-12,-12)*{}; (-12,12)*{} **\dir{-} ?(0)*\dir{<};
   ?(.5)*\dir{};
  (4,12)*{}="t1";  (-4,12)*{}="t2";
  "t2";"t1" **\crv{(-4,5) & (4,5)}; ?(0)*\dir{<}
    ?(.75)*\dir{}+(0,0)*{\bullet}+(3,-1)*{\scs f_1};
  %(16,-10)*{n};
  (2,-5)*{\ccbub{\spadesuit+f_2}{}};
  \endxy }
\end{equation}
\begin{equation}
h'_8 \;\; = \;\; -\;  \vcenter{   \xy 0;/r.15pc/:
    (-18,12)*{};(-10,12)*{} **\crv{(-18,4) & (-10,4)}?(1)*\dir{>};
    (-4,-12)*{};(4,4)*{} **\crv{(-1,-4) & (1,4)}?(0)*\dir{<};
    (-4,12)*{};(4,-4)*{} **\crv{(-1,4) & (1,-4)}?(1)*\dir{};;
    (4,-4)*{};(12,4)*{} **\crv{(7,-4) & (9,4)};
    (4,4)*{};(12,-4)*{} **\crv{(7,4) & (9,-4)};
    (12,-4)*{};(12,4)*{} **\crv{(18,-4) & (18,4)}?(.5)*\dir{<};
  (8,8)*{n};
 \endxy}
 \;\; +\;\;
  \vcenter{
 \xy 0;/r.15pc/:
    (4,4)*{};(-4,-4)*{} **\crv{(4,1) & (-4,-1)}?(0)*\dir{};
    (-4,4)*{};(4,-4)*{} **\crv{(-4,1) & (4,-1)}?(0)*\dir{<};
    (-4,-4)*{};(-12,-12)*{} **\crv{(-4,-7) & (-12,-9)}?(0)*\dir{};
    (-12,-4)*{};(-4,-12)*{} **\crv{(-12,-7) & (-4,-9)}?(1)*\dir{};
    (-4,-12)*{};(4,-12)*{} **\crv{(-3,-15) & (4,-15)}?(0)*\dir{};
    (4,-4)*{}; (4,-12) **\dir{-};
    (-12,4)*{}; (-12,-4) **\dir{-} ?(0)*\dir{};
    (-12,-12)*{}; (-12,-20) **\dir{-}?(1)*\dir{>};
  (12,-8)*{n};
\endxy}
\;\; =\;\;
\Theta(n-1)\;
 \vcenter{
 \xy 0;/r.15pc/:
    (-4,12)*{};(4,12)*{} **\crv{(-4,4) & (4,4)}?(1)*\dir{>};
    (12,12)*{}; (12,-12) **\dir{-}?(1)*\dir{>};
    (18,-8)*{n};
\endxy}
 \;\; -\;\;\sum_{\xy (0,2)*{\scs f_1+f_2+f_3}; (0,-1)*{\scs =n-1};\endxy}
\vcenter{  \xy 0;/r.15pc/:
  (-12,-12)*{}; (-12,12)*{} **\dir{-} ?(0)*\dir{<};
   ?(.5)*\dir{}+(0,0)*{\bullet}+(4,-1)*{\scs f_1};
  (4,12)*{}="t1";  (-4,12)*{}="t2";
  "t2";"t1" **\crv{(-4,5) & (4,5)}; ?(0)*\dir{<}
    ?(.75)*\dir{}+(0,0)*{\bullet}+(3,-1)*{\scs f_3};
  (16,-10)*{n};(2,-5)*{\ccbub{\spadesuit+f_2}{}};
  \endxy }
\end{equation}

The rather nontrivial computation that these maps are the required homotopies is omitted from the paper.

\end{proof}

% ------------------------------------------------------------------------------
%
\subsubsection{Commutativity chain maps $\ogam$ and $\odelt$}
%
% ------------------------------------------------------------------------------

\begin{defn}
We introduce chain maps
\begin{equation} \label{eq_def_overline_gamma}
   \ogam = \ogam\onen:=
   \cal{E} \hat{\varrho}^{\tsigma \tomega} \circ
   \tomega\tsigma(\gam\onen)\circ \varrho^{\tsigma\tomega}\cal{E} \maps \cal{C}\cal{E}\onen \to \cal{E}\cal{C}\onen
 \end{equation}
and
 \begin{equation} \label{eq_def_overline_delta}
   \odelt = \odelt\onen :=
    \hat{\varrho}^{\tsigma \tomega} \cal{E}\circ
   \tomega\tsigma(\delt\onen)\circ \cal{E}\varrho^{\tsigma\tomega} \maps \cal{E}\cal{C}\onen\to \cal{C}\cal{E}\onen.
 \end{equation}
\end{defn}

The diagrams
\begin{equation}
      \xy
   (18,-10)*+{\cal{E}\tsigma\tomega(\cal{C}\onen)\onen}="tl";
   (-18,-10)*+{\tsigma\tomega(\cal{C}\onen)\cal{E}\onen}="tr";
   (18,10)*+{\cal{E}\cal{C}\onen}="bl";
   (-18,10)*+{\cal{C}\cal{E}\onen}="br";
   {\ar_{\tomega\tsigma(\gam)} "tr";"tl"};
   {\ar_{\cal{E} \hat{\varrho}^{\tsigma \tomega}} "tl";"bl"};
   {\ar^{\ogam} "br";"bl"};
   {\ar_{\varrho^{\tsigma \tomega}\cal{E}} "br";"tr"};
  \endxy
  \qquad
    \xy
   (-18,-10)*+{\cal{E}\tsigma\tomega(\cal{C}\onen)\onen}="tl";
   (18,-10)*+{\tsigma\tomega(\cal{C}\onen)\cal{E}\onen}="tr";
   (-18,10)*+{\cal{E}\cal{C}\onen}="bl";
   (18,10)*+{\cal{C}\cal{E}\onen}="br";
   {\ar_{ \tomega\tsigma(\delt)} "tl";"tr"};
   {\ar_{\cal{E}\varrho^{\tsigma \tomega}} "bl";"tl"};
   {\ar^{\odelt} "bl";"br"};
   {\ar_{\hat{\varrho}^{\tsigma \tomega}\cal{E}} "tr";"br"};
  \endxy.
\end{equation}
 commute by definition.

\begin{prop}
The maps $\ogam$ and $\odelt$ defined above are mutually-inverse chain homotopy equivalences between $\cal{E}\cal{C}\onen$ and $\cal{C}\cal{E}\onen$.
\end{prop}

\begin{proof}
The Proposition follows at once from \eqref{eq_def_overline_gamma}, \eqref{eq_def_overline_delta}, and Proposition~\ref{prop_gam_delt}.
\end{proof}

% ====================================================================
%
\subsection{Indecomposability} \label{sec_indec}
%
% ====================================================================

% ---------------------------------------------------------------------------
%
\subsubsection{Indecomposability of Casimir complexes}
%
% ---------------------------------------------------------------------------

Since $\onen$ is indecomposable for all $n$, morphisms $\onen\{1-n\}$, $\onen\{n-1\}$ are as well. The 1-morphisms $\cal{E}\cal{F}\onen\{-2\}$, $\cal{E}\cal{F}\onen$, and $\cal{E}\cal{F}\onen\{2\}$ appearing in various direct summands of the complex $\cal{C}\onen$ are indecomposable when $n \leq 0$.   Assuming $n \leq 0$ each of the eight maps describing the differential in $\cal{C}\onen$ belongs to the graded Jacobson radical of the category $\UcatD(n,n)$. This implies that $\cal{C}\onen$ does not contain any contractible summands if $n\leq 0$.

\begin{prop}
  The complex $\cal{C}\onen$ given by \eqref{eq_casimir_EF_INTRO} is indecomposable in $Kom(\UcatD)$ if $n \leq 0$.
\end{prop}

\begin{proof}
Assume $n \leq 0$ so that $\cal{C}\onen$ consists of 6 terms, all indecomposable.  Those terms belong to the category $\UcatD(n,n)$, which is Krull-Schmidt with finite dimensional hom spaces.

There are no homs (of degree zero) $\onen\{1-n\} \to \cal{E}\cal{F}\onen\{2\}$ for any $n$.  The degree zero hom space $\cal{E}\cal{F}\onen\{2\} \to \onen\{1-n\}$ (for $n\leq 0$) is nontrivial only when $n=0$, and then it is spanned by the diagram $\;\vcenter{\xy (-2,-3)*{}; (2,-3)*{} **\crv{(-2,1) & (2,1)}?(1)*\dir{>};
 \endxy}\;$.  Consequently, $(\cal{C}\onen)^{-1} = \cal{E}\cal{F}\onen\{2\} \oplus \onen\{1-n\}$ has only one possible direct sum decomposition for $n< 0$, and a one-parameter family of direct sum decompositions for $n=0$
\begin{equation}
  (\cal{C}\onenn{0})^{-1} = X \oplus \onenn{0}\{1\},
\end{equation}
where $X\cong \cal{E}\cal{F}\onenn{0}\{2\}$ is the image of $\cal{E}\cal{F}\onenn{0}\{2\}$ in $(\cal{C}\onenn{0})^{-1}$ under the homomorphism
\begin{equation}
  \xymatrix{
  \cal{E}\cal{F}\onenn{0}\{2\}
  \ar[rrr]^-{\left(\begin{array}{c}
             \Id \\ a \Ucapr
           \end{array}\right)
  }
  &&& \cal{E}\cal{F}\onenn{0}\{2\} \oplus \onenn{0}\{1\}
  }
\end{equation}
for $a\in \Bbbk$. Any direct sum decomposition of $(\cal{C}\onen)^0 = \cal{E}\cal{F}\onen \oplus \cal{E}\cal{F}\onen$ is determined by a $2\times 2$ invertible matrix with coefficients in $\Bbbk$.   There are no homs (of degree zero) $\cal{E}\cal{F}\onen\{-2\} \to \onen\{n-1\}$, and the degree zero hom space $\onen\{n-1\} \to \cal{E}\cal{F}\onen\{-2\}$ is nontrivial only when $n=0$, and then it is spanned by $\;\vcenter{\xy (2,3)*{}; (-2,3)*{} **\crv{(2,-1) & (-2,-1)}?(1)*\dir{>};
\endxy}$. Therefore, for $n<0$ direct sum decomposition of $(\cal{C}\onen)^1$ is unique, and for $n=0$ any direct sum decomposition in $\UcatD(0,0)$ of $(\cal{C}\onenn{0})^{1} = \cal{E}\cal{F}\onenn{0}\{-2\} \oplus \onenn{0}\{-1\}$ has the form
\begin{equation}
  (\cal{C}\onenn{0})^1 \simeq \cal{E}\cal{F}\onenn{0}\{-2\} \oplus Y
\end{equation}
where $Y \simeq \onenn{0}\{n-1\}$ is the image of $\onenn{0}\{-1\}$ in $(\cal{C}\onenn{0})^{1}$ under the homomorphism
\begin{equation}
  \xymatrix{
  \onenn{0}\{-1\}
  \ar[rrr]^-{\left(\begin{array}{c}
              b \Ucupl \\ \Id
           \end{array}\right)
  }
  &&& \cal{E}\cal{F}\onenn{0}\{-2\} \oplus \onenn{0}\{-1\}
  }
\end{equation}
for some $b \in \Bbbk$.

Suppose that for some $n \leq 0$ there exists a nontrivial direct sum decomposition $\cal{C}\onen \simeq \cal{C}_1 \oplus \cal{C}_2$ in $Kom(\UcatD)$.  Then, from the above discussion, we know that the summand $\onen\{1-n\} \subset (\cal{C}\onen)^{-1}$ must be either in $\cal{C}_1$ or $\cal{C}_2$. We can assume it belongs to $\cal{C}_1$.  Applying the differential in $\cal{C}\onen$ and the classification of direct sum decompositions of $(\cal{C}\onen)^0$ we see that $\cal{C}_1$ must contain the diagonal summand of $(\cal{C}\onen)^0$, the image of
\begin{equation}
  \xymatrix{
  \cal{E}\cal{F}\onen
  \ar[rrr]^-{\left(\begin{array}{c}
              \Id
              \\
               \Id
           \end{array}\right)}
  &&& \cal{E}\cal{F}\onen \oplus \cal{E}\cal{F}\onen.
  }
\end{equation}
Further application of the differential and very few available direct sum decompositions of $(\cal{C}\onen)^1$ tell us that $\cal{C}_1$ must contain the summand $\cal{E}\cal{F}\onen\{-2\}$ of $(\cal{C}\onen)^1$.

If $(\cal{C}_2)^{-1} \neq 0$, then $(\cal{C}_2)^{-1}=X$, described above, for some $a \in \Bbbk$. Then $dX$ must lie inside a summand of $(\cal{C}\onen)^0$ isomorphic to $\cal{E}\cal{F}\onen$.  A simple computation shows that this is impossible.  Therefore, $(\cal{C}_2)^{-1}=0$ and $(\cal{C}_{-1})=(\cal{C}\onen)^{-1}$.  Applying the differential to $(\cal{C}_1)^{-1}$ we quickly conclude that $(\cal{C}_1)^0=(\cal{C}\onen)^0$ and then $(\cal{C}_1)^{1}=(\cal{C}\onen)^1$.  Hence $\cal{C}_2=0$ supplying a contradiction.
\end{proof}

\begin{cor}
The complex $\cal{C}'\onen=\tsigma(\cal{C}\onenn{-n})$ is indecomposable in $Kom(\UcatD)$ if $n \geq 0$.
\end{cor}

Assume $n \leq 0$. Since $\cal{C}\onen$ and $\cal{C}'\onen$ are isomorphic in $Com(\UcatD)$ by Proposition~\ref{prop_homotopy_sym}, and $\cal{F}\cal{E}\onen \cong \cal{E}\cal{F}\onen \oplus_{[n]}\onen$, we conclude that $\cal{C}'\onen$ is isomorphic in $Kom(\UcatD)$ to the direct sum of $\cal{C}\onen$,  contractible complexes
\begin{equation}
  \xymatrix{
  0 \ar[r] & \onen\{n-1-2\ell+2\} \ar[r]^{\Id} & \onen\{n-1-2\ell+2\} \ar[r] & 0
  }, \qquad 0 \leq \ell \leq n-1
\end{equation}
concentrated in cohomological degrees -1 and 0, and contractible complexes
\begin{equation}
  \xymatrix{
  0 \ar[r] & \onen\{n-1-2\ell-2\} \ar[r]^{\Id} & \onen\{n-1-2\ell-2\} \ar[r] & 0
  }, \qquad 0 \leq \ell \leq n-1
\end{equation}
concentrated in cohomological degrees 0 and 1.  When $n=0$ complexes $\cal{C}\onen$ and $\cal{C}'\onen$ are isomorphic in $Kom(\UcatD)$ via Proposition~\ref{prop_homotopy_sym}.

When $n \geq 0$ there is a similar isomorphism in $Kom(\UcatD)$ between $\cal{C}\onen$ and the direct sum of $\cal{C}'\onen$ and contractible complexes.

In the complex $\cal{C}'\onen=\tsigma(\cal{C}\onenn{-n})$ the 1-morphisms $\onen\{1+n\}$, $\onen\{-n-1\}$, $\cal{F}\cal{E}\onen\{-2\}$, $\cal{F}\cal{E}\onen$, and $\cal{F}\cal{E}\onen\{2\}$ in the direct summands are all indecomposable when $n \geq 0$.  In this way, the symmetry 2-functor $\tsigma$ plays an important role allowing us to switch between complexes $\cal{C}\onen$ and $\cal{C}'\onen$.

The commutativity chain maps studied above reduce drastically when we work with the indecomposable version of the Casimir complex.  Below we collect these maps for later convenience.

%-----------------------------------------------------------------------------
%
\subsubsection{Chain maps $\gam$ and $\delt$ in indecomposable case}
%
%-----------------------------------------------------------------------------

When $n \leq 0$ the maps in Definition~\ref{def_xi_minus} simplify to
\[
 \gam_1 = \gam_3 =\;\;
 \vcenter{
 \xy 0;/r.13pc/:
    (-4,-4)*{};(4,4)*{} **\crv{(-4,-1) & (4,1)} ?(0)*\dir{<};
    (4,-4)*{};(-4,4)*{} **\crv{(4,-1) & (-4,1)};
    (4,4)*{};(12,12)*{} **\crv{(4,7) & (12,9)};
    (12,4)*{};(4,12)*{} **\crv{(12,7) & (4,9)};
    (-4,4)*{}; (-4,12) **\dir{-}?(1)*\dir{>};
    (12,-4)*{}; (12,4) **\dir{-}?(0)*\dir{<};
  (10,9.5)*{\bullet}; %(18,8)*{n};
\endxy}
\;\;  - \;\;\vcenter{
 \xy 0;/r.13pc/:
    (-4,-4)*{};(4,4)*{} **\crv{(-4,-1) & (4,1)} ?(0)*\dir{<};
    (4,-4)*{};(-4,4)*{} **\crv{(4,-1) & (-4,1)};
    (4,4)*{};(12,12)*{} **\crv{(4,7) & (12,9)};
    (12,4)*{};(4,12)*{} **\crv{(12,7) & (4,9)};
    (-4,4)*{}; (-4,12) **\dir{-}?(1)*\dir{>};
    (12,-4)*{}; (12,4) **\dir{-}?(0)*\dir{<};
    (-4.3,6)*{\bullet}; %(18,8)*{n};
\endxy}
\qquad
\gam_2  =\;\; \vcenter{
 \xy 0;/r.13pc/:
    (-4,-4)*{};(4,-4)*{} **\crv{(-4,4) & (4,4)} ?(0)*\dir{<};
    (12,-4)*{}; (12,12) **\dir{-}?(0)*\dir{<};
   %(4,8)*{\bullet}; (18,8)*{n};
\endxy}
\qquad
\gam_4  = \;\; \vcenter{
 \xy 0;/r.13pc/:
    (-4,-8)*{};(-12,8)*{} **\crv{(-4,-1) & (-12,1)}?(1)*\dir{>};
    (-12,-8)*{};(-4,8)*{} **\crv{(-12,-1) & (-4,1)}?(0)*\dir{<};
    (4,-8)*{}; (4,8) **\dir{-}?(0)*\dir{<};
\endxy}
\]
\[
\gam_5 = \gam_6 =\;\; \vcenter{
 \xy 0;/r.13pc/:
    (-4,-4)*{};(4,4)*{} **\crv{(-4,-1) & (4,1)} ?(0)*\dir{<};
    (4,-4)*{};(-4,4)*{} **\crv{(4,-1) & (-4,1)};
    (4,4)*{};(12,12)*{} **\crv{(4,7) & (12,9)};
    (12,4)*{};(4,12)*{} **\crv{(12,7) & (4,9)};
    (-4,4)*{}; (-4,12) **\dir{-}?(1)*\dir{>};
    (12,-4)*{}; (12,4) **\dir{-}?(0)*\dir{<};
    (4,4)*{\bullet}; %(18,8)*{n};
\endxy}
\;\; - \;\; \vcenter{
 \xy 0;/r.13pc/:
    (-4,-4)*{};(4,4)*{} **\crv{(-4,-1) & (4,1)} ?(0)*\dir{<};
    (4,-4)*{};(-4,4)*{} **\crv{(4,-1) & (-4,1)};
    (4,4)*{};(12,12)*{} **\crv{(4,7) & (12,9)};
    (12,4)*{};(4,12)*{} **\crv{(12,7) & (4,9)};
    (-4,4)*{}; (-4,12) **\dir{-}?(1)*\dir{>};
    (12,-4)*{}; (12,4) **\dir{-}?(0)*\dir{<};
    (-4.3,6)*{\bullet}; %(18,8)*{n};
\endxy}
\qquad
\gam_8  =\;\;
-\;\;
 \vcenter{
 \xy 0;/r.13pc/:
    (0,-4)*{}; (0,12) **\dir{-}?(0)*\dir{<};
    (0,4)*{\bullet};
     %(9,8)*{n};
\endxy}
\quad
\gam_9  =
\;\; -\;\; \vcenter{
 \xy 0;/r.13pc/:
    (4,-12)*{};(-4,-12)*{} **\crv{(4,-4) & (-4,-4)} ?(0)*\dir{<};
    (-12,-12)*{}; (-12,4) **\dir{-}?(0)*\dir{<};
    %(2,8)*{n};
\endxy}
\qquad
\gam_7\;\; = 0
\]

\[
\delt_1 = \delt_5\;\;= \;\;\vcenter{
 \xy 0;/r.13pc/:
    (4,-4)*{};(-4,4)*{} **\crv{(4,-1) & (-4,1)}?(0)*\dir{<};
    (-4,-4)*{};(4,4)*{} **\crv{(-4,-1) & (4,1)}?(0)*\dir{<};
    (-4,4)*{};(-12,12)*{} **\crv{(-4,7) & (-12,9)};
    (-12,4)*{};(-4,12)*{} **\crv{(-12,7) & (-4,9)}?(1)*\dir{>};;
    (4,4)*{}; (4,12) **\dir{-};
    (-12,-4)*{}; (-12,4) **\dir{-};
   (-4.3,4)*{\bullet}; %(18,8)*{n};
\endxy}
\;\; - \;\;
 \vcenter{
 \xy 0;/r.13pc/:
    (4,-4)*{};(-4,4)*{} **\crv{(4,-1) & (-4,1)}?(0)*\dir{<};
    (-4,-4)*{};(4,4)*{} **\crv{(-4,-1) & (4,1)}?(0)*\dir{<};
    (-4,4)*{};(-12,12)*{} **\crv{(-4,7) & (-12,9)};
    (-12,4)*{};(-4,12)*{} **\crv{(-12,7) & (-4,9)}?(1)*\dir{>};;
    (4,4)*{}; (4,12) **\dir{-};
    (-12,-4)*{}; (-12,4) **\dir{-};
  (-12.3,2)*{\bullet}; %(18,8)*{n};
\endxy}
\qquad
\delt_2 =\;\;0
\qquad
\delt_3  =
\;\; \vcenter{
 \xy 0;/r.13pc/:
    (4,12)*{};(-4,12)*{} **\crv{(4,4) & (-4,4)} ?(1)*\dir{>};
    (-12,12)*{}; (-12,-4) **\dir{-}?(1)*\dir{>};
    (2,-8)*{n};
\endxy}
\qquad
\delt_4 \;\; =\;\;
-\;\;
 \vcenter{
 \xy 0;/r.13pc/:
    (0,-4)*{}; (0,12) **\dir{-}?(0)*\dir{<};
    (0,4)*{\bullet};
     %(9,8)*{n};
\endxy}
\]
\[
\delt_6 \;\; = \;\;-\; \vcenter{
 \xy 0;/r.13pc/:
    (-4,-4)*{};(-12,12)*{} **\crv{(-4,3) & (-12,5)}?(0)*\dir{<};
    (-12,-4)*{};(-4,12)*{} **\crv{(-12,3) & (-4,5)}?(1)*\dir{>};
    (4,-4)*{}; (4,12) **\dir{-}?(0)*\dir{<};
\endxy}\qquad
 \delt_7 = \delt_8 =\;\; \vcenter{
 \xy 0;/r.13pc/:
    (4,-4)*{};(-4,4)*{} **\crv{(4,-1) & (-4,1)}?(0)*\dir{<};
    (-4,-4)*{};(4,4)*{} **\crv{(-4,-1) & (4,1)}?(0)*\dir{<};
    (-4,4)*{};(-12,12)*{} **\crv{(-4,7) & (-12,9)};
    (-12,4)*{};(-4,12)*{} **\crv{(-12,7) & (-4,9)}?(1)*\dir{>};;
    (4,4)*{}; (4,12) **\dir{-};
    (-12,-4)*{}; (-12,4) **\dir{-};
  (2,-1.5)*{\bullet}; %(18,8)*{n};
\endxy}
\;\; - \;\;
 \vcenter{
 \xy 0;/r.13pc/:
    (4,-4)*{};(-4,4)*{} **\crv{(4,-1) & (-4,1)}?(0)*\dir{<};
    (-4,-4)*{};(4,4)*{} **\crv{(-4,-1) & (4,1)}?(0)*\dir{<};
    (-4,4)*{};(-12,12)*{} **\crv{(-4,7) & (-12,9)};
    (-12,4)*{};(-4,12)*{} **\crv{(-12,7) & (-4,9)}?(1)*\dir{>};;
    (4,4)*{}; (4,12) **\dir{-};
    (-12,-4)*{}; (-12,4) **\dir{-};
  (-12.3,2)*{\bullet}; %(18,8)*{n};
\endxy}
\qquad
\delt_9  =\;\; - \;
 \vcenter{
 \xy 0;/r.13pc/:
    (-4,12)*{};(4,12)*{} **\crv{(-4,4) & (4,4)}?(1)*\dir{>};
    (12,-4)*{}; (12,12) **\dir{-}?(0)*\dir{<};
  %  (18,8)*{n};
\endxy}
\]
In the diagrams above we have omitted the label $n$ on the far right region of each diagram for simplicity.  We will sometimes make use of this convention in the following sections when the labelling of each region is clear from the context.

% ------------------------------------------------------------------------------
%
\subsubsection{Chain maps $\ogam$ and $\odelt$ in indecomposable case}
%
% ------------------------------------------------------------------------------

For $n \leq 0$ the chain maps $\ogam\maps \cal{C}\cal{E}\onen \to \cal{E}\cal{C}\onen$ and $\odelt\maps \cal{E}\cal{C}\onen\to \cal{C}\cal{E}\onen$ defined in \eqref{eq_def_overline_gamma} and \eqref{eq_def_overline_delta} can be written as
\begin{equation}
 \xy
  (-65,35)*+{\E\E\F \onen \{2\}}="1'";
  (-45,15)*+{\E\onen \{1+n\}}="2'";
  (-10,35)*+{\E\E\F \onen }="3'";
  (10,15)*+{\E\E\F \onen}="4'";
  (45,35)*+{\E\E\F \onen \{-2\}}="5'";
  (65,15)*+{\E\onen \{-1-n\}}="6'";
   {\ar^{\sUup\sUupdot\sUdown} "1'";"3'"};
   {\ar^{} "1'" ;"4'"};
   "1'"+(15,-8)*{\sUup\sUup\sUdowndot};
   "3'"+(-15,-5)*{\sUup\sUcupl};
   "3'"+(28,-4)*{\text{$\sUup\sUcapr$}};
   "4'"+(4,8)*{\sUup\sUupdot\sUdown};
   {\ar "2'";"3'"};
   {\ar^(.4){\text{$\sUup\sUcupl$}} "2'";"4'"};
   {\ar^-{-\;\sUup\sUup\sUdowndot} "3'";"5'"};
   {\ar "3'";"6'"};
   {\ar "4'";"5'"};
   {\ar^(.4){-\;\;\text{$\sUup\sUcapr$}} "4'";"6'"};
  (-65,-15)*+{\E\F\E \onen \{2\}}="1";
  (-45,-35)*+{\E\onen \{n+1\}}="2";
  (-10,-15)*+{\E\F\E \onen }="3";
  (10,-35)*+{\E\F\E \onen}="4";
  (45,-15)*+{\E\F\E \onen \{-2\}}="5";
  (65,-35)*+{\E\onen \{n-1\}}="6";
   {\ar^{\sUupdot\sUdown\sUup} "1";"3"};
   {\ar^{} "1" ;"4"};
   "4"+(-12,8)*{\sUup\sUdowndot\sUup};
   "3"+(-16,-5)*{\sUcupl\sUup};
   "6"+(-18,10)*{\text{$\sUcapr\sUup$}};
   "4"+(14,4)*{\sUupdot\sUdown\sUup};
   {\ar "2";"3"};
   {\ar_(.4){\text{$\sUcupl\sUup$}} "2";"4"};
   {\ar^(.63){-\;\sUup\sUdowndot\sUup} "3";"5"};
   {\ar "3";"6"};
   {\ar "4";"5"};
   {\ar_(.4){-\text{$\sUcapr\sUup$}} "4";"6"};
   {\textcolor[rgb]{0.00,0.50,0.25}{\ar@/^0.25pc/^{\ogam_1} "1"+(-6,4);"1'"+(-6,-4)}};
   {\textcolor[rgb]{0.00,0.50,0.25}{\ar@/^0.3pc/^{\ogam_2} "1";"2'"+(-4,-3)}};
   {\textcolor[rgb]{0.00,0.50,0.25}{\ar@/^0.25pc/^{\ogam_3} "3";"3'"}};
   {\textcolor[rgb]{0.00,0.50,0.25}{\ar^(.75){\ogam_4} "3"+(4,3);"4'"}};
   {\textcolor[rgb]{0.00,0.50,0.25}{\ar@/^0.35pc/^(.55){\ogam_5} "4"+(1,4);"4'"+(1,-4)}};
    {\textcolor[rgb]{0.00,0.50,0.25}{\ar^{\ogam_6} "5";"5'"+(-1,-3)}};
   {\textcolor[rgb]{0.00,0.50,0.25}{\ar@/^0.3pc/^(.4){\ogam_8} "5"+(4,3);"6'"+(-1,-3)}};
   %{\textcolor[rgb]{0.00,0.50,0.25}{\ar@/^0.3pc/^(.2){\ogam_7} "6";"5'"+(4,-4)}};
   {\textcolor[rgb]{0.00,0.50,0.25}{\ar^{\ogam_7} "6"+(4,4);"6'"+(4,-4)}};
   {\textcolor[rgb]{0.00,0.00,1.00}{\ar@/^0.25pc/^{\odelt_1} "1'"+(-4,-4);"1"+(-4,4)}};
   {\textcolor[rgb]{0.00,0.00,1.00}{\ar@/^0.3pc/^(.5){\odelt_2} "2'";"1"+(3,3)}};
   %{\textcolor[rgb]{0.00,0.00,1.00}{\ar@/^0.25pc/^(.7){\odelt_2} "1'"+(4,-4);"2"+(-4,4)}};
   {\textcolor[rgb]{0.00,0.00,1.00}{\ar@/^0.3pc/^(.3){\odelt_3} "2'";"2"}};
   {\textcolor[rgb]{0.00,0.00,1.00}{\ar@/^0.35pc/^{\odelt_4} "3'"+(2,-3);"3"+(2,3)}};
   {\textcolor[rgb]{0.00,0.00,1.00}{\ar^(.55){\odelt_5} "3'"+(5,-3);"4"+(-5,3)}};
   {\textcolor[rgb]{0.00,0.00,1.00}{\ar@/^0.3pc/^(.45){\odelt_6} "4'"+(3,-4);"4"+(3,4)}};
   {\textcolor[rgb]{0.00,0.00,1.00}{\ar@/^0.3pc/^(.55){\odelt_8} "6'"+(2,-3);"5"+(6,3)}};
   {\textcolor[rgb]{0.00,0.00,1.00}{\ar@/^0.25pc/^(.55){\odelt_7} "5'"+(1,-3);"5"+(2,3)}};
 \endxy
\end{equation}
\[
 \ogam_1 = \ogam_3 =\;\;
  \vcenter{
 \xy 0;/r.13pc/:
    (4,-4)*{};(-4,4)*{} **\crv{(4,-1) & (-4,1)} ;
    (-4,-4)*{};(4,4)*{} **\crv{(-4,-1) & (4,1)}?(0)*\dir{<};
    (-4,4)*{};(-12,12)*{} **\crv{(-4,7) & (-12,9)}?(1)*\dir{>};
    (-12,4)*{};(-4,12)*{} **\crv{(-12,7) & (-4,9)}?(1)*\dir{>};
    (4,4)*{}; (4,12) **\dir{-};
    (-12,-4)*{}; (-12,4) **\dir{-};
  (-10.5,9.5)*{\bullet}; %(18,8)*{n};
\endxy}
\;\; - \;\; \vcenter{
 \xy 0;/r.13pc/:
    (4,-4)*{};(-4,4)*{} **\crv{(4,-1) & (-4,1)} ;
    (-4,-4)*{};(4,4)*{} **\crv{(-4,-1) & (4,1)}?(0)*\dir{<};
    (-4,4)*{};(-12,12)*{} **\crv{(-4,7) & (-12,9)}?(1)*\dir{>};
    (-12,4)*{};(-4,12)*{} **\crv{(-12,7) & (-4,9)}?(1)*\dir{>};
    (4,4)*{}; (4,12) **\dir{-};
    (-12,-4)*{}; (-12,4) **\dir{-};
    (4.3,6)*{\bullet}; %(18,8)*{n};
\endxy}
\qquad
\ogam_2  =\;\; \vcenter{
 \xy 0;/r.13pc/:
    (4,-4)*{};(-4,-4)*{} **\crv{(4,4) & (-4,4)} ?(1)*\dir{>};
    (-12,-4)*{}; (-12,12) **\dir{-}?(1)*\dir{>};
   %(4,8)*{\bullet}; %(18,8)*{n};
\endxy}
\qquad
\ogam_4  = \;\; \vcenter{
 \xy 0;/r.13pc/:
    (4,-4)*{};(12,12)*{} **\crv{(4,3) & (12,5)}?(0)*\dir{<};
    (12,-4)*{};(4,12)*{} **\crv{(12,3) & (4,5)}?(1)*\dir{>};
    (-4,-4)*{}; (-4,12) **\dir{-}?(1)*\dir{>};
\endxy}
\]

\[
\ogam_5 = \ogam_6 =\;\; \vcenter{
 \xy 0;/r.13pc/:
    (4,-4)*{};(-4,4)*{} **\crv{(4,-1) & (-4,1)} ;
    (-4,-4)*{};(4,4)*{} **\crv{(-4,-1) & (4,1)}?(0)*\dir{<};
    (-4,4)*{};(-12,12)*{} **\crv{(-4,7) & (-12,9)}?(1)*\dir{>};
    (-12,4)*{};(-4,12)*{} **\crv{(-12,7) & (-4,9)}?(1)*\dir{>};
    (4,4)*{}; (4,12) **\dir{-};
    (-12,-4)*{}; (-12,4) **\dir{-};
    (-4,4)*{\bullet}; %(18,8)*{n};
\endxy}
\;\; - \;\; \vcenter{
 \xy 0;/r.13pc/:
    (4,-4)*{};(-4,4)*{} **\crv{(4,-1) & (-4,1)} ;
    (-4,-4)*{};(4,4)*{} **\crv{(-4,-1) & (4,1)}?(0)*\dir{<};
    (-4,4)*{};(-12,12)*{} **\crv{(-4,7) & (-12,9)}?(1)*\dir{>};
    (-12,4)*{};(-4,12)*{} **\crv{(-12,7) & (-4,9)}?(1)*\dir{>};
    (4,4)*{}; (4,12) **\dir{-};
    (-12,-4)*{}; (-12,4) **\dir{-};
    (4.3,6)*{\bullet}; %(18,8)*{n};
\endxy}
\qquad
\ogam_7  =\;\;
-\;\;
 \vcenter{
 \xy 0;/r.13pc/:
    (0,-4)*{}; (0,12) **\dir{-}?(1)*\dir{>};
    (0,4)*{\bullet};
     %(9,8)*{n};
\endxy}
\quad
\ogam_8  =
\;\; -\;\; \vcenter{
 \xy 0;/r.13pc/:
    (-4,-4)*{};(4,-4)*{} **\crv{(-4,4) & (4,4)} ?(1)*\dir{>};
    (12,-4)*{}; (12,12) **\dir{-}?(1)*\dir{>};
   %(4,8)*{\bullet}; %(18,8)*{n};
\endxy}
\]

\[
\odelt_1 = \odelt_4=\;\; \vcenter{
 \xy 0;/r.13pc/:
    (-4,-4)*{};(4,4)*{} **\crv{(-4,-1) & (4,1)};
    (4,-4)*{};(-4,4)*{} **\crv{(4,-1) & (-4,1)};
    (4,4)*{};(12,12)*{} **\crv{(4,7) & (12,9)}?(1)*\dir{>};
    (12,4)*{};(4,12)*{} **\crv{(12,7) & (4,9)};
    (-4,4)*{}; (-4,12) **\dir{-}?(1)*\dir{>};;
    (12,-4)*{}; (12,4) **\dir{-}?(0)*\dir{<};;
   (4.3,4)*{\bullet}; %(18,8)*{n};
\endxy}
\;\; - \;\;
 \vcenter{
 \xy 0;/r.13pc/:
    (-4,-4)*{};(4,4)*{} **\crv{(-4,-1) & (4,1)};
    (4,-4)*{};(-4,4)*{} **\crv{(4,-1) & (-4,1)};
    (4,4)*{};(12,12)*{} **\crv{(4,7) & (12,9)}?(1)*\dir{>};
    (12,4)*{};(4,12)*{} **\crv{(12,7) & (4,9)};
    (-4,4)*{}; (-4,12) **\dir{-}?(1)*\dir{>};;
    (12,-4)*{}; (12,4) **\dir{-}?(0)*\dir{<};;
  (12.3,2)*{\bullet}; %(18,8)*{n};
\endxy}
\qquad
\odelt_2  =
\;\; \vcenter{
 \xy 0;/r.13pc/:
    (-4,12)*{};(4,12)*{} **\crv{(-4,4) & (4,4)}?(0)*\dir{<};
    (12,-4)*{}; (12,12) **\dir{-}?(1)*\dir{>};
   %(12.3,2)*{\bullet}; %(18,8)*{n};
\endxy}
\qquad
\odelt_3 \;\; =\;\;
-\;\;
 \vcenter{
 \xy 0;/r.13pc/:
    (0,-4)*{}; (0,12) **\dir{-}?(1)*\dir{>};
    (0,4)*{\bullet};
     (9,8)*{n};
\endxy}
\]

\[
\odelt_5 \;\; = \;\;-\; \; \vcenter{
 \xy 0;/r.13pc/:
    (4,-4)*{};(12,12)*{} **\crv{(4,3) & (12,5)}?(1)*\dir{>};
    (12,-4)*{};(4,12)*{} **\crv{(12,3) & (4,5)}?(0)*\dir{<};
    (-4,-4)*{}; (-4,12) **\dir{-}?(1)*\dir{>};
\endxy}\qquad
 \odelt_6 = \odelt_7 =\;\; \vcenter{
 \xy 0;/r.13pc/:
    (-4,-4)*{};(4,4)*{} **\crv{(-4,-1) & (4,1)};
    (4,-4)*{};(-4,4)*{} **\crv{(4,-1) & (-4,1)};
    (4,4)*{};(12,12)*{} **\crv{(4,7) & (12,9)}?(1)*\dir{>};
    (12,4)*{};(4,12)*{} **\crv{(12,7) & (4,9)};
    (-4,4)*{}; (-4,12) **\dir{-}?(1)*\dir{>};;
    (12,-4)*{}; (12,4) **\dir{-}?(0)*\dir{<};;
  (-2,-1.5)*{\bullet}; %(18,8)*{n};
\endxy}
\;\; - \;\;
 \vcenter{
 \xy 0;/r.13pc/:
    (-4,-4)*{};(4,4)*{} **\crv{(-4,-1) & (4,1)};
    (4,-4)*{};(-4,4)*{} **\crv{(4,-1) & (-4,1)};
    (4,4)*{};(12,12)*{} **\crv{(4,7) & (12,9)}?(1)*\dir{>};
    (12,4)*{};(4,12)*{} **\crv{(12,7) & (4,9)};
    (-4,4)*{}; (-4,12) **\dir{-}?(1)*\dir{>};;
    (12,-4)*{}; (12,4) **\dir{-}?(0)*\dir{<};;
  (12.3,2)*{\bullet}; %(18,8)*{n};
\endxy}
\quad
\odelt_8  =\;\; - \;
 \vcenter{
 \xy 0;/r.13pc/:
    (4,12)*{};(-4,12)*{} **\crv{(4,4) & (-4,4)}?(0)*\dir{<};
    (-12,-4)*{}; (-12,12) **\dir{-}?(1)*\dir{>};
    %(-12.3,2)*{\bullet}; %(18,8)*{n};
\endxy}
\]

%------------------------------------------------------------------------------
%
\subsubsection{Chain maps $\tsigma(\gam)$ and $\tsigma(\delt)$ in indecomposable case}
%
% ------------------------------------------------------------------------------

When $n \geq 0$ it is useful to work with the indecomposable in $Kom(\UcatD)$ complex $\tsigma(\cal{C}\onenn{-n})$ instead of $\cal{C}\onen$.
For $n \geq 0$ the maps
\begin{align}
  \delts\onen:=\tsigma(\delt\onenn{-n+2}) &\maps \F \tsigma(\cal{C}\onenn{-n})\to\tsigma(\cal{C}\onenn{-n+2})\F\onen
  \\
  \gams\onen:=\tsigma(\gam\onenn{-n+2}) &\maps \tsigma(\cal{C}\onenn{-n+2})\F\onen \to \F \tsigma(\cal{C}\onenn{-n})
\end{align}
simplify to the form
\begin{equation}
 \xy
  (-65,35)*+{\F\F\E \onen \{2\}}="1'";
  (-45,15)*+{\F\onen \{1+n\}}="2'";
  (-10,35)*+{\F\F\E \onen }="3'";
  (10,15)*+{\F\F\E \onen}="4'";
  (65,15)*+{\F\onen \{-n-1\}}="5'";
  (45,35)*+{\F\F\E \onen \{-2\}}="6'";
   {\ar^{\sUdown\sUdown\sUupdot} "1'";"3'"};
   {\ar^{} "1'" ;"4'"};
   "1'"+(15,-8)*{\sUdown\sUdowndot\sUup};
   "3'"+(-15,-5)*{\sUdown\sUcupr};
   "3'"+(28,-4)*{\text{$\sUdown\sUcapl$}};
   "4'"+(4,8)*{\sUdown\sUdown\sUupdot};
   {\ar "2'";"3'"};
   {\ar^(.4){\text{$\sUdown\sUcupr$}} "2'";"4'"};
   {\ar^-{-\;\sUdown\sUdowndot\sUup} "3'";"6'"};
   {\ar "3'";"5'"};
   {\ar "4'";"6'"};
   {\ar^(.4){-\;\;\text{$\sUdown\sUcapl$}} "4'";"5'"};
  (-65,-15)*+{\F\E\F \onen \{2\}}="1";
  (-45,-35)*+{\F\onen \{-1+n\}}="2";
  (-10,-15)*+{\F\E\F \onen }="3";
  (10,-35)*+{\F\E\F \onen}="4";
  (65,-35)*+{\F\onen \{-n+1\}}="5";
  (45,-15)*+{\F\E\F \onen \{-2\}}="6";
   {\ar^{\sUdown\sUupdot\sUdown} "1";"3"};
   {\ar^{} "1" ;"4"};
   "4"+(-12,8)*{\sUdowndot\sUdown\sUup};
   "3"+(-16,-5)*{\sUcupr\sUdown};
   "5"+(-18,10)*{\text{$\sUcapl\sUdown$}};
   "4"+(14,4)*{\sUdown\sUupdot\sUdown};
   {\ar "2";"3"};
   {\ar_(.4){\text{$\sUcupr\sUdown$}} "2";"4"};
   {\ar^(.63){-\;\sUdown\sUup\sUdowndot} "3";"6"};
   {\ar "3";"5"};
   {\ar "4";"6"};
   {\ar_(.4){-\text{$\sUcapl\sUdown$}} "4";"5"};
   {\textcolor[rgb]{0.00,0.00,1.00}{\ar@/^0.25pc/^{(\gams)_1} "1"+(-6,4);"1'"+(-6,-4)}};
   {\textcolor[rgb]{0.00,0.00,1.00}{\ar@/^0.3pc/^{(\gams)_2} "1";"2'"+(-4,-3)}};
   {\textcolor[rgb]{0.00,0.00,1.00}{\ar@/^0.3pc/^{(\gams)_3} "3"+(-4,4);"3'"+(-4,-4)}};
   {\textcolor[rgb]{0.00,0.00,1.00}{\ar^(.35){(\gams)_4} "3";"4'"}};
   {\textcolor[rgb]{0.00,0.00,1.00}{\ar@/^0.35pc/^(.6){(\gams)_5} "4"+(2,4);"4'"+(2,-4)}};
    {\textcolor[rgb]{0.00,0.00,1.00}{\ar@/_0.3pc/^{(\gams)_6} "5";"5'"+(2,-4)}};
   {\textcolor[rgb]{0.00,0.00,1.00}{\ar@/^0.3pc/^{(\gams)_7} "6";"5'"+(-4,-4)}};
   {\textcolor[rgb]{0.00,0.00,1.00}{\ar@/^0.25pc/^{(\gams)_8} "6"+(-4,4);"6'"+(-4,-4)}};
   {\textcolor[rgb]{0.00,0.50,0.25}{\ar@/^0.25pc/^{(\delts)_1} "1'"+(-4,-4);"1"+(-4,4)}};
   {\textcolor[rgb]{0.00,0.50,0.25}{\ar@/^0.3pc/^{(\delts)_2} "2'";"1"+(3,3)}};
   {\textcolor[rgb]{0.00,0.50,0.25}{\ar@/^0.3pc/^(.3){(\delts)_3} "2'";"2"}};
   {\textcolor[rgb]{0.00,0.50,0.25}{\ar@/^0.35pc/^{(\delts)_4} "3'"+(-2,-4);"3"+(-2,4)}};
   {\textcolor[rgb]{0.00,0.50,0.25}{\ar^(.7){(\delts)_5} "3'";"4"+(-2,4)}};
   {\textcolor[rgb]{0.00,0.50,0.25}{\ar@/^0.3pc/^(.4){(\delts)_6} "4'"+(4,-4);"4"+(4,4)}};
   {\textcolor[rgb]{0.00,0.50,0.25}{\ar@/^0.3pc/^{(\delts)_7} "5'"+(-2,-4);"6"+(4,4)}};
   {\textcolor[rgb]{0.00,0.50,0.25}{\ar@/^0.25pc/^{(\delts)_8} "6'"+(-2,-4);"6"+(-2,4)}};
 \endxy
\end{equation}
where
\[
 (\delts)_1 = (\delts)_4 \;\;= \;\;\vcenter{
 \xy 0;/r.13pc/:
    (-4,-4)*{};(4,4)*{} **\crv{(-4,-1) & (4,1)}?(0)*\dir{<};
    (4,-4)*{};(-4,4)*{} **\crv{(4,-1) & (-4,1)}?(0)*\dir{<};
    (4,4)*{};(12,12)*{} **\crv{(4,7) & (12,9)};
    (12,4)*{};(4,12)*{} **\crv{(12,7) & (4,9)}?(1)*\dir{>};;
    (-4,4)*{}; (-4,12) **\dir{-};
    (12,-4)*{}; (12,4) **\dir{-};
   (4.3,4)*{\bullet}; %(18,8)*{n};
\endxy}
\;\; - \;\;
 \vcenter{
 \xy 0;/r.13pc/:
    (-4,-4)*{};(4,4)*{} **\crv{(-4,-1) & (4,1)}?(0)*\dir{<};
    (4,-4)*{};(-4,4)*{} **\crv{(4,-1) & (-4,1)}?(0)*\dir{<};
    (4,4)*{};(12,12)*{} **\crv{(4,7) & (12,9)};
    (12,4)*{};(4,12)*{} **\crv{(12,7) & (4,9)}?(1)*\dir{>};;
    (-4,4)*{}; (-4,12) **\dir{-};
    (12,-4)*{}; (12,4) **\dir{-};
  (12.3,2)*{\bullet}; %(18,8)*{n};
\endxy} \qquad
(\delts)_2 \;\; = \;\; \vcenter{
 \xy 0;/r.13pc/:
    (-4,12)*{};(4,12)*{} **\crv{(-4,4) & (4,4)}?(1)*\dir{>};
    (12,-4)*{}; (12,12) **\dir{-}?(0)*\dir{<};
   %(12.3,2)*{\bullet}; %(18,8)*{n};
\endxy}
\qquad (\delts)_3 \;\; = \;\; -\;
 \vcenter{\xy 0;/r.13pc/:
    (12,-4)*{}; (12,12) **\dir{-}?(0)*\dir{<};
  (12.3,2)*{\bullet}; %(18,8)*{n};
\endxy}
\]

\[
(\delts)_5 \;\; = \;\; \vcenter{
 \xy 0;/r.13pc/:
    (4,-4)*{};(12,12)*{} **\crv{(4,3) & (12,5)}?(0)*\dir{<};
    (12,-4)*{};(4,12)*{} **\crv{(12,3) & (4,5)}?(1)*\dir{>};
    (-4,-4)*{}; (-4,12) **\dir{-}?(0)*\dir{<};
\endxy}
\qquad
 (\delts)_6 = (\delts)_8\;\; = \;\; \vcenter{
 \xy 0;/r.13pc/:
    (-4,-4)*{};(4,4)*{} **\crv{(-4,-1) & (4,1)}?(0)*\dir{<};
    (4,-4)*{};(-4,4)*{} **\crv{(4,-1) & (-4,1)}?(0)*\dir{<};
    (4,4)*{};(12,12)*{} **\crv{(4,7) & (12,9)};
    (12,4)*{};(4,12)*{} **\crv{(12,7) & (4,9)}?(1)*\dir{>};;
    (-4,4)*{}; (-4,12) **\dir{-};
    (12,-4)*{}; (12,4) **\dir{-};
  (-2,-1.5)*{\bullet}; %(18,8)*{n};
\endxy}
\;\; - \;\;
 \vcenter{
 \xy 0;/r.13pc/:
    (-4,-4)*{};(4,4)*{} **\crv{(-4,-1) & (4,1)}?(0)*\dir{<};
    (4,-4)*{};(-4,4)*{} **\crv{(4,-1) & (-4,1)}?(0)*\dir{<};
    (4,4)*{};(12,12)*{} **\crv{(4,7) & (12,9)};
    (12,4)*{};(4,12)*{} **\crv{(12,7) & (4,9)}?(1)*\dir{>};;
    (-4,4)*{}; (-4,12) **\dir{-};
    (12,-4)*{}; (12,4) **\dir{-};
  (12.3,2)*{\bullet}; %(18,8)*{n};
\endxy}
 \qquad
(\delts)_7 \;\; =\;\; - \;
 \vcenter{
 \xy 0;/r.13pc/:
    (4,12)*{};(-4,12)*{} **\crv{(4,4) & (-4,4)}?(1)*\dir{>};
    (-12,-4)*{}; (-12,12) **\dir{-}?(0)*\dir{<};
    %(-12.3,2)*{\bullet}; %(18,8)*{n};
\endxy}
\]

\[
 (\gams)_1 = (\gams)_3\;\; =\;\; \vcenter{
 \xy 0;/r.13pc/:
    (4,-4)*{};(-4,4)*{} **\crv{(4,-1) & (-4,1)} ?(0)*\dir{<};
    (-4,-4)*{};(4,4)*{} **\crv{(-4,-1) & (4,1)};
    (-4,4)*{};(-12,12)*{} **\crv{(-4,7) & (-12,9)};
    (-12,4)*{};(-4,12)*{} **\crv{(-12,7) & (-4,9)};
    (4,4)*{}; (4,12) **\dir{-}?(1)*\dir{>};
    (-12,-4)*{}; (-12,4) **\dir{-}?(0)*\dir{<};
  (-10,9.5)*{\bullet}; %(18,8)*{n};
\endxy}
\;\; - \;\; \vcenter{
 \xy 0;/r.13pc/:
    (4,-4)*{};(-4,4)*{} **\crv{(4,-1) & (-4,1)} ?(0)*\dir{<};
    (-4,-4)*{};(4,4)*{} **\crv{(-4,-1) & (4,1)};
    (-4,4)*{};(-12,12)*{} **\crv{(-4,7) & (-12,9)};
    (-12,4)*{};(-4,12)*{} **\crv{(-12,7) & (-4,9)};
    (4,4)*{}; (4,12) **\dir{-}?(1)*\dir{>};
    (-12,-4)*{}; (-12,4) **\dir{-}?(0)*\dir{<};
    (4.3,6)*{\bullet}; %(18,8)*{n};
\endxy}
    \qquad
(\gams)_2 \;\; =\;\; \vcenter{
 \xy 0;/r.13pc/:
    (4,-4)*{};(-4,-4)*{} **\crv{(4,4) & (-4,4)} ?(0)*\dir{<};
    (-12,-4)*{}; (-12,12) **\dir{-}?(0)*\dir{<};
   %(4,8)*{\bullet}; %(18,8)*{n};
\endxy}
    \qquad
(\gams)_4 \;\; = \;\; \vcenter{
 \xy 0;/r.13pc/:
    (4,-4)*{};(12,12)*{} **\crv{(4,3) & (12,5)}?(1)*\dir{>};
    (12,-4)*{};(4,12)*{} **\crv{(12,3) & (4,5)}?(0)*\dir{<};
    (-4,-4)*{}; (-4,12) **\dir{-}?(0)*\dir{<};
\endxy}
\]

\[
(\gams)_5 = (\gams)_8 \;\; =\;\; \vcenter{
 \xy 0;/r.13pc/:
    (4,-4)*{};(-4,4)*{} **\crv{(4,-1) & (-4,1)} ?(0)*\dir{<};
    (-4,-4)*{};(4,4)*{} **\crv{(-4,-1) & (4,1)};
    (-4,4)*{};(-12,12)*{} **\crv{(-4,7) & (-12,9)};
    (-12,4)*{};(-4,12)*{} **\crv{(-12,7) & (-4,9)};
    (4,4)*{}; (4,12) **\dir{-}?(1)*\dir{>};
    (-12,-4)*{}; (-12,4) **\dir{-}?(0)*\dir{<};
    (-4,4)*{\bullet}; %(18,8)*{n};
\endxy}
\;\; - \;\; \vcenter{
 \xy 0;/r.13pc/:
    (4,-4)*{};(-4,4)*{} **\crv{(4,-1) & (-4,1)} ?(0)*\dir{<};
    (-4,-4)*{};(4,4)*{} **\crv{(-4,-1) & (4,1)};
    (-4,4)*{};(-12,12)*{} **\crv{(-4,7) & (-12,9)};
    (-12,4)*{};(-4,12)*{} **\crv{(-12,7) & (-4,9)};
    (4,4)*{}; (4,12) **\dir{-}?(1)*\dir{>};
    (-12,-4)*{}; (-12,4) **\dir{-}?(0)*\dir{<};
    (4.3,6)*{\bullet}; %(18,8)*{n};
\endxy}
    \qquad
(\gams)_6 \;\; = \;\; -\;
 \vcenter{\xy 0;/r.13pc/:
    (12,-4)*{}; (12,12) **\dir{-}?(0)*\dir{<};
  (12.3,2)*{\bullet}; %(18,8)*{n};
\endxy}
    \qquad
(\gams)_7 \;\; =\;\; -\; \vcenter{
 \xy 0;/r.13pc/:
    (-4,-4)*{};(4,-4)*{} **\crv{(-4,4) & (4,4)} ?(0)*\dir{<};
    (12,-4)*{}; (12,12) **\dir{-}?(0)*\dir{<};
   %(4,8)*{\bullet}; %(18,8)*{n};
\endxy}
\]

% ------------------------------------------------------------------------------
%
\subsubsection{Chain maps $\tsigma(\ogam)$ and $\tsigma(\odelt)$ in indecomposable case}
%
% ------------------------------------------------------------------------------

For $n \geq 0$ the chain maps
\begin{align}
 \odelts\onen := \tsigma(\odelt\onenn{-n-2}) & \maps \tsigma(\cal{C}\onenn{-n-2})\E\onen \to \E
\tsigma(\cal{C}\onenn{-n}) \\
\ogams\onen:= \tsigma(\ogam\onenn{-n-2}) &\maps \E \tsigma(\cal{C}\onenn{-n}) \to
\tsigma(\cal{C}\onenn{-n-2})\E\onen
\end{align}
simplify to the form
\begin{equation}
 \xy
  (-65,35)*+{\F\E\E \onen \{2\}}="1'";
  (-45,15)*+{\E\onen \{3+n\}}="2'";
  (-10,35)*+{\F\E\E \onen }="3'";
  (10,15)*+{\F\E\E \onen}="4'";
  (65,15)*+{\E\onen \{-n-3\}}="5'";
  (45,35)*+{\F\E\E \onen \{-2\}}="6'";
   {\ar^{\sUdown\sUupdot\sUup} "1'";"3'"};
   {\ar^{} "1'" ;"4'"};
   "1'"+(15,-8)*{\sUdowndot\sUup\sUup};
   "3'"+(-15,-5)*{\sUcupr\sUup};
   "3'"+(28,-4)*{\text{$\sUcapl\sUup$}};
   "4'"+(4,8)*{\sUdown\sUupdot\sUup};
   {\ar "2'";"3'"};
   {\ar^(.4){\text{$\sUcupr\sUup$}} "2'";"4'"};
   {\ar^-{-\;\sUdowndot\sUup\sUup} "3'";"6'"};
   {\ar "3'";"5'"};
   {\ar "4'";"6'"};
   {\ar^(.4){-\;\;\text{$\sUcapl\sUup$}} "4'";"5'"};
  (-65,-15)*+{\E\F\E \onen \{2\}}="1";
  (-45,-35)*+{\E\onen \{1+n\}}="2";
  (-10,-15)*+{\E\F\E \onen }="3";
  (10,-35)*+{\E\F\E \onen}="4";
  (65,-35)*+{\E\onen \{-n-1\}}="5";
  (45,-15)*+{\E\F\E \onen \{-2\}}="6";
   {\ar^{\sUup\sUdown\sUupdot} "1";"3"};
   {\ar^{} "1" ;"4"};
   "4"+(-12,8)*{\sUdown\sUupdot\sUup};
   "3"+(-16,-5)*{\sUup\sUcupr};
   "5"+(-18,10)*{\text{$\sUup\sUcapl$}};
   "4"+(14,4)*{\sUup\sUdown\sUupdot};
   {\ar "2";"3"};
   {\ar_(.4){\text{$\sUup\sUcupr$}} "2";"4"};
   {\ar^(.63){-\;\sUup\sUdowndot\sUup} "3";"6"};
   {\ar "3";"5"};
   {\ar "4";"6"};
   {\ar_(.4){-\text{$\sUup\sUcapr$}} "4";"5"};
   {\textcolor[rgb]{0.00,0.00,1.00}{\ar@/^0.25pc/^{(\ogams)_1} "1"+(-6,4);"1'"+(-6,-4)}};
   {\textcolor[rgb]{0.00,0.00,1.00}{\ar@/^0.3pc/^{(\ogams)_2} "1";"2'"+(-4,-3)}};
   {\textcolor[rgb]{0.00,0.00,1.00}{\ar@/^0.3pc/^{(\ogams)_3} "3"+(-4,4);"3'"+(-4,-4)}};
   {\textcolor[rgb]{0.00,0.00,1.00}{\ar^(.35){(\ogams)_4} "3";"4'"}};
   {\textcolor[rgb]{0.00,0.00,1.00}{\ar@/^0.35pc/^(.6){(\ogams)_5} "4"+(2,4);"4'"+(2,-4)}};
    {\textcolor[rgb]{0.00,0.00,1.00}{\ar@/_0.3pc/^{(\ogams)_6} "5";"5'"+(2,-4)}};
   {\textcolor[rgb]{0.00,0.00,1.00}{\ar@/^0.3pc/^{(\ogams)_7} "6";"5'"+(-4,-4)}};
   {\textcolor[rgb]{0.00,0.00,1.00}{\ar@/^0.25pc/^{(\ogams)_8} "6"+(-4,4);"6'"+(-4,-4)}};
   {\textcolor[rgb]{0.00,0.50,0.25}{\ar@/^0.25pc/^{(\odelts)_1} "1'"+(-4,-4);"1"+(-4,4)}};
   {\textcolor[rgb]{0.00,0.50,0.25}{\ar@/^0.3pc/^{(\odelts)_2} "2'";"1"+(3,3)}};
   {\textcolor[rgb]{0.00,0.50,0.25}{\ar@/^0.3pc/^(.3){(\odelts)_3} "2'";"2"}};
   {\textcolor[rgb]{0.00,0.50,0.25}{\ar@/^0.35pc/^{(\odelts)_4} "3'"+(-2,-4);"3"+(-2,4)}};
   {\textcolor[rgb]{0.00,0.50,0.25}{\ar^(.7){(\odelts)_5} "3'";"4"+(-2,4)}};
   {\textcolor[rgb]{0.00,0.50,0.25}{\ar@/^0.3pc/^(.4){(\odelts)_6} "4'"+(4,-4);"4"+(4,4)}};
   {\textcolor[rgb]{0.00,0.50,0.25}{\ar@/^0.3pc/^{(\odelts)_7}  "5'"+(-2,-4);"6"+(4,4)}};
   {\textcolor[rgb]{0.00,0.50,0.25}{\ar@/^0.25pc/^{(\odelts)_8} "6'"+(-2,-4);"6"+(-2,4)}};
 \endxy
\end{equation}
where
\[
 (\odelts)_1 = (\odelts)_4 \;\;= \;\;\vcenter{
 \xy 0;/r.13pc/:
    (4,-4)*{};(-4,4)*{} **\crv{(4,-1) & (-4,1)};
    (-4,-4)*{};(4,4)*{} **\crv{(-4,-1) & (4,1)};
    (-4,4)*{};(-12,12)*{} **\crv{(-4,7) & (-12,9)}?(1)*\dir{>};
    (-12,4)*{};(-4,12)*{} **\crv{(-12,7) & (-4,9)};
    (4,4)*{}; (4,12) **\dir{-}?(1)*\dir{>};;
    (-12,-4)*{}; (-12,4) **\dir{-}?(0)*\dir{<};;
   (-4.3,4)*{\bullet}; %(18,8)*{n};
\endxy}
\;\; - \;\;
 \vcenter{
 \xy 0;/r.13pc/:
    (4,-4)*{};(-4,4)*{} **\crv{(4,-1) & (-4,1)};
    (-4,-4)*{};(4,4)*{} **\crv{(-4,-1) & (4,1)};
    (-4,4)*{};(-12,12)*{} **\crv{(-4,7) & (-12,9)}?(1)*\dir{>};
    (-12,4)*{};(-4,12)*{} **\crv{(-12,7) & (-4,9)};
    (4,4)*{}; (4,12) **\dir{-}?(1)*\dir{>};;
    (-12,-4)*{}; (-12,4) **\dir{-}?(0)*\dir{<};;
  (-12.3,2)*{\bullet}; %(18,8)*{n};
\endxy} \qquad
(\odelts)_2 \;\; = \;\; \vcenter{
 \xy 0;/r.13pc/:
    (4,12)*{};(-4,12)*{} **\crv{(4,4) & (-4,4)}?(0)*\dir{<};
    (-12,-4)*{}; (-12,12) **\dir{-}?(1)*\dir{>};
   %(12.3,2)*{\bullet}; %(18,8)*{n};
\endxy}
\qquad (\odelts)_3 \;\; = \;\; -\;
 \vcenter{\xy 0;/r.13pc/:
    (12,-4)*{}; (12,12) **\dir{-}?(1)*\dir{>};
  (12.3,2)*{\bullet}; %(18,8)*{n};
\endxy}
\]

\[
(\odelts)_5 \;\; = \;\; \vcenter{
 \xy 0;/r.13pc/:
    (-4,-4)*{};(-12,12)*{} **\crv{(-4,3) & (-12,5)}?(1)*\dir{>};
    (-12,-4)*{};(-4,12)*{} **\crv{(-12,3) & (-4,5)}?(0)*\dir{<};
    (4,-4)*{}; (4,12) **\dir{-}?(1)*\dir{>};
\endxy}
\qquad
 (\odelts)_6 = (\odelts)_8\;\; = \;\; \vcenter{
 \xy 0;/r.13pc/:
    (4,-4)*{};(-4,4)*{} **\crv{(4,-1) & (-4,1)};
    (-4,-4)*{};(4,4)*{} **\crv{(-4,-1) & (4,1)};
    (-4,4)*{};(-12,12)*{} **\crv{(-4,7) & (-12,9)}?(1)*\dir{>};
    (-12,4)*{};(-4,12)*{} **\crv{(-12,7) & (-4,9)};
    (4,4)*{}; (4,12) **\dir{-}?(1)*\dir{>};;
    (-12,-4)*{}; (-12,4) **\dir{-}?(0)*\dir{<};;
  (2,-1.5)*{\bullet}; %(18,8)*{n};
\endxy}
\;\; - \;\;
 \vcenter{
 \xy 0;/r.13pc/:
    (4,-4)*{};(-4,4)*{} **\crv{(4,-1) & (-4,1)};
    (-4,-4)*{};(4,4)*{} **\crv{(-4,-1) & (4,1)};
    (-4,4)*{};(-12,12)*{} **\crv{(-4,7) & (-12,9)}?(1)*\dir{>};
    (-12,4)*{};(-4,12)*{} **\crv{(-12,7) & (-4,9)};
    (4,4)*{}; (4,12) **\dir{-}?(1)*\dir{>};;
    (-12,-4)*{}; (-12,4) **\dir{-}?(0)*\dir{<};;
  (-12.3,2)*{\bullet}; %(18,8)*{n};
\endxy}
 \qquad
(\odelts)_7 \;\; =\;\; - \;
 \vcenter{
 \xy 0;/r.13pc/:
    (-4,12)*{};(4,12)*{} **\crv{(-4,4) & (4,4)}?(0)*\dir{<};
    (12,-4)*{}; (12,12) **\dir{-}?(1)*\dir{>};
    %(-12.3,2)*{\bullet}; %(18,8)*{n};
\endxy}
\]

\[
 (\ogams)_1 = (\ogams)_3\;\; =\;\; \vcenter{
 \xy 0;/r.13pc/:
    (-4,-4)*{};(4,4)*{} **\crv{(-4,-1) & (4,1)} ;
    (4,-4)*{};(-4,4)*{} **\crv{(4,-1) & (-4,1)}?(0)*\dir{<};
    (4,4)*{};(12,12)*{} **\crv{(4,7) & (12,9)}?(1)*\dir{>};
    (12,4)*{};(4,12)*{} **\crv{(12,7) & (4,9)}?(1)*\dir{>};
    (-4,4)*{}; (-4,12) **\dir{-};
    (12,-4)*{}; (12,4) **\dir{-};
  (10.5,9.5)*{\bullet}; %(18,8)*{n};
\endxy}
\;\; - \;\; \vcenter{
 \xy 0;/r.13pc/:
    (-4,-4)*{};(4,4)*{} **\crv{(-4,-1) & (4,1)} ;
    (4,-4)*{};(-4,4)*{} **\crv{(4,-1) & (-4,1)}?(0)*\dir{<};
    (4,4)*{};(12,12)*{} **\crv{(4,7) & (12,9)}?(1)*\dir{>};
    (12,4)*{};(4,12)*{} **\crv{(12,7) & (4,9)}?(1)*\dir{>};
    (-4,4)*{}; (-4,12) **\dir{-};
    (12,-4)*{}; (12,4) **\dir{-};
    (-4.3,6)*{\bullet}; %(18,8)*{n};
\endxy}
    \qquad
(\ogams)_2 \;\; =\;\; \vcenter{
 \xy 0;/r.13pc/:
    (-4,-4)*{};(4,-4)*{} **\crv{(-4,4) & (4,4)} ?(1)*\dir{>};
    (12,-4)*{}; (12,12) **\dir{-}?(1)*\dir{>};
   %(4,8)*{\bullet}; %(18,8)*{n};
\endxy}
    \qquad
(\ogams)_4 \;\; = \;\; \vcenter{
 \xy 0;/r.13pc/:
    (-4,-4)*{};(-12,12)*{} **\crv{(-4,3) & (-12,5)}?(0)*\dir{<};
    (-12,-4)*{};(-4,12)*{} **\crv{(-12,3) & (-4,5)}?(1)*\dir{>};
    (4,-4)*{}; (4,12) **\dir{-}?(1)*\dir{>};
\endxy}
\]

\[
(\ogams)_5 = (\ogams)_8 \;\; =\;\; \vcenter{
 \xy 0;/r.13pc/:
    (-4,-4)*{};(4,4)*{} **\crv{(-4,-1) & (4,1)} ;
    (4,-4)*{};(-4,4)*{} **\crv{(4,-1) & (-4,1)}?(0)*\dir{<};
    (4,4)*{};(12,12)*{} **\crv{(4,7) & (12,9)}?(1)*\dir{>};
    (12,4)*{};(4,12)*{} **\crv{(12,7) & (4,9)}?(1)*\dir{>};
    (-4,4)*{}; (-4,12) **\dir{-};
    (12,-4)*{}; (12,4) **\dir{-};
    (4,4)*{\bullet}; %(18,8)*{n};
\endxy}
\;\; - \;\; \vcenter{
 \xy 0;/r.13pc/:
    (-4,-4)*{};(4,4)*{} **\crv{(-4,-1) & (4,1)} ;
    (4,-4)*{};(-4,4)*{} **\crv{(4,-1) & (-4,1)}?(0)*\dir{<};
    (4,4)*{};(12,12)*{} **\crv{(4,7) & (12,9)}?(1)*\dir{>};
    (12,4)*{};(4,12)*{} **\crv{(12,7) & (4,9)}?(1)*\dir{>};
    (-4,4)*{}; (-4,12) **\dir{-};
    (12,-4)*{}; (12,4) **\dir{-};
    (-4.3,6)*{\bullet}; %(18,8)*{n};
\endxy}
    \qquad
(\ogams)_6 \;\; = \;\; -\;
 \vcenter{\xy 0;/r.13pc/:
    (12,-4)*{}; (12,12) **\dir{-}?(1)*\dir{>};
  (12.3,2)*{\bullet}; %(18,8)*{n};
\endxy}
    \qquad
(\ogams)_7 \;\; =\;\; -\; \vcenter{
 \xy 0;/r.13pc/:
    (4,-4)*{};(-4,-4)*{} **\crv{(4,4) & (-4,4)} ?(1)*\dir{>};
    (-12,-4)*{}; (-12,12) **\dir{-}?(1)*\dir{>};
   %(4,8)*{\bullet}; %(18,8)*{n};
\endxy}
\]

% ------------------------------------------------------------------------------
%
\subsection{Symmetries of the commutativity chain maps} \label{sec_symm}
%
% ------------------------------------------------------------------------------

Given a 2-morphism $\alpha$ and a 1-morphism $x$ in $\Ucat$ we write $\alpha x$ in place of the composite $\alpha 1_x$ whenever this composite makes sense, likewise for $x\alpha$.

The following Propositions gives several alternative characterizations of the various commutativity chain maps.

\begin{prop}  \label{prop_psi}
For all $n \in \Z$ the equalities in $Kom(\Ucat)$
 \begin{align}
 \delt\onen &=
  \cal{F}\hat{\varrho}^{\tpsi} \circ \tpsi(\gam\onen) \circ \varrho^{\tpsi}\cal{F} \label{eq_delta_psi_gamma}\\
  \gam\onen &= \hat{\varrho}^{\tpsi}\cal{F}\circ\tpsi(\delt\onen)\circ \cal{F} \varrho^{\tpsi} \label{eq_gamma_psi} \\
  \ogam\onen &=
\cal{E}\hat{\varrho}^{\tpsi}\circ\tpsi(\odelt\onen)\circ \varrho^{\tpsi}\cal{E} \label{eq_bargamma_psi}\\
  \odelt\onen &=
  \hat{\varrho}^{\tpsi}\cal{E} \circ \tpsi(\ogam\onen) \circ \cal{E}\varrho^{\tpsi} \label{eq_bardelta_psi}
\end{align}
hold.
\end{prop}

\begin{proof}
Note that \eqref{eq_delta_psi_gamma} is how the map $\delt$ was defined.   Equation~\eqref{eq_gamma_psi} follows from the strictly commutative square
\begin{equation}
    \xy
   (-15,0)*+{\cal{F}\cal{C}\onen}="bl";
   (15,0)*+{\cal{C}\cal{F}\onen}="br";
   (-45,15)*+{\cal{F}\cal{C}\onen}="tl'";
   (45,15)*+{\cal{C}\cal{F}\onen}="tr'";
   (-45,-15)*+{\cal{F}\tpsi(\cal{C})\onen}="bl'";
   (45,-15)*+{\tpsi(\cal{C})\cal{F}\onen}="br'";
   {\ar_{\gam} "bl";"br"};
      {\ar^{\gam} "tl'";"tr'"};
   {\ar_{\cal{F}\varrho^{\tpsi}} "tl'";"bl'"};
   {\ar_{\tpsi(\delt)} "bl'";"br'"};
   {\ar_{\hat{\varrho}^{\tpsi}\cal{F}} "br'";"tr'"};
   {\ar^{-\Id} "tl'";"bl"};
   {\ar_{\cal{F}\tpsi(\hat{\varrho}^{\tpsi})} "bl'";"bl"};
   {\ar^{-\Id} "br";"tr'"};
   {\ar^{\tpsi(\varrho^{\tpsi})\cal{F}} "br";"br'"};
  \endxy
\end{equation}
where the triangles commute since $\tpsi(\varrho^{\tpsi}) = -\hat{\varrho}^{\tpsi}$, $\tpsi(\hat{\varrho}^{\tpsi})=-\varrho^{\tpsi}$, and $\varrho^{\tpsi}$ has inverse $\hat{\varrho}^{\tpsi}$ by Proposition~\ref{prop_homotopy_sym}. The bottom square is $\tpsi$ applied to \eqref{eq_delta_psi_gamma}.

To prove \eqref{eq_bargamma_psi} consider the diagram:
\begin{equation}
    \xy
   (-25,10)*+{\tsigma\tomega(\cal{C})\cal{E}\onen}="tl";
   (25,10)*+{\cal{E}\tsigma\tomega(\cal{C})\onen}="tr";
   (-25,-10)*+{\tpsi\tsigma\tomega(\cal{C})\cal{E}\onen}="bl";
   (25,-10)*+{\cal{E}\tpsi\tsigma\tomega(\cal{C})\onen}="br";
   (-45,25)*+{\cal{C}\cal{E}\onen}="tl'";
   (45,25)*+{\cal{E}\cal{C}\onen}="tr'";
   (-45,-25)*+{\tpsi(\cal{C})\cal{E}\onen}="bl'";
   (45,-25)*+{\cal{E}\tpsi(\cal{C})\onen}="br'";
   {\ar^{\tsigma\tomega(\gam)} "tl";"tr"};
   {\ar_{\tsigma\tomega(\cal{F}\varrho^{\tpsi})} "tl";"bl"};
   {\ar_{\tsigma\tomega\tpsi(\delt)} "bl";"br"};
   {\ar_{\tsigma\tomega(\hat{\varrho}^{\tpsi}\cal{F})} "br";"tr"};
      {\ar^{\ogam} "tl'";"tr'"};
   {\ar_{\varrho^{\tpsi}\cal{E}} "tl'";"bl'"};
   {\ar_{\tpsi(\odelt)} "bl'";"br'"};
   {\ar_{\cal{E}\hat{\varrho}^{\tpsi}} "br'";"tr'"};
   {\ar^{\varrho^{\tsigma\tomega}\cal{E}} "tl'";"tl"};
   {\ar_{-\tpsi(\hat{\varrho}^{\tsigma\tomega})\cal{E}} "bl'";"bl"};
   {\ar^{\cal{E}\hat{\varrho}^{\tsigma\tomega}} "tr";"tr'"};
   {\ar_{-\cal{E}\tpsi(\varrho^{\tsigma\tomega})} "br";"br'"};
  \endxy
\end{equation}
The bottom square commutes on the nose since it is $\tpsi$ applied to the definition of $\odelt$ in \eqref{eq_def_overline_delta} with two minus signs distributed through the map. The top square is the definition of $\ogam$ in \eqref{eq_def_overline_gamma}.  The center square is $\tsigma\tomega$ applied to \eqref{eq_delta_psi_gamma}.  Noting that
\begin{align}
  \tsigma\tomega(\cal{F}\varrho^{\tpsi})& =\tsigma\tomega(\varrho^{\tpsi})\cal{E}
  \\
  \tsigma\tomega(\hat{\varrho}^{\tpsi}\cal{F}) &
  =\tsigma\tomega(\hat{\varrho}^{\tpsi})\cal{E}
\end{align}
the left and right squares commute since
\begin{equation}
  \tpsi(\hat{\varrho}^{\tsigma\tomega})\circ\varrho^{\tpsi}
  = -\tsigma\tomega(\varrho^{\tpsi})\circ \varrho^{\tsigma\tomega}, \qquad
    \tpsi(\varrho^{\tsigma\tomega})\circ\hat{\varrho}^{\tpsi}
  = -\tsigma\tomega(\hat{\varrho}^{\tpsi})\circ \hat{\varrho}^{\tsigma\tomega}
\end{equation}
by Proposition~\ref{prop_homotopy_sym} part (\ref{rem_varrho_nat}).  Equation \eqref{eq_bardelta_psi} follows from \eqref{eq_bargamma_psi} by applying $\tpsi$ and arguing as in the proof that \eqref{eq_delta_psi_gamma} implies \eqref{eq_gamma_psi} above.
\end{proof}

Just as the maps $\ogam$ and $\odelt$ are defined from $\gam$ and $\delt$ using the symmetry $\tsigma\tomega$, the following proposition characterizes $\gam$ and $\delt$ in terms of $\ogam$ and $\odelt$ via the symmetry $\tsigma\tomega$.
\begin{prop}  \label{prop_sigmaomega}
For all $n \in \Z$ the equalities in $Kom(\Ucat)$
 \begin{align}
\gam\onen & =
    \hat{\varrho}^{\tsigma \tomega}\cal{F} \circ
   \tomega\tsigma(\ogam\onen)\circ \cal{F}\varrho^{\tsigma\tomega} \\
\delt\onen &=
   \cal{F} \hat{\varrho}^{\tsigma \tomega} \circ
   \tomega\tsigma(\odelt\onen)\circ \varrho^{\tsigma\tomega}\cal{F}
\end{align}
hold.
\end{prop}

\begin{proof}
The first equation follows from the commutative square
\begin{equation} \label{eq_nat_dotE}
    \xy
   (-25,00)*+{\cal{F}\cal{C}\onen}="bl";
   (25,00)*+{\cal{C}\cal{F}\onen}="br";
   (-45,15)*+{\cal{F}\cal{C}\onen}="tl'";
   (45,15)*+{\cal{C}\cal{F}\onen}="tr'";
   (-45,-15)*+{\cal{F}\tsigma\tomega(\cal{C})\onen}="bl'";
   (45,-15)*+{\tsigma\tomega(\cal{C})\cal{F}\onen}="br'";
   {\ar_{\gam} "bl";"br"};
      {\ar^{\gam} "tl'";"tr'"};
   {\ar_{\cal{F}\varrho^{\tsigma\tomega}} "tl'";"bl'"};
   {\ar_{\tsigma\tomega(\ogam)} "bl'";"br'"};
   {\ar_{\cal{F}\hat{\varrho}^{\tsigma\tomega}} "br'";"tr'"};
   {\ar^{\Id} "tl'";"bl"};
   {\ar_{\tsigma\tomega(\varrho^{\tsigma\tomega}\cal{E})
   =\cal{F}\tsigma\tomega(\varrho^{\tsigma\tomega})} "bl'";"bl"};
   {\ar^{\Id} "br";"tr'"};
   {\ar_{\tsigma\tomega(\cal{E}\hat{\varrho}^{\tsigma\tomega})
   =\tsigma\tomega(\hat{\varrho}^{\tsigma\tomega}\cal{F})} "br";"br'"};
  \endxy
\end{equation}
where the bottom square commutes on the nose since it is $\tsigma\tomega$ applied to the definition of $\ogam$.  The left and right triangles commute on the nose since $\tsigma\tomega(\varrho^{\tsigma\tomega})=\hat{\varrho}^{\tsigma\tomega}$ and  $\tsigma\tomega(\hat{\varrho}^{\tsigma\tomega})=\varrho^{\tsigma\tomega}$ and $\varrho^{\tsigma\tomega}$ has inverse $\hat{\varrho}^{\tsigma\tomega}$ by Proposition~\ref{prop_homotopy_sym}. The second claim in the Proposition is proven similarly using the definition of $\odelt$.
\end{proof}

Notice that there are four equations in Proposition~\ref{prop_psi} but only two in Proposition~\ref{prop_sigmaomega}.  The missing equalities are definitions \eqref{eq_def_overline_gamma} and \eqref{eq_def_overline_delta}.

\begin{prop} \label{prop_sigma}
For all $n \in \Z$ the equalities (=) and homotopy equivalences ($\simeq$) hold.
 \begin{align}
 \gam\onen &= \hat{\varrho}^{\tsigma}\cal{F}\circ\tsigma(\delt\onenn{-n+2})\circ \cal{F} \varrho^{\tsigma} \label{eq_gamma_sigma}\\
 \delt\onen &\simeq \cal{F}\hat{\varrho}^{\tsigma}\circ\tsigma(\gam\onenn{-n+2})\circ  \varrho^{\tsigma}\cal{F} \label{eq_delt_sigma}\\
  \ogam\onen &= \cal{E} \hat{\varrho}^{\tsigma}\circ\tsigma(\odelt\onenn{-n-2})\circ \varrho^{\tsigma}\cal{E} \label{eq_bargamma_sigma}\\
  \odelt\onen &\simeq \hat{\varrho}^{\tsigma}\cal{E}\circ\tsigma(\ogam\onenn{-n-2})\circ \cal{E} \varrho^{\tsigma} \label{eq_bardelta_sigma}
\end{align}
\end{prop}

\begin{proof}
We need to show the equality of two chain maps in equation \eqref{eq_gamma_sigma}.  Each is given by three two-by-two matrices and we check the equality of coefficients one by one. That is twelve equations to check. Here we prove just one of the more complicated equalities of matrix entries. For example, the equality for the upper left term of the first matrix is
\begin{align} \label{eq_sigma_samp_gamma1}
\gam_1\onen
&= \hat{\varrho}^{\tsigma}_1 \cal{F} \circ \tsigma(\delt_1\onenn{-n+2}) \circ \cal{F} \varrho^{\tsigma}_1
+ \hat{\varrho}^{\tsigma}_3 \cal{F} \circ \tsigma(\delt_2\onenn{-n+2}) \circ \cal{F} \varrho^{\tsigma}_1 \nn \\
&\quad +\hat{\varrho}^{\tsigma}_1 \cal{F} \circ \tsigma(\delt_3\onenn{-n+2}) \circ \cal{F} \varrho^{\tsigma}_2
+ \hat{\varrho}^{\tsigma}_3 \cal{F} \circ \tsigma(\delt_4\onenn{-n+2}) \circ \cal{F} \varrho^{\tsigma}_2.
\end{align}
One can check that
\begin{align}
\hat{\varrho}^{\tsigma}_1 \cal{F} \circ \tsigma(\delt_1\onenn{-n+2}) \circ \cal{F} \varrho^{\tsigma}_1 \;\; &= \;\;
    \vcenter{
 \xy 0;/r.13pc/:
    (-8,-16)*{};(8,16)*{} **\crv{(-8,-4) & (8,4)} ?(0)*\dir{<};
    (0,-16)*{};(-8,16)*{} **\crv{(0,-10) &(8,-5) &(8,2) & (-8,10)} ?(1)*\dir{>};
    (0,16)*{};(8,-16)*{} **\crv{(0,10) &(-8,5) &(-8,-2) & (8,-10)} ?(1)*\dir{>};
      (15,8)*{n};
        (7,0)*{\bullet};
  \endxy}
 \;\; - \;\;
     \vcenter{
 \xy 0;/r.13pc/:
    (-8,-16)*{};(8,16)*{} **\crv{(-8,-4) & (8,4)} ?(0)*\dir{<};
    (0,-16)*{};(-8,16)*{} **\crv{(0,-10) &(8,-5) &(8,2) & (-8,10)} ?(1)*\dir{>};
    (0,16)*{};(8,-16)*{} **\crv{(0,10) &(-8,5) &(-8,-2) & (8,-10)} ?(1)*\dir{>};
      (15,8)*{n};
      (0,0)*{\bullet};
  \endxy}
 \;\; -\;\;\sum_{\xy (0,2)*{\scs f_1+f_2+f_3}; (0,-1)*{\scs =-n+1};
 (0,-4)*{\scs 1 \leq f_1}; \endxy} \;
\vcenter{
 \xy 0;/r.13pc/:
    (8,-16)*{};(8,16)*{} **\crv{(8,-10) & (0,-10) & (0,-5) & (8,10)}?(0)*\dir{<};
    (0,-16)*{};(-8,-16)*{} **\crv{(0,-10) & (8,-10) & (8,-2) & (-8,-2)}?(1)*\dir{>};
    (-16,1)*{\cbub{\spadesuit+f_2}{}};
    (0,8)*{};(-8,16)*{} **\crv{(0,11) & (-8,13)}?(1)*\dir{>};
    (-8,8)*{};(0,16)*{} **\crv{(-8,11) & (0,13)};
    (-8,8)*{};(0,8)*{} **\crv{(-8,4) & (0,4)}
        ?(.5)*\dir{}+(0,-0.1)*{\bullet}+(-1,-4)*{\scs f_3};
    (7,-5.5)*{\bullet}+(4.5,1)*{\scs f_1};
  (15,8)*{n};
\endxy} \\
\;\; &= \;\;
\vcenter{
 \xy 0;/r.13pc/:
    (-8,-16)*{};(8,16)*{} **\crv{(-8,-4) & (8,4)} ?(0)*\dir{<};
    (0,-16)*{};(-8,16)*{} **\crv{(0,-10) &(8,-5) &(8,2) & (-8,10)} ?(1)*\dir{>};
    (0,16)*{};(8,-16)*{} **\crv{(0,10) &(-8,5) &(-8,-2) & (8,-10)} ?(1)*\dir{>};
      (15,8)*{n};
        (7,0)*{\bullet};
  \endxy}
 \;\; - \;\;
     \vcenter{
 \xy 0;/r.13pc/:
    (-8,-16)*{};(8,16)*{} **\crv{(-8,-4) & (8,4)} ?(0)*\dir{<};
    (0,-16)*{};(-8,16)*{} **\crv{(0,-10) &(8,-5) &(8,2) & (-8,10)} ?(1)*\dir{>};
    (0,16)*{};(8,-16)*{} **\crv{(0,10) &(-8,5) &(-8,-2) & (8,-10)} ?(1)*\dir{>};
      (15,8)*{n};
      (0,0)*{\bullet};
  \endxy}
\end{align}
where the summation term is zero by the dotted curl relation \eqref{eq_reduction-dots} since
\begin{equation}
  \sum_{\xy (0,2)*{\scs f_1+f_2+f_3}; (0,-1)*{\scs =-n+1};
 (0,-4)*{\scs 1 \leq f_1}; \endxy} \;
\vcenter{
 \xy 0;/r.13pc/:
    (0,8)*{};(-8,16)*{} **\crv{(0,11) & (-8,13)}?(1)*\dir{>};
    (-8,8)*{};(0,16)*{} **\crv{(-8,11) & (0,13)};
    (-8,8)*{};(0,8)*{} **\crv{(-8,4) & (0,4)}
        ?(.5)*\dir{}+(0,-0.1)*{\bullet}+(-1,-4)*{\scs f_3};
  (11,3)*{n};
\endxy}
\;\; = \;\;
  \sum_{\xy (0,2)*{\scs f_1+f_2+f_3}; (0,-1)*{\scs =-n+1};
 (0,-4)*{\scs 1 \leq f_1}; \endxy} \;
 \sum_{\xy (0,2)*{\scs g_1+g_2}; (0,-1)*{\scs =n-2+f_3}; \endxy} \;
\vcenter{
 \xy 0;/r.13pc/:
    (0,16)*{};(-8,16)*{} **\crv{(0,10) & (-8,10)}?(1)*\dir{>}
        ?(.5)*\dir{}+(0,-0.1)*{\bullet}+(-1,-4)*{\scs g_1};
     (-1,-3)*{\ccbub{\spadesuit+g_2}{}};
  (11,3)*{n};
\endxy}
\end{equation}
and the for both summations to be nonzero we must have $-n+1\geq 0$ and $f_3+n-2\geq 0$ which is impossible since $f_3 \leq -n+1$. Similarly,
\begin{equation}
  \hat{\varrho}^{\tsigma}_3 \cal{F} \circ \tsigma(\delt_2\onenn{-n+2}) \circ \cal{F} \varrho^{\tsigma}_1 \;\; = \;\;
  -\;\;\sum_{\xy (0,2)*{\scs f_1+f_2}; (0,-1)*{\scs =-n+1};
 (0,-4)*{\scs 1 \leq f_1}; \endxy} \;
 \sum_{\xy (0,2)*{\scs g_1+g_2}; (0,-1)*{\scs =n-3};\endxy} \;
\vcenter{
 \xy 0;/r.13pc/:
    (8,-16)*{};(8,20)*{} **\crv{(8,-10) & (0,-10) & (0,-5) & (8,10)}?(0)*\dir{<};
    (0,-16)*{};(-8,-16)*{} **\crv{(0,-10) & (8,-10) & (8,-2) & (-8,-2)}?(1)*\dir{>};
    (-18,-7)*{\cbub{\spadesuit+f_2}{}};
   (0,20)*{};(-8,20)*{} **\crv{(0,14) & (-8,14)}?(1)*\dir{>}
        ?(.25)*\dir{}+(0,-0.1)*{\bullet}+(4,-1)*{\scs g_1};
     (-5.5,6)*{\ccbub{\spadesuit+g_2}{}};
    (7,-5.5)*{\bullet}+(4.5,1)*{\scs f_1};
  (15,8)*{n};
\endxy} \;\; = \;\; 0
\end{equation}
since the summation indices can never both be non-negative.

We also have
\begin{equation}
  \hat{\varrho}^{\tsigma}_1 \cal{F} \circ \tsigma(\delt_3\onenn{-n+2}) \circ \cal{F} \varrho^{\tsigma}_2 \;\; = \;\;
\vcenter{
 \xy 0;/r.13pc/:
    (-8,-16)*{};(8,16)*{} **\crv{(-8,-4) & (8,4)} ?(0)*\dir{<};
    (0,-16)*{};(8,-16)*{} **\crv{(0,-10) & (8,-10)}?(1)*\dir{>};
    (0,8)*{};(-8,16)*{} **\crv{(0,11) & (-8,13)}?(1)*\dir{>};
    (-8,8)*{};(0,16)*{} **\crv{(-8,11) & (0,13)};
    (-8,8)*{};(0,8)*{} **\crv{(-8,4) & (0,4)};
  (15,6)*{n};
\endxy}
\;\; = \;\;
\sum_{\xy (0,2)*{\scs g_1+g_2}; (0,-1)*{\scs =n-2};\endxy} \;
\vcenter{
 \xy 0;/r.13pc/:
    (-8,-16)*{};(8,20)*{} **\crv{(-8,-8) & (8,-8)} ?(0)*\dir{<};
    (0,-16)*{};(8,-16)*{} **\crv{(0,-10) & (8,-10)}?(1)*\dir{>};
     (0,20)*{};(-8,20)*{} **\crv{(0,14) & (-8,14)}?(1)*\dir{>}
        ?(.25)*\dir{}+(0,-0.1)*{\bullet}+(4,-1)*{\scs g_1};
     (-5.5,6)*{\ccbub{\spadesuit+g_2}{}};
  (15,6)*{n};
\endxy}
\end{equation}
and
\begin{equation}
  \hat{\varrho}^{\tsigma}_3 \cal{F} \circ \tsigma(\delt_4\onenn{-n+2}) \circ \cal{F} \varrho^{\tsigma}_2 \;\; = \;\; -\;
 \sum_{\xy (0,2)*{\scs g_1+g_2}; (0,-1)*{\scs =n-3};\endxy} \;
\vcenter{
 \xy 0;/r.13pc/:
    (-8,-16)*{};(8,20)*{} **\crv{(-8,-8) & (8,-8)} ?(0)*\dir{<}
        ?(.45)*\dir{}+(0,0)*{\bullet};
    (0,-16)*{};(8,-16)*{} **\crv{(0,-10) & (8,-10)}?(1)*\dir{>};
     (0,20)*{};(-8,20)*{} **\crv{(0,14) & (-8,14)}?(1)*\dir{>}
        ?(.25)*\dir{}+(0,-0.1)*{\bullet}+(4,-1)*{\scs g_1};
     (-5.5,6)*{\ccbub{\spadesuit+g_2}{}};
  (15,6)*{n};
\endxy}
\end{equation}
Hence, equation \eqref{eq_sigma_samp_gamma1} amounts to proving the equality
\begin{equation} \label{eq_gamma1_samp}
  \gam_1 \onen \;\; = \;\;
  \vcenter{
 \xy 0;/r.13pc/:
    (-8,-16)*{};(8,16)*{} **\crv{(-8,-4) & (8,4)} ?(0)*\dir{<};
    (0,-16)*{};(-8,16)*{} **\crv{(0,-10) &(8,-5) &(8,2) & (-8,10)} ?(1)*\dir{>};
    (0,16)*{};(8,-16)*{} **\crv{(0,10) &(-8,5) &(-8,-2) & (8,-10)} ?(1)*\dir{>};
      (15,8)*{n};
        (7,0)*{\bullet};
  \endxy}
 \;\; - \;\;
     \vcenter{
 \xy 0;/r.13pc/:
    (-8,-16)*{};(8,16)*{} **\crv{(-8,-4) & (8,4)} ?(0)*\dir{<};
    (0,-16)*{};(-8,16)*{} **\crv{(0,-10) &(8,-5) &(8,2) & (-8,10)} ?(1)*\dir{>};
    (0,16)*{};(8,-16)*{} **\crv{(0,10) &(-8,5) &(-8,-2) & (8,-10)} ?(1)*\dir{>};
      (15,8)*{n};
      (0,0)*{\bullet};
  \endxy}
  \;\; + \;\;
  \vcenter{
 \xy 0;/r.13pc/:
    (-8,-16)*{};(8,16)*{} **\crv{(-8,-4) & (8,4)} ?(0)*\dir{<};
    (0,-16)*{};(8,-16)*{} **\crv{(0,-10) & (8,-10)}?(1)*\dir{>};
    (0,8)*{};(-8,16)*{} **\crv{(0,11) & (-8,13)}?(1)*\dir{>};
    (-8,8)*{};(0,16)*{} **\crv{(-8,11) & (0,13)};
    (-8,8)*{};(0,8)*{} **\crv{(-8,4) & (0,4)};
  (15,6)*{n};
\endxy}
\;\; - \;\;
 \sum_{\xy (0,2)*{\scs g_1+g_2}; (0,-1)*{\scs =n-3};\endxy} \;
\vcenter{
 \xy 0;/r.13pc/:
    (-8,-16)*{};(8,20)*{} **\crv{(-8,-8) & (8,-8)} ?(0)*\dir{<}
        ?(.45)*\dir{}+(0,0)*{\bullet};
    (0,-16)*{};(8,-16)*{} **\crv{(0,-10) & (8,-10)}?(1)*\dir{>};
     (0,20)*{};(-8,20)*{} **\crv{(0,14) & (-8,14)}?(1)*\dir{>}
        ?(.25)*\dir{}+(0,-0.1)*{\bullet}+(4,-1)*{\scs g_1};
     (-5.5,6)*{\ccbub{\spadesuit+g_2}{}};
  (15,6)*{n};
\endxy}
\end{equation}

Using the nilHecke relation \eqref{eq_nil_dotslide} and cancelling terms we can write
\begin{align}
    \vcenter{
 \xy 0;/r.13pc/:
    (-8,-16)*{};(8,16)*{} **\crv{(-8,-4) & (8,4)} ?(0)*\dir{<};
    (0,-16)*{};(-8,16)*{} **\crv{(0,-10) &(8,-5) &(8,2) & (-8,10)} ?(1)*\dir{>};
    (0,16)*{};(8,-16)*{} **\crv{(0,10) &(-8,5) &(-8,-2) & (8,-10)} ?(1)*\dir{>};
      (15,8)*{n};
        (7,0)*{\bullet};
  \endxy}
 \;\; - \;\;
     \vcenter{
 \xy 0;/r.13pc/:
    (-8,-16)*{};(8,16)*{} **\crv{(-8,-4) & (8,4)} ?(0)*\dir{<};
    (0,-16)*{};(-8,16)*{} **\crv{(0,-10) &(8,-5) &(8,2) & (-8,10)} ?(1)*\dir{>};
    (0,16)*{};(8,-16)*{} **\crv{(0,10) &(-8,5) &(-8,-2) & (8,-10)} ?(1)*\dir{>};
      (15,8)*{n};
      (0,0)*{\bullet};
  \endxy}
  \;\; & = \;\;
      \vcenter{
 \xy 0;/r.13pc/:
    (-8,-16)*{};(8,16)*{} **\crv{(-8,-4) & (8,4)} ?(0)*\dir{<};
    (0,-16)*{};(-8,16)*{} **\crv{(0,-10) &(8,-5) &(8,2) & (-8,10)} ?(1)*\dir{>};
    (0,16)*{};(8,-16)*{} **\crv{(0,10) &(-8,5) &(-8,-2) & (8,-10)} ?(1)*\dir{>};
      (15,8)*{n};
        (0,5.5)*{\bullet};
  \endxy}
 \;\; - \;\;
     \vcenter{
 \xy 0;/r.13pc/:
    (-8,-16)*{};(8,16)*{} **\crv{(-8,-4) & (8,4)} ?(0)*\dir{<};
    (0,-16)*{};(-8,16)*{} **\crv{(0,-10) &(8,-5) &(8,2) & (-8,10)} ?(1)*\dir{>};
    (0,16)*{};(8,-16)*{} **\crv{(0,10) &(-8,5) &(-8,-2) & (8,-10)} ?(1)*\dir{>};
      (15,8)*{n};
      (5.5,7)*{\bullet};
  \endxy}
\end{align}
Now we can apply \eqref{eq_r3_extra} to both terms on the right-hand
side, sliding the bottom right crossing in each diagram through the
diagonal line.  Note that after applying \eqref{eq_r3_extra} the
location of the dots above produce some terms that cancel leaving
only
\begin{equation}
    \vcenter{
 \xy 0;/r.13pc/:
    (-8,-16)*{};(8,16)*{} **\crv{(-8,-8) & (8,-8) & (8,14)} ?(0)*\dir{<};
     (-8,16)*{};(0,8)*{} **\crv{(-8,11) & (0,11)}?(0)*\dir{<};
    (0,16)*{};(-8,8)*{} **\crv{(0,13) & (-8,11)};
    (-8,8)*{};(8,-16)*{} **\crv{(-8,3) & (8,-3)} ?(1)*\dir{>};
    (0,8)*{};(-8,0)*{} **\crv{(0,5) & (-8,3)};
    (-8,0)*{};(0,-16)*{} **\crv{(-8,-4) & (0,-8)};
     (0,8)*{\bullet}; (15,-4)*{n};
  \endxy}
  \;\; - \;\;
 \vcenter{
 \xy 0;/r.13pc/:
    (-8,-16)*{};(8,16)*{} **\crv{(-8,-8) & (8,-8) & (8,14)} ?(0)*\dir{<};
     (-8,16)*{};(0,8)*{} **\crv{(-8,11) & (0,11)}?(0)*\dir{<};
    (0,16)*{};(-8,8)*{} **\crv{(0,13) & (-8,11)};
    (-8,8)*{};(8,-16)*{} **\crv{(-8,3) & (8,-3)} ?(1)*\dir{>};
    (0,8)*{};(-8,0)*{} **\crv{(0,5) & (-8,3)};
    (-8,0)*{};(0,-16)*{} **\crv{(-8,-4) & (0,-8)};
     (7.5,8)*{\bullet}; (15,-4)*{n};
  \endxy}
  \;\; + \;\;
   \sum_{\xy (0,2)*{\scs g_1+g_2+g_3}; (0,-1)*{\scs +g_4=-n};\endxy} \;
\vcenter{
 \xy 0;/r.13pc/:
    (8,-16)*{};(8,16)*{} **\dir{-} ?(0)*\dir{<}
        ?(.45)*\dir{}+(0,0)*{\bullet}+(4.5,2)*{\scs g_4};
  (0,-16)*{};(-8,-16)*{} **\crv{(0,-10) & (-8,-10)}?(1)*\dir{>}
        ?(.35)*\dir{}+(0,0)*{\bullet}+(4,3)*{\scs g_3};
   (0,8)*{};(-8,16)*{} **\crv{(0,11) & (-8,11)}?(1)*\dir{>};
    (-8,8)*{};(0,16)*{} **\crv{(-8,11) & (0,11)};
    (-8,8)*{};(0,8)*{} **\crv{(-8,4) & (0,4)}
        ?(.5)*\dir{}+(0,-0.1)*{\bullet}+(-1,-4)*{\scs g_1+1};
     (-15,-4)*{\cbub{\spadesuit+g_2}{}};
  (15,-4)*{n};
\endxy}
\;\; - \;\;
    \sum_{\xy (0,2)*{\scs g_1+g_2+g_3}; (0,-1)*{\scs +g_4=-n};\endxy} \;
\vcenter{
 \xy 0;/r.13pc/:
    (8,-16)*{};(8,16)*{} **\dir{-} ?(0)*\dir{<}
        ?(.45)*\dir{}+(0,0)*{\bullet}+(8.5,2)*{\scs g_4+1};
  (0,-16)*{};(-8,-16)*{} **\crv{(0,-10) & (-8,-10)}?(1)*\dir{>}
        ?(.35)*\dir{}+(0,0)*{\bullet}+(4,3)*{\scs g_3};
   (0,8)*{};(-8,16)*{} **\crv{(0,11) & (-8,11)}?(1)*\dir{>};
    (-8,8)*{};(0,16)*{} **\crv{(-8,11) & (0,11)};
    (-8,8)*{};(0,8)*{} **\crv{(-8,4) & (0,4)}
        ?(.5)*\dir{}+(0,-0.1)*{\bullet}+(-1,-4)*{\scs g_1};
     (-15,-4)*{\cbub{\spadesuit+g_2}{}};
  (15,-4)*{n};
\endxy}
\end{equation}
where the last two terms can be shown to be zero by simplifying the dotted curl and arguing as above. Therefore,
\begin{align}
    \vcenter{
 \xy 0;/r.13pc/:
    (-8,-16)*{};(8,16)*{} **\crv{(-8,-4) & (8,4)} ?(0)*\dir{<};
    (0,-16)*{};(-8,16)*{} **\crv{(0,-10) &(8,-5) &(8,2) & (-8,10)} ?(1)*\dir{>};
    (0,16)*{};(8,-16)*{} **\crv{(0,10) &(-8,5) &(-8,-2) & (8,-10)} ?(1)*\dir{>};
      (14,8)*{n};
        (7,0)*{\bullet};
  \endxy}
 \;\; - \;\;
     \vcenter{
 \xy 0;/r.13pc/:
    (-8,-16)*{};(8,16)*{} **\crv{(-8,-4) & (8,4)} ?(0)*\dir{<};
    (0,-16)*{};(-8,16)*{} **\crv{(0,-10) &(8,-5) &(8,2) & (-8,10)} ?(1)*\dir{>};
    (0,16)*{};(8,-16)*{} **\crv{(0,10) &(-8,5) &(-8,-2) & (8,-10)} ?(1)*\dir{>};
      (14,8)*{n};
      (0,0)*{\bullet};
  \endxy}
  \;\; & = \;\;
  \vcenter{
 \xy 0;/r.13pc/:
    (-8,-16)*{};(8,16)*{} **\crv{(-8,-8) & (8,-8) & (8,14)} ?(0)*\dir{<};
     (-8,16)*{};(0,8)*{} **\crv{(-8,11) & (0,11)}?(0)*\dir{<};
    (0,16)*{};(-8,8)*{} **\crv{(0,13) & (-8,11)};
    (-8,8)*{};(8,-16)*{} **\crv{(-8,3) & (8,-3)} ?(1)*\dir{>};
    (0,8)*{};(-8,0)*{} **\crv{(0,5) & (-8,3)};
    (-8,0)*{};(0,-16)*{} **\crv{(-8,-4) & (0,-8)};
     (0,8)*{\bullet}; (14,-4)*{n};
  \endxy}
  \;\; - \;\;
 \vcenter{
 \xy 0;/r.13pc/:
    (-8,-16)*{};(8,16)*{} **\crv{(-8,-8) & (8,-8) & (8,14)} ?(0)*\dir{<};
     (-8,16)*{};(0,8)*{} **\crv{(-8,11) & (0,11)}?(0)*\dir{<};
    (0,16)*{};(-8,8)*{} **\crv{(0,13) & (-8,11)};
    (-8,8)*{};(8,-16)*{} **\crv{(-8,3) & (8,-3)} ?(1)*\dir{>};
    (0,8)*{};(-8,0)*{} **\crv{(0,5) & (-8,3)};
    (-8,0)*{};(0,-16)*{} **\crv{(-8,-4) & (0,-8)};
     (7.5,8)*{\bullet}; (14,-4)*{n};
  \endxy} \\
  \;\; &\refequal{\eqref{eq_nil_dotslide}} \;\;
   \vcenter{
 \xy 0;/r.13pc/:
    (-8,-16)*{};(8,16)*{} **\crv{(-8,-8) & (8,-8) & (8,14)} ?(0)*\dir{<};
     (-8,16)*{};(0,8)*{} **\crv{(-8,11) & (0,11)}?(0)*\dir{<};
    (0,16)*{};(-8,8)*{} **\crv{(0,13) & (-8,11)};
    (-8,8)*{};(8,-16)*{} **\crv{(-8,3) & (8,-3)} ?(1)*\dir{>};
    (0,8)*{};(-8,0)*{} **\crv{(0,5) & (-8,3)};
    (-8,0)*{};(0,-16)*{} **\crv{(-8,-4) & (0,-8)};
     (-7,-3)*{\bullet}; (15,-4)*{n};
  \endxy}
  \;\; - \;\;
 \vcenter{
 \xy 0;/r.13pc/:
    (-8,-16)*{};(8,16)*{} **\crv{(-8,-8) & (8,-8) & (8,14)} ?(0)*\dir{<};
     (-8,16)*{};(0,8)*{} **\crv{(-8,11) & (0,11)}?(0)*\dir{<};
    (0,16)*{};(-8,8)*{} **\crv{(0,13) & (-8,11)};
    (-8,8)*{};(8,-16)*{} **\crv{(-8,3) & (8,-3)} ?(1)*\dir{>};
    (0,8)*{};(-8,0)*{} **\crv{(0,5) & (-8,3)};
    (-8,0)*{};(0,-16)*{} **\crv{(-8,-4) & (0,-8)};
     (7.5,8)*{\bullet}; (15,-4)*{n};
  \endxy}
 \;\; + \;\;
    \vcenter{
 \xy 0;/r.13pc/:
    (-8,-16)*{};(8,16)*{} **\crv{(-8,-8) & (8,-8) & (8,14)} ?(0)*\dir{<};
     (-8,16)*{};(0,8)*{} **\crv{(-8,11) & (0,11)}?(0)*\dir{<};
    (0,16)*{};(-8,8)*{} **\crv{(0,13) & (-8,11)};
    (0,8)*{};(-8,8) **\crv{ (0,2) & (-8,2)};
    (8,-16)*{};(0,-16)*{} **\crv{(8,-12) & (0,2) & (-8,-2) & (0,-10)}?(0)*\dir{<};
     (15,-4)*{n};
  \endxy} \hspace{1in}
\end{align}
\begin{align}
 &\refequal{\eqref{eq_ident_decomp}}
- \;\;\vcenter{
 \xy 0;/r.13pc/:
    (-4,-4)*{};(4,4)*{} **\crv{(-4,-1) & (4,1)} ?(0)*\dir{<};
    (4,-4)*{};(-4,4)*{} **\crv{(4,-1) & (-4,1)};
    (4,4)*{};(12,12)*{} **\crv{(4,7) & (12,9)};
    (12,4)*{};(4,12)*{} **\crv{(12,7) & (4,9)};
    (-4,4)*{}; (-4,12) **\dir{-}?(1)*\dir{>};
    (12,-4)*{}; (12,4) **\dir{-}?(0)*\dir{<};
    (-4.3,6)*{\bullet}; %(18,8)*{n};
\endxy} \;\; + \;\;
 \vcenter{
 \xy 0;/r.13pc/:
    (-4,-4)*{};(4,4)*{} **\crv{(-4,-1) & (4,1)} ?(0)*\dir{<};
    (4,-4)*{};(-4,4)*{} **\crv{(4,-1) & (-4,1)};
    (4,4)*{};(12,12)*{} **\crv{(4,7) & (12,9)};
    (12,4)*{};(4,12)*{} **\crv{(12,7) & (4,9)};
    (-4,4)*{}; (-4,12) **\dir{-}?(1)*\dir{>};
    (12,-4)*{}; (12,4) **\dir{-}?(0)*\dir{<};
  (10,9.5)*{\bullet}; %(18,8)*{n};
\endxy}
 \;\; + \;\;
 \sum_{\xy (0,2)*{\scs f_1+f_2+f_3}; (0,-1)*{\scs =n-3};\endxy} \;
    \vcenter{
 \xy 0;/r.13pc/:
    (-8,-16)*{};(8,16)*{} **\crv{(-8,-8) & (8,-8) & (8,14)} ?(0)*\dir{<};
    (0,16)*{};(-8,16)*{} **\crv{(0,10) & (-8,10)}
      ?(1)*\dir{>} ?(.35)*\dir{} +(0,0)*{\bullet}+(4,-2)*{\scs f_1};
    (8,-16)*{};(0,-16)*{} **\crv{(8,-12) & (0,0) & (-8,-4) & (0,-10)}?(0)*\dir{<}
        ?(.7)*\dir{} +(0,0)*{\bullet}+(-7.5,-2)*{\scs f_3+1};
    (-14,4)*{\ccbub{\spadesuit+f_2}{}};
    %(15,-4)*{n};
  \endxy}
  \;\; -\;\;
   \sum_{\xy (0,2)*{\scs f_1+f_2+f_3}; (0,-1)*{\scs =n-3};\endxy} \;
    \vcenter{
 \xy 0;/r.13pc/:
    (-8,-16)*{};(8,16)*{} **\crv{(-8,-8) & (8,-8) & (8,14)} ?(0)*\dir{<};
    (0,16)*{};(-8,16)*{} **\crv{(0,10) & (-8,10)}
      ?(1)*\dir{>} ?(.35)*\dir{} +(0,0)*{\bullet}+(4,-2)*{\scs f_1};
    (8,-16)*{};(0,-16)*{} **\crv{(8,-12) & (0,0) & (-8,-4) & (0,-10)}?(0)*\dir{<}
        ?(.7)*\dir{} +(0,0)*{\bullet}+(-4,-2)*{\scs f_3};
    (-14,4)*{\ccbub{\spadesuit+f_2}{}};
    (7.5,8)*{\bullet};
    %(15,-4)*{n};
  \endxy}
 \;\; + \;\;
    \vcenter{
 \xy 0;/r.13pc/:
    (-8,-16)*{};(8,16)*{} **\crv{(-8,-8) & (8,-8) & (8,14)} ?(0)*\dir{<};
     (-8,16)*{};(0,8)*{} **\crv{(-8,11) & (0,11)}?(0)*\dir{<};
    (0,16)*{};(-8,8)*{} **\crv{(0,13) & (-8,11)};
    (0,8)*{};(-8,8) **\crv{ (0,2) & (-8,2)};
    (8,-16)*{};(0,-16)*{} **\crv{(8,-12) & (0,2) & (-8,-2) & (0,-10)}?(0)*\dir{<};
     %(15,-4)*{n};
  \endxy}
 \end{align}
Now plug this into the right-hand side of \eqref{eq_gamma1_samp}, use the identity decomposition equation \eqref{eq_ident_decomp} on the last three terms above and note that the additional bubble terms arising from the application of \eqref{eq_ident_decomp} vanish by considering the conditions on the summation indices as above.  After reordering the non-vanishing terms of the right-hand side of \eqref{eq_gamma1_samp} we are left with
\begin{align}
  - \;\;\vcenter{
 \xy 0;/r.13pc/:
    (-4,-4)*{};(4,4)*{} **\crv{(-4,-1) & (4,1)} ?(0)*\dir{<};
    (4,-4)*{};(-4,4)*{} **\crv{(4,-1) & (-4,1)};
    (4,4)*{};(12,12)*{} **\crv{(4,7) & (12,9)};
    (12,4)*{};(4,12)*{} **\crv{(12,7) & (4,9)};
    (-4,4)*{}; (-4,12) **\dir{-}?(1)*\dir{>};
    (12,-4)*{}; (12,4) **\dir{-}?(0)*\dir{<};
    (-4.3,6)*{\bullet}; %(18,8)*{n};
\endxy} \;\; + \;\;
 \vcenter{
 \xy 0;/r.13pc/:
    (-4,-4)*{};(4,4)*{} **\crv{(-4,-1) & (4,1)} ?(0)*\dir{<};
    (4,-4)*{};(-4,4)*{} **\crv{(4,-1) & (-4,1)};
    (4,4)*{};(12,12)*{} **\crv{(4,7) & (12,9)};
    (12,4)*{};(4,12)*{} **\crv{(12,7) & (4,9)};
    (-4,4)*{}; (-4,12) **\dir{-}?(1)*\dir{>};
    (12,-4)*{}; (12,4) **\dir{-}?(0)*\dir{<};
  (10,9.5)*{\bullet}; %(18,8)*{n};
\endxy}
 \;\; - \;\;
 \sum_{\xy (0,2)*{\scs f_1+f_2+f_3}; (0,-1)*{\scs =n-3};\endxy} \;
  \vcenter{
 \xy 0;/r.13pc/:
    (-8,-16)*{};(8,20)*{} **\crv{(-8,-8) & (8,-8)} ?(0)*\dir{<};
    (0,-16)*{};(8,-16)*{} **\crv{(0,-10) & (8,-10)}?(1)*\dir{>}
        ?(.5)*\dir{} +(0,0)*{\bullet}+(5,2.5)*{\scs f_3+1};
     (0,20)*{};(-8,20)*{} **\crv{(0,14) & (-8,14)}?(1)*\dir{>}
        ?(.25)*\dir{}+(0,-0.1)*{\bullet}+(4,-1)*{\scs f_1};
     (-5.5,6)*{\ccbub{\spadesuit+f_2}{}};
  %(15,6)*{n};
\endxy}
  \;\; +\;\;
   \sum_{\xy (0,2)*{\scs f_1+f_2+f_3}; (0,-1)*{\scs =n-3};\endxy} \;
  \vcenter{
 \xy 0;/r.13pc/:
    (-8,-16)*{};(8,20)*{} **\crv{(-8,-8) & (8,-8)} ?(0)*\dir{<}
        ?(.45)*\dir{}+(0,0)*{\bullet};
    (0,-16)*{};(8,-16)*{} **\crv{(0,-10) & (8,-10)}?(1)*\dir{>}
        ?(.5)*\dir{} +(0,0)*{\bullet}+(4,2)*{\scs f_3};
     (0,20)*{};(-8,20)*{} **\crv{(0,14) & (-8,14)}?(1)*\dir{>}
        ?(.25)*\dir{}+(0,-0.1)*{\bullet}+(4,-1)*{\scs f_1};
     (-5.5,6)*{\ccbub{\spadesuit+f_2}{}};
  %(15,6)*{n};
\endxy}
\;\; - \;\;
 \sum_{\xy (0,2)*{\scs g_1+g_2}; (0,-1)*{\scs =n-3};\endxy} \;
\vcenter{
 \xy 0;/r.13pc/:
    (-8,-16)*{};(8,20)*{} **\crv{(-8,-8) & (8,-8)} ?(0)*\dir{<}
        ?(.45)*\dir{}+(0,0)*{\bullet};
    (0,-16)*{};(8,-16)*{} **\crv{(0,-10) & (8,-10)}?(1)*\dir{>};
     (0,20)*{};(-8,20)*{} **\crv{(0,14) & (-8,14)}?(1)*\dir{>}
        ?(.25)*\dir{}+(0,-0.1)*{\bullet}+(4,-1)*{\scs g_1};
     (-5.5,6)*{\ccbub{\spadesuit+g_2}{}};
  %(15,6)*{n};
\endxy}
\end{align}
By rescalling the third term and simplifying the last two terms,
the above becomes
\begin{align}
  - \;\;\vcenter{
 \xy 0;/r.13pc/:
    (-4,-4)*{};(4,4)*{} **\crv{(-4,-1) & (4,1)} ?(0)*\dir{<};
    (4,-4)*{};(-4,4)*{} **\crv{(4,-1) & (-4,1)};
    (4,4)*{};(12,12)*{} **\crv{(4,7) & (12,9)};
    (12,4)*{};(4,12)*{} **\crv{(12,7) & (4,9)};
    (-4,4)*{}; (-4,12) **\dir{-}?(1)*\dir{>};
    (12,-4)*{}; (12,4) **\dir{-}?(0)*\dir{<};
    (-4.3,6)*{\bullet}; %(18,8)*{n};
\endxy} \;\; + \;\;
 \vcenter{
 \xy 0;/r.13pc/:
    (-4,-4)*{};(4,4)*{} **\crv{(-4,-1) & (4,1)} ?(0)*\dir{<};
    (4,-4)*{};(-4,4)*{} **\crv{(4,-1) & (-4,1)};
    (4,4)*{};(12,12)*{} **\crv{(4,7) & (12,9)};
    (12,4)*{};(4,12)*{} **\crv{(12,7) & (4,9)};
    (-4,4)*{}; (-4,12) **\dir{-}?(1)*\dir{>};
    (12,-4)*{}; (12,4) **\dir{-}?(0)*\dir{<};
  (10,9.5)*{\bullet}; %(18,8)*{n};
\endxy}
 \;\; - \;\;
 \sum_{\xy (0,2)*{\scs f_1+f_2+f'_3}; (0,-1)*{\scs =n-2};
        (0,-4)*{\scs 1 \leq f'_3};\endxy} \;
  \vcenter{
 \xy 0;/r.13pc/:
    (-8,-16)*{};(8,20)*{} **\crv{(-8,-8) & (8,-8)} ?(0)*\dir{<};
    (0,-16)*{};(8,-16)*{} **\crv{(0,-10) & (8,-10)}?(1)*\dir{>}
        ?(.5)*\dir{} +(0,0)*{\bullet}+(5,2.5)*{\scs f'_3};
     (0,20)*{};(-8,20)*{} **\crv{(0,14) & (-8,14)}?(1)*\dir{>}
        ?(.25)*\dir{}+(0,-0.1)*{\bullet}+(4,-1)*{\scs f_1};
     (-5.5,6)*{\ccbub{\spadesuit+f_2}{}};
  %(15,6)*{n};
\endxy}
  \;\; +\;\;
   \sum_{\xy (0,2)*{\scs f_1+f_2+f_3}; (0,-1)*{\scs =n-3};
        (0,-4)*{\scs 1 \leq f_3};\endxy} \;
  \vcenter{
 \xy 0;/r.13pc/:
    (-8,-16)*{};(8,20)*{} **\crv{(-8,-8) & (8,-8)} ?(0)*\dir{<}
        ?(.45)*\dir{}+(0,0)*{\bullet};
    (0,-16)*{};(8,-16)*{} **\crv{(0,-10) & (8,-10)}?(1)*\dir{>}
        ?(.5)*\dir{} +(0,0)*{\bullet}+(4,2)*{\scs f_3};
     (0,20)*{};(-8,20)*{} **\crv{(0,14) & (-8,14)}?(1)*\dir{>}
        ?(.25)*\dir{}+(0,-0.1)*{\bullet}+(4,-1)*{\scs f_1};
     (-5.5,6)*{\ccbub{\spadesuit+f_2}{}};
  %(15,6)*{n};
\endxy} \hspace{0.8in}
\end{align}
But by sliding the bubbles using a rotated version of \eqref{eq_bubbleslide_cc_r} we can simplify the last two terms
\begin{align}
\sum_{\xy (0,2)*{\scs f_1+f_2+f_3}; (0,-1)*{\scs =n-3};
        (0,-4)*{\scs 1 \leq f_3};\endxy} \;
  \vcenter{
 \xy 0;/r.13pc/:
    (-8,-16)*{};(8,20)*{} **\crv{(-8,-8) & (8,-8)} ?(0)*\dir{<}
        ?(.45)*\dir{}+(0,0)*{\bullet};
    (0,-16)*{};(8,-16)*{} **\crv{(0,-10) & (8,-10)}?(1)*\dir{>}
        ?(.5)*\dir{} +(0,0)*{\bullet}+(4,2)*{\scs f_3};
     (0,20)*{};(-8,20)*{} **\crv{(0,14) & (-8,14)}?(1)*\dir{>}
        ?(.25)*\dir{}+(0,-0.1)*{\bullet}+(4,-1)*{\scs f_1};
     (-5.5,6)*{\ccbub{\spadesuit+f_2}{}};
  %(15,6)*{n};
\endxy} \;\; - \;\;
 \sum_{\xy (0,2)*{\scs f_1+f_2+f'_3}; (0,-1)*{\scs =n-2};
        (0,-4)*{\scs 1 \leq f'_3};\endxy} \;
  \vcenter{
 \xy 0;/r.13pc/:
    (-8,-16)*{};(8,20)*{} **\crv{(-8,-8) & (8,-8)} ?(0)*\dir{<};
    (0,-16)*{};(8,-16)*{} **\crv{(0,-10) & (8,-10)}?(1)*\dir{>}
        ?(.5)*\dir{} +(0,0)*{\bullet}+(5,2.5)*{\scs f'_3};
     (0,20)*{};(-8,20)*{} **\crv{(0,14) & (-8,14)}?(1)*\dir{>}
        ?(.25)*\dir{}+(0,-0.1)*{\bullet}+(4,-1)*{\scs f_1};
     (-5.5,6)*{\ccbub{\spadesuit+f_2}{}};
  %(15,6)*{n};
\endxy}
 &\refequal{\eqref{eq_bubbleslide_cc_r}}
   \sum_{\xy (0,2)*{\scs f_1+f_2+f_3}; (0,-1)*{\scs =n-3};
        (0,-4)*{\scs 1 \leq f_3};\endxy} \;
    \sum_{\xy (0,2)*{\scs g_1+g_2}; (0,-1)*{\scs =f_2};
       \endxy} (f_2+1-g_2)\;
  \vcenter{
 \xy 0;/r.13pc/:
    (-8,-16)*{};(8,20)*{} **\crv{(-8,10) & (8,8)} ?(0)*\dir{<}
         ?(.45)*\dir{}+(0,0)*{\bullet}+(-7.5,2)*{\scs g_1+1};
    (0,-16)*{};(8,-16)*{} **\crv{(0,-10) & (8,-10)}?(1)*\dir{>}
        ?(.25)*\dir{} +(0,0)*{\bullet}+(-5,2.5)*{\scs f_3};
     (0,20)*{};(-8,20)*{} **\crv{(0,14) & (-8,14)}?(1)*\dir{>}
        ?(.7)*\dir{}+(0,-0.1)*{\bullet}+(-4,-1)*{\scs f_1};
     (8,0)*{\ccbub{\spadesuit+g_2}{}};
  %(15,6)*{n};
\endxy} \nn\\
 & \qquad  \;\; - \;\;
 \sum_{\xy (0,2)*{\scs f_1+f_2+f'_3}; (0,-1)*{\scs =n-2};
        (0,-4)*{\scs 1 \leq f'_3};\endxy}
 \sum_{\xy (0,2)*{\scs g_1+g_2}; (0,-1)*{\scs =f_2};
       \endxy} (f_2+1-g_2) \;
  \vcenter{
 \xy 0;/r.13pc/:
    (-8,-16)*{};(8,20)*{} **\crv{(-8,10) & (8,8)} ?(0)*\dir{<}
         ?(.45)*\dir{}+(0,0)*{\bullet}+(-4,2)*{\scs g_1};
    (0,-16)*{};(8,-16)*{} **\crv{(0,-10) & (8,-10)}?(1)*\dir{>}
        ?(.25)*\dir{} +(0,0)*{\bullet}+(-5,2.5)*{\scs f'_3};
     (0,20)*{};(-8,20)*{} **\crv{(0,14) & (-8,14)}?(1)*\dir{>}
        ?(.7)*\dir{}+(0,-0.1)*{\bullet}+(-4,-1)*{\scs f_1};
     (8,0)*{\ccbub{\spadesuit+g_2}{}};
  %(15,6)*{n};
\endxy} \nn \\
&=
\;\;- \sum_{\xy (0,2)*{\scs f_1+f_2+f_3+f_4}; (0,-1)*{\scs =n-2}; (0,-4)*{\scs 1 \leq f_3};\endxy} \;
  \vcenter{
 \xy 0;/r.13pc/:
    (-8,-16)*{};(8,20)*{} **\crv{(-8,10) & (8,8)} ?(0)*\dir{<}
         ?(.45)*\dir{}+(0,0)*{\bullet}+(-4,2)*{\scs f_2};
    (0,-16)*{};(8,-16)*{} **\crv{(0,-10) & (8,-10)}?(1)*\dir{>}
        ?(.25)*\dir{} +(0,0)*{\bullet}+(-5,2.5)*{\scs f_3};
     (0,20)*{};(-8,20)*{} **\crv{(0,14) & (-8,14)}?(1)*\dir{>}
        ?(.7)*\dir{}+(0,-0.1)*{\bullet}+(-4,-1)*{\scs f_1};
     (8,0)*{\ccbub{\spadesuit+f_4}{}};
  %(15,6)*{n};
\endxy}
\end{align}
proving equation \eqref{eq_sigma_samp_gamma1}.  The rest of the proof of
\eqref{eq_gamma_sigma} follows by many more computations analogous to the one above.

Equations \eqref{eq_delt_sigma}--\eqref{eq_bardelta_sigma} follow by applying various symmetries to \eqref{eq_gamma_sigma}. The left and right hand sides of \eqref{eq_delt_sigma} constitute the perimeter of the following diagram
\begin{equation} \label{eq_delta_sigma_fill}
    \xy
   (-25,0)*+{\cal{C}\cal{F}\onen}="bl";
   (25,0)*+{\cal{F}\cal{C}\onen}="br";
   (-45,15)*+{\cal{C}\cal{F}\onen}="tl'";
   (45,15)*+{\cal{F}\cal{C}\onen}="tr'";
   (-45,-15)*+{\tsigma(\cal{C})\cal{F}\onen}="bl'";
   (45,-15)*+{\cal{F}\tsigma(\cal{C})\onen}="br'";
   {\ar_{\delt\onen} "bl";"br"};
      {\ar^{\delt\onen} "tl'";"tr'"};
   {\ar_{\varrho^{\tsigma}\cal{F}} "tl'";"bl'"};
   {\ar_{\tsigma(\gam\onenn{-n+2})} "bl'";"br'"};
   {\ar_{\cal{}\hat{\varrho}^{\tsigma}} "br'";"tr'"};
   {\ar^{\Id} "tl'";"bl"};
   {\ar_{\tsigma(\cal{F}\varrho^{\tsigma})=\tsigma(\varrho^{\tsigma})\cal{F}} "bl'";"bl"};
   {\ar^{\Id} "br";"tr'"};
   {\ar_{\tsigma(\hat{\varrho}^{\tsigma}\cal{F})=\cal{F}\tsigma(\hat{\varrho}^{\tsigma})} "br";"br'"};
  \endxy
\end{equation}
Since $\tsigma(\varrho^{\tsigma})=\hat{\varrho}^{\tsigma}$ and $\tsigma(\hat{\varrho}^{\tsigma})=\varrho^{\tsigma}$  the left and right triangles commute up to homotopy by Proposition~\ref{prop_homotopy_sym}.  The bottom square is $\tsigma$ applied to \eqref{eq_gamma_sigma} with parameter $-n+2$.

Equation~\eqref{eq_bargamma_sigma} follows from the commutativity of the diagram
\begin{equation}
    \xy
   (-25,10)*+{\tsigma\tomega(\cal{C})\cal{E}\onen}="tl";
   (25,10)*+{\cal{E}\tsigma\tomega(\cal{C})\onen}="tr";
   (-25,-10)*+{\tomega(\cal{C})\cal{E}\onen}="bl";
   (25,-10)*+{\cal{E}\tomega(\cal{C})\onen}="br";
   (-45,25)*+{\cal{C}\cal{E}\onen}="tl'";
   (45,25)*+{\cal{E}\cal{C}\onen}="tr'";
   (-45,-25)*+{\tsigma(\cal{C})\cal{E}\onen}="bl'";
   (45,-25)*+{\cal{E}\tsigma(\cal{C})\onen}="br'";
   {\ar^{\tsigma\tomega(\gam)} "tl";"tr"};
   {\ar^{\tomega(\varrho^{\tsigma}\cal{F})} "bl";"tl"};
   {\ar_{\tomega(\delt)} "bl";"br"};
   {\ar^{\tomega(\cal{F}\hat{\varrho}^{\tsigma})} "tr";"br"};
      {\ar^{\ogam} "tl'";"tr'"};
   {\ar_{\varrho^{\tsigma}\cal{E}} "tl'";"bl'"};
   {\ar_{\tsigma(\odelt)} "bl'";"br'"};
   {\ar_{\cal{E}\hat{\varrho}^{\tsigma}} "br'";"tr'"};
   {\ar^{\varrho^{\tsigma\tomega}\cal{E}} "tl'";"tl"};
   {\ar_{\tsigma(\varrho^{\tsigma\tomega})\cal{E}=\tsigma(\varrho^{\tsigma\tomega})\cal{E}} "bl'";"bl"};
   {\ar^{\cal{E}\hat{\varrho}^{\tsigma\tomega}} "tr";"tr'"};
   {\ar_{\cal{E}\tsigma(\hat{\varrho}^{\tsigma\tomega})=
   \cal{E}\tsigma(\hat{\varrho}^{\tsigma\tomega})} "br";"br'"};
  \endxy
\end{equation}
Here $\tomega(\varrho^{\tsigma}\cal{F})=\tomega(\varrho^{\tsigma})\cal{E}$ and
$\tomega(\cal{F}\hat{\varrho}^{\tsigma})=\cal{E}\tomega(\hat{\varrho}^{\tsigma})$, the left and right squares commute on the nose since
\begin{equation}
  \varrho^{\tomega,\tsigma\tomega} \circ \tsigma(\varrho^{\tsigma\tomega})\circ \varrho^{\tsigma} = \varrho^{\tsigma\tomega}, \qquad
  \hat{\varrho}^{\tomega}
  \circ \tsigma(\hat{\varrho}^{\tsigma\tomega})\circ \hat{\varrho}^{\tomega,\tsigma\tomega} = \hat{\varrho}^{\tsigma\tomega}
\end{equation}
by Proposition~\ref{prop_homotopy_sym} part (\ref{rem_varrho_nat}).  The center square commutes since it is $\tomega$ applied to \eqref{eq_gamma_sigma}.  The bottom square is $\tsigma$ applied to the definition of $\odelt$, where we used that $\tsigma^2=\Id$.  The top square commutes by definition of $\ogam$.

A similar homotopy commutative square to \eqref{eq_delta_sigma_fill} shows that \eqref{eq_bardelta_sigma} follows from \eqref{eq_bargamma_sigma}.
\end{proof}

\begin{prop} \label{prop_omega}
For all $n \in \Z$ the equalities (=) and homotopy equivalences ($\simeq$) hold.
 \begin{align}
 \gam\onen &= \hat{\varrho}^{\tomega}\cal{F}\circ\tomega(\odelt\onenn{-n})\circ \cal{F} \varrho^{\tomega} \label{prop_omega1}\\
 \delt\onen &\simeq \cal{F}\hat{\varrho}^{\tomega}\circ\tomega(\ogam\onenn{-n})\circ  \varrho^{\tomega}\cal{F} \\
  \ogam\onen &= \cal{E} \hat{\varrho}^{\tomega}\circ\tomega(\delt\onenn{-n})\circ \varrho^{\tomega}\cal{E} \\
  \odelt\onen &\simeq \hat{\varrho}^{\tomega}\cal{E}\circ\tomega(\gam\onenn{-n})\circ \cal{E} \varrho^{\tomega} \label{prop_omega4}
\end{align}
\end{prop}

\begin{proof}
The proof follows from Proposition~\ref{prop_sigma} and the definitions of $\ogam$ and $\odelt$. For example, the first two equations are proven by the diagrams
\begin{equation}
    \xy
   (-18,0)*+{\tsigma(\cal{C})\cal{F}\onen}="bl";
   (18,0)*+{\tsigma(\cal{C})\cal{F}\onen}="br";
   (-30,15)*+{\cal{C}\cal{F}\onen}="tl'";
   (30,15)*+{\cal{F}\cal{C}\onen}="tr'";
   (-30,-15)*+{\tomega(\cal{C})\cal{F}\onen}="bl'";
   (30,-15)*+{\cal{F}\tomega(\cal{C})\onen}="br'";
   {\ar_{\tsigma(\delt\onenn{-n+2})} "bl";"br"};
      {\ar^{\gam\onen} "tl'";"tr'"};
   {\ar_{\varrho^{\tomega}\cal{F}} "tl'";"bl'"};
   {\ar_{\tomega(\odelt\onen{-n})} "bl'";"br'"};
   {\ar_{\cal{F}\hat{\varrho}^{\tomega}} "br'";"tr'"};
   {\ar^{\varrho^{\tsigma}\cal{F}} "tl'";"bl"};
   {\ar_{%\varrho^{\tomega, \tsigma}\cal{F}=
   \tomega(\varrho^{\tsigma\tomega})\cal{F}} "bl'";"bl"};
   {\ar^{\cal{F}\hat{\varrho}^{\tsigma}} "br";"tr'"};
   {\ar_{%\cal{F}\hat{\varrho}^{\tomega, \tsigma}=
   \cal{F}\tomega(\hat{\varrho}^{\tsigma\tomega})} "br";"br'"};
  \endxy
\quad
    \xy
   (-18,0)*+{\cal{F}\tsigma(\cal{C})\onen}="bl";
   (18,0)*+{\tsigma(\cal{C})\cal{F}\onen}="br";
   (-30,15)*+{\cal{F}\cal{C}\onen}="tl'";
   (30,15)*+{\cal{C}\cal{F}\onen}="tr'";
   (-30,-15)*+{\cal{F}\tomega(\cal{C})\onen}="bl'";
   (30,-15)*+{\tomega(\cal{C})\cal{F}\onen}="br'";
   {\ar_{\tsigma(\gam\onenn{-n+2})} "bl";"br"};
   {\ar^{\delt\onen} "tl'";"tr'"};
   {\ar_{\cal{F}\varrho^{\tomega}} "tl'";"bl'"};
   {\ar_{\tomega(\ogam\onenn{-n})} "bl'";"br'"};
   {\ar_{\hat{\varrho}^{\tomega}\cal{F}} "br'";"tr'"};
   {\ar^{\cal{F}\varrho^{\tsigma}} "tl'";"bl"};
   {\ar_{\cal{F}\tomega(\varrho^{\tsigma\tomega})} "bl'";"bl"};
   {\ar^{\hat{\varrho}^{\tsigma}\cal{F}} "br";"tr'"};
   {\ar_{\tomega(\hat{\varrho}^{\tsigma\tomega})\cal{F}} "br";"br'"};
  \endxy
\end{equation}
The top squares commute by Proposition~\ref{prop_sigma}. The bottom
squares are $\tomega$ applied to the definitions of $\odelt$ and
$\ogam$, respectively.  Note that
$\tomega(\varrho^{\tsigma\tomega})=\tsigma(\hat{\varrho}^{\tsigma\tomega})$
and
$\tomega(\hat{\varrho}^{\tsigma\tomega})=\tsigma(\varrho^{\tsigma\tomega})$.
The left and right triangles in both squares commute on the nose
since
\begin{equation}
  \tsigma(\hat{\varrho}^{\tsigma\tomega}) \circ \varrho^{\tomega} = \varrho^{\tsigma}, \qquad
  \hat{\varrho}^{\tomega} \circ \tsigma(\varrho^{\tsigma\tomega}) = \hat{\varrho}^{\tsigma}.
\end{equation}

Using Proposition~\ref{prop_sigmaomega}, similar arguments as the above prove the last two equations in the Proposition.
\end{proof}

% ===========================================================================
%
\section{Naturality of the Casimir complex} \label{sec_nat}
%
% ===========================================================================

In the previous section we have shown that the Casimir complex commutes with generating 1-morphisms $\cal{E}\onen$, $\cal{F}\onen$ in $Com(\Ucat)$.  In this section we show that this commutativity is natural with respect to 2-morphisms.

% --------------------------------------------------------
%
\subsection{Natural transformations $\kappa$ and $\hat{\kappa}$}
%
% --------------------------------------------------------

Throughout this section we will find it convenient to view $Kom(\Ucat)$ and $Com(\Ucat)$ as idempotented additive monoidal categories as explained in the introduction.  Consider the complex
\begin{equation}
  \cal{C}  \;\; := \;\; \bigoplus_{n \in \Z} \cal{C}\onen.
\end{equation}
In this section we show that the functor
\begin{equation}
  -\otimes \cal{C} \maps  Com(\Ucat) \to Com(\Ucat),
\end{equation}
is naturally isomorphic to the functor
\begin{equation}
  \cal{C} \otimes - \maps Com(\Ucat) \to Com(\Ucat)
\end{equation}
via an invertible natural transformation
\begin{equation}
  \kappa \maps - \otimes \cal{C} \To \cal{C} \otimes -
\end{equation}
with inverse
\begin{equation}
  \hat{\kappa} \maps \cal{C} \otimes - \To - \otimes \cal{C}.
\end{equation}
Recall that the tensor product of complexes and juxtaposition of diagrams gives the composition operation in categories $Com(\Ucat)$ and $Com(\UcatD)$.  Here we will use composition notation rather than the tensor notation.

Defining the natural transformation $\kappa$ and its inverse
$\hat{\kappa}$ amounts to specifying for any complex $X$ in
$Com(\Ucat)$ a chain map
\begin{equation}
  \kappa_{X} \maps X \cal{C} \to \cal{C} X, \qquad
  \hat{\kappa}_{X} \maps \cal{C} X \to X \cal{C},
\end{equation}
such that for any chain map $f \maps X \to Y$ the squares
\begin{equation} \label{eq_nat_squares}
    \xy
   (-10,-10)*+{X\cal{C}}="tl";
   (10,-10)*+{\cal{C}X}="tr";
   (-10,10)*+{Y\cal{C}}="bl";
   (10,10)*+{\cal{C}Y}="br";
   {\ar_-{\kappa_{X}} "tl";"tr"};
   {\ar^{f\Ucas} "tl";"bl"};
   {\ar^{\kappa_{Y}} "bl";"br"};
   {\ar_{\Ucas f} "tr";"br"};
  \endxy
  \qquad
      \xy
   (-10,-10)*+{\cal{C}X}="tl";
   (10,-10)*+{X\cal{C}}="tr";
   (-10,10)*+{\cal{C}Y}="bl";
   (10,10)*+{Y\cal{C}}="br";
   {\ar_-{\hat{\kappa}_{X}} "tl";"tr"};
   {\ar^{\Ucas f} "tl";"bl"};
   {\ar^{\hat{\kappa}_{Y}} "bl";"br"};
   {\ar_{f\Ucas} "tr";"br"};
  \endxy
\end{equation}
commute (up to chain homotopy).

On generating 1-morphisms we define
\begin{align} \label{eq_kappa}
  \kappa_{\cal{F}\onen} &:=\gam \maps   \cal{F} \cal{C}\onen\to \cal{C} \cal{F}\onen &\hat{\kappa}_{\cal{F}\onen} &:= \delt \maps \cal{C}\cal{F}\onen \to  \cal{F}\cal{C}\onen
   \\ \label{eq_kappabar}
  \kappa_{\cal{E}\onen} &:= \odelt\maps   \cal{E}\cal{C}\onen \to \cal{C}\cal{E}\onen  &
  \hat{\kappa}_{\cal{E}\onen} &:= \ogam \maps \cal{C}\cal{E}\onen \to  \cal{E}\cal{C}\onen.
\end{align}
For an arbitrary complex $X$ in $Com(\Ucat)$ the chain maps
$\kappa_{X}$ and $\hat{\kappa}_{X}$ are determined from the
assignments above.  For example, if
$X=\cal{E}\cal{F}\cal{F}\onen\{t\}$ then
\begin{equation}
  \kappa_{X} =
 \xy
 (-58,0)*+{\cal{E}\cal{F}\cal{F}\cal{C}\onen\{t\} }="1";
 (-20,0)*+{\cal{E}\cal{F}\cal{C}\cal{F}\onen\{t\}}="2" ;  (20,0)*+{\cal{E}\cal{C}\cal{F}\cal{F}\onen\{t\}}="3";
 (58,0)*+{\cal{C}\cal{E}\cal{F}\cal{F}\onen\{t\} .}="4"; {\ar^-{\cal{E}\cal{F}\gam\{t\}} "1";"2"};
 {\ar^-{\cal{E}\gam\cal{F}\{t\}} "2";"3"};
 {\ar^-{\odelt\cal{F}\cal{F}\{t\}} "3";"4"};
 \endxy
\end{equation}
This definition of $\kappa_X$ produces a commutative diagram
\begin{equation}
  \xy
   (-20,15)*+{\cal{C}XY\onen}="l";
   (20,15)*+{XY\cal{C}\onen}="r";
   (0,0)*+{X\cal{C}Y\onen}="b";
   {\ar^-{\kappa_{XY}} "l";"r"};
   {\ar_{\kappa_X Y} "l";"b"};
   {\ar_{X\kappa_Y} "b";"r"};
  \endxy
\end{equation}
in $Com(\Ucat)$ for complexes $Y=\onenp Y\onen$ and $X=\onenpp X \onenp$, with $\cal{C}XY\onen = \onenpp\cal{C}\onenpp X\onenp Y\onen$.

\begin{prop}
Equations \eqref{eq_nat_squares} hold for all 2-morphisms in $Com(\Ucat)$.
\end{prop}

\begin{proof}
It is enough to check naturality squares \eqref{eq_nat_squares} on generating 2-morphisms (dots, crossing, cups, and caps) in $\Ucat$.   This will be done in Section~\ref{sec_nat_2morph}.
\end{proof}

It is clear from the definitions in \eqref{eq_kappa} and \eqref{eq_kappabar} of $\kappa$ and $\hat{\kappa}$ and the results in section~\ref{subsec_commutativity} that
\begin{equation}
  \hat{\kappa}_X\kappa_X = \Id_{\cal{C}X}, \qquad
  \kappa_X\hat{\kappa}_X = \Id_{X\cal{C}}
\end{equation}
in $Com(\Ucat)$, so that $\kappa$ and $\hat{\kappa}$ are inverse.

Naturality of $\kappa$ and $\hat{\kappa}$ and the universality of the Karoubian envelope allow us to extend $\kappa$ and $\hat{\kappa}$ to isomorphisms between functors
\begin{equation}
  -\otimes \cal{C}\; \maps \;  Kar(Com(\Ucat)) \to Kar(Com(\Ucat))
\end{equation}
and
\begin{equation}
  \cal{C} \otimes - \; \maps \; Kar(Com(\Ucat)) \to Kar(Com(\Ucat)).
\end{equation}
The equivalence $Kar(Com(\Ucat)) \cong Com(\UcatD)$ allows us to treat $\kappa$ and $\hat{\kappa}$ as isomorphisms
\begin{eqnarray}
  \kappa \maps - \otimes \cal{C} \To \cal{C} \otimes - \quad \text{and} \quad
  \hat{\kappa} \maps \cal{C} \otimes - \To - \otimes \cal{C} \nn
\end{eqnarray}
between endofunctors on $Com(\UcatD)$ concluding the proof of Theorem~\ref{thm_main}.

% --------------------------------------------------------
%
\subsection{Naturality with respect to 2-morphisms} \label{sec_nat_2morph}
%
% --------------------------------------------------------

It is immediate from the axioms of a 2-category and will be used throughout this section that for any $g \in \cal{G}$  the chain maps $\varrho^{g} \maps \cal{C} \to g(\cal{C})$ induce natural transformations between functors
\begin{align}
  \varrho^{g} \otimes - &\maps \cal{C} \otimes - \To g(\cal{C}) \otimes -, &
  - \otimes \varrho^{g} &\maps - \otimes \cal{C} \To - \otimes g(\cal{C}).
\end{align}

% --------------------------------------------------------
%
\subsubsection{Naturality of $\kappa$ for dot 2-morphism}
%
% --------------------------------------------------------

We will show that the diagram
\begin{equation} \label{eq_nat_dotF}
    \xy
   (-14,-10)*+{\cal{F}\cal{C}\onen}="tl";
   (14,-10)*+{\cal{C}\cal{F}\onen}="tr";
   (-14,10)*+{\cal{F}\cal{C}\onen}="bl";
   (14,10)*+{\cal{C}\cal{F}\onen}="br";
   {\ar_{\kappa_{\cal{F}\onen}} "tl";"tr"};
   {\ar^{\Udowndot\Ucas} "tl";"bl"};
   {\ar^{\kappa_{\cal{F}\onen}} "bl";"br"};
   {\ar_{\Ucas\Udowndot} "tr";"br"};
  \endxy
  \quad := \quad
    \xy
   (-10,-10)*+{\cal{F}\cal{C}\onen}="tl";
   (10,-10)*+{\cal{C}\cal{F}\onen}="tr";
   (-10,10)*+{\cal{F}\cal{C}\onen}="bl";
   (10,10)*+{\cal{C}\cal{F}\onen}="br";
   {\ar_{\gam} "tl";"tr"};
   {\ar^{\Udowndot\Ucas} "tl";"bl"};
   {\ar^{\gam} "bl";"br"};
   {\ar_{\Ucas\Udowndot} "tr";"br"};
  \endxy
\end{equation}
commutes up to homotopy (commutes in $Com(\Ucat)$).  Let
\begin{equation}
  (\gam)' = \gam \circ \big( \Udowndot\Ucas\big) - \big(\Ucas\Udowndot\big) \circ \gam.
\end{equation}
We construct a chain homotopy $(\gam)' \simeq 0$.
\begin{equation}
  \xy
  (-64,15)*+{
    \left(
    \begin{array}{c}
    \scs
    \E{}\F{}\cal{F} \onen \{2\} \\
    \scs  \cal{F}\onen \{1+n\} \\
    \end{array}
    \right)}="1t";
  (-5,15)*+{
    \und{\left(
    \begin{array}{c}
    \scs \E{}\F{}\cal{F} \onen \\
     \scs  \E{}\F{} \cal{F}\onen \\
    \end{array}
    \right)}}="3t";
  (64,15)*+{
    \left(\begin{array}{c}
    \scs \E{}\F{} \cal{F}\onen \{-2\} \\
     \scs  \cal{F}\onen \{n-3\} \\
    \end{array}
    \right) }="5t";
  {\ar^-{  \left(
    \begin{array}{cc}
      \text{$\Uupdot\Udown\Udown$} & \Ucupl\Udown \\ & \\
      \text{$\Uup\Udowndot\Udown $} & \Ucupl\Udown \\
    \end{array}
    \right)}   "1t";"3t"};
  {\ar^-{  \left(
    \begin{array}{cc}
      -\;\Uup\Udowndot\Udown  &  \text{$\Uupdot\Udown\Udown $}  \\ & \\
      \Ucapr\Udown & -\;\Ucapr\Udown\\
    \end{array}
  \right)} "3t";"5t"};
  (-64,-15)*+{
    \left( \begin{array}{c}
    \scs
    \cal{F}\E{}\F{} \onen \{2\} \\
    \scs  \cal{F}\onen \{1-n\} \\
   \end{array}
    \right)}="1";
 (-5,-15)*+{    \und{\left(
    \begin{array}{c}
    \scs \cal{F}\E{}\F{} \onen \\
     \scs  \cal{F}\E{}\F{} \onen \\
    \end{array}\right)}}="3";
 (64,-15)*+{  \left(
    \begin{array}{c}
    \scs \cal{F}\E{}\F{} \onen \{-2\} \\
     \scs  \cal{F}\onen \{n-1\} \\
    \end{array}\right) }="5";
 {\ar_-{   \left(
    \begin{array}{cc}
      \text{$\Udown\Uupdot\Udown$} & \Udown\Ucupl \\ & \\
      \text{$\Udown\Uup\Udowndot $} & \Udown\Ucupl \\
    \end{array}
    \right)} "1";"3"};
 {\ar_-{\left(
    \begin{array}{cc}
      -\;\Udown\Uup\Udowndot  &  \text{$\Udown\Uupdot\Udown $}  \\ & \\
      \Udown\Ucapr & -\;\Udown\Ucapr\\
     \end{array}
    \right)} "3";"5"};
 {\ar^{(\gamp)^{-1}} "1";"1t"};
 {\ar^{(\gamp)^{0}} "3";"3t"};
 {\ar^{(\gamp)^{1}} "5";"5t"};
 {\ar^{h^0} "3";"1t"};
 {\ar^{h^1} "5";"3t"};
 \endxy
\end{equation}
where, after simplifying the map $(\gam)'$, we have
\begin{align}
 (\gamp)^{-1}_{11} &= \;\;
 \vcenter{
 \xy 0;/r.13pc/:
    (-4,-4)*{};(-12,12)*{} **\crv{(-4,3) & (-12,5)}?(1)*\dir{>};
    (-12,-4)*{};(-4,12)*{} **\crv{(-12,3) & (-4,5)}?(0)*\dir{<};
    (4,-4)*{}; (4,12) **\dir{-}?(0)*\dir{<};
    (-5,0)*{\bullet};
\endxy}
 \;\; - \;\;
  \vcenter{
 \xy 0;/r.13pc/:
    (-4,-4)*{};(-12,12)*{} **\crv{(-4,3) & (-12,5)}?(1)*\dir{>};
    (-12,-4)*{};(-4,12)*{} **\crv{(-12,3) & (-4,5)}?(0)*\dir{<};
    (4,-4)*{}; (4,12) **\dir{-}?(0)*\dir{<};
    (4,4)*{\bullet};
\endxy}
  &  (\gamp)^{-1}_{12} &= 0 \\
  (\gamp)^{-1}_{21} &=
  \;\; \vcenter{
 \xy 0;/r.13pc/:
    (-4,-12)*{};(4,-12)*{} **\crv{(-4,-4) & (4,-4)}?(0)*\dir{<};
    (12,12)*{}; (12,-12) **\dir{-}?(1)*\dir{>};
    (0,-6)*{\bullet};
    %(18,-8)*{n};
\endxy}
 \;\; - \;\;
   \vcenter{
 \xy 0;/r.13pc/:
    (-4,-12)*{};(4,-12)*{} **\crv{(-4,-4) & (4,-4)}?(0)*\dir{<};
    (12,12)*{}; (12,-12) **\dir{-}?(1)*\dir{>};
    (12,0)*{\bullet};
    %(18,-8)*{n};
\endxy}
  &  (\gamp)^{-1}_{22} &= 0
\end{align}

\begin{align}
(\gamp)^{0}_{11} &=
  \;\; \vcenter{
 \xy 0;/r.13pc/:
    (-4,-4)*{};(-12,12)*{} **\crv{(-4,3) & (-12,5)}?(1)*\dir{>};
    (-12,-4)*{};(-4,12)*{} **\crv{(-12,3) & (-4,5)}?(0)*\dir{<};
    (4,-4)*{}; (4,12) **\dir{-}?(0)*\dir{<};
    (-5,0)*{\bullet};
\endxy}
 \;\; - \;\;
  \vcenter{
 \xy 0;/r.13pc/:
    (-4,-4)*{};(-12,12)*{} **\crv{(-4,3) & (-12,5)}?(1)*\dir{>};
    (-12,-4)*{};(-4,12)*{} **\crv{(-12,3) & (-4,5)}?(0)*\dir{<};
    (4,-4)*{}; (4,12) **\dir{-}?(0)*\dir{<};
    (4,4)*{\bullet};
\endxy}
  \\
 (\gamp)^{0}_{12} &= 0
  \\
  (\gamp)^{0}_{21} &=
  \;\; \vcenter{
 \xy 0;/r.13pc/:
    (-4,-4)*{};(-12,12)*{} **\crv{(-4,3) & (-12,5)}?(1)*\dir{>};
    (-12,-4)*{};(-4,12)*{} **\crv{(-12,3) & (-4,5)}?(0)*\dir{<};
    (4,-4)*{}; (4,12) **\dir{-}?(0)*\dir{<};
    (-11,0)*{\bullet};
\endxy}
 \;\; - \;\;
  \vcenter{
 \xy 0;/r.13pc/:
    (-4,-4)*{};(-12,12)*{} **\crv{(-4,3) & (-12,5)}?(1)*\dir{>};
    (-12,-4)*{};(-4,12)*{} **\crv{(-12,3) & (-4,5)}?(0)*\dir{<};
    (4,-4)*{}; (4,12) **\dir{-}?(0)*\dir{<};
    (4,4)*{\bullet};
\endxy}
 \;\; +\;\;
 \sum_{\xy (0,2)*{\scs f_1+f_2+f_3}; (0,-1)*{\scs =n-1}; \endxy}
\vcenter{  \xy 0;/r.13pc/:
  (-4,12)*{};(-12,-12)*{} **\crv{(-4,4) & (-12,-4)} ?(1)*\dir{>};
  (4,12)*{}="t1";
  (-12,12)*{}="t2";
  "t1";"t2" **\crv{(4,5) & (-12,5)};?(1)*\dir{>}?
   ?(.25)*\dir{}+(0,0)*{\bullet}+(1,-3.5)*{\scs f_1};
  (4,-12)*{}="t1";  (-4,-12)*{}="t2";
  "t2";"t1" **\crv{(-4,-5) & (4,-5)}; ?(1)*\dir{>}
    ?(.5)*\dir{}+(0,0)*{\bullet}+(1,3.5)*{\scs f_3};
  %(16,-10)*{n};
  (14,0)*{\ccbub{\spadesuit+f_2}{}};
  (-10,-4)*{\bullet};
  \endxy }
  \;\; -\;\;
 \sum_{\xy (0,2)*{\scs f_1+f_2+f_3}; (0,-1)*{\scs =n-1}; \endxy}
\vcenter{  \xy 0;/r.13pc/:
  (-4,12)*{};(-12,-12)*{} **\crv{(-4,4) & (-12,-4)} ?(1)*\dir{>};
  (4,12)*{}="t1";
  (-12,12)*{}="t2";
  "t1";"t2" **\crv{(4,5) & (-12,5)};?(1)*\dir{>}?
   ?(.25)*\dir{}+(0,0)*{\bullet}+(-1,-3.5)*{\scs f_1+1};
  (4,-12)*{}="t1";  (-4,-12)*{}="t2";
  "t2";"t1" **\crv{(-4,-5) & (4,-5)}; ?(1)*\dir{>}
    ?(.5)*\dir{}+(0,0)*{\bullet}+(1,3.5)*{\scs f_3};
  %(17,-10)*{n};
  (15,0)*{\ccbub{\spadesuit+f_2}{}};
  \endxy }
  \\
  (\gamp)^{0}_{22} &=
  \;\; \vcenter{
 \xy 0;/r.13pc/:
    (-4,-4)*{};(-12,12)*{} **\crv{(-4,3) & (-12,5)}?(1)*\dir{>};
    (-12,-4)*{};(-4,12)*{} **\crv{(-12,3) & (-4,5)}?(0)*\dir{<};
    (4,-4)*{}; (4,12) **\dir{-}?(0)*\dir{<};
    (-5,0)*{\bullet};
\endxy}
 \;\; -  \;\; \vcenter{
 \xy 0;/r.13pc/:
    (-4,-4)*{};(-12,12)*{} **\crv{(-4,3) & (-12,5)}?(1)*\dir{>};
    (-12,-4)*{};(-4,12)*{} **\crv{(-12,3) & (-4,5)}?(0)*\dir{<};
    (4,-4)*{}; (4,12) **\dir{-}?(0)*\dir{<};
    (-11,0)*{\bullet};
\endxy}
\;\; + \;\;
 \sum_{\xy (0,2)*{\scs f_1+f_2+f_3}; (0,-1)*{\scs =n-1}; \endxy}
\vcenter{  \xy 0;/r.13pc/:
  (-4,12)*{};(-12,-12)*{} **\crv{(-4,4) & (-12,-4)} ?(1)*\dir{>};
  (4,12)*{}="t1";
  (-12,12)*{}="t2";
  "t1";"t2" **\crv{(4,5) & (-12,5)};?(1)*\dir{>}?
   ?(.75)*\dir{}+(0,0)*{\bullet}+(-1,-3.5)*{\scs f_1};
  (4,-12)*{}="t1";  (-4,-12)*{}="t2";
  "t2";"t1" **\crv{(-4,-5) & (4,-5)}; ?(1)*\dir{>}
    ?(.5)*\dir{}+(0,0)*{\bullet}+(1,3.5)*{\scs f_3};
  %(16,-10)*{n};
  (14,0)*{\ccbub{\spadesuit+f_2}{}};
  (-9,-2)*{\bullet};
  \endxy }
  \;\; - \;\;
   \sum_{\xy (0,2)*{\scs f_1+f_2+f_3}; (0,-1)*{\scs =n-1}; \endxy}
\vcenter{  \xy 0;/r.13pc/:
  (-4,12)*{};(-12,-12)*{} **\crv{(-4,4) & (-12,-4)} ?(1)*\dir{>};
  (4,12)*{}="t1";
  (-12,12)*{}="t2";
  "t1";"t2" **\crv{(4,5) & (-12,5)};?(1)*\dir{>} ?
   ?(.75)*\dir{}+(0,0)*{\bullet}+(-1,-3.5)*{\scs f_1}
   ?(.25)*\dir{}+(0,0)*{\bullet};
  (4,-12)*{}="t1";  (-4,-12)*{}="t2";
  "t2";"t1" **\crv{(-4,-5) & (4,-5)}; ?(1)*\dir{>}
    ?(.5)*\dir{}+(0,0)*{\bullet}+(1,3.5)*{\scs f_3};
  %(16,-10)*{n};
  (14,0)*{\ccbub{\spadesuit+f_2}{}};
  \endxy }
\end{align}

\begin{align}
 (\gamp)^{1}_{11} &=
  \;\; \vcenter{
 \xy 0;/r.13pc/:
    (-4,-4)*{};(-12,12)*{} **\crv{(-4,3) & (-12,5)}?(1)*\dir{>};
    (-12,-4)*{};(-4,12)*{} **\crv{(-12,3) & (-4,5)}?(0)*\dir{<};
    (4,-4)*{}; (4,12) **\dir{-}?(0)*\dir{<};
    (-11,7.5)*{\bullet};
\endxy}
 \;\; -  \;\; \vcenter{
 \xy 0;/r.13pc/:
    (-4,-4)*{};(-12,12)*{} **\crv{(-4,3) & (-12,5)}?(1)*\dir{>};
    (-12,-4)*{};(-4,12)*{} **\crv{(-12,3) & (-4,5)}?(0)*\dir{<};
    (4,-4)*{}; (4,12) **\dir{-}?(0)*\dir{<};
    (-5,7.5)*{\bullet};
\endxy}
 \;\; + \;\;
 \sum_{\xy (0,2)*{\scs f_1+f_2+f_3}; (0,-1)*{\scs =n-1}; (0,-4)*{\scs 1 \leq f_1};\endxy}
\vcenter{  \xy 0;/r.13pc/:
  (-4,12)*{};(-12,-12)*{} **\crv{(-4,4) & (-12,-4)} ?(1)*\dir{>};
  (4,12)*{}="t1";
  (-12,12)*{}="t2";
  "t1";"t2" **\crv{(4,5) & (-12,5)};?(1)*\dir{>}?
   ?(.75)*\dir{}+(0,0)*{\bullet}+(-1,-3.5)*{\scs f_1};
  (4,-12)*{}="t1";  (-4,-12)*{}="t2";
  "t2";"t1" **\crv{(-4,-5) & (4,-5)}; ?(1)*\dir{>}
    ?(.5)*\dir{}+(0,0)*{\bullet}+(1,3.5)*{\scs f_3};
  %(16,-10)*{n};
  (14,0)*{\ccbub{\spadesuit+f_2}{}};
  (-9,-2)*{\bullet};
  \endxy }
  \;\; - \;\;
   \sum_{\xy (0,2)*{\scs f_1+f_2+f_3}; (0,-1)*{\scs =n-1};
   (0,-4)*{\scs 1 \leq f_1 }; \endxy}
\vcenter{  \xy 0;/r.13pc/:
  (-4,12)*{};(-12,-12)*{} **\crv{(-4,4) & (-12,-4)} ?(1)*\dir{>};
  (4,12)*{}="t1";
  (-12,12)*{}="t2";
  "t1";"t2" **\crv{(4,5) & (-12,5)};?(1)*\dir{>} ?
   ?(.75)*\dir{}+(0,0)*{\bullet}+(-1,-3.5)*{\scs f_1}
   ?(.25)*\dir{}+(0,0)*{\bullet};
  (4,-12)*{}="t1";  (-4,-12)*{}="t2";
  "t2";"t1" **\crv{(-4,-5) & (4,-5)}; ?(1)*\dir{>}
    ?(.5)*\dir{}+(0,0)*{\bullet}+(1,3.5)*{\scs f_3};
  %(16,-10)*{n};
  (14,0)*{\ccbub{\spadesuit+f_2}{}};
  \endxy }
\\
   (\gamp)^{1}_{12} &=
    \sum_{\xy (0,2)*{\scs f_1+f_2=n}; (0,-1)*{\scs 1 \leq f_1}; \endxy}
\vcenter{  \xy 0;/r.13pc/:
  (-4,12)*{};(-4,-12)*{} **\crv{(-4,4) & (-4,-4)} ?(1)*\dir{>};
  (4,12)*{}="t1";
  (-12,12)*{}="t2";
  "t1";"t2" **\crv{(4,5) & (-12,5)};?(1)*\dir{>}?
   ?(.75)*\dir{}+(0,0)*{\bullet}+(-1,-3.5)*{\scs f_1} ?(.25)*\dir{}+(0,0)*{\bullet};
  %(22,-9)*{n};
  (10,-3)*{\ccbub{\spadesuit+f_2}{}};
  \endxy }
  \;\; -\;\;
      \sum_{\xy (0,2)*{\scs f_1+f_2=n}; (0,-1)*{\scs 1 \leq f_1}; \endxy}
\vcenter{  \xy 0;/r.13pc/:
  (-4,12)*{};(-4,-12)*{} **\crv{(-4,4) & (-4,-4)} ?(1)*\dir{>};
  (4,12)*{}="t1";
  (-12,12)*{}="t2";
  "t1";"t2" **\crv{(4,5) & (-12,5)};?(1)*\dir{>}?
   ?(.75)*\dir{}+(0,0)*{\bullet}+(-1,-3.5)*{\scs f_1};
  %(22,-9)*{n};
  (10,-3)*{\ccbub{\spadesuit+f_2}{}};
  (-4,0)*{\bullet};
  \endxy }\\
  (\gamp)^{1}_{21}  &= (\gamp)^{1}_{22} = 0
\end{align}
and the chain homotopy is given by
\begin{align}
  \big(h^0\big)_{11} &= - \big(h^0\big)_{12} \;\;= \;\;
     \;\; \vcenter{
 \xy 0;/r.13pc/:
    (-4,-4)*{};(-12,12)*{} **\crv{(-4,3) & (-12,5)}?(1)*\dir{>};
    (-12,-4)*{};(-4,12)*{} **\crv{(-12,3) & (-4,5)}?(0)*\dir{<};
    (4,-4)*{}; (4,12) **\dir{-}?(0)*\dir{<};
\endxy} \;\; - \;\;   \sum_{\xy (0,2)*{\scs f_1+f_2+f_3}; (0,-1)*{\scs =n-1}; \endxy}
\vcenter{  \xy 0;/r.13pc/:
  (-4,12)*{};(-12,-12)*{} **\crv{(-4,4) & (-12,-4)} ?(1)*\dir{>};
  (4,12)*{}="t1";
  (-12,12)*{}="t2";
  "t1";"t2" **\crv{(4,5) & (-12,5)};?(1)*\dir{>} ?
   ?(.75)*\dir{}+(0,0)*{\bullet}+(-1,-3.5)*{\scs f_1};
  (4,-12)*{}="t1";  (-4,-12)*{}="t2";
  "t2";"t1" **\crv{(-4,-5) & (4,-5)}; ?(1)*\dir{>}
    ?(.5)*\dir{}+(0,0)*{\bullet}+(1,3.5)*{\scs f_3};
  %(16,-10)*{n};
  (14,0)*{\ccbub{\spadesuit+f_2}{}};
  \endxy }
  \\
  \big(h^0\big)_{21} &=  - \big(h^0\big)_{22} \;\;=\;\;
  - \;\; \vcenter{
 \xy 0;/r.13pc/:
    (-4,-12)*{};(4,-12)*{} **\crv{(-4,-4) & (4,-4)}?(0)*\dir{<};
    (12,12)*{}; (12,-12) **\dir{-}?(1)*\dir{>};
    %(18,-8)*{n};
\endxy}
\end{align}

\begin{align}
  \big(h^1\big)_{11} &=
   \;\; \vcenter{
 \xy 0;/r.13pc/:
    (-4,-4)*{};(-12,12)*{} **\crv{(-4,3) & (-12,5)}?(1)*\dir{>};
    (-12,-4)*{};(-4,12)*{} **\crv{(-12,3) & (-4,5)}?(0)*\dir{<};
    (4,-4)*{}; (4,12) **\dir{-}?(0)*\dir{<};
\endxy}
 \;\; - \;\;
 \sum_{\xy (0,2)*{\scs f_1+f_2+f_3}; (0,-1)*{\scs =n-1}; (0,-4)*{\scs 1 \leq f_1}; \endxy}
\vcenter{  \xy 0;/r.13pc/:
  (-4,12)*{};(-12,-12)*{} **\crv{(-4,4) & (-12,-4)} ?(1)*\dir{>};
  (4,12)*{}="t1";
  (-12,12)*{}="t2";
  "t1";"t2" **\crv{(4,5) & (-12,5)};?(1)*\dir{>} ?
   ?(.75)*\dir{}+(0,0)*{\bullet}+(-1,-3.5)*{\scs f_1};
  (4,-12)*{}="t1";  (-4,-12)*{}="t2";
  "t2";"t1" **\crv{(-4,-5) & (4,-5)}; ?(1)*\dir{>}
    ?(.5)*\dir{}+(0,0)*{\bullet}+(1,3.5)*{\scs f_3};
  %(16,-10)*{n};
  (14,0)*{\ccbub{\spadesuit+f_2}{}};
  \endxy }
  &   \big(h^1\big)_{12} &=
    \;\; \vcenter{
 \xy 0;/r.13pc/:
    (-4,-4)*{};(-12,12)*{} **\crv{(-4,3) & (-12,5)}?(1)*\dir{>};
    (-12,-4)*{};(-4,12)*{} **\crv{(-12,3) & (-4,5)}?(0)*\dir{<};
    (4,-4)*{}; (4,12) **\dir{-}?(0)*\dir{<};
\endxy}
 \;\; - \;\;
 \sum_{\xy (0,2)*{\scs f_1+f_2+f_3}; (0,-1)*{\scs =n-1}; (0,-4)*{\scs 1 \leq f_1}; \endxy}
\vcenter{  \xy 0;/r.13pc/:
  (-4,12)*{};(-12,-12)*{} **\crv{(-4,4) & (-12,-4)} ?(1)*\dir{>};
  (4,12)*{}="t1";
  (-12,12)*{}="t2";
  "t1";"t2" **\crv{(4,5) & (-12,5)};?(1)*\dir{>} ?
   ?(.75)*\dir{}+(0,0)*{\bullet}+(-1,-3.5)*{\scs f_1};
  (4,-12)*{}="t1";  (-4,-12)*{}="t2";
  "t2";"t1" **\crv{(-4,-5) & (4,-5)}; ?(1)*\dir{>}
    ?(.5)*\dir{}+(0,0)*{\bullet}+(1,3.5)*{\scs f_3};
  %(16,-10)*{n};
  (14,0)*{\ccbub{\spadesuit+f_2}{}};
  \endxy }  \\
   \big(h^1\big)_{21} &=
    \;\; \sum_{\xy (0,2)*{\scs f_1+f_2=n}; (0,-1)*{\scs 1 \leq f_1}; \endxy}
\vcenter{  \xy 0;/r.13pc/:
  (-4,12)*{};(-4,-12)*{} **\crv{(-4,4) & (-4,-4)} ?(1)*\dir{>};
  (4,12)*{}="t1";
  (-12,12)*{}="t2";
  "t1";"t2" **\crv{(4,5) & (-12,5)};?(1)*\dir{>}?
   ?(.75)*\dir{}+(0,0)*{\bullet}+(-1,-3.5)*{\scs f_1};
  %(22,-9)*{n};
  (10,-3)*{\ccbub{\spadesuit+f_2}{}};
  \endxy }
  &   \big(h^1\big)_{22} &= 0
\end{align}

The naturality square for the map $\kappa_{\cal{E}\onen}=\odelt$ follows from the naturality square in \eqref{eq_nat_dotF}:
\begin{equation} \label{eq_nat_dotE}
    \xy
   (-15,-10)*+{\cal{E}\tomega(\cal{C})\onen}="tl";
   (15,-10)*+{\tomega(\cal{C})\cal{E}\onen}="tr";
   (-15,10)*+{\cal{E}\tomega(\cal{C})\onen}="bl";
   (15,10)*+{\tomega(\cal{C})\cal{E}\onen}="br";
   (-35,-25)*+{\cal{E}\cal{C}\onen}="tl'";
   (35,-25)*+{\cal{C}\cal{E}\onen}="tr'";
   (-35,25)*+{\cal{E}\cal{C}\onen}="bl'";
   (35,25)*+{\cal{C}\cal{E}\onen}="br'";
   {\ar_{\tomega(\gam)} "tl";"tr"};
   {\ar^{\Uupdot\Ucas} "tl";"bl"};
   {\ar^{\tomega(\gam)} "bl";"br"};
   {\ar_{\Ucas\Uupdot} "tr";"br"};
   {\ar_{\odelt} "tl'";"tr'"};
   {\ar^{\Uupdot\Ucas} "tl'";"bl'"};
   {\ar^{\odelt} "bl'";"br'"};
   {\ar_{\Ucas\Uupdot} "tr'";"br'"};
   {\ar_{} "tl'";"tl"}; (-20,20)*{\UupD\varrho^{\tomega}};
   {\ar^{} "bl'";"bl"}; (-20,-20)*{\UupD\varrho^{\tomega}};
   {\ar_{} "tr";"tr'"}; (20,20)*{\hat{\varrho}^{\tomega}\UupD};
   (20,-20)*{\hat{\varrho}^{\tomega}\UupD};
   {\ar_{} "br";"br'"};
  \endxy
\end{equation}
where the middle square is the image of \eqref{eq_nat_dotF} under the 2-functor $\tomega$.  The left and right squares commute by the naturality of $\varrho^{\tomega}$ and $\hat{\varrho}^{\tomega}$, and the top and bottom squares commutes by Proposition~\ref{prop_omega}.

% --------------------------------------------------------
%
\subsubsection{Naturality of $\hat{\kappa}$ for dot 2-morphism}
%
% --------------------------------------------------------

We will show that the diagram
\begin{equation}
    \xy
   (-14,-10)*+{\cal{C}\cal{F}\onen}="tl";
   (14,-10)*+{\cal{F}\cal{C}\onen}="tr";
   (-14,10)*+{\cal{C}\cal{F}\onen}="bl";
   (14,10)*+{\cal{F}\cal{C}\onen}="br";
   {\ar_{\hat{\kappa}_{\cal{F}\onen}} "tl";"tr"};
   {\ar^{\Ucas\Udowndot} "tl";"bl"};
   {\ar^{\hat{\kappa}_{\cal{F}\onen}} "bl";"br"};
   {\ar_{\Udowndot\Ucas} "tr";"br"};
  \endxy
  \quad := \quad
    \xy
   (-10,-10)*+{\cal{C}\cal{F}\onen}="tl";
   (10,-10)*+{\cal{F}\cal{C}\onen}="tr";
   (-10,10)*+{\cal{C}\cal{F}\onen}="bl";
   (10,10)*+{\cal{F}\cal{C}\onen}="br";
   {\ar_{\delt} "tl";"tr"};
   {\ar^{\Ucas\Udowndot} "tl";"bl"};
   {\ar^{\delt} "bl";"br"};
   {\ar_{\Udowndot\Ucas} "tr";"br"};
  \endxy
\end{equation}
commutes up to homotopy.  To see this apply $\tsigma\tomega$ to the diagram \eqref{eq_nat_dotE}, and use the naturality of $\varrho^{\tsigma\tomega}$ and $\hat{\varrho}^{\tsigma\tomega}$, and Proposition~\ref{prop_sigmaomega} to see that each of the five small squares in \eqref{eq_5.31} commutes.
\begin{equation} \label{eq_5.31}
    \xy
   (-25,-10)*+{\tsigma\tomega(\cal{C})\cal{F}\onen}="tl";
   (25,-10)*+{\cal{F}\tsigma\tomega(\cal{C})\onen}="tr";
   (-25,10)*+{\tsigma\tomega(\cal{C})\cal{F}\onen}="bl";
   (25,10)*+{\cal{F}\tsigma\tomega(\cal{C})\onen}="br";
   (-55,-25)*+{\cal{C}\cal{F}\onen}="tl'";
   (55,-25)*+{\cal{F}\cal{C}\onen}="tr'";
   (-55,25)*+{\cal{C}\cal{F}\onen}="bl'";
   (55,25)*+{\cal{F}\cal{C}\onen}="br'";
   {\ar_{\tsigma\tomega(\odelt)} "tl";"tr"};
   {\ar^{\tsigma\tomega(\cal{C})\Udowndot} "tl";"bl"};
   {\ar^{\tsigma\tomega(\odelt)} "bl";"br"};
   {\ar_{\Udowndot\tsigma\tomega(\cal{C})} "tr";"br"};
   {\ar_{\delt} "tl'";"tr'"};
   {\ar^{\Ucas\Udowndot} "tl'";"bl'"};
   {\ar^{\delt} "bl'";"br'"};
   {\ar_{\Udowndot\Ucas} "tr'";"br'"};
   {\ar_{} "tl'";"tl"};
    (-34,20)*{\varrho^{\tsigma\tomega}\UdownD};
    (-34,-20)*{\varrho^{\tsigma\tomega}\UdownD};
   {\ar^{} "bl'";"bl"};
   {\ar_{} "tr";"tr'"};
   {\ar_{} "br";"br'"};
    (34,20)*{\UdownD\hat{\varrho}^{\tsigma\tomega}};
    (34,-20)*{\UdownD\hat{\varrho}^{\tsigma\tomega}};
  \endxy
\end{equation}

The naturality square
\begin{equation}
      \xy
   (-14,-10)*+{\cal{C}\cal{E}\onen}="tl";
   (14,-10)*+{\cal{E}\cal{C}\onen}="tr";
   (-14,10)*+{\cal{C}\cal{E}\onen}="bl";
   (14,10)*+{\cal{E}\cal{C}\onen}="br";
   {\ar_{\hat{\kappa}_{\cal{E}\onen}} "tl";"tr"};
   {\ar^{\Ucas\Uupdot} "tl";"bl"};
   {\ar^{\hat{\kappa}_{\cal{E}\onen}} "bl";"br"};
   {\ar_{\Uupdot\Ucas} "tr";"br"};
  \endxy
    \quad := \quad
    \xy
   (-10,-10)*+{\cal{C}\cal{E}\onen}="tl";
   (10,-10)*+{\cal{E}\cal{C}\onen}="tr";
   (-10,10)*+{\cal{C}\cal{E}\onen}="bl";
   (10,10)*+{\cal{E}\cal{C}\onen}="br";
   {\ar_{\ogam} "tl";"tr"};
   {\ar^{\Ucas\Udowndot} "tl";"bl"};
   {\ar^{\ogam} "bl";"br"};
   {\ar_{\Uupdot\Ucas} "tr";"br"};
  \endxy
\end{equation}
can similarly be shown to commute by applying $\tsigma\tomega$ to the square \eqref{eq_nat_dotF} and appealing to Proposition~\ref{prop_sigmaomega}.

% --------------------------------------------------------
%
\subsubsection{Naturality of $\kappa$ and $\hat{\kappa}$ for crossing 2-morphisms}
%
% --------------------------------------------------------

We will show that the diagram
\begin{equation} \label{eq_cross_nat}
     \xy
   (-14,-10)*+{\cal{F}\cal{F}\cal{C}\onen}="tl";
   (14,-10)*+{\cal{C}\cal{F}\cal{F}\onen}="tr";
   (-14,10)*+{\cal{F}\cal{F}\cal{C}\onen}="bl";
   (14,10)*+{\cal{C}\cal{F}\cal{F}\onen}="br";
   {\ar_{\kappa_{\cal{F}\cal{F}\onen}} "tl";"tr"};
   {\ar^{\Ucrossd\Ucas} "tl";"bl"};
   {\ar^{\kappa_{\cal{F}\cal{F}\onen}} "bl";"br"};
   {\ar_{\Ucas\Ucrossd} "tr";"br"};
  \endxy
  \quad := \quad
     \xy
   (-25,-10)*+{\cal{F}\cal{F}\cal{C}\onen}="tl";
   (0,-10)*+{\cal{F}\cal{C}\cal{F}\onen}="tm";
   (25,-10)*+{\cal{C}\cal{F}\cal{F}\onen}="tr";
   (-25,10)*+{\cal{F}\cal{F}\cal{C}\onen}="bl";
   (0,10)*+{\cal{F}\cal{C}\cal{F}\onen}="bm";
   (25,10)*+{\cal{C}\cal{F}\cal{F}\onen}="br";
   {\ar_-{\UdownD\gam} "tl";"tm"};
   {\ar_-{\gam\UdownD} "tm";"tr"};
   {\ar^{\Ucrossd\Ucas} "tl";"bl"};
   {\ar^{\UdownD\gam} "bl";"bm"};
   {\ar^{\gam\UdownD} "bm";"br"};
   {\ar_{\Ucas\Ucrossd} "tr";"br"};
  \endxy
\end{equation}
commutes for all $n \in \Z$ by treating the cases $n \leq 0$ and $n>0$ separately.  For $n \leq 0$ consider the map
\begin{equation}
  Z:= \big(\Ucas\Ucrossd \big) \circ \big(\gam\UdownD\big) \circ \big(\UdownD\gam\big)
 - \big(\gam\UdownD\big)\circ \big(\cal{F}\gam\big) \circ  \big(\Ucrossd\Ucas \big).
\end{equation}
One can check that the map $Z$ is identically zero, so that the diagram \eqref{eq_cross_nat} commutes on the nose for $n \leq 0$ (i.e. commutes in $Kom(\Ucat)$).

For $n>0$ we deduce commutativity from the commutativity of the diagram
\begin{equation}
    \xy
   (-45,-10)*+{\cal{F}\cal{F}\tsigma(\cal{C})\onen}="tl";
   (-15,-10)*+{\cal{F}\tsigma(\cal{C})\cal{F}\onen}="tm";
   (15,-10)*+{\cal{F}\tsigma(\cal{C})\cal{F}\onen}="tm2";
   (45,-10)*+{\tsigma(\cal{C})\cal{F}\cal{F}\onen}="tr";
   (-45,10)*+{\cal{F}\cal{F}\tsigma(\cal{C})\onen}="bl";
   (-15,10)*+{\cal{F}\tsigma(\cal{C})\cal{F}\onen}="bm";
    (15,10)*+{\cal{F}\tsigma(\cal{C})\cal{F}\onen}="bm2";
   (45,10)*+{\tsigma(\cal{C})\cal{F}\cal{F}\onen}="br";
   (-60,-25)*+{\cal{F}\cal{F}\cal{C}\onen}="tl'";
   (0,-25)*+{\cal{F}\cal{C}\cal{F}\onen}="tm'";
   (60,-25)*+{\cal{C}\cal{F}\cal{F}\onen}="tr'";
   (-60,25)*+{\cal{F}\cal{F}\cal{C}\onen}="bl'";
   (0,25)*+{\cal{F}\cal{F}\cal{C}\onen}="bm'";
   (60,25)*+{\cal{C}\cal{F}\cal{F}\onen}="br'";
   {\ar_-{\UdownD\delts} "tl";"tm"};
   {\ar_{\Ucrossd\tsigma(\cal{C})} "tl";"bl"};
   {\ar^{\UdownD\delts} "bl";"bm"};
   {\ar^-{} "bm";"bm'"};
   {\ar^-{} "bm'";"bm2"};
   {\ar_-{} "tm";"tm'"};
   {\ar_-{} "tm'";"tm2"};
   {\ar_-{\delts\UdownD} "tm2";"tr"};
   {\ar^-{\delts\UdownD} "bm2";"br"};
   {\ar^-{\Id} "bm";"bm2"};
   {\ar^-{\Id} "tm";"tm2"};
   {\ar^{\tsigma(\cal{C})\Ucrossd} "tr";"br"};
   {\ar_{\UdownD\gam} "tl'";"tm'"};
   {\ar_{\gam\UdownD} "tm'";"tr'"};
   {\ar^{\Ucrossd\Ucas} "tl'";"bl'"};
   {\ar^{\UdownD\gam} "bl'";"bm'"};
   {\ar^{\gam\UdownD} "bm'";"br'"};
   {\ar_{\Ucas\Ucrossd} "tr'";"br'"};
   {\ar_{} "tl'";"tl"};
   {\ar^{} "bl'";"bl"};
   {\ar_{} "tr";"tr'"};
   {\ar_{} "br";"br'"};
    (-46,20)*{\UdownD\UdownD\varrho^{\tsigma}};
    (-46,-20)*{\UdownD\UdownD\varrho^{\tsigma}};
    (-14,20)*{\UdownD\hat{\varrho}^{\tsigma}\UdownD};
    (-14,-20)*{\UdownD\hat{\varrho}^{\tsigma}\UdownD};
    (14,20)*{\UdownD\varrho^{\tsigma}\UdownD};
    (14,-20)*{\UdownD\varrho^{\tsigma}\UdownD};
    (46,20)*{\hat{\varrho}^{\tsigma}\UdownD\UdownD};
    (46,-20)*{\hat{\varrho}^{\tsigma}\UdownD\UdownD};
  \endxy
\end{equation}
where the center rectangle
\begin{equation}
    \xy
   (-35,-10)*+{\cal{F}\cal{F}\tsigma(\cal{C})\onen}="tl";
   (0,-10)*+{\cal{F}\tsigma(\cal{C})\cal{F}\onen}="tm";
   (35,-10)*+{\tsigma(\cal{C})\cal{F}\cal{F}\onen}="tr";
   (-35,10)*+{\cal{F}\cal{F}\tsigma(\cal{C})\onen}="bl";
   (0,10)*+{\cal{F}\tsigma(\cal{C})\cal{F}\onen}="bm";
   (35,10)*+{\tsigma(\cal{C})\cal{F}\cal{F}\onen}="br";
   {\ar_-{\UdownD\delts} "tl";"tm"};
   {\ar_-{\delts\UdownD} "tm";"tr"};
   {\ar^{\Ucrossd\tsigma(\cal{C})} "tl";"bl"};
   {\ar^{\UdownD\delts} "bl";"bm"};
   {\ar^{\delts\UdownD} "bm";"br"};
   {\ar_{\tsigma(\cal{C})\Ucrossd} "tr";"br"};
  \endxy
\end{equation}
commutes on the nose since the map (of complexes)
\begin{equation}
  Z':= \big(\tsigma(\cal{C})\Ucrossd \big) \circ \big(\delts\UdownD\big)\circ \big(\UdownD\delts\big)
 - \big(\delts\UdownD\big)\circ \big(\UdownD\delts\big) \circ  \big(\Ucrossd\tsigma(\cal{C}) \big).
\end{equation}
is already zero in $Kom(\Ucat)$ for $n>0$.  The triangles commute up to homotopy since $\varrho^{\tsigma}$ is the homotopy inverse of $\hat{\varrho}^{\tsigma}$ by Proposition~\ref{prop_homotopy_sym}.  The left and right squares follow from the naturality of $\varrho^{\tsigma}$, and the commutativity up to homotopy of the remaining squares is implied by \eqref{eq_gamma_sigma}.

For all $n\in \Z$ the naturality square for $\kappa_{\cal{E}\cal{E}\onen}$ and the crossing is proven as follows
\begin{equation}
    \xy
   (-45,-10)*+{\cal{E}\cal{E}\tomega(\cal{C})\onen}="tl";
   (-15,-10)*+{\cal{E}\tomega(\cal{C})\cal{E}\onen}="tm";
   (15,-10)*+{\cal{E}\tomega(\cal{C})\cal{E}\onen}="tm2";
   (45,-10)*+{\tomega(\cal{C})\cal{E}\cal{E}\onen}="tr";
   (-45,10)*+{\cal{E}\cal{E}\tomega(\cal{C})\onen}="bl";
   (-15,10)*+{\cal{E}\tomega(\cal{C})\cal{E}\onen}="bm";
    (15,10)*+{\cal{E}\tomega(\cal{C})\cal{E}\onen}="bm2";
   (45,10)*+{\tomega(\cal{C})\cal{E}\cal{E}\onen}="br";
   (-60,-25)*+{\cal{E}\cal{E}\cal{C}\onen}="tl'";
   (0,-25)*+{\cal{E}\cal{C}\cal{E}\onen}="tm'";
   (60,-25)*+{\cal{C}\cal{E}\cal{E}\onen}="tr'";
   (-60,25)*+{\cal{E}\cal{E}\cal{C}\onen}="bl'";
   (0,25)*+{\cal{E}\cal{C}\cal{E}\onen}="bm'";
   (60,25)*+{\cal{C}\cal{E}\cal{E}\onen}="br'";
   {\ar_-{\UupD\tomega(\gam)} "tl";"tm"};
   {\ar_-{\Id} "tm";"tm2"};
   {\ar_-{\tomega(\gam)\UupD} "tm2";"tr"};
   {\ar_-{\Ucrossu\;\tomega(\cal{C})} "tl";"bl"};
   {\ar^-{\UupD\tomega(\gam)} "bl";"bm"};
   {\ar^-{\Id} "bm";"bm2"};
   {\ar^-{\tomega(\gam)\UupD} "bm2";"br"};
   {\ar^-{\tomega(\cal{C})\;\Ucrossu} "tr";"br"};
   {\ar_-{\UupD\odelt} "tl'";"tm'"};
   {\ar_-{\odelt\UupD} "tm'";"tr'"};
   {\ar^-{\Ucrossu\;\Ucas} "tl'";"bl'"};
   {\ar^-{\UupD\odelt} "bl'";"bm'"};
   {\ar^-{\odelt\UupD} "bm'";"br'"};
   {\ar_-{\Ucas\;\Ucrossu} "tr'";"br'"};
   {\ar_-{} "tl'";"tl"};
   {\ar^-{} "bl'";"bl"};
   {\ar_-{} "tr";"tr'"};
   {\ar^-{} "br";"br'"};
   {\ar^-{} "bm";"bm'"};
   {\ar^-{} "bm'";"bm2"};
   {\ar_-{} "tm";"tm'"};
   {\ar_-{} "tm'";"tm2"};
   (-46,20)*{\UupD\UupD\varrho^{\tomega}};
    (-46,-20)*{\UupD\UupD\varrho^{\tomega}};
    (-14,20)*{\UupD\hat{\varrho}^{\tomega}\UupD};
    (-14,-20)*{\UupD\hat{\varrho}^{\tomega}\UupD};
    (14,20)*{\UupD\varrho^{\tomega}\UupD};
    (14,-20)*{\UupD\varrho^{\tomega}\UupD};
    (46,20)*{\hat{\varrho}^{\tomega}\UupD\UupD};
    (46,-20)*{\hat{\varrho}^{\tomega}\UupD\UupD};
  \endxy
\end{equation}
where the center rectangle is the image of \eqref{eq_cross_nat} under $\tomega$.  Each of the two triangles commutes up to homotopy since $\varrho^{\tomega}$ is the homotopy inverse of $\hat{\varrho}^{\tomega}$.  The squares on the left and right commute by the naturality of $\varrho^{\tomega}$ and $\hat{\varrho}^{\tomega}$.  The remaining squares commute by equation \eqref{prop_omega4} in Proposition~\ref{prop_omega}.

The naturality of $\hat{\kappa}$ and the crossing is established by applying the symmetry 2-functor $\tpsi$ to the above arguments.

% --------------------------------------------------------
%
\subsubsection{Naturality of $\kappa$ for cap 2-morphisms}
%
% --------------------------------------------------------

We show that the naturality square
\begin{equation} \label{eq_cup_nat_square2}
     \xy
   (-15,-10)*+{\cal{E}\cal{F}\cal{C}\onen}="tl";
   (15,-10)*+{\cal{C}\cal{E}\cal{F}\onen}="tr";
   (-15,10)*+{\cal{C}\onen}="bl";
   (15,10)*+{\cal{C}\onen}="br";
   {\ar_-{\kappa_{\cal{E}\cal{F}\onen}} "tl";"tr"};
   {\ar^{\Ucapr\Ucas} "tl";"bl"};
   {\ar^{\kappa_{\onen}} "bl";"br"};
   {\ar_{\Ucas\Ucapr} "tr";"br"};
  \endxy
  \quad := \quad   \xy
   (-25,-10)*+{\cal{E}\cal{F}\cal{C}\onen}="tl";
   (0,-10)*+{\cal{E}\cal{C}\cal{F}\onen}="tm";
   (25,-10)*+{\cal{C}\cal{E}\cal{F}\onen}="tr";
   (-25,10)*+{\cal{C}\onen}="bl";
   (25,10)*+{\cal{C}\onen}="br";
   {\ar_-{\UupD\gam} "tl";"tm"};
   {\ar_-{\odelt\UdownD} "tm";"tr"};
   {\ar^{\Ucapr\Ucas} "tl";"bl"};
   {\ar^{} "bl";"br"};(0,13)*{\Ucas};
   {\ar_{\Ucas\Ucapr} "tr";"br"};
  \endxy
\end{equation}
commutes up to homotopy by considering the cases $n> 0$ and $n \leq 0$ separately.

For $n \leq 0$ the complex $\cal{C}\onen$ is indecomposable. In this case, let
\begin{equation}
Z = \left(\Ucas\Ucapr \right)\circ \left(\odelt\UdownD\right) \circ
\left(\UupD \gam \right)- \left(\Ucapr \Ucas \right).
\end{equation}
We will specify a homotopy $Z \simeq 0$.
\begin{equation}
  \xy
  (-68,15)*+{
    \left(
    \begin{array}{c}
    \scs
    \E{}\F{} \onen \{2\} \\
    \scs  \onen \{1-n\} \\
    \end{array}
    \right)}="1t";
  (-5,15)*+{
    \und{\left(
    \begin{array}{c}
    \scs \E{}\F{} \onen \\
     \scs  \E{}\F{} \onen \\
    \end{array}
    \right)}}="3t";
  (68,15)*+{
    \left(\begin{array}{c}
    \scs \E{}\F{} \onen \{-2\} \\
     \scs  \onen \{n-1\} \\
    \end{array}
    \right) }="5t";
  {\ar^-{  \left(
    \begin{array}{cc}
      \text{$\Uupdot\Udown$} & \Ucupl \\ & \\
      \text{$\Uup\Udowndot$} & \Ucupl \\
    \end{array}
    \right)}   "1t";"3t"};
  {\ar^-{  \left(
    \begin{array}{cc}
      -\;\Uup\Udowndot  &  \text{$\Uupdot\Udown$}  \\ & \\
      \Ucapr & -\;\Ucapr\\
    \end{array}
  \right)} "3t";"5t"};
  (-68,-15)*+{
    \left( \begin{array}{c}
    \scs
    \cal{E}\cal{F}\E{}\F{} \onen \{2\} \\
    \scs  \cal{E}\cal{F}\onen \{1-n\} \\
   \end{array}
    \right)}="1";
 (-5,-15)*+{    \und{\left(
    \begin{array}{c}
    \scs \cal{E}\cal{F}\E{}\F{} \onen \\
     \scs  \cal{E}\cal{F}\E{}\F{} \onen \\
    \end{array}\right)}}="3";
 (68,-15)*+{  \left(
    \begin{array}{c}
    \scs \cal{E}\cal{F}\E{}\F{} \onen \{-2\} \\
     \scs  \cal{E}\cal{F}\onen \{n-1\} \\
    \end{array}\right) }="5";
 {\ar_-{   \left(
    \begin{array}{cc}
      \text{$\Uup\Udown\Uupdot\Udown$} & \text{$\Uup\Udown\Ucupl$} \\ & \\
      \text{$\Uup\Udown\Uup\Udowndot $} & \text{$\Uup\Udown\Ucupl$} \\
    \end{array}
    \right)} "1";"3"};
 {\ar_-{\left(
    \begin{array}{cc}
      -\;\Uup\Udown\Uup\Udowndot  &  \text{$\Uup\Udown\Uupdot\Udown $}  \\ & \\
      \text{$\Uup\Udown\Ucapr$} & -\;\Uup\Udown\Ucapr\\
     \end{array}
    \right)} "3";"5"};
 {\ar^{(Z)^{-1}} "1";"1t"};
 {\ar^{(Z)^{0}} "3";"3t"};
 {\ar^{(Z)^{1}} "5";"5t"};
 {\ar^{h^0} "3";"1t"};
 {\ar^{h^1} "5";"3t"};
 \endxy
\end{equation}
where, after simplifying,  $Z$ is given by
terms
\begin{align}
(Z)^{-1}_{11} &= \;\;
  \vcenter{
 \xy 0;/r.13pc/:
    (-4,12)*{};(-4,0)*{} **\crv{(-4,3) & (-4,-1)};
    (-12,12)*{};(-12,0)*{} **\crv{(-12,3) & (-12,-1)}?(0)*\dir{<};
    (-4,0)*{};(-12,-12)*{} **\crv{(-4,-3) & (-12,-9)}?(1)*\dir{>};
    (-12,0)*{};(-4,-12)*{} **\crv{(-12,-3) & (-4,-9)};
    (-20,-12)*{};(4,-12)*{} **\crv{(-20,6) & (4,6)}?(1)*\dir{>};
    (-8,1)*{\bullet}; (5,8)*{n};
\endxy}
 \;\; - \;\;
 \vcenter{
 \xy 0;/r.13pc/:
    (-4,12)*{};(-4,0)*{} **\crv{(-4,3) & (-4,-1)};
    (-12,12)*{};(-12,0)*{} **\crv{(-12,3) & (-12,-1)}?(0)*\dir{<};
    (-4,0)*{};(-12,-12)*{} **\crv{(-4,-3) & (-12,-9)}?(1)*\dir{>};
    (-12,0)*{};(-4,-12)*{} **\crv{(-12,-3) & (-4,-9)};
    (-20,-12)*{};(4,-12)*{} **\crv{(-20,6) & (4,6)}?(1)*\dir{>};
    (-10.5,-3)*{\bullet}; (5,8)*{n};
\endxy}
 \;\; - \;\;
 \vcenter{
 \xy 0;/r.13pc/:
    (-12,12)*{};(-4,-12)*{} **\crv{(-12,3) & (-4,-1)}?(0)*\dir{<};
    (-4,12)*{};(4,-12)*{} **\crv{(-4,3) & (4,-1)}?(1)*\dir{>};
    (-20,-12)*{};(-12,-12)*{} **\crv{(-20,-4) & (-12,-4)}?(1)*\dir{>};
     (5,8)*{n};
\endxy}
 \;\; + \;\;
 \vcenter{
 \xy 0;/r.13pc/:
    (-20,-12)*{};(4,-12)*{} **\crv{(-20,6) & (4,6)}?(1)*\dir{>};
    (-12,-12)*{};(-4,-12)*{} **\crv{(-12,-4) & (-4,-4)}?(1)*\dir{>};
    (-12,12)*{};(-4,12)*{} **\crv{(-12,4) & (-4,4)}?(0)*\dir{<};
    (5,8)*{n};
\endxy}
\\
(Z)^{-1}_{12} &=  0\\
 (Z)^{-1}_{21} &= \;\; -\;\;
  \vcenter{
 \xy 0;/r.13pc/:
    (-20,-12)*{};(4,-12)*{} **\crv{(-20,6) & (4,6)}?(1)*\dir{>};
    (-12,-12)*{};(-4,-12)*{} **\crv{(-12,-4) & (-4,-4)}?(0)*\dir{<};
    (-4,1)*{\bullet};
    (5,8)*{n};
\endxy}
 \\
 (Z)^{-1}_{22} &= \;\; - \;\;
   \vcenter{
 \xy 0;/r.13pc/:
    (-12,-12)*{};(-4,-12)*{} **\crv{(-12,-4) & (-4,-4)}?(1)*\dir{>};
    (5,4)*{n};
\endxy}
\end{align}

\begin{align}
 (Z)^{0}_{11} &= \;\;
  \vcenter{
 \xy 0;/r.13pc/:
    (-4,12)*{};(-4,0)*{} **\crv{(-4,3) & (-4,-1)};
    (-12,12)*{};(-12,0)*{} **\crv{(-12,3) & (-12,-1)}?(0)*\dir{<};
    (-4,0)*{};(-12,-12)*{} **\crv{(-4,-3) & (-12,-9)}?(1)*\dir{>};
    (-12,0)*{};(-4,-12)*{} **\crv{(-12,-3) & (-4,-9)};
    (-20,-12)*{};(4,-12)*{} **\crv{(-20,6) & (4,6)}?(1)*\dir{>};
    (-17.5,-3)*{\bullet}; (5,8)*{n};
\endxy}
 \;\; - \;\;
 \vcenter{
 \xy 0;/r.13pc/:
    (-4,12)*{};(-4,0)*{} **\crv{(-4,3) & (-4,-1)};
    (-12,12)*{};(-12,0)*{} **\crv{(-12,3) & (-12,-1)}?(0)*\dir{<};
    (-4,0)*{};(-12,-12)*{} **\crv{(-4,-3) & (-12,-9)}?(1)*\dir{>};
    (-12,0)*{};(-4,-12)*{} **\crv{(-12,-3) & (-4,-9)};
    (-20,-12)*{};(4,-12)*{} **\crv{(-20,6) & (4,6)}?(1)*\dir{>};
    (-10.5,-3)*{\bullet}; (5,8)*{n};
\endxy}
 \;\; - \;\;
 \vcenter{
 \xy 0;/r.13pc/:
    (-12,12)*{};(-4,-12)*{} **\crv{(-12,3) & (-4,-1)}?(0)*\dir{<};
    (-4,12)*{};(4,-12)*{} **\crv{(-4,3) & (4,-1)}?(1)*\dir{>};
    (-20,-12)*{};(-12,-12)*{} **\crv{(-20,-4) & (-12,-4)}?(1)*\dir{>};
     (5,8)*{n};
\endxy}
 \;\; - \;\;
 \vcenter{
 \xy 0;/r.13pc/:
    (-4,-4)*{};(4,4)*{} **\crv{(-4,-1) & (4,1)}?(0)*\dir{<};
    (4,-4)*{};(-4,4)*{} **\crv{(4,-1) & (-4,1)}?(1)*\dir{};
    (4,4)*{};(12,12)*{} **\crv{(4,7) & (12,9)}?(1)*\dir{};
    (12,4)*{};(4,12)*{} **\crv{(12,7) & (4,9)}?(1)*\dir{};
    (4,12)*{};(-4,12)*{} **\crv{(3,15) & (-4,15)}?(0)*\dir{};
    (-4,4)*{}; (-4,12) **\dir{-};
    (12,-4)*{}; (12,4) **\dir{-} ?(0)*\dir{<};
    (-12,-4)*{}; (-12,20) **\dir{-}?(1)*\dir{>};
    (12,12)*{}; (12,20) **\dir{-};
  (16,8)*{n};
\endxy}
\\
 (Z)^{0}_{12} &=  \;\;
  \vcenter{
 \xy 0;/r.13pc/:
    (-4,12)*{};(-4,0)*{} **\crv{(-4,3) & (-4,-1)};
    (-12,12)*{};(-12,0)*{} **\crv{(-12,3) & (-12,-1)}?(0)*\dir{<};
    (-4,0)*{};(-12,-12)*{} **\crv{(-4,-3) & (-12,-9)}?(1)*\dir{>};
    (-12,0)*{};(-4,-12)*{} **\crv{(-12,-3) & (-4,-9)};
    (-20,-12)*{};(4,-12)*{} **\crv{(-20,6) & (4,6)}?(1)*\dir{>};
    (-8,1)*{\bullet}; (5,8)*{n};
\endxy}
 \;\; - \;\;
  \vcenter{
 \xy 0;/r.13pc/:
    (-4,12)*{};(-4,0)*{} **\crv{(-4,3) & (-4,-1)};
    (-12,12)*{};(-12,0)*{} **\crv{(-12,3) & (-12,-1)}?(0)*\dir{<};
    (-4,0)*{};(-12,-12)*{} **\crv{(-4,-3) & (-12,-9)}?(1)*\dir{>};
    (-12,0)*{};(-4,-12)*{} **\crv{(-12,-3) & (-4,-9)};
    (-20,-12)*{};(4,-12)*{} **\crv{(-20,6) & (4,6)}?(1)*\dir{>};
    (-5,-3)*{\bullet}; (5,8)*{n};
\endxy}
\\
 (Z)^{0}_{21} &= 0
 \\
(Z)^{0}_{22} &= \;\;
  \vcenter{
 \xy 0;/r.13pc/:
    (-4,12)*{};(-4,0)*{} **\crv{(-4,3) & (-4,-1)};
    (-12,12)*{};(-12,0)*{} **\crv{(-12,3) & (-12,-1)}?(0)*\dir{<};
    (-4,0)*{};(-12,-12)*{} **\crv{(-4,-3) & (-12,-9)}?(1)*\dir{>};
    (-12,0)*{};(-4,-12)*{} **\crv{(-12,-3) & (-4,-9)};
    (-20,-12)*{};(4,-12)*{} **\crv{(-20,6) & (4,6)}?(1)*\dir{>};
    (-5,-3)*{\bullet}; (5,8)*{n};
\endxy}
 \;\; - \;\;
 \vcenter{
 \xy 0;/r.13pc/:
    (-4,12)*{};(-4,0)*{} **\crv{(-4,3) & (-4,-1)};
    (-12,12)*{};(-12,0)*{} **\crv{(-12,3) & (-12,-1)}?(0)*\dir{<};
    (-4,0)*{};(-12,-12)*{} **\crv{(-4,-3) & (-12,-9)}?(1)*\dir{>};
    (-12,0)*{};(-4,-12)*{} **\crv{(-12,-3) & (-4,-9)};
    (-20,-12)*{};(4,-12)*{} **\crv{(-20,6) & (4,6)}?(1)*\dir{>};
    (-10.5,-3)*{\bullet}; (5,8)*{n};
\endxy}
 \;\; - \;\;
 \vcenter{
 \xy 0;/r.13pc/:
    (-12,12)*{};(-4,-12)*{} **\crv{(-12,3) & (-4,-1)}?(0)*\dir{<};
    (-4,12)*{};(4,-12)*{} **\crv{(-4,3) & (4,-1)}?(1)*\dir{>};
    (-20,-12)*{};(-12,-12)*{} **\crv{(-20,-4) & (-12,-4)}?(1)*\dir{>};
     (5,8)*{n};
\endxy}
\end{align}

\begin{align}
(Z)^{1}_{11} &= \;\;
   \vcenter{
 \xy 0;/r.13pc/:
    (-4,12)*{};(-4,0)*{} **\crv{(-4,3) & (-4,-1)};
    (-12,12)*{};(-12,0)*{} **\crv{(-12,3) & (-12,-1)}?(0)*\dir{<};
    (-4,0)*{};(-12,-12)*{} **\crv{(-4,-3) & (-12,-9)}?(1)*\dir{>};
    (-12,0)*{};(-4,-12)*{} **\crv{(-12,-3) & (-4,-9)};
    (-20,-12)*{};(4,-12)*{} **\crv{(-20,6) & (4,6)}?(1)*\dir{>};
    (-5,-3)*{\bullet}; (5,8)*{n};
\endxy}
 \;\; - \;\;
 \vcenter{
 \xy 0;/r.13pc/:
    (-4,12)*{};(-4,0)*{} **\crv{(-4,3) & (-4,-1)};
    (-12,12)*{};(-12,0)*{} **\crv{(-12,3) & (-12,-1)}?(0)*\dir{<};
    (-4,0)*{};(-12,-12)*{} **\crv{(-4,-3) & (-12,-9)}?(1)*\dir{>};
    (-12,0)*{};(-4,-12)*{} **\crv{(-12,-3) & (-4,-9)};
    (-20,-12)*{};(4,-12)*{} **\crv{(-20,6) & (4,6)}?(1)*\dir{>};
    (-10.5,-3)*{\bullet}; (5,8)*{n};
\endxy}
 \;\; - \;\;
 \vcenter{
 \xy 0;/r.13pc/:
    (-12,12)*{};(-4,-12)*{} **\crv{(-12,3) & (-4,-1)}?(0)*\dir{<};
    (-4,12)*{};(4,-12)*{} **\crv{(-4,3) & (4,-1)}?(1)*\dir{>};
    (-20,-12)*{};(-12,-12)*{} **\crv{(-20,-4) & (-12,-4)}?(1)*\dir{>};
     (5,8)*{n};
\endxy}
 \;\; - \;\;
 \vcenter{
 \xy 0;/r.13pc/:
    (12,12)*{};(4,-12)*{} **\crv{(12,3) & (4,-1)}?(1)*\dir{>};
    (4,12)*{};(-4,-12)*{} **\crv{(4,3) & (-4,-1)}?(0)*\dir{<};
    (20,-12)*{};(12,-12)*{} **\crv{(20,-4) & (12,-4)}?(0)*\dir{<};
     (25,8)*{n};
\endxy}
\\ \nn \\
 (Z)^{1}_{12} &=
 \vcenter{
 \xy 0;/r.13pc/:
    (12,10)*{};(12,-10)*{} **\dir{-}?(1)*\dir{>};
    (4,10)*{};(4,-10)*{} **\dir{-}?(0)*\dir{<};
     (12,0)*{\bullet};(20,6)*{n};
\endxy}
\\
(Z)^{1}_{21} &= 0
 \\
(Z)^{1}_{22} &=  \;\; - \;\;
   \vcenter{
 \xy 0;/r.13pc/:
    (-12,-12)*{};(-4,-12)*{} **\crv{(-12,-4) & (-4,-4)}?(1)*\dir{>};
    (5,4)*{n};
\endxy}
\end{align}
and the chain homotopy is given by
\begin{align}
  \big(h^0\big)_{11} &= \;\; -
    \;\;\vcenter{
 \xy 0;/r.13pc/:
    (-4,12)*{};(-4,0)*{} **\crv{(-4,3) & (-4,-1)};
    (-12,12)*{};(-12,0)*{} **\crv{(-12,3) & (-12,-1)}?(0)*\dir{<};
    (-4,0)*{};(-12,-12)*{} **\crv{(-4,-3) & (-12,-9)}?(1)*\dir{>};
    (-12,0)*{};(-4,-12)*{} **\crv{(-12,-3) & (-4,-9)};
    (-20,-12)*{};(4,-12)*{} **\crv{(-20,6) & (4,6)}?(1)*\dir{>};
   (5,8)*{n};
\endxy}
\;\; -\;\;
\sum_{\xy (0,2)*{\scs g_1+g_2+g_3}; (0,-1)*{\scs =-n}; \endxy}
\vcenter{  \xy 0;/r.13pc/:
  (-14,12)*{}; (-14,-12)*{} **\dir{-} ?(0)*\dir{<};
   ?(.5)*\dir{}+(0,0)*{\bullet}+(-5,1)*{\scs g_1};
  (12,12)*{}; (12,-12)*{} **\dir{-} ?(1)*\dir{>};
   %?(.5)*\dir{}+(0,0)*{\bullet}+(5,1)*{\scs g_3};
  (-4,-12)*{}="t1";  (4,-12)*{}="t2";
  "t2";"t1" **\crv{(4,-5) & (-4,-5)}; ?(1)*\dir{>};
    ?(.75)*\dir{}+(0,0)*{\bullet}+(-3.5,1)*{\scs g_3};
  (20,-10)*{n};(-2,5)*{\cbub{\spadesuit+g_2}{}};
  \endxy }
  \\
  \big(h^0\big)_{12} &=   0
  \\
  \big(h^0\big)_{21} &=
   \;\;\vcenter{
 \xy 0;/r.13pc/:
    (-4,12)*{};(-4,0)*{} **\crv{(-4,3) & (-4,-1)};
    (-12,12)*{};(-12,0)*{} **\crv{(-12,3) & (-12,-1)}?(0)*\dir{<};
    (-4,0)*{};(-12,-12)*{} **\crv{(-4,-3) & (-12,-9)}?(1)*\dir{>};
    (-12,0)*{};(-4,-12)*{} **\crv{(-12,-3) & (-4,-9)};
    (-20,-12)*{};(4,-12)*{} **\crv{(-20,6) & (4,6)}?(1)*\dir{>};
     (5,8)*{n};
\endxy}
\;\; + \;\;
\sum_{\xy (0,2)*{\scs g_1+g_2+g_3}; (0,-1)*{\scs =-n}; \endxy}
\vcenter{  \xy 0;/r.13pc/:
  (-14,12)*{}; (-14,-12)*{} **\dir{-} ?(0)*\dir{<};
   ?(.5)*\dir{}+(0,0)*{\bullet}+(-4,1)*{\scs g_1};
  (12,12)*{}; (12,-12)*{} **\dir{-} ?(1)*\dir{>};
   ?(.5)*\dir{}+(0,0)*{\bullet}+(5,1)*{\scs g_3};
  (-4,-12)*{}="t1";  (4,-12)*{}="t2";
  "t2";"t1" **\crv{(4,-5) & (-4,-5)}; ?(1)*\dir{>};
    %?(.75)*\dir{}+(0,0)*{\bullet}+(-3.5,1)*{\scs g_2};
  (20,-10)*{n};(-2,5)*{\cbub{\spadesuit+g_2}{}};
  \endxy }
  \\ \big(h^0\big)_{22} &=  \;\; - \;\;
  \vcenter{
 \xy 0;/r.13pc/:
    (-20,-12)*{};(4,-12)*{} **\crv{(-20,6) & (4,6)}?(1)*\dir{>};
    (-12,-12)*{};(-4,-12)*{} **\crv{(-12,-4) & (-4,-4)}?(0)*\dir{<};
    (5,8)*{n};
\endxy}
\end{align}

\begin{align}
  \big(h^1\big)_{11} &= \;\; -
    \;\;\vcenter{
 \xy 0;/r.13pc/:
    (-4,12)*{};(-4,0)*{} **\crv{(-4,3) & (-4,-1)};
    (-12,12)*{};(-12,0)*{} **\crv{(-12,3) & (-12,-1)}?(0)*\dir{<};
    (-4,0)*{};(-12,-12)*{} **\crv{(-4,-3) & (-12,-9)}?(1)*\dir{>};
    (-12,0)*{};(-4,-12)*{} **\crv{(-12,-3) & (-4,-9)};
    (-20,-12)*{};(4,-12)*{} **\crv{(-20,6) & (4,6)}?(1)*\dir{>};
   (5,8)*{n};
\endxy}
\;\; -\;\;
\sum_{\xy (0,2)*{\scs g_1+g_2+g_3}; (0,-1)*{\scs =-n}; \endxy}
\vcenter{  \xy 0;/r.13pc/:
  (-14,12)*{}; (-14,-12)*{} **\dir{-} ?(0)*\dir{<};
   %?(.5)*\dir{}+(0,0)*{\bullet}+(-4,1)*{\scs g_1};
  (12,12)*{}; (12,-12)*{} **\dir{-} ?(1)*\dir{>};
   ?(.5)*\dir{}+(0,0)*{\bullet}+(5,1)*{\scs g_1};
  (-4,-12)*{}="t1";  (4,-12)*{}="t2";
  "t2";"t1" **\crv{(4,-5) & (-4,-5)}; ?(1)*\dir{>};
    ?(.75)*\dir{}+(0,0)*{\bullet}+(-5,1)*{\scs g_3};
  (20,-10)*{n};(-2,5)*{\cbub{\spadesuit+g_2}{}};
  \endxy }
  \\
  \big(h^1\big)_{12} &=
  \;\; - \;\;
 \vcenter{
 \xy 0;/r.13pc/:
    (12,10)*{};(12,-10)*{} **\dir{-}?(1)*\dir{>};
    (4,10)*{};(4,-10)*{} **\dir{-}?(0)*\dir{<};
     (25,8)*{n};
\endxy}
  \\
  \big(h^1\big)_{21} &=
  \;\;- \;\;\vcenter{
 \xy 0;/r.13pc/:
    (-4,12)*{};(-4,0)*{} **\crv{(-4,3) & (-4,-1)};
    (-12,12)*{};(-12,0)*{} **\crv{(-12,3) & (-12,-1)}?(0)*\dir{<};
    (-4,0)*{};(-12,-12)*{} **\crv{(-4,-3) & (-12,-9)}?(1)*\dir{>};
    (-12,0)*{};(-4,-12)*{} **\crv{(-12,-3) & (-4,-9)};
    (-20,-12)*{};(4,-12)*{} **\crv{(-20,6) & (4,6)}?(1)*\dir{>};
     (5,8)*{n};
\endxy}
\;\; - \;\;
\sum_{\xy (0,2)*{\scs g_1+g_2+g_3}; (0,-1)*{\scs =-n}; \endxy}
\vcenter{  \xy 0;/r.13pc/:
  (-14,12)*{}; (-14,-12)*{} **\dir{-} ?(0)*\dir{<};
   ?(.5)*\dir{}+(0,0)*{\bullet}+(-4,1)*{\scs g_1};
  (12,12)*{}; (12,-12)*{} **\dir{-} ?(1)*\dir{>};
   %?(.5)*\dir{}+(0,0)*{\bullet}+(5,1)*{\scs g_3};
  (-4,-12)*{}="t1";  (4,-12)*{}="t2";
  "t2";"t1" **\crv{(4,-5) & (-4,-5)}; ?(1)*\dir{>};
    ?(.75)*\dir{}+(0,0)*{\bullet}+(-5,1)*{\scs g_3};
  (20,-10)*{n};(-2,5)*{\cbub{\spadesuit+g_2}{}};
  \endxy }
  \\ \big(h^0\big)_{22} &=  \;\;
0.
\end{align}
It remains to prove the homotopy commutativity of \eqref{eq_cup_nat_square2} in the case $n >0$.  This follows from the commutativity of the diagram
\begin{equation}
    \xy
   (-35,-10)*+{\cal{E}\cal{F}\tsigma(\cal{C})\onen}="tl";
   (0,-10)*+{\cal{E}\tsigma(\cal{C})\cal{F}\onen}="tm";
   (35,-10)*+{\tsigma(\cal{C})\cal{E}\cal{F}\onen}="tr";
   (-35,10)*+{\tsigma(\cal{C})\onen}="bl";
   (35,10)*+{\tsigma(\cal{C})\onen}="br";
   (-55,-25)*+{\cal{E}\cal{F}\cal{C}\onen}="tl'";
   (0,-25)*+{\cal{E}\cal{C}\cal{F}\onen}="tm'";
   (55,-25)*+{\cal{C}\cal{E}\cal{F}\onen}="tr'";
   (-55,25)*+{\cal{C}\onen}="bl'";
   (55,25)*+{\cal{C}\onen}="br'";
   {\ar_-{\UupD\delts} "tl";"tm"};
   {\ar_-{\ogams\UdownD} "tm";"tr"};
   {\ar^{\Ucapr\tsigma(\Ucas)} "tl";"bl"};
   {\ar^{} "bl";"br"};
   {\ar_{\tsigma(\Ucas)\Ucapr} "tr";"br"};
   {\ar_{\UupD\gam} "tl'";"tm'"};
   {\ar_{\odelt\UdownD} "tm'";"tr'"};
   {\ar^{\Ucapr\Ucas} "tl'";"bl'"};
   {\ar^{} "bl'";"br'"};
   {\ar_{\Ucas\Ucapr} "tr'";"br'"};
   {\ar_{} "tl'";"tl"};
   {\ar^{\varrho^{\tsigma}} "bl'";"bl"};
   {\ar_{} "tr";"tr'"};
   {\ar_{\hat{\varrho}^{\tsigma}} "br";"br'"};
   {\ar_{\UupD\hat{\varrho}^{\tsigma}\UdownD} "tm";"tm'"};
    (-38,-20)*{\scs\UupD\UdownD\varrho^{\tsigma}};
    (38,-20)*{\scs\hat{\varrho}^{\tsigma}\UupD\UdownD};
    (0,29)*{\Ucas};
    (0,14)*{\scs \tsigma(\Ucas)};
  \endxy
\end{equation}The left and right squares commute (up to homotopy) by definition.  The top square commutes up to homotopy by Proposition~\ref{prop_homotopy_sym}.  The bottom two squares commute by Proposition~\ref{prop_sigma}.
Hence, naturality follows from the commutativity of the diagram below:
\begin{equation} \label{eq_cup_nat_square}
    \xy
   (-35,-10)*+{\cal{E}\cal{F}\tsigma(\cal{C})\onen}="tl";
   (0,-10)*+{\cal{E}\tsigma(\cal{C})\cal{F}\onen}="tm";
   (35,-10)*+{\tsigma(\cal{C})\cal{E}\cal{F}\onen}="tr";
   (-35,10)*+{\tsigma(\cal{C})\onen}="bl";
   (35,10)*+{\tsigma(\cal{C})\onen}="br";
   {\ar_-{\cal{E}\delts} "tl";"tm"};
   {\ar_-{\ogams\cal{F}} "tm";"tr"};
   {\ar^{\Ucapr\tsigma(\cal{C})} "tl";"bl"};
   {\ar^{\tsigma(\cal{C})} "bl";"br"};
   {\ar_{\tsigma(\cal{C})\Ucapr} "tr";"br"};
  \endxy
\end{equation}
when $n > 0$.

To prove the homotopy commutativity of \eqref{eq_cup_nat_square} let
\begin{equation}
Z' = \left(\tsigma(\cal{C})\Ucapr \right)\circ \left(\ogams\UdownD\right) \circ
\left(\UupD \delts\right) - \left(\Ucapr \tsigma(\cal{C}) \right).
\end{equation}
We will show that $Z'$ is homotopic to zero
\begin{equation}
  \xy
  (-68,15)*+{
    \left(
    \begin{array}{c}
    \scs
    \F{}\E{} \onen \{2\} \\
    \scs  \onen \{1+n\} \\
    \end{array}
    \right)}="1t";
  (-5,15)*+{
    \und{\left(
    \begin{array}{c}
    \scs \F{} \E{}\onen \\
     \scs  \F{}\E{} \onen \\
    \end{array}
    \right)}}="3t";
  (68,15)*+{
    \left(\begin{array}{c}
    \scs \F{}\E{} \onen \{-2\} \\
     \scs  \onen \{-n-1\} \\
    \end{array}
    \right) }="5t";
   {\ar^-{
  \left(
    \begin{array}{cc}
      \text{$\Udown\Uupdot$} & \Ucupr \\ & \\
      \text{$\Udowndot \Uup$} & \Ucupr \\
    \end{array}
  \right)
   } "1t";"3t"};
   {\ar^-{
  \left(
    \begin{array}{cc}
      -\;\Udowndot \Uup &  \text{$\Udown \Uupdot$}  \\ & \\
      \Ucapl & -\;\Ucapl\\
    \end{array}
  \right)
   } "3t";"5t"};
(-68,-15)*+{
    \left(
    \begin{array}{c}
    \scs
    \cal{E}\cal{F} \F{}\E{} \onen \{2\} \\
    \scs  \cal{E}\cal{F}\onen \{1+n\} \\
    \end{array}
    \right)}="1";
  (-5,-15)*+{
    \und{\left(
    \begin{array}{c}
    \scs \cal{E}\cal{F}\F{} \E{}\onen \\
     \scs  \cal{E}\cal{F}\F{}\E{} \onen \\
    \end{array}
    \right)}}="3";
  (68,-15)*+{
    \left(\begin{array}{c}
    \scs \cal{E}\cal{F}\F{}\E{} \onen \{-2\} \\
     \scs  \cal{E}\cal{F}\onen \{-n-1\} \\
    \end{array}
    \right) }="5";
   {\ar_-{
  \left(
    \begin{array}{cc}
      \text{$\Uup\Udown\Udown\Uupdot$} & \text{$\Uup\Udown\Ucupr$} \\ & \\
      \text{$\Uup\Udown\Udowndot \Uup$} & \text{$\Uup\Udown\Ucupr$} \\
    \end{array}
  \right)
   } "1";"3"};
   {\ar_-{
  \left(
    \begin{array}{cc}
      -\;\Uup\Udown\Udowndot \Uup &  \text{$\Uup\Udown\Udown \Uupdot$}  \\ & \\
      \text{$\Uup\Udown\Ucapl$} & -\;\Uup\Udown\Ucapl\\
    \end{array}
  \right)
   } "3";"5"};
 {\ar^{(Z')^{-1}} "1";"1t"};
 {\ar^{(Z')^{0}} "3";"3t"};
 {\ar^{(Z')^{1}} "5";"5t"};
 {\ar^{h^0} "3";"1t"};
 {\ar^{h^1} "5";"3t"};
 \endxy
\end{equation}
where the chain map $Z'$ is given by
\begin{align}
 (Z')^{-1}_{11} &= \;\;
\vcenter{  \xy 0;/r.13pc/:
  (-4,-12)*{};(4,12)*{} **\crv{(-4,-4) & (4,-4)} ?(0)*\dir{<};
  %?(.75)*\dir{}+(0,0)*{\bullet}+(3.5,1)*{\scs f_2};
  (4,-12)*{}="t1";  (-12,-12)*{}="t2";
  "t1";"t2" **\crv{(4,-2) & (-12,-2)};?(0)*\dir{<};
   %?(.75)*\dir{}+(0,0)*{\bullet}+(-1,3.5)*{\scs f_1};
   (12,-12)*{}; (12,12)*{} **\dir{-} ?(1)*\dir{>};
  (19,6)*{n}; (12,0)*{\bullet};
  \endxy }
 \;\; - \;\;
\vcenter{  \xy 0;/r.13pc/:
  (-4,-12)*{};(4,12)*{} **\crv{(-4,-4) & (4,-4)} ?(0)*\dir{<};
  %?(.75)*\dir{}+(0,0)*{\bullet}+(3.5,1)*{\scs f_2};
  (4,-12)*{}="t1";  (-12,-12)*{}="t2";
  "t1";"t2" **\crv{(4,-2) & (-12,-2)};?(0)*\dir{<};
   ?(.2)*\dir{}+(0,0)*{\bullet};
   (12,-12)*{}; (12,12)*{} **\dir{-} ?(1)*\dir{>};
  (19,6)*{n};
  \endxy }
\\
 (Z')^{-1}_{12} &=
 (Z')^{-1}_{21} =
(Z')^{-1}_{22} = \;\; 0.
\end{align}

\begin{align}
 (Z')^{0}_{11} &= \;\;
\vcenter{  \xy 0;/r.13pc/:
  (-4,-12)*{};(4,12)*{} **\crv{(-4,-4) & (4,-4)} ?(0)*\dir{<};
  %?(.75)*\dir{}+(0,0)*{\bullet}+(3.5,1)*{\scs f_2};
  (4,-12)*{}="t1";  (-12,-12)*{}="t2";
  "t1";"t2" **\crv{(4,-2) & (-12,-2)};?(0)*\dir{<};
   %?(.75)*\dir{}+(0,0)*{\bullet}+(-1,3.5)*{\scs f_1};
   (12,-12)*{}; (12,12)*{} **\dir{-} ?(1)*\dir{>};
  (19,6)*{n}; (12,0)*{\bullet};
  \endxy }
 \;\; - \;\;
\vcenter{  \xy 0;/r.13pc/:
  (-4,-12)*{};(4,12)*{} **\crv{(-4,-4) & (4,-4)} ?(0)*\dir{<};
  %?(.75)*\dir{}+(0,0)*{\bullet}+(3.5,1)*{\scs f_2};
  (4,-12)*{}="t1";  (-12,-12)*{}="t2";
  "t1";"t2" **\crv{(4,-2) & (-12,-2)};?(0)*\dir{<};
   ?(.2)*\dir{}+(0,0)*{\bullet};
   (12,-12)*{}; (12,12)*{} **\dir{-} ?(1)*\dir{>};
  (19,6)*{n};
  \endxy }
\\
 (Z')^{0}_{12} &= 0\\
 (Z')^{0}_{22} &= -(Z')^{0}_{21} =\;\;
 \vcenter{  \xy 0;/r.13pc/:
  (-4,-12)*{};(4,12)*{} **\crv{(-4,-4) & (4,-4)} ?(0)*\dir{<};
  %?(.75)*\dir{}+(0,0)*{\bullet}+(3.5,1)*{\scs f_2};
  (4,-12)*{}="t1";  (-12,-12)*{}="t2";
  "t1";"t2" **\crv{(4,-2) & (-12,-2)};?(0)*\dir{<};
   %?(.75)*\dir{}+(0,0)*{\bullet}+(-1,3.5)*{\scs f_1};
   (12,-12)*{}; (12,12)*{} **\dir{-} ?(1)*\dir{>};
  (19,6)*{n}; (12,0)*{\bullet};
  \endxy }
 \;\; - \;\;
\vcenter{  \xy 0;/r.13pc/:
  (-4,-12)*{};(4,12)*{} **\crv{(-4,-4) & (4,-4)} ?(0)*\dir{<}
  ?(.6)*\dir{}+(0,0)*{\bullet};
  %?(.75)*\dir{}+(0,0)*{\bullet}+(3.5,1)*{\scs f_2};
  (4,-12)*{}="t1";  (-12,-12)*{}="t2";
  "t1";"t2" **\crv{(4,-2) & (-12,-2)};?(0)*\dir{<};
   %?(.7)*\dir{}+(0,0)*{\bullet};
   (12,-12)*{}; (12,12)*{} **\dir{-} ?(1)*\dir{>};
  (19,6)*{n};
  \endxy }
\end{align}

\begin{align}
 (Z')^{1}_{11} &= \;\;
 \vcenter{  \xy 0;/r.13pc/:
  (-4,-12)*{};(4,12)*{} **\crv{(-4,-4) & (4,-4)} ?(0)*\dir{<};
  %?(.75)*\dir{}+(0,0)*{\bullet}+(3.5,1)*{\scs f_2};
  (4,-12)*{}="t1";  (-12,-12)*{}="t2";
  "t1";"t2" **\crv{(4,-2) & (-12,-2)};?(0)*\dir{<};
   %?(.75)*\dir{}+(0,0)*{\bullet}+(-1,3.5)*{\scs f_1};
   (12,-12)*{}; (12,12)*{} **\dir{-} ?(1)*\dir{>};
  (19,6)*{n}; (12,0)*{\bullet};
  \endxy }
 \;\; - \;\;
\vcenter{  \xy 0;/r.13pc/:
  (-4,-12)*{};(4,12)*{} **\crv{(-4,-4) & (4,-4)} ?(0)*\dir{<}
  ?(.6)*\dir{}+(0,0)*{\bullet};
  %?(.75)*\dir{}+(0,0)*{\bullet}+(3.5,1)*{\scs f_2};
  (4,-12)*{}="t1";  (-12,-12)*{}="t2";
  "t1";"t2" **\crv{(4,-2) & (-12,-2)};?(0)*\dir{<};
   %?(.7)*\dir{}+(0,0)*{\bullet};
   (12,-12)*{}; (12,12)*{} **\dir{-} ?(1)*\dir{>};
  (19,6)*{n};
  \endxy }
\\
 (Z')^{1}_{12} &=
 (Z')^{1}_{22} =
 (Z')^{0}_{11} = \;\;0
\end{align}
and the chain homotopy is given by
\begin{align}
  \big(h^0\big)_{11} &= -\big(h^0\big)_{12} = \;\;
\vcenter{  \xy 0;/r.13pc/:
  (-4,-12)*{};(4,12)*{} **\crv{(-4,-4) & (4,-4)} ?(0)*\dir{<};
  %?(.75)*\dir{}+(0,0)*{\bullet}+(3.5,1)*{\scs f_2};
  (4,-12)*{}="t1";  (-12,-12)*{}="t2";
  "t1";"t2" **\crv{(4,-2) & (-12,-2)};?(0)*\dir{<};
   %?(.7)*\dir{}+(0,0)*{\bullet};
   (12,-12)*{}; (12,12)*{} **\dir{-} ?(1)*\dir{>};
  (19,6)*{n};
  \endxy }
  \\
  \big(h^0\big)_{21} &=
  \big(h^0\big)_{22} =  \;\; 0
\end{align}
and
\begin{align}
  \big(h^1\big)_{11} &= \big(h^1\big)_{21} = \;\;
\vcenter{  \xy 0;/r.13pc/:
  (-4,-12)*{};(4,12)*{} **\crv{(-4,-4) & (4,-4)} ?(0)*\dir{<};
  %?(.75)*\dir{}+(0,0)*{\bullet}+(3.5,1)*{\scs f_2};
  (4,-12)*{}="t1";  (-12,-12)*{}="t2";
  "t1";"t2" **\crv{(4,-2) & (-12,-2)};?(0)*\dir{<};
   %?(.7)*\dir{}+(0,0)*{\bullet};
   (12,-12)*{}; (12,12)*{} **\dir{-} ?(1)*\dir{>};
  (19,6)*{n};
  \endxy }
  \\
  \big(h^1\big)_{12} &=
  \big(h^1\big)_{22} =  \;\; 0.
\end{align}

Naturality of the other cup follows from the homotopy commutative diagram below:
\begin{equation} \label{eq_nat_cup2}
    \xy
   (-50,-10)*+{\cal{F}\cal{E}\tomega(\cal{C})\onen}="tl";
   (-14,-10)*+{\cal{F}\tomega(\cal{C})\cal{E}\onen}="tm";
   (14,-10)*+{\cal{F}\tomega(\cal{C})\cal{E}\onen}="tm'";
   (50,-10)*+{\tomega(\cal{C})\cal{F}\cal{E}\onen}="tr";
   (-50,10)*+{\tomega(\cal{C})\onen}="bl";
   (50,10)*+{\tomega(\cal{C})\onen}="br";
   (-65,-25)*+{\cal{F}\cal{E}\cal{C}\onen}="tl2";
   (0,-25)*+{\cal{F}\cal{C}\cal{E}\onen}="tm2";
   (65,-25)*+{\cal{C}\cal{F}\cal{E}\onen}="tr2";
   (-65,25)*+{\cal{C}\onen}="bl2";
   (65,25)*+{\cal{C}\onen}="br2";
   {\ar_-{\UdownD\tomega(\gam\onenn{-n})} "tl";"tm"};
   {\ar_-{\Id} "tm";"tm'"};
   {\ar_-{\tomega(\odelt\onenn{-n})\UupD} "tm'";"tr"};
   {\ar_{\Ucapl\tomega(\Ucas)} "tl";"bl"};
   {\ar^{} "bl";"br"};
   {\ar^{\tomega(\Ucas)\Ucapl} "tr";"br"};
   {\ar_{} "tr";"tr2"};
   {\ar_{\hat{\varrho}^{\tomega}} "br";"br2"};
   {\ar_{\UdownD\odelt} "tl2";"tm2"};
   {\ar_{\gam\UupD} "tm2";"tr2"};
   {\ar^{\Ucapl\Ucas} "tl2";"bl2"};
   {\ar^{} "bl2";"br2"};
   {\ar_{\Ucas\Ucapl} "tr2";"br2"};
   {\ar_{} "tm";"tm2"};
   {\ar_{} "tm2";"tm'"};
   {\ar_{} "tl2";"tl"};
   {\ar_{\varrho^{\tomega}} "bl2";"bl"};
        (0,14)*{\scs \tomega(\Ucas)};
        (0,29)*{\scs\Ucas};
        (-46,-20)*{\scs\UdownD\UupD\varrho^{\tomega}};
        (-14,-20)*{\scs\UdownD\varrho^{\tomega}\UupD};
        (14,-20)*{\scs\UdownD\hat{\varrho}^{\tomega}\UupD};
        (46,-20)*{\scs\hat{\varrho}^{\tomega}\UdownD\UupD};
  \endxy
\end{equation}
The center rectangle follows by applying the symmetry $\tomega$ to the homotopy commutative square \eqref{eq_cup_nat_square2} and replacing $n$ by $-n$.  The two squares on the left and right commute by naturality of $\varrho^{\tomega}$, $\hat{\varrho}^{\tomega}$.  The top square and the bottom middle triangle are commutative since $\varrho^{\tomega}$ is the homotopy inverse of $\hat{\varrho}^{\tomega}$ as shown in Proposition~\ref{prop_homotopy_sym}.  The bottom two squares follow from equations \eqref{prop_omega1} and \eqref{prop_omega4} in Proposition~\ref{prop_omega}.

% --------------------------------------------------------
%
\subsubsection{Naturality of $\hat{\kappa}$ for cap 2-morphisms}
%
% --------------------------------------------------------

We show that the squares
\begin{equation} \label{eq_nat_barkappa_cup1}
    \xy
   (-14,-10)*+{\cal{C}\cal{E}\cal{F}\onen}="tl";
   (14,-10)*+{\cal{E}\cal{F}\cal{C}\onen}="tr";
   (-14,10)*+{\cal{C}\onen}="bl";
   (14,10)*+{\cal{C}\onen}="br";
   {\ar_{\hat{\kappa}_{\cal{E}\cal{F}\onen}} "tl";"tr"};
   {\ar^{\Ucas\Ucapr} "tl";"bl"};
   {\ar^{\hat{\kappa}_{\onen}} "bl";"br"};
   {\ar_{\Ucapr\Ucas} "tr";"br"};
  \endxy
\qquad
    \xy
   (-14,-10)*+{\cal{C}\cal{F}\cal{E}\onen}="tl";
   (14,-10)*+{\cal{F}\cal{E}\cal{C}\onen}="tr";
   (-14,10)*+{\cal{C}\onen}="bl";
   (14,10)*+{\cal{C}\onen}="br";
   {\ar_{\hat{\kappa}_{\cal{F}\cal{E}\onen}} "tl";"tr"};
   {\ar^{\Ucas\Ucapl} "tl";"bl"};
   {\ar^{\hat{\kappa}_{\onen}} "bl";"br"};
   {\ar_{\Ucapl\Ucas} "tr";"br"};
  \endxy
\end{equation}
commute up to homotopy.  To show that the first square is homotopy commutative consider the diagram below:
\begin{equation}
    \xy
   (-48,-10)*+{\tsigma\tomega(\cal{C})\cal{E}\cal{F}\onen}="tl";
   (-15,-10)*+{\cal{E}\tsigma\tomega(\cal{C})\cal{F}\onen}="tm";
   (15,-10)*+{\cal{E}\tsigma\tomega(\cal{C})\cal{F}\onen}="tm2";
   (48,-10)*+{\cal{E}\cal{F}\tsigma\tomega(\cal{C})\onen}="tr";
   (-48,10)*+{\tsigma\tomega(\cal{C})\onen}="bl";
   (48,10)*+{\tsigma\tomega(\cal{C})\onen}="br";
   (-60,-25)*+{\cal{C}\cal{E}\cal{F}\onen}="tl'";
   (0,-25)*+{\cal{E}\cal{C}\cal{F}\onen}="tm'";
   (60,-25)*+{\cal{E}\cal{F}\cal{C}\onen}="tr'";
   (-60,25)*+{\cal{C}\onen}="bl'";
   (60,25)*+{\cal{C}\onen}="br'";
   {\ar_-{\tsigma\tomega(\gam)\UdownD} "tl";"tm"};
   {\ar^-{\Id} "tm";"tm2"};
   {\ar_-{\UupD\tsigma\tomega(\odelt)} "tm2";"tr"};
   {\ar_{\Ucas\Ucapr} "tl";"bl"};
   {\ar^{} "bl";"br"};
   {\ar^{\Ucapr\Ucas} "tr";"br"};
   {\ar_{\ogam\UdownD} "tl'";"tm'"};
   {\ar_{\UupD\delt} "tm'";"tr'"};
   {\ar^{\Ucas\Ucapr} "tl'";"bl'"};
   {\ar^{} "bl'";"br'"};
   {\ar_{\Ucapr\Ucas} "tr'";"br'"};
   {\ar_{} "tl'";"tl"};
   {\ar^{\varrho^{\tsigma\tomega}} "bl'";"bl"};
   {\ar_{} "tr";"tr'"};
   {\ar_{\hat{\varrho}^{\tsigma\tomega}} "br";"br'"};
   {\ar_{} "tm";"tm'"};
   {\ar_{} "tm'";"tm2"};
    (0,14)*{\scs \tsigma\tomega(\Ucas)};
    (0,29)*{\Ucas};
    (-46,-20)*{\scs \varrho^{\tsigma\tomega}\UupD\UdownD};
    (-14,-20)*{\scs \UupD\hat{\varrho}^{\tsigma\tomega}\UdownD};
    (14,-20)*{\scs \UupD\varrho^{\tsigma\tomega}\UdownD};
    (46,-20)*{\scs \UupD\UdownD\hat{\varrho}^{\tsigma\tomega}};
  \endxy
\end{equation}
The center rectangle is obtained by applying $\tsigma\tomega$ to the homotopy commutative square \eqref{eq_cup_nat_square2}.  The triangle commutes since $\varrho^{\tsigma\tomega}$ has inverse $\hat{\varrho}^{\tsigma\tomega}$ by Proposition~\ref{prop_homotopy_sym}.  The bottom two squares were shown to commute in Proposition~\ref{prop_sigmaomega}. The remaining squares commute up to homotopy by the naturality of $\varrho^{\tsigma\tomega}$ and Propositions~\ref{prop_homotopy_sym}.

The second naturality square in \eqref{eq_nat_barkappa_cup1} can be shown to commute by applying $\tsigma\tomega$ to \eqref{eq_nat_cup2} and arguing as above.

% --------------------------------------------------------
%
\subsubsection{Naturality of $\kappa$ for the cup 2-morphisms}
%
% --------------------------------------------------------

We must show that the squares
\begin{equation}
    \xy
   (-14,-10)*+{\cal{C}\onen}="tl";
   (14,-10)*+{\cal{C}\onen}="tr";
   (-14,10)*+{\cal{E}\cal{F}\cal{C}\onen}="bl";
   (14,10)*+{\cal{C}\cal{E}\cal{F}\onen}="br";
   {\ar_{\hat{\kappa}_{\onen}} "tl";"tr"};
   {\ar^{\Ucupr\Ucas} "tl";"bl"};
   {\ar^{\hat{\kappa}_{\cal{E}\cal{F}\onen}} "bl";"br"};
   {\ar_{\Ucas\Ucupr} "tr";"br"};
  \endxy
\qquad
    \xy
   (-14,-10)*+{\cal{C}\onen}="tl";
   (14,-10)*+{\cal{C}\onen}="tr";
   (-14,10)*+{\cal{F}\cal{E}\cal{C}\onen}="bl";
   (14,10)*+{\cal{C}\cal{F}\cal{E}\onen}="br";
   {\ar_{\hat{\kappa}_{\onen}} "tl";"tr"};
   {\ar^{\Ucupr\Ucas} "tl";"bl"};
   {\ar^{\hat{\kappa}_{\cal{F}\cal{E}\onen}} "bl";"br"};
   {\ar_{\Ucas\Ucupr} "tr";"br"};
  \endxy
\end{equation}
commute up to homotopy.  The proof is given by the following diagrams:
\begin{equation}
    \xy
   (-48,10)*+{\cal{E}\cal{F}\tpsi(\cal{C})\onen}="tl";
   (-15,10)*+{\cal{E}\tpsi(\cal{C})\cal{F}\onen}="tm";
   (15,10)*+{\cal{E}\tpsi(\cal{C})\cal{F}\onen}="tm2";
   (48,10)*+{\tpsi(\cal{C})\cal{E}\cal{F}\onen}="tr";
   (-48,-10)*+{\tpsi(\cal{C})\onen}="bl";
   (48,-10)*+{\tpsi(\cal{C})\onen}="br";
   (-60,25)*+{\cal{E}\cal{F}\cal{C}\onen}="tl'";
   (0,25)*+{\cal{E}\cal{C}\cal{F}\onen}="tm'";
   (60,25)*+{\cal{C}\cal{E}\cal{F}\onen}="tr'";
   (-60,-25)*+{\cal{C}\onen}="bl'";
   (60,-25)*+{\cal{C}\onen}="br'";
   {\ar^-{\UupD\tpsi(\delt)} "tl";"tm"};
   {\ar^-{\Id} "tm";"tm2"};
   {\ar^-{\tpsi(\ogam)\UdownD} "tm2";"tr"};
   {\ar_{\Ucupl\tpsi(\Ucas)} "bl";"tl"};
   {\ar^{} "bl";"br"};
   {\ar^{\tpsi(\Ucas)\Ucupl} "br";"tr"};
   {\ar^{\UupD\gam} "tl'";"tm'"};
   {\ar^{\odelt\UdownD} "tm'";"tr'"};
   {\ar^{\Ucupl\Ucas} "bl'";"tl'"};
   {\ar^{} "bl'";"br'"};
   {\ar_{\Ucas\Ucupl} "br'";"tr'"};
   {\ar^{} "tl'";"tl"};
   {\ar_{\varrho^{\tpsi}} "bl'";"bl"};
   {\ar^{} "tr";"tr'"};
   {\ar_{\hat{\varrho}^{\tpsi}} "br";"br'"};
   {\ar^{} "tm";"tm'"};
   {\ar^{} "tm'";"tm2"};
    (0,-14)*{\scs \tpsi(\Ucas)};
    (0,-31)*{\scs \Ucas};
    (-46,20)*{\scs\UupD\UdownD\varrho^{\tpsi}};
    (-14,20)*{\scs\UupD\hat{\varrho}^{\tpsi}\UdownD};
    (14,20)*{\scs\UupD\varrho^{\tpsi}\UdownD};
    (46,20)*{\scs\hat{\varrho}^{\tpsi}\UupD\UdownD};
  \endxy
\end{equation}
\begin{equation}
    \xy
   (-48,10)*+{\cal{F}\cal{E}\tpsi(\cal{C})\onen}="tl";
   (-15,10)*+{\cal{F}\tpsi(\cal{C})\cal{E}\onen}="tm";
   (15,10)*+{\cal{F}\tpsi(\cal{C})\cal{E}\onen}="tm2";
   (48,10)*+{\tpsi(\cal{C})\cal{F}\cal{E}\onen}="tr";
   (-48,-10)*+{\tpsi(\cal{C})\onen}="bl";
   (48,-10)*+{\tpsi(\cal{C})\onen}="br";
   (-60,25)*+{\cal{F}\cal{E}\cal{C}\onen}="tl'";
   (0,25)*+{\cal{F}\cal{C}\cal{E}\onen}="tm'";
   (60,25)*+{\cal{C}\cal{F}\cal{E}\onen}="tr'";
   (-60,-25)*+{\cal{C}\onen}="bl'";
   (60,-25)*+{\cal{C}\onen}="br'";
   {\ar^-{\UdownD\tpsi(\ogam)} "tl";"tm"};
   {\ar^-{\Id} "tm";"tm2"};
   {\ar^-{\tpsi(\delt)\UupD} "tm2";"tr"};
   {\ar_{\Ucupr\tpsi(\Ucas)} "bl";"tl"};
   {\ar^{} "bl";"br"};
   {\ar^{\tpsi(\Ucas)\Ucupr} "br";"tr"};
   {\ar^{\UdownD\odelt} "tl'";"tm'"};
   {\ar^{\gam\UdownD} "tm'";"tr'"};
   {\ar^{\Ucupr\Ucas} "bl'";"tl'"};
   {\ar^{} "bl'";"br'"};
   {\ar_{\Ucas\Ucupr} "br'";"tr'"};
   {\ar^{} "tl'";"tl"};
   {\ar_{\varrho^{\tpsi}} "bl'";"bl"};
   {\ar^{} "tr";"tr'"};
   {\ar_{\hat{\varrho}^{\tpsi}} "br";"br'"};
   {\ar^{} "tm";"tm'"};
   {\ar^{} "tm'";"tm2"};
    (0,-14)*{\scs \tpsi(\Ucas)};
    (0,-31)*{\scs \Ucas};
    (-46,20)*{\scs\UdownD\UupD\varrho^{\tpsi}};
    (-14,20)*{\scs\UdownD\hat{\varrho}^{\tpsi}\UupD};
    (14,20)*{\scs\UdownD\varrho^{\tpsi}\UupD};
    (46,20)*{\scs\hat{\varrho}^{\tpsi}\UdownD\UupD};
  \endxy
\end{equation}
The center rectangles commute up to homotopy since they are obtained from the homotopy commutative squares in \eqref{eq_nat_barkappa_cup1}.  The triangles and the left, bottom and right squares in both diagrams above commute since $\varrho^{\tpsi}$ is inverse $\hat{\varrho}^{\tpsi}$ by Proposition~\ref{prop_homotopy_sym}.  The remaining squares commute by Proposition~\ref{prop_psi}.

% --------------------------------------------------------
%
\subsubsection{Naturality of $\hat{\kappa}$ for the cup 2-morphisms}
%
% --------------------------------------------------------

The naturality for one cap is given by the commutative diagram in $Com(\Ucat)$
\begin{equation}
    \xy
   (-35,10)*+{\tpsi(\cal{C})\cal{E}\cal{F}\onen}="tl";
   (0,10)*+{\cal{E}\tpsi(\cal{C})\cal{F}\onen}="tm";
   (35,10)*+{\cal{E}\cal{F}\tpsi(\cal{C})\onen}="tr";
   (-35,-10)*+{\tpsi(\cal{C})\onen}="bl";
   (35,-10)*+{\tpsi(\cal{C})\onen}="br";
   (-55,25)*+{\cal{C}\cal{E}\cal{F}\onen}="tl'";
   (0,25)*+{\cal{E}\cal{C}\cal{F}\onen}="tm'";
   (55,25)*+{\cal{E}\cal{F\cal{C}}\onen}="tr'";
   (-55,-25)*+{\cal{C}\onen}="bl'";
   (55,-25)*+{\cal{C}\onen}="br'";
   {\ar^-{\tpsi(\odelt)\UdownD} "tl";"tm"};
   {\ar^-{\UupD\tpsi(\gam)} "tm";"tr"};
   {\ar^{\tpsi(\Ucas)\Ucupl} "bl";"tl"};
   {\ar^{} "bl";"br"};
   {\ar_{\Ucupl\tpsi(\Ucas)} "br";"tr"};
   {\ar^{\UupD\ogam} "tl'";"tm'"};
   {\ar^{\delt\UdownD} "tm'";"tr'"};
   {\ar^{\Ucas\Ucupl} "bl'";"tl'"};
   {\ar^{} "bl'";"br'"};
   {\ar_{\Ucupl\Ucas} "br'";"tr'"};
   {\ar^{} "tl'";"tl"};
   {\ar^{\varrho^{\tpsi}} "bl'";"bl"};
   {\ar^{} "tr";"tr'"};
   {\ar_{\hat{\varrho}^{\tpsi}} "br";"br'"};
   {\ar^{} "tm";"tm'"};
    (0,-14)*{\scs \tpsi(\Ucas)};
    (0,-31)*{\scs \Ucas};
    (-38,20)*{\scs\varrho^{\tpsi}\UupD\UdownD};
    (-7,20)*{\scs\UupD\hat{\varrho}^{\tpsi}\UdownD};
    (38,20)*{\scs\UupD\UdownD\hat{\varrho}^{\tpsi}};
  \endxy
\end{equation}
where the middle square commutes up to homotopy by applying $\tpsi$ to the homotopy commutative diagram \eqref{eq_cup_nat_square2}.  The left and right squares commute by the naturality of $\varrho^{\tpsi}$ and $\hat{\varrho}^{\tpsi}$.  The top square commutes on the nose since $\varrho^{\tpsi}$ has inverse $\hat{\varrho}^{\tpsi}.  $The bottom two squares commute by Proposition~\ref{prop_psi}.

Naturality for $\hat{\kappa}$ with respect to the other cap is proven similarly.

%\section{Questions}
%
%\paragraph{QUESTION:} \AL{Does $C$ generate the center of $\UA$?}
%
%\paragraph{PROBLEM:} Can we compute $Z(\UA)$?

\addcontentsline{toc}{section}{References}

%=============================================================================
% REFERENCES

%\bibliographystyle{hplain}
%\bibliography{bib_casimir}

%
% ========================================================================

\vspace{0.1in}

\noindent A.B.:  { \sl \small Institut f\"ur Mathematik,
Universit\"at Z\"urich, Winterthurerstr. 190
CH-8057 Z\"urich} \newline \noindent
  {\tt \small email: anna@math.uzh.ch}

\vspace{0.1in}

\noindent M.K.: { \sl \small Department of Mathematics, Columbia University, New
York, NY 10027, USA} \newline \noindent {\tt \small email: khovanov@math.columbia.edu}

\vspace{0.1in}

\noindent A.L.:  { \sl \small Department of Mathematics, Columbia University, New
York, NY 10027, USA} \newline \noindent
  {\tt \small email: lauda@math.columbia.edu}

% ==============================================================================
%
\end{document}